\documentclass[12pt]{amsart}

\usepackage{fullpage}

\usepackage{amsaddr}
\usepackage{amsmath}
\usepackage{amssymb}
\usepackage{amsthm}
\usepackage{enumerate}
\usepackage{bbm}

\allowdisplaybreaks

\usepackage{mathrsfs}
\usepackage{stackengine}
\usepackage{todonotes}

\newtheorem{thm}{Theorem}[subsection]

\newtheorem{lemma}[thm]{Lemma}

\newtheorem{corollary}[thm]{Corollary}

\theoremstyle{definition}
\newtheorem{defn}[thm]{Definition}
\newtheorem{algorithm}[thm]{Algorithm}
\newtheorem{remark}[thm]{Remark}
\newtheorem{example}[thm]{Example}
\newtheorem{conditions}[thm]{Conditions}

\usepackage{lipsum}
\usepackage{ulem}

\usepackage{wrapfig}

\usepackage{color, tikz}
\usetikzlibrary{positioning}
\usetikzlibrary{arrows}

\usetikzlibrary{shapes}

\usepackage{tikz-cd}
\tikzset{commutative diagrams/.cd,arrow style=tikz,diagrams={>=latex'}}

\usepackage[hypertexnames=false]{hyperref}

\newcommand{\sAB}{\mathscr{A}\!\!\mathscr{B}}
\newcommand{\cAB}{\mathcal{A}\mathcal{B}}
\newcommand{\Cyl}{\text{Cyl}}
\newcommand{\I}{\text{I}}
\newcommand{\II}{\text{I}\!\protect\rule{0.015in}{0in}\text{I}}
\newcommand{\III}{\text{I}\!\protect\rule{0.015in}{0in}\text{I}\!\protect\rule{0.015in}{0in}\text{I}}
\newcommand{\base}{\text{base}}
\DeclareMathOperator{\vspan}{span}
\DeclareMathOperator{\sgn}{sgn}

\begin{document}

%% define your title in the usual way
\title{Double-dimer condensation and the PT-DT correspondence}

\author{Helen Jenne, Gautam Webb, \and Benjamin Young}
%\address{CNRS, Institut Denis Poisson, Universit\'e de Tours and Universit\'e d'Orl\'eans, France}

%\author{Gautam Webb}
%\author{Benjamin Young}
%\address{Department of Mathematics, University of Oregon, Eugene, OR, USA}

\email{helenkjenne@gmail.com$^*$, gwebb@uoregon.edu, bjy@uoregon.edu}

\thanks{*Helen Jenne has received funding from the European Research Council (ERC) through the European Union's Horizon 2020 research and innovation programme under the Grant Agreement No 759702.}
%Gautam Webb was partially supported by the National Science Foundation, DMS-2039316. 
%Benjamin Young was partially supported by the Knut and Alice Wallenberg Foundation Grant KAW:2010.0063. This work was supported by a grant from the Simons Foundation (637746, BY)}

\date{\today}

\begin{abstract}
	We resolve an open conjecture from algebraic geometry, which states that two generating functions for plane partition-like objects (the ``box-counting'' formulae for the Calabi-Yau topological vertices in Donaldson-Thomas theory and Pandharipande-Thomas theory) are equal up to a factor of MacMahon's generating function for plane partitions. The main tools in our proof are a Desnanot-Jacobi-type condensation identity, and a novel application of the tripartite double-dimer model of Kenyon-Wilson.
\end{abstract}


\maketitle

\tableofcontents

\section{Introduction}

Donaldson-Thomas (DT) theory and Pandharipande-Thomas (PT) theory are branches of enumerative geometry closely related to mirror symmetry and string theory (for an introduction to these theories, see~\cite[Sections $3\frac{1}{2}$, $4\frac{1}{2}$]{counting-curves}).  In both theories, generating functions arise known as the \emph{combinatorial Calabi-Yau topological vertices}.  These generating functions enumerate seemingly different plane partition-like objects. In this paper, we prove that these generating functions coincide up to a factor of $M(q)$, MacMahon's generating function for plane partitions \cite{macmahon}. Our result, taken together with a substantial body of geometric work, proves a geometric conjecture in the foundational work of Pandharipande-Thomas theory that has been open for over 20 years. 

The generating function from Donaldson-Thomas theory is known as the DT topological vertex. Denoted $V(\mu_1, \mu_2, \mu_3)$, where each $\mu_i$ is a partition, it counts {\em plane partitions asymptotic to $(\mu_1, \mu_2, \mu_3)$} (see Section~\ref{sec:DTboxconfigs}). The PT topological vertex, denoted by $W(\mu_1, \mu_2, \mu_3)$, is a generating function for a certain class of finitely generated $\mathbb{C}[x_1,x_2,x_3]$-modules (see Section~\ref{sec:PTboxconfigs}). 

We prove that 
\begin{thm}\cite[Calabi-Yau case of Conjecture 4]{PT2}
\label{thm:ptdt}
\begin{equation}
	\label{eqn:ptdt}
	V(\mu_1, \mu_2, \mu_3) = M(q) W(\mu_1, \mu_2, \mu_3), 
\end{equation}
where $M(q) = \prod\limits_{i\geq 1}\left(1-q^i\right)^{-i}$. 
\end{thm}

The geometric corollary of this theorem is a proof of Theorem/Conjecture 2 of \cite{PT2}, which, loosely speaking, states that $W(\mu_1, \mu_2, \mu_3)$ computes the local contribution to the geometric Calabi-Yau topological vertex in Pandharipande-Thomas theory. The proof of this corollary combines Theorem~\ref{thm:ptdt} with the analogous result in DT theory \cite{mnop1, mnop2, mpt}, along with \cite[Section 4.1.2]{moop}; it is a consequence of the fact that both DT and PT theory give the same invariants as a third enumerative theory, Gromov-Witten theory.\footnote{In~\cite{mnop1, mnop2, PT2} and in general elsewhere in the geometry literature, all of the formulas have $q$ replaced by $-q$.  The sign is there for geometric reasons which are immaterial for us.} 

To be specific, let $Z_{DT}(\mu_1, \mu_2, \mu_3)$ be the geometric Calabi-Yau topological vertex in Donaldson-Thomas theory, and let $Z_{PT}(\mu_1, \mu_2, \mu_3)$ be the geometric Calabi-Yau topological vertex in Pandharipande-Thomas theory.  We have the following system of equalities, which we have temporarily labelled $G$, $E_{DT}$, $E_{PT}$ and $C$ (G for geometry, E for enumeration, C for combinatorics): 

\begin{center}
\begin{tikzcd}[swap,bend angle=45]
	Z_{DT}(\mu_1, \mu_2, \mu_3) \arrow[r,equal, "G"] \arrow[d, equal,"E_{DT}"] 
	& M(q) Z_{PT}(\mu_1, \mu_2, \mu_3) \arrow[d,equal,"E_{PT}"] \\
	V(\mu_1, \mu_2, \mu_3) \arrow[r,equal,"C"] 
	& M(q)W(\mu_1, \mu_2, \mu_3) 
\end{tikzcd}
\end{center}

In the above, Equation $G$ is the geometric PT-DT correspondence; it says that the two enumerative theories are equivalent at the level of the topological vertex. The technique involves showing that both theories are in fact equivalent to Gromov-Witten theory.  On the DT side, this was done in~\cite{mnop1, mnop2}. For proofs that PT theory is equivalent to Gromov-Witten theory, we refer the reader to a series of papers of Pandharipande and Pixton, culminating in~\cite{pandharipande-pixton}. 

Equation $E_{DT}$ is proven in~\cite{mnop1, mnop2}; it says that in the Calabi-Yau case, one can compute Donaldson-Thomas invariants by enumerating plane partitions asymptotic to $(\mu_1, \mu_2, \mu_3)$.  Proving it, and various generalizations of it, has represented a massive amount of work by many geometers over several decades. 

Equation $E_{PT}$ was conjectured in~\cite[Theorem/Conjecture 2]{PT2}, and proven in the ``two-leg'' case where $\mu_3$ is the empty partition; it says (after cancelling the factor of $M(q)$) that one can compute Pandharipande-Thomas invariants by counting labelled box configurations of shape $(\mu_1, \mu_2, \mu_3)$. 

Equation $C$ is the titular combinatorial PT-DT correspondence; we prove it in this paper.  Taken together with Equations $E_{DT}$ and $G$, this establishes the general case of Equation $E_{PT}$~\cite[Theorem/Conjecture 2]{PT2}. 

We now turn to a discussion of the methods that we use to show that $V(\mu_1, \mu_2, \mu_3)=M(q)W(\mu_1,\mu_2,\mu_3).$  The combinatorics problems which we solve are stated in the geometry literature as ``box-counting'' problems; that is, the objects of interest are plane partition-like.  The following bijections are well-known: 
%The reader is doubtless aware of two ``folklore'' bijections:
\[
	\text{\stackanchor{dimer configurations on}{the honeycomb graph}}
	\leftrightarrow
	\text{plane partitions}
	\leftrightarrow
	\text{\stackanchor{finite-length monomial}{ideals in $\mathbb{C}[x_1,x_2,x_3]$}}
\]
The first one is a 3D version of the correspondence between partitions and their Maya diagrams; it is stated explicitly in Section~\ref{sec:DTtheoryAndDimers}. We use essentially the same correspondence to give a dimer model description of the DT topological vertex $V(\mu_1, \mu_2, \mu_3)$. On the PT side, the correspondences are: 
\[
	\text{\stackanchor{tripartite double-dimer configs.}{on the honeycomb graph}}
	\stackrel{(1)}{\leftrightarrow}
	\text{\stackanchor{labelled box}{configurations}}
	\stackrel{(2)}{\leftrightarrow}
	\text{\stackanchor{$\mathbb{C}[x_1, x_2, x_3]$-modules}{$(M_1 \oplus M_2 \oplus M_3)/\left<(1,1,1)\right>$}}
\]
The correspondence (1) is new, as far as we are aware.  We describe labelled box configurations, and the generating functions for them which arise in PT theory, carefully in Section~\ref{sec:PT}.  Interestingly, though (1) is a purely combinatorial correspondence, it is not bijective---rather, it is a weight-preserving, 1-to-many correspondence. Here $M_1 \subseteq \mathbb{C}[x_1, x_1^{-1}, x_2, x_3]$ is spanned by all monomials $x_1^ix_2^jx_3^k$ where $i \in \mathbb{Z}$ and $(j,k)$ ranges over some fixed partition $\mu_1$, with $M_2, M_3$ defined similarly; the quotient is killing the diagonal of the direct sum. 

The correspondence (2) is incidental to this work and is described in~\cite{PT2}; nor will we need to discuss the structure of the modules in the codomain. We expect that our methods will be relevant in other similar situations (one such situation arises in rank 2 DT theory~\cite{GKY2017}) and we would be eager to learn of other instances in which our techniques may apply. 

We prove Theorem~\ref{thm:ptdt} by observing that both $V/M(q)$ and $W$ are 
%the unique solution of the same recurrence, with the same initial conditions.  The recurrence in question is called the \emph{condensation recurrence}. 
%Finally, let $X=X(\mu_1, \mu_2, \mu_3)$ be a power series in $q$, depending on three partitions $\mu_1, \mu_2, \mu_3$, which is symmetric with respect to cyclic permutation of these partitions.  We will show that $V$ and $W$ are 
solutions $X$ to the following functional equation: 
\begin{equation}\small
	\label{eqn:vertex_condensation}
	q^K
	X(\mu_1, \mu_2, \mu_3)
	X(\mu_1^{rc}, \mu_2^{rc}, \mu_3)
	=
	q^K
	X(\mu_1^{rc}, \mu_2, \mu_3)
	X(\mu_1, \mu_2^{rc}, \mu_3)
	+
	X(\mu_1^{r}, \mu_2^{c}, \mu_3)
	X(\mu_1^{c}, \mu_2^{r}, \mu_3).
\end{equation}
This recurrence is called the \emph{condensation recurrence}. We postpone the definitions of $\mu_i^{r}$, $\mu_i^{c}$, and $\mu_i^{rc}$ to Section~\ref{sec:definitions}. Here, $K:=1+(\mu_1)_{d(\mu_1)}-d(\mu_1)+(\mu_2')_{d(\mu_2)}-d(\mu_2)$, where $d(\lambda)$ is the diagonal of $\lambda$. This constant is discussed further in Section~\ref{sec:weights}. 

The partitions $\mu_i^r$, $\mu_i^c$, and $\mu_i^{rc}$ are all of smaller length than $\mu_i$, and none of the topological vertex terms are equal to zero, so we can divide both sides of the condensation recurrence by $q^K X(\mu_1^{rc}, \mu_2^{rc}, \mu_3)$. 
%and obtain a recursive characterization of $V(\mu_1, \mu_2, \mu_3)$. 
%Note also that $V(\mu_1, \mu_2, \mu_3) = V(\mu_2, \mu_3, \mu_1)$ by symmetry - so we can say that $V(\mu_1, \mu_2, \mu_3)$ is the \emph{unique} power series which satisfies the condensation recurrence, where we take the base cases to be the (known) values of $V(\mu_1, \mu_2, \varnothing)$ for all partitions $\mu_1, \mu_2$. 
%we postpone its definition to Section~\ref{sec:definitions}, after we have made the required definitions. 
Viewed as a recurrence in $\mu_1$ and $\mu_2$, the resulting equation uniquely characterizes $V/M(q)$ and $W$. The base case is when one of the partitions $\mu_i$ is equal to $\emptyset$; equation~\eqref{eqn:ptdt} is known to hold in this situation~\cite{PT2}. 

When recast in terms of the dimer model, $V/M(q)$ is easily seen to satisfy equation~\eqref{eqn:vertex_condensation} by Kuo's \emph{graphical condensation}~\cite{kuo}; this is essentially the content of Section~\ref{sec:DT}. 

Showing that $W$ satisfies equation~\eqref{eqn:vertex_condensation} is considerably more intricate, but once we translate to the double-dimer model, the bulk of the work was done elsewhere, in work of Jenne~\cite{jenne}.  Essentially,~\cite{jenne} evaluates a certain determinant by the classical Desnanot-Jacobi identity, and then interprets all six terms in the identity in terms of $W$.

\section{Definitions}
\label{sec:definitions}

Fix three partitions $\mu=(\mu_1, \mu_2, \mu_3)$.  For this paper, we identify $\mu_i$ with the coordinates of the boxes of its Young diagram, with the corner of the diagram located at $(0,0)$ and the rows of the diagram extending in the horizontal direction.  Define the following subsets of $\mathbb{Z}^3$, thought of as sets of boxes: 
\begin{align*}
	\Cyl_1 &= \{(x,u,v) \in \mathbb{Z}^3 \;|\; (u,v) \in \mu_1\}, \\
	\Cyl_2 &= \{(v,y,u) \in \mathbb{Z}^3 \;|\; (u,v) \in \mu_2\}, \\
	\Cyl_3 &= \{(u,v,z) \in \mathbb{Z}^3 \;|\; (u,v) \in \mu_3\}. 
\end{align*}
Moreover, let $\mathbb{Z}^3_{\geq 0}$ denote the integer points in the first octant (including the coordinate planes and axes). Let $\Cyl_i^{+} = \Cyl_i \cap \mathbb{Z}^3_{\geq 0}$ and $\Cyl_i^{-} = \Cyl_i \setminus \mathbb{Z}^3_{\geq 0}$. 
Finally, let 
\begin{align*}
	&&\II_{\bar{1}} &= \Cyl_2 \cap \Cyl_3\setminus\Cyl_1, \\
	\I^- &= \Cyl_1^- \cup \Cyl_2^- \cup \Cyl_3^-, &
	\II_{\bar{2}} &= \Cyl_3 \cap \Cyl_1\setminus\Cyl_2, & \III &= \Cyl_1 \cap \Cyl_2 \cap \Cyl_3, \\
	&&\II_{\bar{3}} &= \Cyl_1 \cap \Cyl_2\setminus\Cyl_3, \\
	&&\II &= \II_{\bar{1}} \cup \II_{\bar{2}} \cup \II_{\bar{3}}, 
\end{align*}
and let \[\I^+ = \left(\Cyl_1^+ \cup \Cyl_2^+ \cup \Cyl_3^+\right)\setminus\left(\II\cup\III\right).\] 
When we wish to emphasize the dependence of $\Cyl_1$, $\Cyl_2$, $\Cyl_3$, $\I^-$, $\II$, $\III$, or $\I^+$ on $\mu$, we will write $\Cyl_1(\mu)$, $\Cyl_2(\mu)$, $\Cyl_3(\mu)$, $\I^-(\mu)$, $\II(\mu)$, $\III(\mu)$, or $\I^+(\mu)$, respectively. Throughout this paper, $M$ will denote the quantity $\max\{(\mu_1)_1, \ell(\mu_1), (\mu_2)_1, \ell(\mu_2), (\mu_3)_1, \ell(\mu_3)\}$. 

We will need the following standard notions of Maya diagrams. 

\begin{defn}
If $\lambda = (\lambda_1, \lambda_2, \ldots, \lambda_k)$ is a partition with $k$ parts, define $\lambda_t=0$ for $t>k$. The \emph{Maya diagram of $\lambda$} is the set $\{\lambda_t - t + \frac{1}{2}\} \subseteq \mathbb{Z}+\frac{1}{2}$.  
\end{defn}

We frequently associate a partition with its Maya diagram by drawing a Maya diagram as a doubly infinite sequence of beads and holes, indexed by $\mathbb{Z}+\frac{1}{2}$, with the beads representing elements of the above set.  For instance, the Maya diagrams of the empty partition and of the partition $\lambda = (4,2,1)$ are the sets $\{ -\frac{1}{2}, -\frac{3}{2}, \ldots \}$ and $\{ \frac{7}{2}, \frac{1}{2}, -\frac{3}{2},-\frac{7}{2}, -\frac{9}{2},  \ldots \}$, respectively, which are drawn as 
\[
\cdots  \circ \circ \circ | \bullet \bullet \bullet \cdots
\qquad
\text{and}
\qquad
\cdots \circ \circ \circ \bullet \circ \circ \bullet | \circ \bullet \circ \bullet \bullet \bullet \cdots. 
%\cdots \bullet \bullet \bullet \circ \bullet \circ | \bullet \circ \circ \bullet \circ \circ \circ \cdots
\]
When convenient, we simply mark the location of 0 with a vertical line, rather than labelling the beads with elements of $\mathbb{Z}+\frac{1}{2}$. 

\begin{defn}
Conversely, if $S$ is a subset of $\mathbb{Z}+\frac{1}{2}$,  define $S^+ = \left\{x \in S \mid x > 0\right\}$ and $S^- = \left\{x \in \mathbb{Z}+\frac{1}{2} \setminus S \mid x < 0\right\}$.  If both $S^+$ and $S^-$ are finite, then define the \emph{charge} of $S$, $c(S)$, to be $|S^+|-|S^-|$; then it is easy to check that the set $\{s - c(S) \mid s \in S\}$ is the Maya diagram of some partition $\lambda$; we say that $S$ itself is the \emph{charge $c(S)$ Maya diagram of $\lambda$}. 
\end{defn}

\begin{defn}
If $\lambda$ is a partition with Maya diagram $S$, let $\lambda^r$ (resp.~$\lambda^c$) be the partition associated to the charge $-1$ (resp.~$1$) Maya diagram $S \setminus \{\min S^+\}$ (resp.~$S \cup \{ \max S^-\}$). Let $\lambda^{rc}$ be the partition associated to the Maya diagram $(S\setminus\{ \min{S^+}\})\cup\{\max{S^-}\}$. 
\end{defn}

\begin{figure}[h]
\begin{center}
\begin{tikzpicture} [ hexa/.style= {shape=regular polygon,
                                   regular polygon sides=6,
                                   minimum size=1cm, draw,
                                   inner sep=0,anchor=south,
                                   fill=white}]
\node[hexa] (hex1) at (0, 0) {};
\foreach \x in {1, 3, 5}
  \fill[color = white, draw = black] (hex1.corner \x) circle[radius=2pt];
  \foreach \x in {2, 4, 6}
  \fill[color = black, draw = black] (hex1.corner \x) circle[radius=2pt];

\node[hexa] (hex2) at (0, {-sin(60)} ) {};
\foreach \x in {1, 3, 5}
  \fill[color =white, draw = black] (hex2.corner \x) circle[radius=2pt];
  \foreach \x in {2, 4, 6}
  \fill[color = black, draw = black] (hex2.corner \x) circle[radius=2pt];

\node[hexa] (hex3) at (0, {-2*sin(60)} ) {};
\foreach \x in {1, 3, 5}
  \fill[color = white, draw = black] (hex3.corner \x) circle[radius=2pt];
  \foreach \x in {2, 4, 6}
  \fill[color = black, draw = black] (hex3.corner \x) circle[radius=2pt];

\node[hexa] (hex3) at (0, {-3*sin(60)} ) {};
\foreach \x in {1, 3, 5}
  \fill[color = white, draw = black] (hex3.corner \x) circle[radius=2pt];
  \foreach \x in {2, 4, 6}
  \fill[color = black, draw = black] (hex3.corner \x) circle[radius=2pt];

\node[hexa] (hex4) at (-.76, {-sin(60)/2} ) {};
\foreach \x in {1, 3, 5}
  \fill[color = white, draw = black] (hex4.corner \x) circle[radius=2pt];
  \foreach \x in {2, 4, 6}
  \fill[color = black, draw = black] (hex4.corner \x) circle[radius=2pt];

\node[hexa] (hex5) at (-.76, { -sin(60) -sin(60)/2} ) {};
\foreach \x in {1, 3, 5}
  \fill[color = white, draw = black] (hex5.corner \x) circle[radius=2pt];
  \foreach \x in {2, 4, 6}
  \fill[color = black, draw = black] (hex5.corner \x) circle[radius=2pt];

\node[hexa] (hex5) at (-.76, { -2*sin(60) -sin(60)/2} ) {};
\foreach \x in {1, 3, 5}
  \fill[color = white, draw = black] (hex5.corner \x) circle[radius=2pt];
  \foreach \x in {2, 4, 6}
  \fill[color = black, draw = black] (hex5.corner \x) circle[radius=2pt];

%%%%THIRD COLUMN OF HEXAGONS
\node[hexa] (hex6) at (.76, {-sin(60)/2}) {};
\foreach \x in {1, 3, 5}
  \fill[color = white, draw = black] (hex6.corner \x) circle[radius=2pt];
  \foreach \x in {2, 4, 6}
  \fill[color = black, draw = black] (hex6.corner \x) circle[radius=2pt];

\node[hexa] (hex7) at (.76, { -sin(60) -sin(60)/2} ) {};
\foreach \x in {1, 3, 5}
  \fill[color = white, draw = black] (hex7.corner \x) circle[radius=2pt];
  \foreach \x in {2, 4, 6}
  \fill[color = black, draw = black] (hex7.corner \x) circle[radius=2pt];

\node[hexa] (hex8) at (.76, {-sin(60)/2 + sin(60)}) {};
\foreach \x in {1, 3, 5}
  \fill[color = white, draw = black] (hex8.corner \x) circle[radius=2pt];
  \foreach \x in {2, 4, 6}
  \fill[color = black, draw = black] (hex8.corner \x) circle[radius=2pt];

\node[hexa] (hex9) at (.76, {-sin(60)/2 -2*sin(60)}) {};
\foreach \x in {1, 3, 5}
  \fill[color = white, draw = black] (hex9.corner \x) circle[radius=2pt];
  \foreach \x in {2, 4, 6}
  \fill[color = black, draw = black] (hex9.corner \x) circle[radius=2pt];

\node[hexa] (hex9) at (.76, {-sin(60)/2 -3*sin(60)}) {};
\foreach \x in {1, 3, 5}
  \fill[color = white, draw = black] (hex9.corner \x) circle[radius=2pt];
  \foreach \x in {2, 4, 6}
  \fill[color = black, draw = black] (hex9.corner \x) circle[radius=2pt];

%%%%FOURTH COLUMN OF HEXAGONS
\node[hexa] (hex10) at (2*.76, {-sin(60) + sin(60)}) {};
\foreach \x in {1, 3, 5}
  \fill[color = white, draw = black] (hex10.corner \x) circle[radius=2pt];
  \foreach \x in {2, 4, 6}
  \fill[color = black, draw = black] (hex10.corner \x) circle[radius=2pt];

\node[hexa] (hex12) at (2*.76, {-sin(60)}) {};
\foreach \x in {1, 3, 5}
  \fill[color = white, draw = black] (hex12.corner \x) circle[radius=2pt];
  \foreach \x in {2, 4, 6}
  \fill[color = black, draw = black] (hex12.corner \x) circle[radius=2pt];

\node[hexa] (hex13) at (2*.76, {-2*sin(60)}) {};
\foreach \x in {1, 3, 5}
  \fill[color = white, draw = black] (hex13.corner \x) circle[radius=2pt];
  \foreach \x in {2, 4, 6}
  \fill[color = black, draw = black] (hex13.corner \x) circle[radius=2pt];

\node[hexa] (hex13) at (2*.76, {-3*sin(60)}) {};
\foreach \x in {1, 3, 5}
  \fill[color = white, draw = black] (hex13.corner \x) circle[radius=2pt];
  \foreach \x in {2, 4, 6}
  \fill[color = black, draw = black] (hex13.corner \x) circle[radius=2pt];

%%%%FIFTH COLUMN OF HEXAGONS
\node[hexa] (hex16) at (3*.76, {-3*sin(60) + sin(60)/2}) {};
\foreach \x in {1, 3, 5}
  \fill[color = white, draw = black] (hex16.corner \x) circle[radius=2pt];
  \foreach \x in {2, 4, 6}
  \fill[color = black, draw = black] (hex16.corner \x) circle[radius=2pt];

\node[hexa] (hex14) at (3*.76, {-2*sin(60) + sin(60)/2}) {};
\foreach \x in {1, 3, 5}
  \fill[color = white, draw = black] (hex14.corner \x) circle[radius=2pt];
  \foreach \x in {2, 4, 6}
  \fill[color = black, draw = black] (hex14.corner \x) circle[radius=2pt];

\node[hexa] (hex15) at (3*.76, {-sin(60) + sin(60)/2}) {};
\foreach \x in {1, 3, 5}
  \fill[color = white, draw = black] (hex15.corner \x) circle[radius=2pt];
  \foreach \x in {2, 4, 6}
  \fill[color = black, draw = black] (hex15.corner \x) circle[radius=2pt];

	\node[left] at (-1.2, -2.85) {{\color{blue} \Large{$1$}}};
	\node[right] at (2.75, -2.85)  {{\color{blue} \Large{$2$}}};
	\node[left] at (1.05, 2)  {{\color{blue} \Large{$3$}}};
                
	%\node at (-.766, 0) {$*$};
	%\node at (2.27, 0) {$\dagger$};
	%\node at (0.766, -2.65) {$\triangle$};
	
	\node[left] at (0.55,1.35) {$\frac{1}{2}$};
	\node[right] at (1,1.35) {$-\frac{1}{2}$};
     
	\node[left] at (-.25,0.9) {$\frac{3}{2}$};
	\node[right] at (1.75,0.9) {$-\frac{3}{2}$};
  
	\node[left] at (-1.3,-.2) {$\vdots$};
	\node[right] at (2.8,-.2) {$\vdots$};
	
	\node[left] at (-1.25,-.85) {$-\frac{3}{2}$};
	\node[right] at (2.75,-.85) {$\frac{3}{2}$};

         \node[left] at (-1.25,-1.7) {$-\frac{1}{2}$};
	\node[right] at (2.75,-1.7) {$\frac{1}{2}$};
	
	%\node[right] at (2.8,0) {$\frac{1}{2}$};
	%\node[left] at (-1,-2.25) {$\frac{1}{2}$};
	
	\node[left] at (-.95,-2.2) {$\frac{1}{2}$};
	\node[right] at (2.42,-2.27) {$-\frac{1}{2}$};
	
	\node[left] at (-.25,-2.65) {$\frac{3}{2}$};
	\node[right] at (1.67,-2.72) {$-\frac{3}{2}$};
     
%\filldraw[fill=black, draw=black] (0.76, -.86) circle (0.05cm); 
\draw[dashed]  (0.76, -.86)-- (0.76, -3.75);
\draw[dashed]  (0.76, -.86)-- (-1.76, -.86 + 1.4);
\draw[dashed]  (0.76, -.86)-- (.76*2 + 1.76, -.86 + 1.4);
\end{tikzpicture} \hspace{1cm}
\begin{tikzpicture} [ hexa/.style= {shape=regular polygon,
                                   regular polygon sides=6,
                                   minimum size=1cm, draw,
                                   inner sep=0,anchor=south,
                                   fill=white}]
\node[hexa] (hex1) at (0, 0) {};
\foreach \x in {1, 3, 5}
  \fill[color = white, draw = black] (hex1.corner \x) circle[radius=2pt];
  \foreach \x in {2, 4, 6}
  \fill[color = black, draw = black] (hex1.corner \x) circle[radius=2pt];

\node[hexa] (hex2) at (0, {-sin(60)} ) {};
\foreach \x in {1, 3, 5}
  \fill[color =white, draw = black] (hex2.corner \x) circle[radius=2pt];
  \foreach \x in {2, 4, 6}
  \fill[color = black, draw = black] (hex2.corner \x) circle[radius=2pt];

\node[hexa] (hex3) at (0, {-2*sin(60)} ) {};
\foreach \x in {1, 3, 5}
  \fill[color = white, draw = black] (hex3.corner \x) circle[radius=2pt];
  \foreach \x in {2, 4, 6}
  \fill[color = black, draw = black] (hex3.corner \x) circle[radius=2pt];

\node[hexa] (hex3) at (0, {-3*sin(60)} ) {};
\foreach \x in {1, 3, 5}
  \fill[color = white, draw = black] (hex3.corner \x) circle[radius=2pt];
  \foreach \x in {2, 4, 6}
  \fill[color = black, draw = black] (hex3.corner \x) circle[radius=2pt];

\node[hexa] (hex4) at (-.76, {-sin(60)/2} ) {};
\foreach \x in {1, 3, 5}
  \fill[color = white, draw = black] (hex4.corner \x) circle[radius=2pt];
  \foreach \x in {2, 4, 6}
  \fill[color = black, draw = black] (hex4.corner \x) circle[radius=2pt];

\node[hexa] (hex5) at (-.76, { -sin(60) -sin(60)/2} ) {};
\foreach \x in {1, 3, 5}
  \fill[color = white, draw = black] (hex5.corner \x) circle[radius=2pt];
  \foreach \x in {2, 4, 6}
  \fill[color = black, draw = black] (hex5.corner \x) circle[radius=2pt];

\node[hexa] (hex5) at (-.76, { -2*sin(60) -sin(60)/2} ) {};
\foreach \x in {1, 3, 5}
  \fill[color = white, draw = black] (hex5.corner \x) circle[radius=2pt];
  \foreach \x in {2, 4, 6}
  \fill[color = black, draw = black] (hex5.corner \x) circle[radius=2pt];

%%%%THIRD COLUMN OF HEXAGONS
\node[hexa] (hex6) at (.76, {-sin(60)/2}) {};
\foreach \x in {1, 3, 5}
  \fill[color = white, draw = black] (hex6.corner \x) circle[radius=2pt];
  \foreach \x in {2, 4, 6}
  \fill[color = black, draw = black] (hex6.corner \x) circle[radius=2pt];

\node[hexa] (hex7) at (.76, { -sin(60) -sin(60)/2} ) {};
\foreach \x in {1, 3, 5}
  \fill[color = white, draw = black] (hex7.corner \x) circle[radius=2pt];
  \foreach \x in {2, 4, 6}
  \fill[color = black, draw = black] (hex7.corner \x) circle[radius=2pt];

\node[hexa] (hex8) at (.76, {-sin(60)/2 + sin(60)}) {};
\foreach \x in {1, 3, 5}
  \fill[color = white, draw = black] (hex8.corner \x) circle[radius=2pt];
  \foreach \x in {2, 4, 6}
  \fill[color = black, draw = black] (hex8.corner \x) circle[radius=2pt];

\node[hexa] (hex9) at (.76, {-sin(60)/2 -2*sin(60)}) {};
\foreach \x in {1, 3, 5}
  \fill[color = white, draw = black] (hex9.corner \x) circle[radius=2pt];
  \foreach \x in {2, 4, 6}
  \fill[color = black, draw = black] (hex9.corner \x) circle[radius=2pt];

\node[hexa] (hex9) at (.76, {-sin(60)/2 -3*sin(60)}) {};
\foreach \x in {1, 3, 5}
  \fill[color = white, draw = black] (hex9.corner \x) circle[radius=2pt];
  \foreach \x in {2, 4, 6}
  \fill[color = black, draw = black] (hex9.corner \x) circle[radius=2pt];

%%%%FOURTH COLUMN OF HEXAGONS
\node[hexa] (hex10) at (2*.76, {-sin(60) + sin(60)}) {};
\foreach \x in {1, 3, 5}
  \fill[color = white, draw = black] (hex10.corner \x) circle[radius=2pt];
  \foreach \x in {2, 4, 6}
  \fill[color = black, draw = black] (hex10.corner \x) circle[radius=2pt];

\node[hexa] (hex12) at (2*.76, {-sin(60)}) {};
\foreach \x in {1, 3, 5}
  \fill[color = white, draw = black] (hex12.corner \x) circle[radius=2pt];
  \foreach \x in {2, 4, 6}
  \fill[color = black, draw = black] (hex12.corner \x) circle[radius=2pt];

\node[hexa] (hex13) at (2*.76, {-2*sin(60)}) {};
\foreach \x in {1, 3, 5}
  \fill[color = white, draw = black] (hex13.corner \x) circle[radius=2pt];
  \foreach \x in {2, 4, 6}
  \fill[color = black, draw = black] (hex13.corner \x) circle[radius=2pt];

\node[hexa] (hex13) at (2*.76, {-3*sin(60)}) {};
\foreach \x in {1, 3, 5}
  \fill[color = white, draw = black] (hex13.corner \x) circle[radius=2pt];
  \foreach \x in {2, 4, 6}
  \fill[color = black, draw = black] (hex13.corner \x) circle[radius=2pt];

%%%%FIFTH COLUMN OF HEXAGONS
\node[hexa] (hex16) at (3*.76, {-3*sin(60) + sin(60)/2}) {};
\foreach \x in {1, 3, 5}
  \fill[color = white, draw = black] (hex16.corner \x) circle[radius=2pt];
  \foreach \x in {2, 4, 6}
  \fill[color = black, draw = black] (hex16.corner \x) circle[radius=2pt];

\node[hexa] (hex14) at (3*.76, {-2*sin(60) + sin(60)/2}) {};
\foreach \x in {1, 3, 5}
  \fill[color = white, draw = black] (hex14.corner \x) circle[radius=2pt];
  \foreach \x in {2, 4, 6}
  \fill[color = black, draw = black] (hex14.corner \x) circle[radius=2pt];

\node[hexa] (hex15) at (3*.76, {-sin(60) + sin(60)/2}) {};
\foreach \x in {1, 3, 5}
  \fill[color = white, draw = black] (hex15.corner \x) circle[radius=2pt];
  \foreach \x in {2, 4, 6}
  \fill[color = black, draw = black] (hex15.corner \x) circle[radius=2pt];
  
	\node[left] at (-1, 1.25) {{\color{blue} \Large{$2$}}};
	\node[right] at (2.65, 1.25)  {{\color{blue} \Large{$1$}}};
	\node[left] at (1.10, -3.5)  {{\color{blue} \Large{$3$}}};
                
	%\node at (-.766, 0) {$*$};
	%\node at (2.27, 0) {$\dagger$};
	%\node at (0.766, -2.65) {$\triangle$};

	\node[left] at (-.2,1) {$\frac{3}{2}$};
	\node[right] at (1.5,1.05) {$-\frac{3}{2}$};
	
	\node[left] at (-1,0.5) {$\frac{1}{2}$};
	\node[right] at (2.3,0.65) {$-\frac{1}{2}$};
	
	\node[left] at (-1.25,0) {$-\frac{1}{2}$};
	\node[right] at (2.8,0) {$\frac{1}{2}$};
	
	\node[left] at (-1.25,-.8) {$-\frac{3}{2}$};
	\node[right] at (2.8,-.8) {$\frac{3}{2}$};

	\node[left] at (-1.3,-1.3) {$\vdots$};
	\node[right] at (2.85,-1.3) {$\vdots$};
	
	%\node[left] at (-1,-2.25) {$\frac{1}{2}$};
	\node at (-.5,-2.7) {$\frac{3}{2}$};
	\node at (1.9,-2.73) {$-\frac{3}{2}$};

	\node[left] at (0.5,-3.15) {$\frac{1}{2}$};
	\node[right] at (0.75,-3.15) {$-\frac{1}{2}$};

%\filldraw[fill=black, draw=black] (0.76, -.86) circle (0.05cm); 
\draw[dashed]  (0.76, -.86)-- (0.76, 1.75);
\draw[dashed]  (0.76, -.86)-- (-1.76, -1.86 - .4);
\draw[dashed]  (0.76, -.86)-- (.76*2 + 1.76, -1.86 - .4);
\end{tikzpicture}
\end{center}
\caption{The graph $H(3)$. Left: The division into sectors for DT. Right: The division into sectors for PT.}
\label{fig:sectors}
\end{figure}
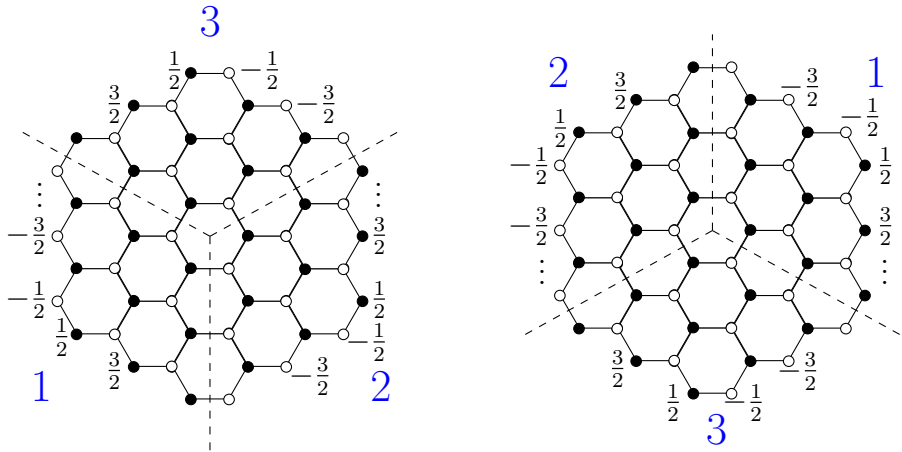

In both DT and PT, it will be convenient to divide the $N \times N \times N$ honeycomb graph $H(N)$ into three sectors and label some of the vertices on the outer face, as shown in Figure~\ref{fig:sectors} for $H(3)$. We remark that the divisions into sectors make sense as $N \to \infty$.  The reason for this choice of labels is that we will need to specify these particular vertices, both in DT and PT, based on the Maya diagrams of $\mu_1$, $\mu_2$, $\mu_3$, and various other partitions. Furthermore, if a vertex $u$ on the outer face in sector $i$ is labelled by a positive (resp.~negative) number, we will say that $u$ is in sector $i^+$ (resp.~sector $i^-$). 

We will weight the edges of $H(N)$ following Kuo \cite{kuo}. 
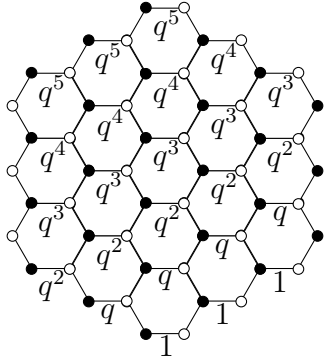
\begin{figure}[htb]
\centering
\begin{tikzpicture} [ hexa/.style= {shape=regular polygon,
                                   regular polygon sides=6,
                                   minimum size=1cm, draw,
                                   inner sep=0,anchor=south,
                                   fill=white}]
\node[hexa] (hex1) at (0, 0) {};
\foreach \x in {1, 3, 5}
  \fill[color = white, draw = black] (hex1.corner \x) circle[radius=2pt];
  \foreach \x in {2, 4, 6}
  \fill[color = black, draw = black] (hex1.corner \x) circle[radius=2pt];

\node[hexa] (hex2) at (0, {-sin(60)} ) {};
\foreach \x in {1, 3, 5}
  \fill[color =white, draw = black] (hex2.corner \x) circle[radius=2pt];
  \foreach \x in {2, 4, 6}
  \fill[color = black, draw = black] (hex2.corner \x) circle[radius=2pt];

\node[hexa] (hex3) at (0, {-2*sin(60)} ) {};
\foreach \x in {1, 3, 5}
  \fill[color = white, draw = black] (hex3.corner \x) circle[radius=2pt];
  \foreach \x in {2, 4, 6}
  \fill[color = black, draw = black] (hex3.corner \x) circle[radius=2pt];

\node[hexa] (hex3) at (0, {-3*sin(60)} ) {};
\foreach \x in {1, 3, 5}
  \fill[color = white, draw = black] (hex3.corner \x) circle[radius=2pt];
  \foreach \x in {2, 4, 6}
  \fill[color = black, draw = black] (hex3.corner \x) circle[radius=2pt];

\node[hexa] (hex4) at (-.76, {-sin(60)/2} ) {};
\foreach \x in {1, 3, 5}
  \fill[color = white, draw = black] (hex4.corner \x) circle[radius=2pt];
  \foreach \x in {2, 4, 6}
  \fill[color = black, draw = black] (hex4.corner \x) circle[radius=2pt];

\node[hexa] (hex5) at (-.76, { -sin(60) -sin(60)/2} ) {};
\foreach \x in {1, 3, 5}
  \fill[color = white, draw = black] (hex5.corner \x) circle[radius=2pt];
  \foreach \x in {2, 4, 6}
  \fill[color = black, draw = black] (hex5.corner \x) circle[radius=2pt];

\node[hexa] (hex5) at (-.76, { -2*sin(60) -sin(60)/2} ) {};
\foreach \x in {1, 3, 5}
  \fill[color = white, draw = black] (hex5.corner \x) circle[radius=2pt];
  \foreach \x in {2, 4, 6}
  \fill[color = black, draw = black] (hex5.corner \x) circle[radius=2pt];

%%%%THIRD COLUMN OF HEXAGONS
\node[hexa] (hex6) at (.76, {-sin(60)/2}) {};
\foreach \x in {1, 3, 5}
  \fill[color = white, draw = black] (hex6.corner \x) circle[radius=2pt];
  \foreach \x in {2, 4, 6}
  \fill[color = black, draw = black] (hex6.corner \x) circle[radius=2pt];

\node[hexa] (hex7) at (.76, { -sin(60) -sin(60)/2} ) {};
\foreach \x in {1, 3, 5}
  \fill[color = white, draw = black] (hex7.corner \x) circle[radius=2pt];
  \foreach \x in {2, 4, 6}
  \fill[color = black, draw = black] (hex7.corner \x) circle[radius=2pt];

\node[hexa] (hex8) at (.76, {-sin(60)/2 + sin(60)}) {};
\foreach \x in {1, 3, 5}
  \fill[color = white, draw = black] (hex8.corner \x) circle[radius=2pt];
  \foreach \x in {2, 4, 6}
  \fill[color = black, draw = black] (hex8.corner \x) circle[radius=2pt];

\node[hexa] (hex9) at (.76, {-sin(60)/2 -2*sin(60)}) {};
\foreach \x in {1, 3, 5}
  \fill[color = white, draw = black] (hex9.corner \x) circle[radius=2pt];
  \foreach \x in {2, 4, 6}
  \fill[color = black, draw = black] (hex9.corner \x) circle[radius=2pt];

\node[hexa] (hex9) at (.76, {-sin(60)/2 -3*sin(60)}) {};
\foreach \x in {1, 3, 5}
  \fill[color = white, draw = black] (hex9.corner \x) circle[radius=2pt];
  \foreach \x in {2, 4, 6}
  \fill[color = black, draw = black] (hex9.corner \x) circle[radius=2pt];

%%%%FOURTH COLUMN OF HEXAGONS
\node[hexa] (hex10) at (2*.76, {-sin(60) + sin(60)}) {};
\foreach \x in {1, 3, 5}
  \fill[color = white, draw = black] (hex10.corner \x) circle[radius=2pt];
  \foreach \x in {2, 4, 6}
  \fill[color = black, draw = black] (hex10.corner \x) circle[radius=2pt];

\node[hexa] (hex12) at (2*.76, {-sin(60)}) {};
\foreach \x in {1, 3, 5}
  \fill[color = white, draw = black] (hex12.corner \x) circle[radius=2pt];
  \foreach \x in {2, 4, 6}
  \fill[color = black, draw = black] (hex12.corner \x) circle[radius=2pt];

\node[hexa] (hex13) at (2*.76, {-2*sin(60)}) {};
\foreach \x in {1, 3, 5}
  \fill[color = white, draw = black] (hex13.corner \x) circle[radius=2pt];
  \foreach \x in {2, 4, 6}
  \fill[color = black, draw = black] (hex13.corner \x) circle[radius=2pt];

\node[hexa] (hex13) at (2*.76, {-3*sin(60)}) {};
\foreach \x in {1, 3, 5}
  \fill[color = white, draw = black] (hex13.corner \x) circle[radius=2pt];
  \foreach \x in {2, 4, 6}
  \fill[color = black, draw = black] (hex13.corner \x) circle[radius=2pt];

%%%%FIFTH COLUMN OF HEXAGONS
\node[hexa] (hex16) at (3*.76, {-3*sin(60) + sin(60)/2}) {};
\foreach \x in {1, 3, 5}
  \fill[color = white, draw = black] (hex16.corner \x) circle[radius=2pt];
  \foreach \x in {2, 4, 6}
  \fill[color = black, draw = black] (hex16.corner \x) circle[radius=2pt];

\node[hexa] (hex14) at (3*.76, {-2*sin(60) + sin(60)/2}) {};
\foreach \x in {1, 3, 5}
  \fill[color = white, draw = black] (hex14.corner \x) circle[radius=2pt];
  \foreach \x in {2, 4, 6}
  \fill[color = black, draw = black] (hex14.corner \x) circle[radius=2pt];

\node[hexa] (hex15) at (3*.76, {-sin(60) + sin(60)/2}) {};
\foreach \x in {1, 3, 5}
  \fill[color = white, draw = black] (hex15.corner \x) circle[radius=2pt];
  \foreach \x in {2, 4, 6}
  \fill[color = black, draw = black] (hex15.corner \x) circle[radius=2pt];

	\node[right] at (-1.07,0.2) {$q^5$};
	\node[right] at (-.32,0.63) {$q^5$};
	\node[right] at (0.46,1.07) {$q^5$};
	
	\node[right] at (-1.05,-0.67) {$q^4$};
	\node[right] at (-.29,-.25) {$q^4$};
	\node[right] at (0.455,0.19) {$q^4$};
	\node[right] at (1.22,0.63) {$q^4$};

	\node[right] at (-1.07,-1.53) {$q^3$};
	\node[right] at (-.31,-1.09) {$q^3$};
	\node[right] at (0.455,-0.66) {$q^3$};
	\node[right] at (1.22,-.23) {$q^3$};
      	\node[right] at (1.98,0.2) {$q^3$};

	\node[right] at (-1.07,-2.39) {$q^2$};
	\node[right] at (-.3,-1.97) {$q^2$};
	\node[right] at (0.465,-1.53) {$q^2$};
	\node[right] at (1.22,-1.1) {$q^2$};
	\node[right] at (1.97,-0.67) {$q^2$};

	\node[right] at (-0.24,-2.76) {$q$};
	\node[right] at (0.52,-2.32) {$q$};
	\node[right] at (1.28,-1.9) {$q$};
	\node[right] at (2.05,-1.46) {$q$};

	\node[right] at (0.53,-3.19) {$1$};
	\node[right] at (1.28,-2.76) {$1$};
	\node[right] at (2.04,-2.33) {$1$};
\end{tikzpicture}
\caption{The graph $H(3)$ with edges weighted as specified in Definition~\ref{def:kuoweighting}.}
\label{fig:kuoweighting}
\end{figure}

\begin{defn}\cite[Section 6]{kuo}
\label{def:kuoweighting}
Weight the edges of $H(N)$ so that the non-horizontal edges have weight 1 and the horizontal edges are weighted by powers of $q$. Specifically, the $N$ horizontal edges along the bottom right diagonal have weight 1. On the next diagonal, the horizontal edges have weight $q$. In general, the weight of the edges on a diagonal is $q$ times the weight of the edges on the previous diagonal. This is illustrated in Figure~\ref{fig:kuoweighting}. 
\end{defn}

\section{DT}
\label{sec:DT}

\subsection{DT box configurations}
\label{sec:DTboxconfigs}

We say that a \emph{plane partition asymptotic to $(\mu_1, \mu_2, \mu_3)$} is an order ideal under the product order in $\mathbb{Z}^3_{\geq 0}$ which contains $\I^+ \cup \II \cup \III$, together with only finitely many other points in $\mathbb{Z}_{\geq 0}^3$.  We let $P(\mu_1,\mu_2,\mu_3)$ denote the set of plane partitions asymptotic to $(\mu_1, \mu_2, \mu_3)$.  

If any of $\mu_1, \mu_2, \mu_3$ are nonzero, then every $\pi \in P(\mu_1, \mu_2, \mu_3)$ is an infinite subset of $\mathbb{Z}^3_{\geq 0}$.  We define
$w(\pi) = |\pi \setminus (\I^+ \cup \II \cup \III)| - |\II| -2|\III|$, the customary measure of ``size'' of such a plane partition in the geometry literature (see, for instance,~\cite{mnop1}).  
%\subsection{The Combinatorial DT Topological Vertex}

Define 
	\[ V(\mu_1, \mu_2, \mu_3) = \sum\limits_{\pi \in P(\mu_1, \mu_2, \mu_3)} q^{w(\pi)}.\]
We call $V(\mu_1, \mu_2, \mu_3)$ the \emph{topological vertex in Donaldson-Thomas theory}.  Note that if $\pi \in P(\emptyset, \emptyset, \emptyset)$ with $|\pi|=n$, then $\pi$ is a plane partition of $n$ in the conventional sense, that is, a finite array of integers such that each row and column is a weakly decreasing sequence of nonnegative integers. Thus MacMahon's enumeration of plane partitions~\cite{macmahon} gives us $V(\emptyset, \emptyset, \emptyset) = \prod_{i = 1}^{\infty}\left(1-q^i\right)^{-i}$.

In~\cite{orv}, there is an expansion of $V(\mu_1,\mu_2,\mu_3)$ in terms of Schur functions.  However, since no similar expansion is known in PT theory, this expansion 
%won't help us.  
does not help prove Theorem~\ref{thm:ptdt}.

\subsection{DT theory and the dimer model}
\label{sec:DTtheoryAndDimers}

Before giving the dimer description of $V(\mu_1,\mu_2,\mu_3)$, we review the correspondence between plane partitions and dimer configurations of a honeycomb graph. By representing each integer $i$ in a plane partition as a stack of $i$ unit boxes, a plane partition can be visualized as a collection of boxes which is stacked stably in the positive octant, with gravity pulling them in the direction $(-1, -1, -1)$. This collection of boxes can be viewed as a lozenge tiling of a hexagonal region of triangles that are the faces of a finite planar graph $T$. This lozenge tiling is then equivalent to a dimer configuration (also called a perfect matching) of the dual graph of $T$, which is a honeycomb graph $H(N)$. 

Just as a plane partition can be visualized as a collection of boxes, a plane partition asymptotic to $(\mu_1, \mu_2, \mu_3)$ can be visualized as a collection of boxes, as shown in Figure~\ref{fig:DTdimers}, left picture. Moreover, a version of the above correspondence puts these box collections in bijection with dimer configurations on the honeycomb graph $H(N)$ with some outer vertices removed, which we call $H(N; \mu)$. Specifically, 
%construct $H(N; \mu)$ as follows. 
let $S_i$ be the Maya diagram of $\mu_i$. Construct the sets $S_i^+$, $S_i^-$ for $i=1,2,3$ and then remove the vertices with the labels in $S_i^+\cup S_i^-$ from sector $i$ of $H(N)$ to obtain $H(N; \mu)$ (here, we are referring to the labelling of the boundary vertices illustrated in Figure~\ref{fig:sectors}, left picture). 

\begin{figure}[htb]
\begin{center}
\includegraphics[width=0.3\textwidth]{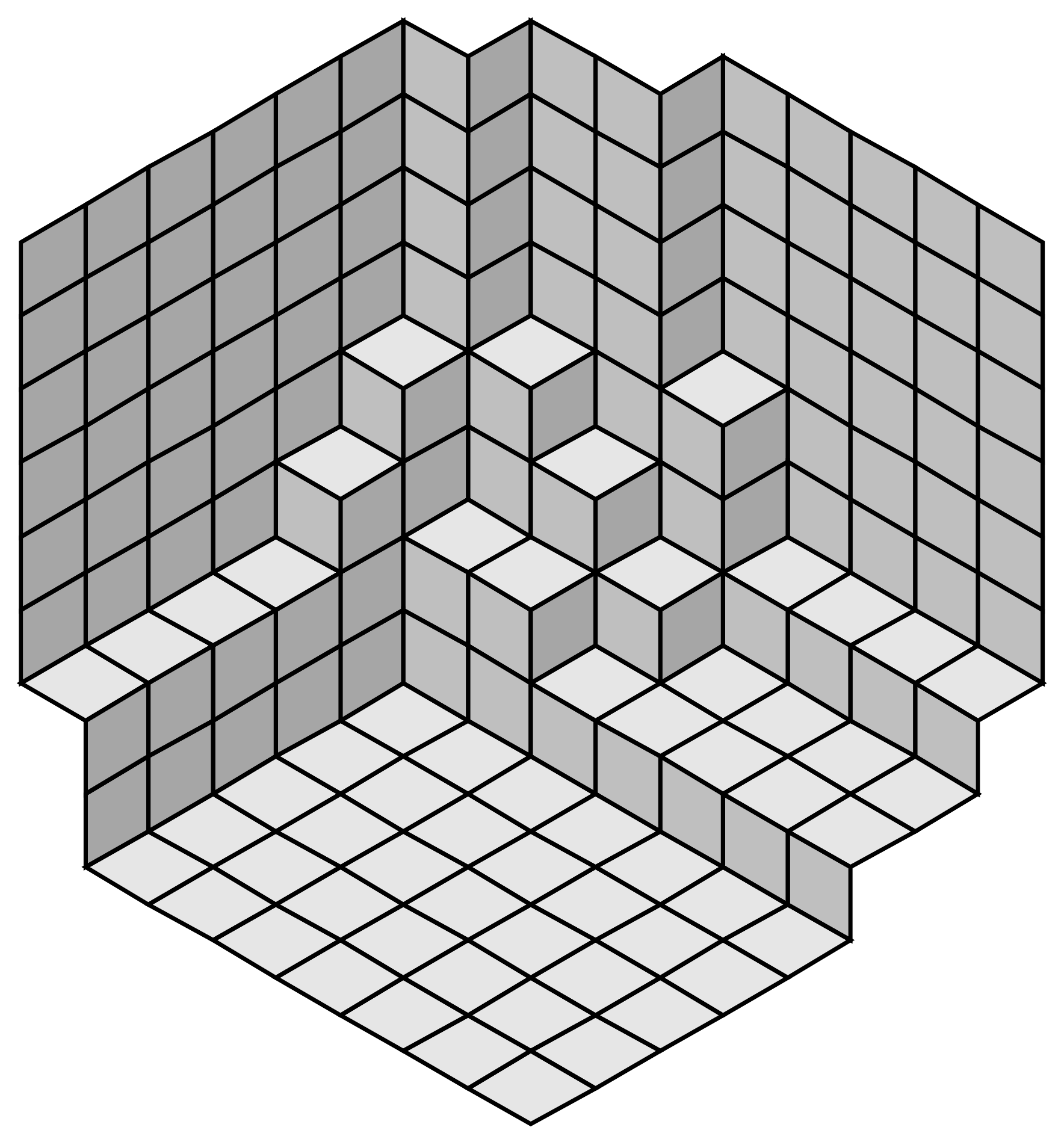}\hfill
\includegraphics[width=0.3\textwidth]{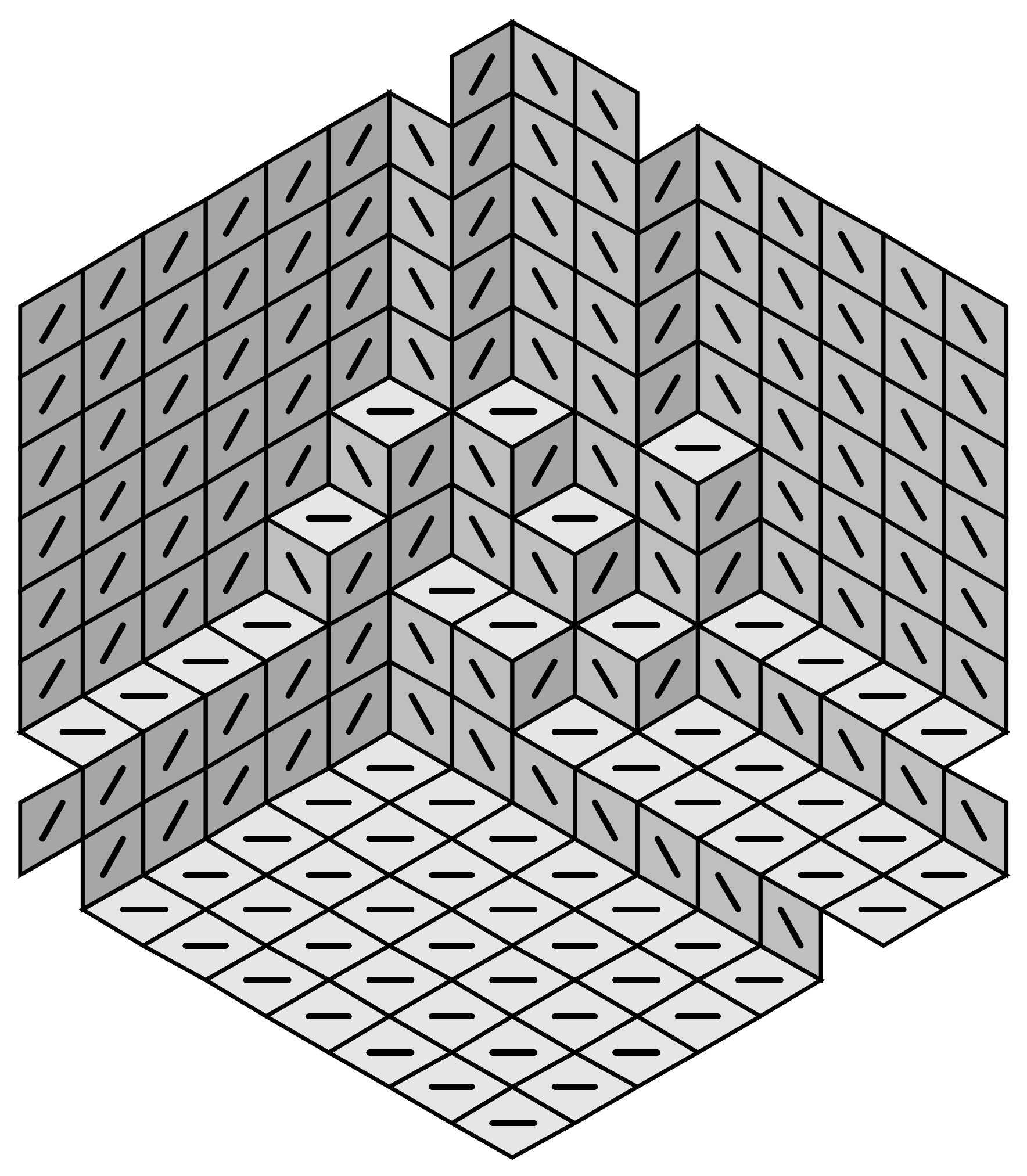}\hfill
\includegraphics[width=0.3\textwidth]{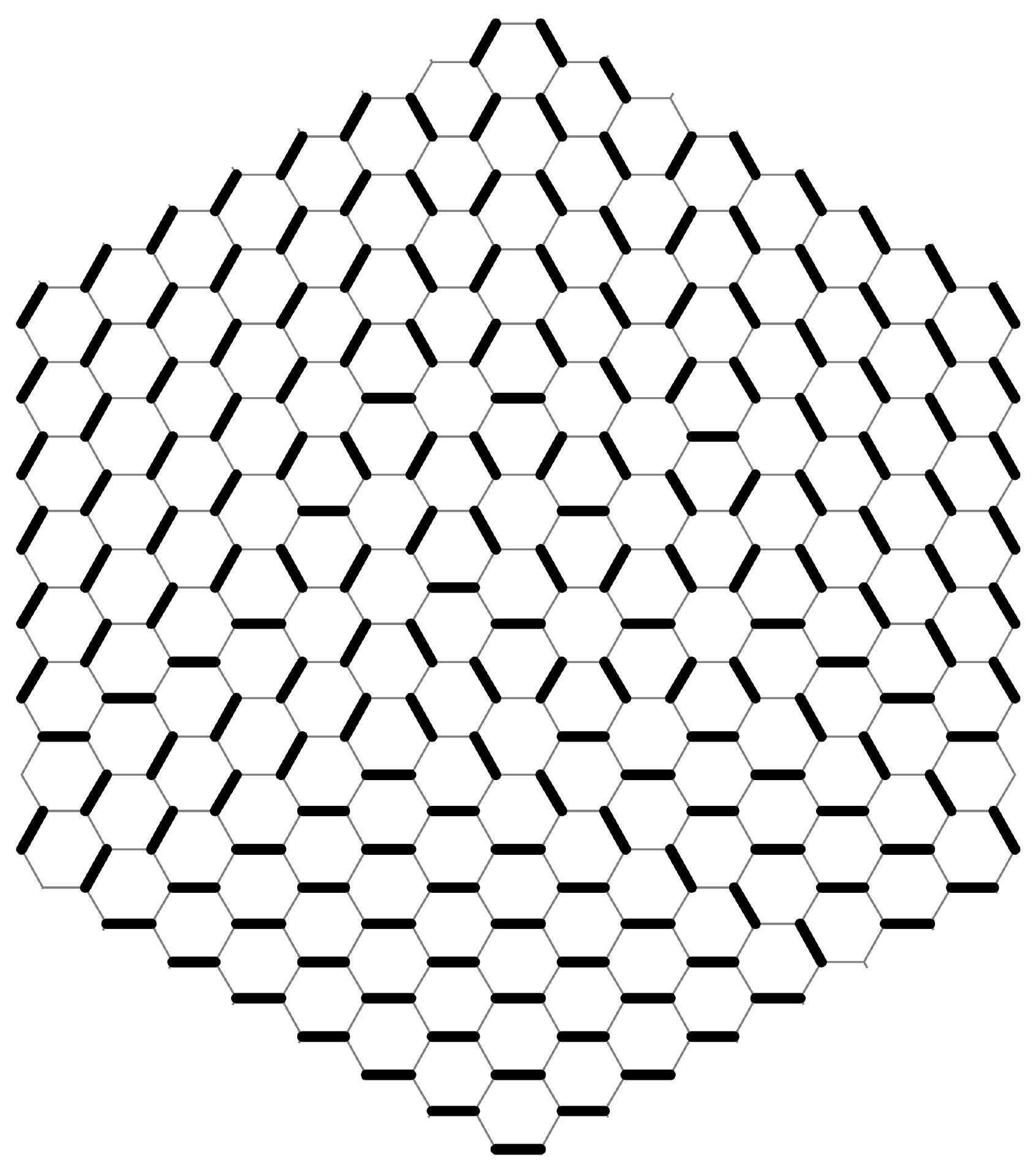}
\end{center}
\caption{Shown left is a plane partition $\pi$ asymptotic to $(\mu_1, \mu_2, \mu_3)$, where $\mu_1 =(1, 1)$, $\mu_2 = \mu_3 = (2, 1, 1)$, $|\II| = 9$, $|\III| = 3$, and $w(\pi) =13-|\II|-2|\III| = -2$. We see that $\pi$ is equivalent to a tiling, which is truncated in the center image so that it corresponds to a dimer configuration of $H(7)$ with a few vertices on the outer face deleted.}
\label{fig:DTdimers}
\end{figure}

Assume $N\geq M$. The bijection described above preserves weight up to an overall multiplicative constant, if we choose the edge weights in the dimer model correctly. The edge weights we use are shown in Figure~\ref{fig:kuoweighting}. Let $Z^D(G)$ denote the weighted sum of all dimer configurations on $G$. Let $M_{\min}(\mu)$ be the unique dimer configuration on $H(N; \mu)$ of minimal weight -- equivalently, the one whose height function is minimal. We call $M_{\min}(\mu)$ the \emph{minimal dimer configuration}; see Section~\ref{sec:minconfigs}. This dimer configuration corresponds to the unique plane partition $\pi_{\min}(\mu)$ asymptotic to $(\mu_1, \mu_2, \mu_3)$ that has no ``extra'' boxes, i.e., the one that contains only $\I^+\cup\II\cup\III$. Observe that $M_{\min}(\mu)$ contributes to the lowest-degree term of $Z^D(H(N; \mu))$, while $\pi_{\min}(\mu)$ contributes to the lowest-degree term of $V$. In fact, adding a box to a plane partition asymptotic to $(\mu_1, \mu_2, \mu_3)$ increases the weight of the corresponding dimer configuration by a factor of $q$, and removing a box decreases the weight by a factor of $q$ (this is a consequence of the particular choice of edge weights). So, if the weight of $M_{\min}(\mu)$ is $q^{w_{\min}(\mu)}$, then $q^{-w_{\min}(\mu)}Z^D(H(N; \mu))$ and $q^{|\II(\mu)|+2|\III(\mu)|}V(\mu_1, \mu_2, \mu_3)$ agree, at least up to degree $N-M$. In other words, if $\tilde{w}_{\min}(\mu):=w_{\min}(\mu)+|\II(\mu)|+2|\III(\mu)|$, 

\begin{thm}
\label{thm:ZD convergence}
As $N\to\infty$, $\widetilde{Z}^D(H(N; \mu)):=q^{-\tilde{w}_{\min}(\mu)}Z^D(H(N; \mu))$ converges to $V(\mu_1, \mu_2, \mu_3)$, where the limit is taken in the sense of formal Laurent series. 
\end{thm}

When $\mu_1 = \mu_2 = \mu_3 = \emptyset$, the weight $q^{w_{\min}(\mu)}$ of $M_{\min}(\mu)$ is computed, for instance in~\cite{kuo}. For general $\mu$, the computation is substantially messier, and is postponed to Section~\ref{sec:minconfigs}.

\subsection{The condensation recurrence in DT theory}
\label{sec:DTcond}

We now show that the DT partition function satisfies the condensation recurrence; this is a corollary of the well-known ``graphical condensation'' theorem of Kuo:

\begin{thm}\cite[Theorem 5.1]{kuo}
\label{thm:kuo}
Let $G = (V_1, V_2, E)$ be a weighted planar bipartite graph with a given planar embedding in which $|V_1| = |V_2|$. Let vertices $a, b, c,$ and $d$ appear in a cyclic order on a face of $G$. If $a, c \in V_1$ and $b, d \in V_2$, then 
\small
\begin{equation}
\label{eqn:DTcond}
Z^{D}(G)Z^{D}(G - \{a, b, c, d\}) = Z^{D}(G - \{a, b\})Z^{D}(G - \{c, d\}) + Z^{D}(G - \{a, d\})Z^{D}(G - \{b, c\}).
\end{equation}
\normalsize
\end{thm}

Take $G$ to be $H(N; \mu_1^{rc}, \mu_2^{rc}, \mu_3)$ for $N\geq M$. Let $a$ and $b$ be the vertices in sector 1 labelled by $\max S_1^-$ and $\min S_1^+$, respectively. Similarly, we let $c$ and $d$ be the vertices in sector 2 labelled by $\max S_2^-$ and $\min S_2^+$. The resulting six dimer model partition functions are all instances of the topological vertex, up to degree $N-M$. 
%Take $G$ to be $H(N; \mu_1,\mu_2,\mu_3) \cup \{a, b, c, d\}$

The graph $G-\{a, b, c, d\}$ is $H(N; \mu_1, \mu_2, \mu_3)$, 
%G-\{a, b, c, d\} &=& H(N; \mu_1,\mu_2,\mu_3),\\
%Then 
\begin{align*}
G-\{a, b\} &= H(N; \mu_1, \mu_2^{rc}, \mu_3),\text{ and} & G-\{c, d\} &= H(N; \mu_1^{rc}, \mu_2, \mu_3). 
\end{align*}
On the other hand, the graphs $G-\{a, d\}$ and $G-\{b, c\}$ are not equal to $H(N; \lambda_1, \lambda_2, \lambda_3)$ for any partitions $\lambda_1, \lambda_2, \lambda_3$, since such partitions would have to satisfy $|S_i^+|=|S_i^-|\pm 1$ for $i=1, 2$, which is impossible (the Maya diagram $S$ of a partition $\lambda$ always satisfies $|S^+|=|S^-|$). Instead, these graphs are associated with Maya diagrams of nonzero charge: $G-\{a, d\}$ is constructed from the charge $-1$ Maya diagram associated to $\mu_1^r$ and the charge $1$ Maya diagram associated to $\mu_2^c$, and $G-\{b, c\}$ is constructed from the charge $1$ Maya diagram associated to $\mu_1^c$ and the charge $-1$ Maya diagram associated to $\mu_2^r$. However, the correspondence discussed in Section~\ref{sec:DTtheoryAndDimers} can still be applied in these cases, with minor modifications: plane partitions asymptotic to $(\mu_1^r, \mu_2^c, \mu_3)$ correspond to dimer configurations on $G-\{a, d\}$, with the origin in $\mathbb{Z}^3$ corresponding to the face directly above the central face of $H(N)$, and plane partitions asymptotic to $(\mu_1^c, \mu_2^r, \mu_3)$ correspond to dimer configurations on $G-\{b, c\}$, with the origin in $\mathbb{Z}^3$ corresponding to the face directly below the central face of $H(N)$. For this reason, we refer to the dimer configurations on $G-\{a, d\}$ and $G-\{b, c\}$ of minimal weight by $M_{\min}^{u}$ and $M_{\min}^{d}$, respectively. 

\begin{figure}[htb]
\begin{center}
\includegraphics[width=2in]{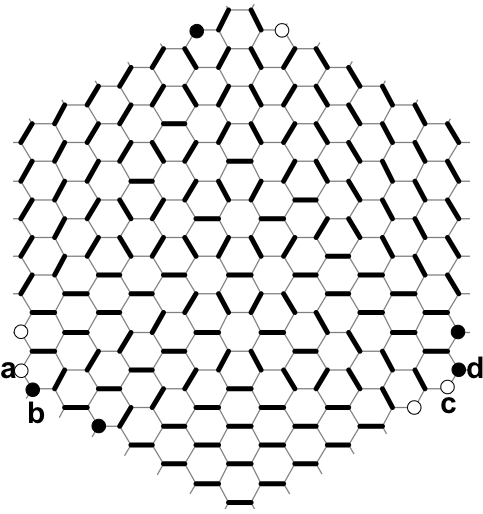}
\end{center}
\caption{A dimer configuration of $H(7; \mu_1, \mu_2, \mu_3)$, and the vertices $a$, $b$, $c$, and $d$, where $\mu_1 = (3, 2)$, $\mu_2 = (2, 2)$, and $\mu_3 = (2, 1)$. }
\label{fig:dtfloor}
\end{figure}

\begin{figure}[htb]
\centering
\includegraphics[width=1.5in]{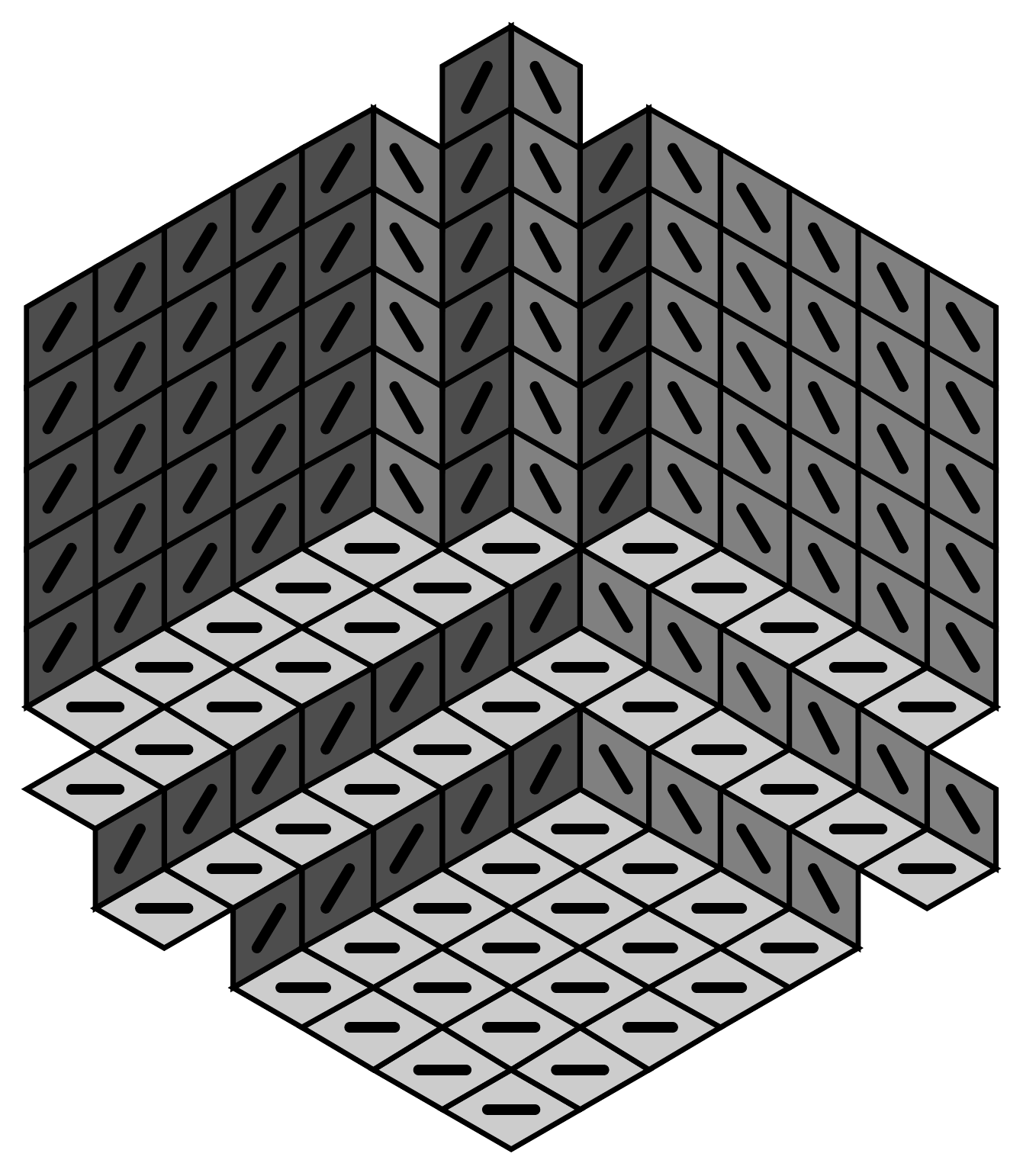} %\hspace{.15cm}
\includegraphics[width=1.5in]{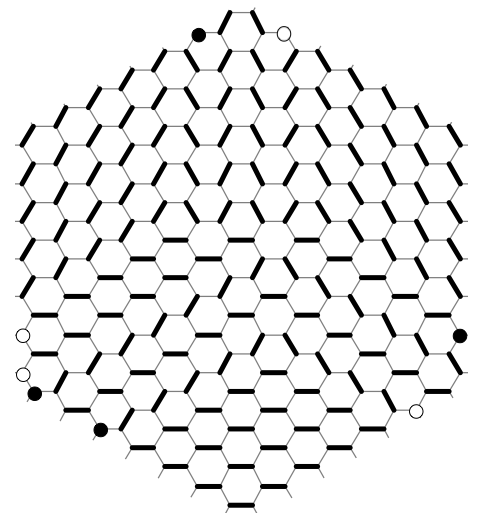}

\includegraphics[width=1.5in]{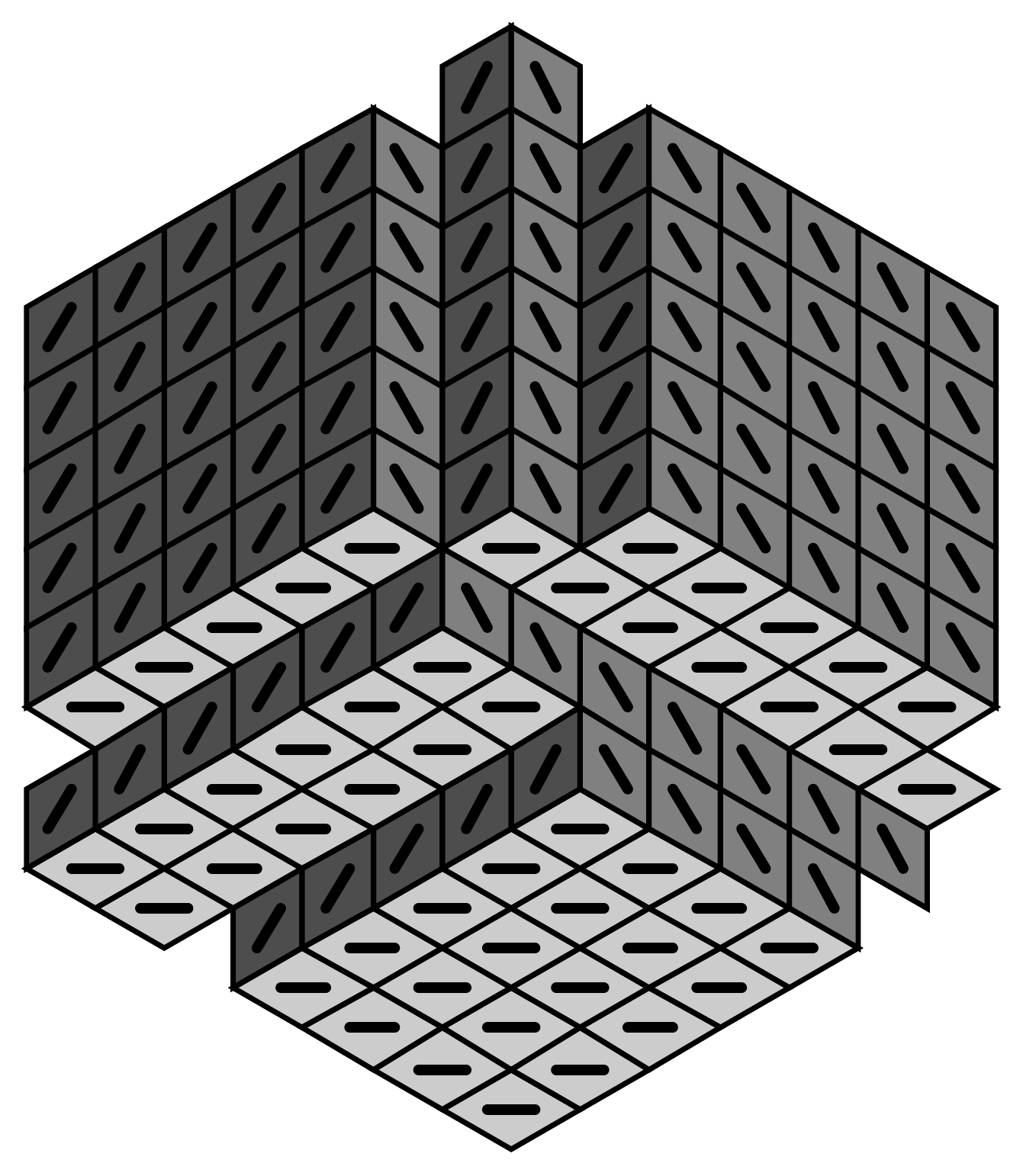}
\includegraphics[width=1.5in]{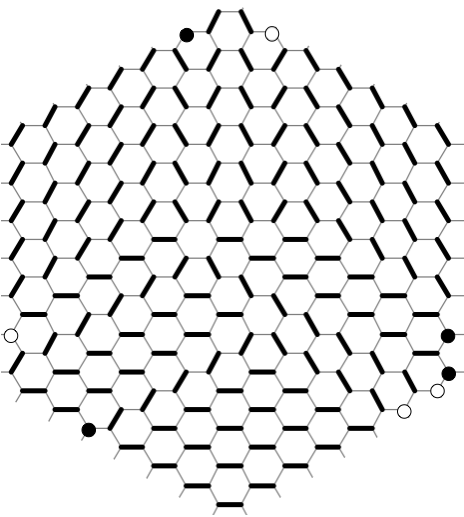}

\includegraphics[width=1.5in]{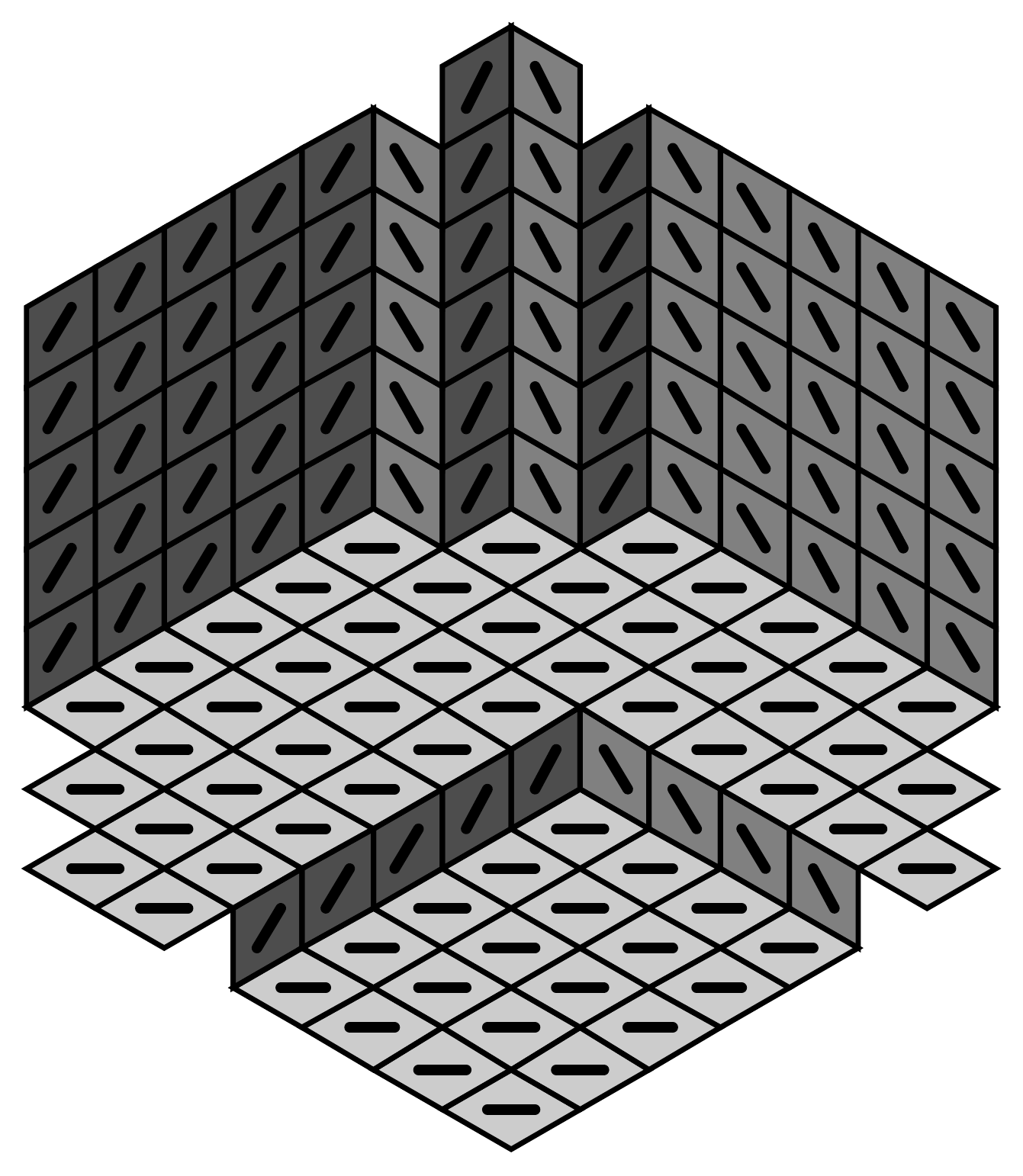} %\hspace{.15cm}
\includegraphics[width=1.5in]{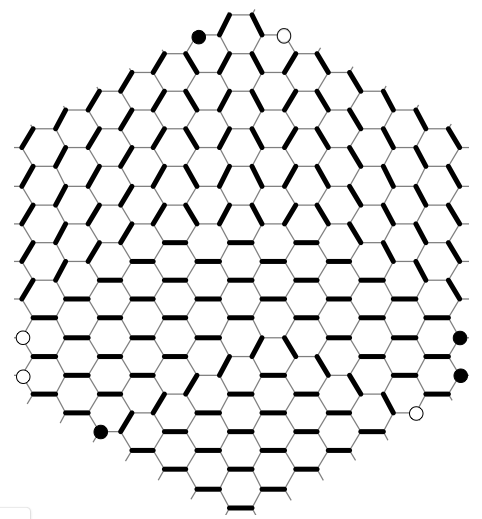}

\includegraphics[width=1.5in]{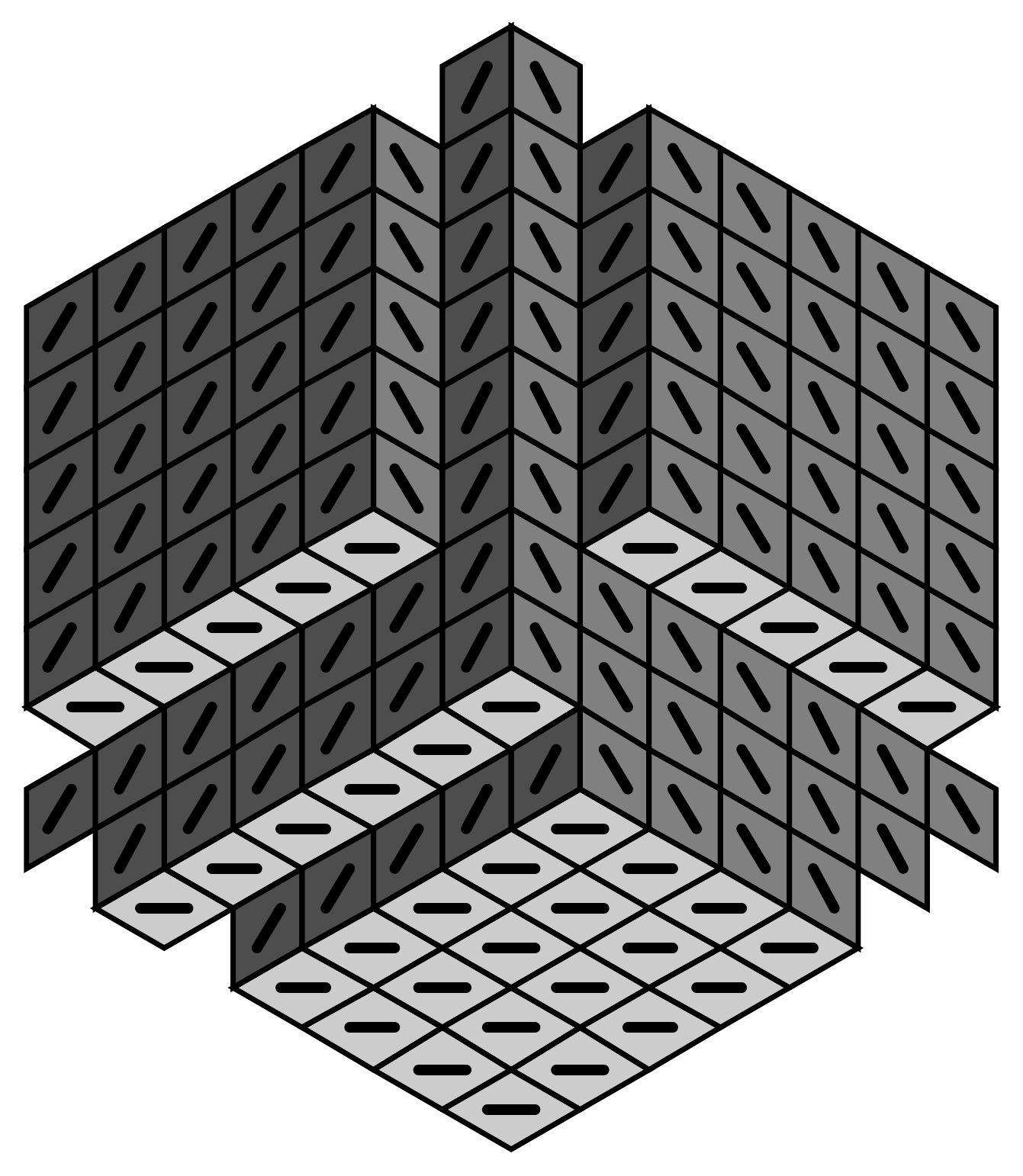}
\includegraphics[width=1.5in]{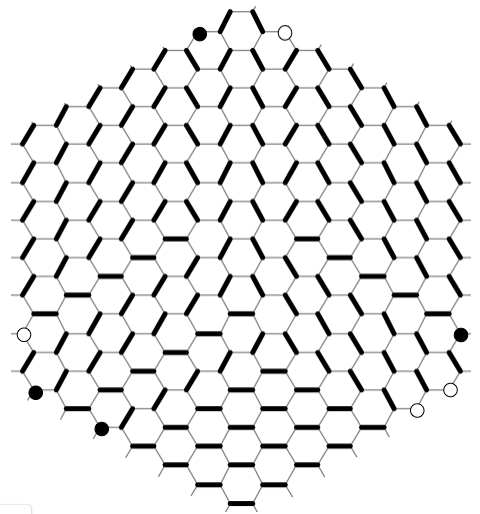}
\caption{Modifications of the graph $G$ from Example~\ref{ex:dtfloorex}, and their minimal dimer configurations. First row: The graph $G-\{a, b\}$ and its minimal dimer configuration. Second row: The graph $G-\{c, d\}$ and its minimal dimer configuration. Third row: The graph $G - \{a, d\}$ and its minimal dimer configuration. Fourth row: The graph $G - \{b, c\}$ and its minimal dimer configuration. }
\label{fig:dtfloor2}
\end{figure}

\begin{example}
\label{ex:dtfloorex}
Let $N = 7$, and let $\mu_1 = (3, 2)$, $\mu_2 = (2, 2)$, and $\mu_3=(2, 1)$. Figure~\ref{fig:dtfloor} shows a dimer configuration of $G - \{a, b, c, d\} = H(N; \mu_1, \mu_2, \mu_3)$ and the vertices $a$, $b$, $c$, and $d$. 

We note that $\mu_1^{rc} = (3, 1)$ and $\mu_2^{rc} = (2, 1)$. The graphs $G-\{a, b\}=H(N; \mu_1, \mu_2^{rc}, \mu_3)$ and $G-\{c, d\}=H(N; \mu_1^{rc}, \mu_2, \mu_3)$, along with their minimal dimer configurations, are shown in Figure~\ref{fig:dtfloor2}. 

We have $\mu_1^r = (4)$, $\mu_2^c = (1, 1, 1)$, $\mu_1^c = (2, 1, 1)$, and $\mu_2^r = (3)$. The graphs $G - \{a, d\}$, $G-\{b, c\}$ and their minimal dimer configurations are also shown in Figure~\ref{fig:dtfloor2}. This figure illustrates the fact that the correspondence between plane partitions asymptotic to $(\mu_1^r, \mu_2^c, \mu_3)$ (resp.~$(\mu_1^c, \mu_2^r, \mu_3)$) and dimer configurations on $G-\{a, d\}$ (resp.~$G-\{b, c\}$) requires a shift; the image shows that the ``floor'' of the plane partition is shifted up (resp.~down). 
\end{example}

Let $q^{w_{\min}^u}$ and $q^{w_{\min}^d}$ be the weights of $M_{\min}^{u}$ and $M_{\min}^{d}$, respectively. Then let $\tilde{w}_{\min}^u=w_{\min}^u+|\II(\mu_1^r, \mu_2^c, \mu_3)|+2|\III(\mu_1^r, \mu_2^c, \mu_3)|$, $\tilde{w}_{\min}^d=w_{\min}^d+|\II(\mu_1^c, \mu_2^r, \mu_3)|+2|\III(\mu_1^c, \mu_2^r, \mu_3)|$, \[\widetilde{Z}^D(H(N; \mu_1^{rc},\mu_2^{rc},\mu_3)-\{a, d\})=q^{-\tilde{w}_{\min}^u}Z^D(H(N; \mu_1^{rc},\mu_2^{rc},\mu_3)-\{a, d\}),\] and \[\widetilde{Z}^D(H(N; \mu_1^{rc},\mu_2^{rc},\mu_3)-\{b, c\})=q^{-\tilde{w}_{\min}^d}Z^D(H(N; \mu_1^{rc},\mu_2^{rc},\mu_3)-\{b, c\}).\] Also, let 
\begin{align*}
A &= \tilde{w}_{\min}(\mu_1, \mu_2, \mu_3)+\tilde{w}_{\min}(\mu_1^{rc}, \mu_2^{rc}, \mu_3), \\
B &= \tilde{w}_{\min}(\mu_1^{rc}, \mu_2, \mu_3)+\tilde{w}_{\min}(\mu_1, \mu_2^{rc}, \mu_3), \text{ and} \\
C &= \tilde{w}_{\min}^u+\tilde{w}_{\min}^d. 
\end{align*}

From (\ref{eqn:DTcond}) and the preceding remarks, we have 
\begin{eqnarray}
&&q^A\widetilde{Z}^D(H(N; \mu_1, \mu_2, \mu_3))\widetilde{Z}^D(H(N; \mu_1^{rc}, \mu_2^{rc}, \mu_3)) \label{eqn:intermediateDTcond} \\
&=&q^B\widetilde{Z}^D(H(N; \mu_1^{rc}, \mu_2, \mu_3))\widetilde{Z}^D(H(N; \mu_1, \mu_2^{rc}, \mu_3)) \nonumber \\
&&{}+{}q^C\widetilde{Z}^D(H(N; \mu_1^{rc},\mu_2^{rc},\mu_3)-\{a, d\})\widetilde{Z}^D(H(N; \mu_1^{rc},\mu_2^{rc},\mu_3)-\{b, c\}). \nonumber 
\end{eqnarray}
From Lemma~\ref{cor:DTweightGabcd}, we see that $A=B$, and we multiply equation (\ref{eqn:intermediateDTcond}) by $q^{-A}$. In Section~\ref{sec:DTalg}, we show that $C-A=-K$, which does not depend on the variable $N$. For this reason, we can take $N \rightarrow \infty$; in this limit, all six of the Laurent series $\widetilde{Z}^D$ converge to instances of $V$, with different partitions as parameters. By Theorem~\ref{thm:ZD convergence}, the first four Laurent series $\widetilde{Z}^D$ converge to $V(\mu_1, \mu_2, \mu_3)$, $V(\mu_1^{rc}, \mu_2^{rc}, \mu_3)$, $V(\mu_1^{rc}, \mu_2, \mu_3)$, and $V(\mu_1, \mu_2^{rc}, \mu_3)$, respectively. Similarly, $\widetilde{Z}^D(H(N; \mu_1^{rc},\mu_2^{rc},\mu_3)-\{a, d\})$ converges to $V(\mu_1^{r}, \mu_2^{c}, \mu_3)$, and $\widetilde{Z}^D(H(N; \mu_1^{rc},\mu_2^{rc},\mu_3)-\{b, c\})$ converges to $V(\mu_1^{c}, \mu_2^{r}, \mu_3)$. Thus, 
\begin{equation}\small
	\label{eqn:vertex_condensation_DT}
	V(\mu_1, \mu_2, \mu_3)
	V(\mu_1^{rc}, \mu_2^{rc}, \mu_3)
	=
	V(\mu_1^{rc}, \mu_2, \mu_3)
	V(\mu_1, \mu_2^{rc}, \mu_3)
	+
	q^{-K}
	V(\mu_1^{r}, \mu_2^{c}, \mu_3)
	V(\mu_1^{c}, \mu_2^{r}, \mu_3). 
\end{equation}
Multiplying by $\frac{q^K}{(M(q))^2}$, we conclude that $V/M(q)$ satisfies the condensation recurrence~\eqref{eqn:vertex_condensation}.

\section{PT}
\label{sec:PT}

This section is, in principle, parallel to the previous one, except our computations are done in PT theory~\cite{PT2}, rather than DT theory.  However, the computations in question are substantially more intricate.  

The overall plan is as follows.  In Section~\ref{sec:PTboxconfigs}, we describe the original index set for the generating function $W(\mu_1, \mu_2, \mu_3)$ that was introduced in~\cite{PT2}; it consists of certain novel plane-partition-like objects that we call \emph{PT box configurations}. These configurations come with a notion of \emph{labelling}, which is needed to describe the coefficients of the generating function $W$.  We introduce two alternate combinatorial models for the index set for $W$: namely \emph{$AB$ configurations} in Section~\ref{sec:ABconfigs}, and \emph{double-dimer configurations} in Section~\ref{sec:double_dimer_configs}. We demonstrate in Section~\ref{sec:labelling_algorithm_proofs} that these combinatorial objects are computing the same generating function $W(\mu_1, \mu_2, \mu_3)$ by describing and analyzing algorithms, called the \emph{labelling algorithms}, which are used in recovering PT box configurations from the other models. Finally, in Section~\ref{sec:pt_condensation_identity}, we review the facts we need from~\cite{jenne} about the condensation identity in the double-dimer model, and explain how this identity is applied to compute $W(\mu_1,\mu_2, \mu_3)$.

\subsection{Labelled PT box configurations}
\label{sec:PTboxconfigs}

We refer to elements of $\mathbb{Z}^3$ as \emph{cells}. 

\begin{defn}
If $w = (w_1, w_2, w_3)$ is a cell, the set of \emph{back neighbors} of $w$, denoted $BN(w)$, is \[\left\{(w_1 - 1, w_2, w_3), (w_1, w_2 - 1, w_3), (w_1, w_2, w_3 - 1)\right\}.\]
\end{defn}

We now introduce labelled box configurations. Their definition is taken from~\cite{PT2}. 

\begin{defn}
A \emph{set of labelled boxes} is a finite subset of $\I^-\cup\II\cup\III$, whose elements are referred to as \emph{boxes}, where each type $\III$ box $w$ may be labelled by an element of 
\[\mathbb{P}^1_w:=\mathbb{P}\left(\frac{
	\mathbb{C} \cdot \mathbf{1}_w \oplus
	\mathbb{C} \cdot \mathbf{2}_w \oplus
	\mathbb{C} \cdot \mathbf{3}_w
	}{
	\mathbb{C} \cdot (1,1,1)_w}\right).\] 
\end{defn}

\begin{defn}
A \emph{labelled box configuration} is a set of labelled boxes that satisfies the following \emph{box-stacking rules}. 
\end{defn}

\begin{conditions}
\begin{enumerate}[1.]
\item If $w \in \I^-$ and any cell in $BN(w)$ is a box, then $w$ must be a box.
\item If $w \in \II_{\bar i}$ and any cell $n\in BN(w)$ is a box that is not a type $\III$ box labelled $\vspan\{\mathbf{i}_n+\mathbb{C}\cdot(1, 1, 1)_n\}$, then $w$ must be a box.
\item If $w \in \III$ and the span of subspaces of \[
	\frac{
	\mathbb{C} \cdot \mathbf{1}_w \oplus
	\mathbb{C} \cdot \mathbf{2}_w \oplus
	\mathbb{C} \cdot \mathbf{3}_w
	}{\mathbb{C} \cdot (1,1,1)_w}\]
induced by boxes in $BN(w)$ is nonzero, then $w$ must be a box.  If the dimension of the span is 1, then $w$ may either be labelled by the span or be unlabelled.  If the dimension of the span is 2, then $w$ must be unlabelled.
\end{enumerate}
\label{conditions:labelled box stacking}
\end{conditions}

\begin{remark}
\label{rem:unlabelled in front of unlabelled}
By Conditions~\ref{conditions:labelled box stacking}.3, if $w \in \III$ and $n \in BN(w)$ is an unlabelled type $\III$ box, then $w$ must be an unlabelled box. This is because unlabelled type $\III$ boxes induce the whole $2$-dimensional space $\frac{\mathbb{C} \cdot \mathbf{1}_w \oplus\mathbb{C} \cdot \mathbf{2}_w \oplus\mathbb{C} \cdot \mathbf{3}_w}{\mathbb{C} \cdot (1,1,1)_w}$.
\end{remark}

We then define
\[
	W(\mu_1, \mu_2, \mu_3) = q^{-|\II|-2|\III|}\sum_{\text{labelled box configs.~}\pi} \chi_{\text{top}}(\pi) q^{|\pi|},
\]
where $|\pi|$ is the number of boxes in $\pi$ plus the number of unlabelled type $\III$ boxes in $\pi$, and $\chi_{\text{top}}(\pi)$ is the topological Euler characteristic of the moduli space of labellings of $\pi$. When we wish to emphasize the variable being used, we will write $W(\mu_1, \mu_2, \mu_3; q)$ instead of $W(\mu_1, \mu_2, \mu_3)$. 

We will also use the terminology introduced in the following definition. 

\begin{defn}
We say that a type $\III$ box $w$ of a labelled box configuration $\pi$ is \emph{freely labelled} if $w$ is labelled and for any $\ell\in\mathbb{P}^1_w$, there is a labelling of $\pi$ in which $w$ is labelled $\ell$. In this case, we also say that $w$ is labelled by a \emph{freely chosen} element of $\mathbb{P}^1$. 
\end{defn}

The following example appears in \cite[Section 5.4]{PT2}. 

\begin{example}
\label{ex:labelledboxconfigsfrompt}
Let $\mu_1 = (1), \mu_2 = (2),$ and $\mu_3 = (1)$. Then $\III = \{ (0, 0, 0)\}$ and $\II = \II_{\bar{1}} = \{ (0, 0, 1)\}$. We list labelled box configurations $\pi$ with $|\pi| \leq 3$. 

There is a unique empty labelled box configuration. There are two labelled box configurations $\pi$ with $|\pi| = 1$: 
\begin{enumerate}
\item a box at $(0, 0, 0)$ labelled with $\mathbb{C} \cdot {\bf 1}_{(0, 0, 0)}+\mathbb{C} \cdot (1, 1, 1)_{(0, 0, 0)}$, 
\item a box at $(0, 0, 1)$. 
\end{enumerate}
There are three labelled box configurations with $|\pi| = 2$: 
\begin{enumerate}
\item boxes at $(0, -1, 1)$ and $(0, 0, 1)$, 
\item a box at $(0, 0, 0)$ labelled with $\mathbb{C} \cdot {\bf 1}_{(0, 0, 0)}+\mathbb{C} \cdot (1, 1, 1)_{(0, 0, 0)}$ and a box at $(-1, 0, 0)$, 
\item a freely labelled box at $(0, 0, 0)$ and a box at $(0, 0, 1)$. 
\end{enumerate}
There are six labelled box configurations with $|\pi| = 3$: 
\begin{enumerate}
\item an unlabelled box at $(0, 0, 0)$ and a box at $(0, 0, 1)$, 
\item a freely labelled box at $(0, 0, 0)$, and boxes at $(0, -1, 1)$ and $(0, 0, 1)$, 
\item a box at $(-1, 0, 0)$, a box at $(0, 0, 0)$ labelled with $\mathbb{C} \cdot {\bf 1}_{(0, 0, 0)}+\mathbb{C} \cdot (1, 1, 1)_{(0, 0, 0)}$, and a box at $(0, 0, 1)$, 
\item a box at $(0, 0, -1)$, a box at $(0, 0, 0)$ labelled with $\mathbb{C} \cdot {\bf 3}_{(0, 0, 0)}+\mathbb{C} \cdot (1, 1, 1)_{(0, 0, 0)}$, and a box at $(0, 0, 1)$, 
\item a box at $(0, 0, 0)$ labelled with $\mathbb{C} \cdot {\bf 1}_{(0, 0, 0)}+\mathbb{C} \cdot (1, 1, 1)_{(0, 0, 0)}$, and boxes at $(-2, 0, 0)$ and $(-1, 0, 0)$, 
\item boxes at $(0, -2, 1)$, $(0, -1, 1)$, and $(0, 0, 1)$. 
\end{enumerate}
\end{example}

\subsection{Labelled \texorpdfstring{$AB$}{AB} configurations}
\label{sec:ABconfigs}

Given a labelled box configuration $\pi$ such that $\chi_{\text{top}}(\pi) = 2^k$, our objective is to associate to $\pi$ a certain collection of $2^k$ pairs $(A, B)$, called {\em labelled $AB$ configurations.}
%In this section, we introduce one of the main objects of our study: labelled $AB$ configurations. 
We begin by defining $AB$ configurations, and then describe how to label these configurations. 

\begin{defn}
An \emph{$AB$ configuration} is a pair $(A, B)$ of finite sets $A \subseteq \I^- \cup \III$ and $B \subseteq \II \cup \III$, whose elements are referred to as \emph{boxes}, which satisfies the following conditions. 
\end{defn}

\begin{conditions}
\begin{enumerate}[1.]
	\item If $w \in \I^- \cup \III$ and $BN(w)\cap A\neq\varnothing$, then $w\in A$.
	\item If $w \in \II \cup \III$ and $BN(w)\cap B\neq\varnothing$, then $w\in B$.
\end{enumerate}
\label{conditions:ab box stacking}
\end{conditions}
We remark that these are the familiar conditions for plane partitions, except that gravity is pulling the boxes in the direction $(1,1,1)$. Also, we call an $AB$ configuration $(A, B)$ \emph{empty} (resp.~\emph{nonempty}) if $A\cup B$ is empty (resp.~nonempty). 

If there is a labelled box configuration $\pi$ so that the additional conditions below are satisfied, then we say that $(A, B)$ is an $AB$ configuration \emph{on $\pi$}.

\begin{conditions}
\begin{enumerate}[1.]
	\item $A \cup B$ is the set of boxes in $\pi$.
	\item $A \cap B$ is the set of unlabelled type $\III$ boxes in $\pi$.
\end{enumerate}
\label{conditions:ab on pi}
\end{conditions}

\begin{figure}[htb]
\begin{center}
\includegraphics[width=1.5in]{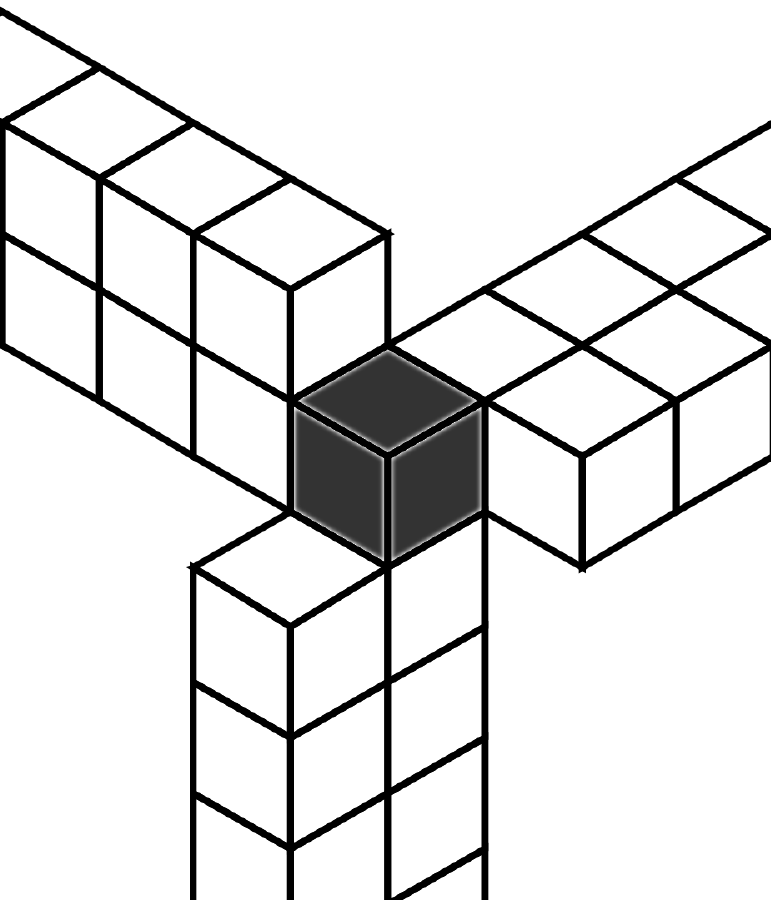}
\includegraphics[width=1.5in]{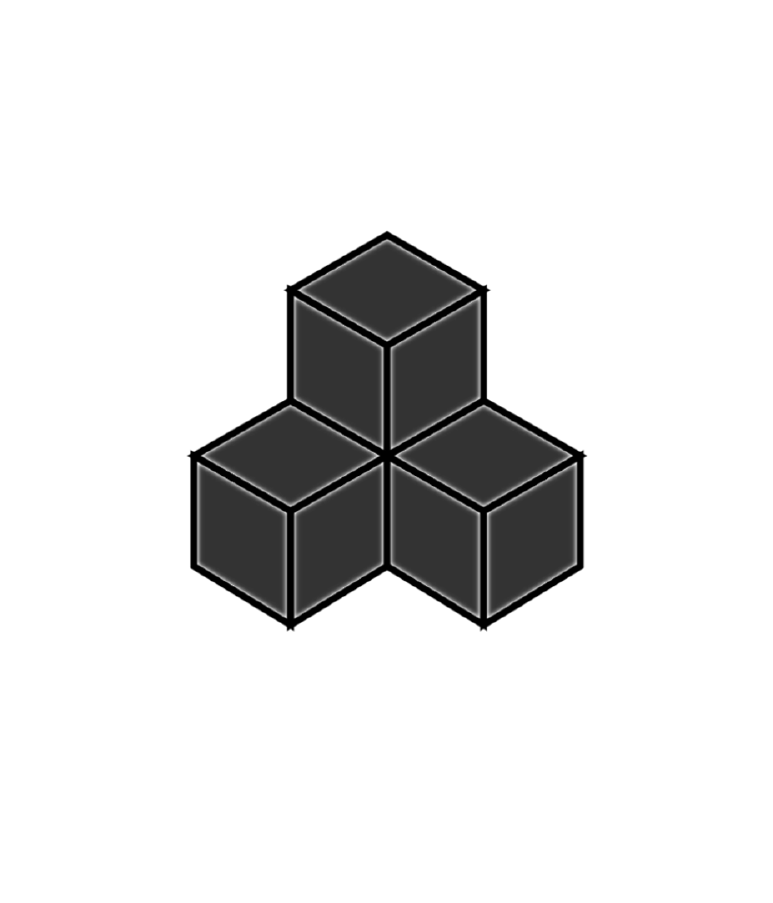}
\end{center}
	\caption{The $AB$ configuration $(\III, \II\cup\III)$, in the case where $\mu_1 = \mu_2 = \mu_3 = (2)$.}
	%, the partition with one part of size 2.}
\label{fig:AB}
\end{figure}

\subsubsection{The base \texorpdfstring{$AB$}{AB} configuration}
\label{sec:base_AB_config}

The set of all $AB$ configurations on $\pi$ will be denoted $\sAB(\pi)$. There is always at least one way to construct an $AB$ configuration on $\pi$. This will be called the \emph{base $AB$ configuration}, $AB_{\base}(\pi)$.

\begin{defn}
	Construct $A$ and $B$ from the boxes of $\pi$ as follows. Let $A$ consist of the type $\I^-$ boxes and the type $\III$ boxes. Let $B$ consist of the type $\II$ boxes and the unlabelled type $\III$ boxes. Define $AB_{\base}(\pi) = (A, B)$.
\end{defn}

\begin{example}
If $\mu_1=\mu_2=\mu_3=(2)$, then there is a labelled box configuration $\pi$ consisting of an unlabelled type $\III$ box $(0, 0, 0)$, and type $\II$ boxes $(1, 0, 0)$, $(0, 1, 0)$, and $(0, 0, 1)$. The base $AB$ configuration is $AB_{\base}(\pi)=(A, B)$, where $A=\{(0, 0, 0)\}$ and $B=\{(0, 0, 0), (1, 0, 0), (0, 1, 0), (0, 0, 1)\}$, and is illustrated in Figure~\ref{fig:AB}. In this case, $A=\III$ and $B=\II\cup\III$. 
\end{example}

We will now show that $AB_{\base}(\pi) \in \sAB(\pi)$. To establish this fact as well as subsequent statements, the following lemmas will be needed. 

\begin{lemma}
	\label{lemma:cylinder back neighbor}
Suppose $w\in\Cyl_j$ and $n(i)\in BN(w)$ is the back neighbor obtained by subtracting $1$ from the $i$th coordinate of $w$. Then, if $i=j$ or the $i$th coordinate of $w$ is positive, $n(i)\in\Cyl_j$.
\end{lemma}

\begin{proof}
Let $w=(w_1, w_2, w_3)$ and $n(i)=(n_1, n_2, n_3)$, so that $n_i=w_i-1$ and $n_l = w_l$ for $l \neq i$. In what follows, all indices should be considered modulo $3$. Since $w\in\Cyl_j$, $(w_{j+1}, w_{j+2})\in\mu_j$. Suppose $i=j$. Then $(n_{j+1}, n_{j+2})=(w_{j+1}, w_{j+2})\in\mu_j$, so $n(i)\in\Cyl_j$. Suppose $w_i>0$. We may assume $i\neq j$, so $i=j+1$ or $i=j+2$. In the first case, $w_{j+1}-1\geq 0$, so $(n_{j+1}, n_{j+2})=(w_{j+1}-1, w_{j+2})\in\mu_j$, while in the second case, $w_{j+2}-1\geq 0$, so $(n_{j+1}, n_{j+2})=(w_{j+1}, w_{j+2}-1)\in\mu_j$. In both cases, $n(i)\in\Cyl_j$.
\end{proof}

\begin{lemma}
	\label{lemma:adjacent types I- and II}
Let $i\in\{1, 2, 3\}$. If $w\in\I^-$ is adjacent to $w'\in\II_{\bar{i}}$, then $w\in\Cyl_j^-$ for some $j\in\{1, 2, 3\}\setminus\{i\}$. 
\end{lemma}

\begin{proof}
Either $w\in BN(w')$ or $w'\in BN(w)$. Since $w'\in\II\subseteq\mathbb{Z}^3_{\geq 0}$, if $w'\in BN(w)$, then $w\in\mathbb{Z}^3_{\geq 0}$. However, $w\in\I^-$, so $w\not\in\mathbb{Z}^3_{\geq 0}$. Thus, $w\in BN(w')$. Since $w\in\I^-$, $w\in\Cyl_j^-$ for some $j\in\{1, 2, 3\}$, so the $j$th coordinate of $w$ must be negative. Since $w'\in\mathbb{Z}^3_{\geq 0}$, $w$ must be the back neighbor obtained by subtracting $1$ from the $j$th coordinate of $w'$. Since $w\in\Cyl_j$, we find that $w'\in\Cyl_j$. On the other hand, since $w'\in\II_{\bar{i}}$, $w'\not\in\Cyl_i$, so $j\neq i$. 
\end{proof}

\begin{lemma}
	\label{lemma:adjacent type II}
Let $i\in\{1, 2, 3\}$. If $w\in\II$ is adjacent to $w'\in\II_{\bar{i}}$, then $w\in\II_{\bar{i}}$.
\end{lemma}

\begin{proof}
Either $w\in BN(w')$ or $w'\in BN(w)$. If $w\in BN(w')$, observe that $w'\in\Cyl_j$ for $j\neq i$, so by Lemma~\ref{lemma:cylinder back neighbor}, $BN(w')\cap\mathbb{Z}^3_{\geq 0}\subseteq\Cyl_j$. Since $w\in\II\subseteq\mathbb{Z}^3_{\geq 0}$, we have $w\in\Cyl_j$ for $j\neq i$, so $w\in\II_{\bar{i}}$. Otherwise, $w'\in BN(w)$. Then $w\in\II_{\bar{j}}$ for some $j\in\{1, 2, 3\}$, and by the same argument, $w'\in\II_{\bar{j}}$. We deduce that $j=i$, so $w\in\II_{\bar{i}}$.
\end{proof}

\begin{lemma}
	\label{lemma:type III back neighbors}
Suppose $w\in\III$ and $n\in BN(w)$.  Then $n\in\I^-\cup\III$.
\end{lemma}

\begin{proof}
Let $n(i)\in BN(w)$ be the neighbor obtained by subtracting $1$ from the $i$th coordinate of $w$. If $n(i)\not\in\III$, then $n(i)\not\in\Cyl_j$ for some $j\in\{1, 2, 3\}$, so by Lemma~\ref{lemma:cylinder back neighbor}, the $i$th coordinate of $w$ is not positive. Since $w\in\III\subseteq\mathbb{Z}^3_{\geq 0}$, it follows that the $i$th coordinate of $w$ is $0$, so the $i$th coordinate of $n(i)$ is $-1$. Therefore, by the same lemma, $n(i)\in\Cyl_i\setminus\mathbb{Z}^3_{\geq 0}=\Cyl_i^-\subseteq\I^-$.
\end{proof}

\begin{lemma}
	If $\pi$ is a labelled box configuration, then $AB_{\base}(\pi)$ satisfies Conditions~\ref{conditions:ab box stacking} and Conditions~\ref{conditions:ab on pi}, i.e., $AB_{\base}(\pi)\in\sAB(\pi)$.
\end{lemma}

\begin{proof}
Let $(A, B)=AB_{\base}(\pi)$. Conditions ~\ref{conditions:ab on pi} are immediate. To check Conditions~\ref{conditions:ab box stacking}.1, suppose that $w \in \I^- \cup \III$ and $n\in BN(w)\cap A$. We must show that $w\in A$. Since $n\in A$, $n$ is a box of $\pi$ in $\I^-\cup\III$. If $w \in \I^-$, the claim follows from Conditions~\ref{conditions:labelled box stacking}.1. If $w \in \III$, the claim follows from Conditions~\ref{conditions:labelled box stacking}.3. 

Similarly, to check Conditions~\ref{conditions:ab box stacking}.2, suppose that $w \in \II \cup \III$ and $n\in BN(w)\cap B$. We must show that $w\in B$. Since $n\in B$, $n$ is a type $\II$ box of $\pi$ or an unlabelled type $\III$ box of $\pi$. If $w \in \II$, then the claim follows from Conditions~\ref{conditions:labelled box stacking}.2. If $w \in \III$, then $w$ is a box of $\pi$, by Conditions~\ref{conditions:labelled box stacking}.3, but we need to check that $w$ is unlabelled. Since $w\in\III$ and $n\in BN(w)$, Lemma~\ref{lemma:type III back neighbors} shows that $n$ cannot be in $\II$, so it must be an unlabelled type $\III$ box. Since $n$ is unlabelled, $w$ must be unlabelled as well by Remark~\ref{rem:unlabelled in front of unlabelled}. 
\end{proof}

We will also need the following definitions.

\begin{defn}
Let $\text{PT-box}$ be the set of all labelled box configurations, and let $\sAB_{\text{all}}$ be the set of all $AB$ configurations. 
\end{defn}

Let $\phi_{\base}:\text{PT-box} \rightarrow \sAB_{\text{all}}$ be the map that sends $\pi$ to $AB_{\base}(\pi)$, and let $\sAB_{\base} = \phi_{\base}(\text{PT-box})$. Observe that \[\sAB_{\base} = \bigcup_{\pi\in\text{PT-box}}\left\{AB_{\base}(\pi)\right\}.\]

\subsubsection{The labelling algorithm for \texorpdfstring{$AB$}{AB} configurations}
\label{sec:AB_labelling_alg}

Thus far, we have described a method for constructing an $AB$ configuration from a labelled box configuration. We now describe an algorithm that labels $AB$ configurations. When successful, its output can be used to construct a labelled box configuration from an $AB$ configuration. Note that the algorithm assigns labels to cells, not boxes.

\begin{defn}
Let $(A, B)\in\sAB_{\text{all}}$. We call the set \[\mathcal{L}(A, B):=(\I^-\cap A)\cup(\II\setminus B)\cup(\III\cap(A\triangle B))\] the \emph{labelling set} of $(A, B)$.
\end{defn}

We label cells by assigning labels to connected components of $\mathcal{L}(A, B)$ using the following algorithm. 

\begin{algorithm}
	\begin{enumerate}[1.]
		\item If a connected component of $\mathcal{L}(A, B)$ contains a cell in $\Cyl_i^-\cup\II_{\bar{i}}$ and a cell in $\Cyl_j^-\cup\II_{\bar{j}}$, where $i\neq j$, terminate with failure.
		\item For each connected component $C$ of $\mathcal{L}(A, B)$ that contains a cell in $\Cyl_i^-\cup\II_{\bar{i}}$, label each element of $C$ by $i$.
		\item For each remaining connected component $C$ of $\mathcal{L}(A, B)$, label each element of $C$ by the same freely chosen element of $\mathbb{P}\left(\frac{\mathbb{C}\cdot\mathbf{1}\oplus\mathbb{C}\cdot\mathbf{2}\oplus\mathbb{C}\cdot\mathbf{3}}{\mathbb{C}\cdot(1, 1, 1)}\right)$.
	\end{enumerate}
\label{algorithm:AB labelling algorithm}
\end{algorithm}

\begin{remark}
When the context is clear, we will denote $\mathbb{P}\left(\frac{\mathbb{C}\cdot\mathbf{1}\oplus\mathbb{C}\cdot\mathbf{2}\oplus\mathbb{C}\cdot\mathbf{3}}{\mathbb{C}\cdot(1, 1, 1)}\right)$ by $\mathbb{P}^1$. We will also use $\langle z_1, z_2, z_3\rangle_w$ to denote $\vspan\left\{z_1\mathbf{1}_w+z_2\mathbf{2}_w+z_3\mathbf{3}_w+\mathbb{C}\cdot(1, 1, 1)_w\right\}\in\mathbb{P}\left(\frac{\mathbb{C}\cdot\mathbf{1}_w\oplus\mathbb{C}\cdot\mathbf{2}_w\oplus\mathbb{C}\cdot\mathbf{3}_w}{\mathbb{C}\cdot(1, 1, 1)_w}\right)$ and $\langle z_1, z_2, z_3\rangle$ to denote $\vspan\left\{z_1\mathbf{1}+z_2\mathbf{2}+z_3\mathbf{3}+\mathbb{C}\cdot(1, 1, 1)\right\}\in\mathbb{P}\left(\frac{\mathbb{C}\cdot\mathbf{1}\oplus\mathbb{C}\cdot\mathbf{2}\oplus\mathbb{C}\cdot\mathbf{3}}{\mathbb{C}\cdot(1, 1, 1)}\right)$. 
\end{remark}

\begin{defn}
For $i\in\{1, 2, 3\}$, if $w\in\Cyl_i^-\cup\II_{\bar{i}}$, set $\ell(w):=i$.
\end{defn}

\begin{lemma}
	\label{remark:definition of ell}
If $w\in\I^-\cup\II$ is labelled at any point in Algorithm~\ref{algorithm:AB labelling algorithm}, then it is labelled by $\ell(w)$. 
\end{lemma}
%\begin{remark}

\begin{proof}
Let $w\in\I^-\cup\II$. Suppose $w$ is labelled at some point in Algorithm~\ref{algorithm:AB labelling algorithm}. Then $w$ is an element of some connected component $C$ of $\mathcal{L}(A, B)$. If $w\in\I^-$, then $w\in\Cyl_i^-$ for some $i\in\{1, 2, 3\}$, so $w$ is labelled by $i$ in step 2 of Algorithm~\ref{algorithm:AB labelling algorithm}, and $\ell(w)=i$. Otherwise, $w\in\II$, so $w\in\II_{\bar{i}}$ for some $i\in\{1, 2, 3\}$. In this case, $w$ is labelled by $i$ in step 2 of Algorithm~\ref{algorithm:AB labelling algorithm} and $\ell(w)=i$. %To summarize, 
\end{proof}
%\end{remark}

\begin{defn}
Given $(A, B)\in\sAB_{\text{all}}$ and a connected component $C$ of $\mathcal{L}(A, B)$, let \[\mathcal{N}(C)=\left\lvert\left\{\ell\left(w\right)\mid w\in C\cap\left(\I^-\cup\II\right)\right\}\right\rvert.\]
\end{defn}

\begin{remark}
\label{remark:N(C)}
Let $(A, B)\in\sAB_{\text{all}}$. Algorithm~\ref{algorithm:AB labelling algorithm} terminates if and only if there is a connected component $C$ of $\mathcal{L}(A, B)$ such that $\mathcal{N}(C)>1$. Moreover, if Algorithm~\ref{algorithm:AB labelling algorithm} does not terminate, then a connected component $C$ of $\mathcal{L}(A, B)$ is labelled in step 2 if and only if $\mathcal{N}(C)=1$, and $C$ is labelled in step 3 if and only if $\mathcal{N}(C)=0$. Finally, if $w$ is labelled in step 3 of Algorithm~\ref{algorithm:AB labelling algorithm}, then $w\in C$, where $C$ is a connected component of $\mathcal{L}(A, B)$ that does not contain any cells in $\Cyl_i^-\cup\II_{\bar{i}}$ for any $i\in\{1, 2, 3\}$, so \[w\in C\subseteq\mathcal{L}(A, B)\setminus\left(\bigcup_{i=1}^3\Cyl_i^-\cup\II_{\bar{i}}\right)=\mathcal{L}(A, B)\setminus(\I^-\cup\II)\subseteq\III\cap(A\triangle B).\] 
\end{remark}

Because Algorithm~\ref{algorithm:AB labelling algorithm} may fail in step 1, there are $AB$ configurations that cannot be labelled. 

\begin{defn}
A \emph{labelled $AB$ configuration} is an $AB$ configuration for which Algorithm~\ref{algorithm:AB labelling algorithm} succeeds.
\end{defn}

\begin{figure}[htb]
\centering
\includegraphics[width=1in]{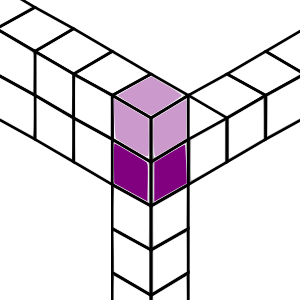}
	\hspace{0.25in}
\includegraphics[width=1in]{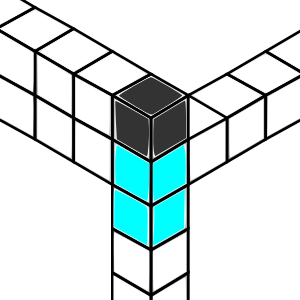}
	\hspace{0.25in}
\includegraphics[width=1in]{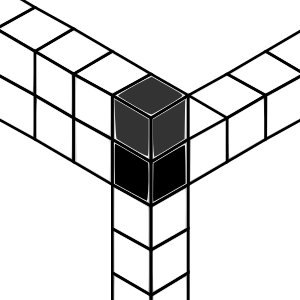}
	\hspace{0.25in}
\includegraphics[width=1in]{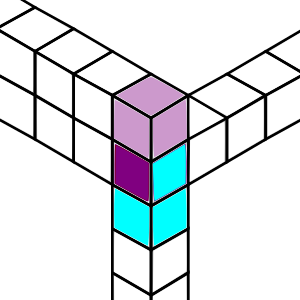}
\caption{The $AB$ configurations from Example~\ref{ex:ABlabellings}.}
\label{fig:ABlabellings}
\end{figure}

\begin{example}
\label{ex:ABlabellings}
As in Example~\ref{ex:labelledboxconfigsfrompt}, let $\mu_1 = (1), \mu_2 = (2),$ and $\mu_3 = (1)$, so $\III = \{ (0, 0, 0)\}$ and $\II = \II_{\bar{1}} = \{ (0, 0, 1)\}$. In Figure~\ref{fig:ABlabellings}, we illustrate four $AB$ configurations, three of which are labelled $AB$ configurations. The first three of these configurations appear in Example~\ref{ex:labelledboxconfigsfrompt} as the configuration (1) with $|\pi| = 1$, the configuration (4) with $|\pi| = 3$, and the configuration (3) with $|\pi| = 2$. 

%as the length 1 configuration (i), the length 3 configuration (iv), and the length 2 configuration (iii).}

\begin{enumerate}
\item %In the first $AB$ configuration, 
$A$ consists of a single box at $(0, 0, 0)$ and $B = \varnothing$. Step 2 of Algorithm~\ref{algorithm:AB labelling algorithm} gives the connected component consisting of cells $(0, 0, 0)$ and $(0, 0, 1)$ the label 1, which is indicated by the color purple. The cell $(0, 0, 0)$ is opaque because it is a box; the cell $(0, 0, 1)$ is not. 
\item %In the second $AB$ configuration, 
$A = \{ (0, 0, 0), (0, 0, -1)\}$ and $B = \{ (0, 0, 1)\}$. The box in $B$ is not in the labelling set. Step 2 labels the cells in $A$ by 3, which we illustrate by coloring the two boxes cyan. The box at $(0, 0, 1)$ is colored gray because it does not get a label.
\item %In the third $AB$ configuration, 
$A = \varnothing$ and $B =\{ (0, 0, 0), (0, 0, 1)\}$. Again, the box at $(0, 0, 1)$ is not in the labelling set. The box at $(0, 0, 0)$ has a freely chosen label in $\mathbb{P}^1$.
\item %In the final $AB$ configuration, 
$B = \varnothing$ and $A = \{ (0, 0, 0), (0, 0, -1)\}$. The algorithm terminates with failure in step 1 because $(0, 0, -1) \in\Cyl_{3}^{-}$ and $(0, 0, 1) \in \II_{\bar{1}}$, and these cells are in the same connected component. In the figure, $(0, 0, 0)$ is colored both cyan, required by the box at $(0, 0, -1)$, and purple, required by the cell at $(0, 0, 1)$. 
\end{enumerate}
\end{example}

\begin{figure}[htb]
\centering
\includegraphics[width=1.5in]{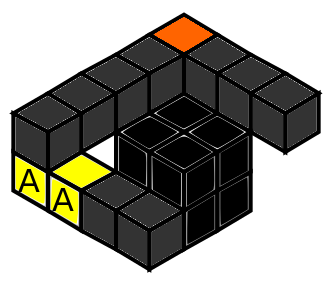}
\includegraphics[width=1.5in]{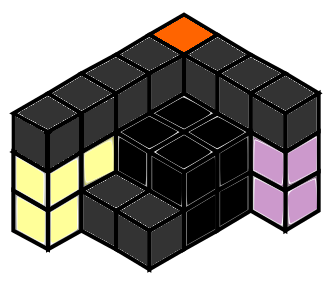}
\caption{The $AB$ configuration from Example~\ref{ex:spicy}.}
\label{fig:tripartiteex3}
\end{figure}

\begin{example}
\label{ex:spicy}
Figure~\ref{fig:tripartiteex3} shows a labelled $AB$ configuration with $\mu_1 = (3, 3, 1)$, $\mu_2 = (3, 2, 2, 1)$, and $\mu_3 = (5, 3, 3, 1)$. The left image shows the configuration. The boxes belonging to $A$ are marked; all other boxes are in $B$. The right image includes surrounding cells in $\II$. In both images, yellow cells are labelled 2 and purple cells are labelled 1. Opaque cells are boxes in the configuration and transparent cells are not. The two connected components of $\mathcal{L}(A, B)$ labelled by freely chosen elements of $\mathbb{P}^1$ are colored black and orange, respectively. 
\end{example}

\subsubsection{Projection to the base \texorpdfstring{$AB$}{AB} configuration}
\label{sec:box_config_maps}

Given a labelled $AB$ configuration $(A, B)$, we can define a set of labelled boxes $\pi(A, B)$ as follows. 

\begin{defn}
\label{defn:pi(A, B)}
Take $A\cup B$ to be the set of boxes of $\pi(A, B)$ and label type $\III$ boxes using the labels specified by Algorithm~\ref{algorithm:AB labelling algorithm}. More precisely, given a connected component $C$ of $\mathcal{L}(A, B)$, if Algorithm~\ref{algorithm:AB labelling algorithm} labels $C$ by $i\in\{1, 2, 3\}$, let the label of $w\in\III\cap C$ in $\pi(A, B)$ be $\vspan\left\{\mathbf{i}_w+\mathbb{C}\cdot(1, 1, 1)_w\right\}$, while if Algorithm~\ref{algorithm:AB labelling algorithm} labels $C$ by a freely chosen element $\langle z_1, z_2, z_3\rangle$ of $\mathbb{P}^1$, let the label of $w\in\III\cap C$ in $\pi(A, B)$ be the same freely chosen element $\langle z_1, z_2, z_3\rangle_w$ of $\mathbb{P}^1_w$. 
\end{defn}

Define a map $P:\sAB_{\text{all}} \rightarrow \sAB_{\text{all}}$ by letting $P(A,B)$ be the $AB$ configuration obtained by moving all multiplicity 1 type $\III$ boxes into $A$. That is, let \[P(A, B)=(A\cup(\III\cap(B\setminus A)), B\setminus(\III\cap(B\setminus A))).\]

\begin{lemma}
	\label{lemma:well-definedness of P}
This map is well-defined.  In fact, $P$ takes every element of $\sAB(\pi)$ to $AB_{\base}(\pi)$.
\end{lemma}

\begin{proof}
Given $(A, B)\in\sAB_{\text{all}}$, let $(A', B')=(A\cup(\III\cap(B\setminus A)), B\setminus(\III\cap(B\setminus A)))$. We need to see that $(A', B')\in\sAB_{\text{all}}$.

First, $A'\subseteq A\cup\III\subseteq\I^-\cup\III$ and $B'\subseteq B\subseteq\II\cup\III$ are finite, since $A'\subseteq A\cup B$ and since $A$ and $B$ are both finite. Also, note that $A\subseteq A'$. To check Conditions~\ref{conditions:ab box stacking}.1, suppose $w\in\I^-\cup\III$ and $n\in BN(w)\cap A'$.  If $n\in A$, then $w\in A\subseteq A'$, by Conditions~\ref{conditions:ab box stacking}.1 and the fact that $(A, B)$ is an $AB$ configuration. Otherwise, $n\in A'\setminus A$, i.e., $n\in\III\cap(B\setminus A)$. Then $n\in\III\subseteq\mathbb{Z}^3_{\geq 0}$, so $w\in\mathbb{Z}^3_{\geq 0}$. Since $w\in\I^-\cup\III$, it follows that $w\in\III$, so by Conditions~\ref{conditions:ab box stacking}.2 and the fact that $(A, B)$ is an $AB$ configuration, $w\in\III\cap B\subseteq\III\cap(A\cup B)\subseteq A'$.

Similarly, to check Conditions~\ref{conditions:ab box stacking}.2, suppose $w\in\II\cup\III$ and $n\in BN(w)\cap B'$. Since $B'\subseteq B$, $w\in B$, by Conditions~\ref{conditions:ab box stacking}.2 and the fact that $(A, B)$ is an $AB$ configuration. If $w\in\II$, then $w\in B'$. Otherwise, $w\in\III$. By Lemma~\ref{lemma:type III back neighbors}, $n\in\I^-\cup\III$. However, $n\in B'\subseteq\II\cup\III$.  Thus, $n\in\III$. Since $n\in B'$, $n\in\III\cap B'\subseteq\III\cap B\setminus(B\setminus A)\subseteq A\cap B$. In particular, $n\in A$, so by Conditions~\ref{conditions:ab box stacking}.1 and the fact that $(A, B)$ is an $AB$ configuration, $w\in A$, i.e., $w\in A\cap B\subseteq B'$.

Finally, suppose $(A, B)\in\sAB(\pi)$.  The fact that $(A', B')\in\sAB(\pi)$ is a consequence of the equalities $A'\cup B'=A\cup B$, and $A'\cap B'=A\cap B$, which are both clear.  We claim that $(A', B')=AB_{\base}(\pi)$.  To see this, we must show that $A'$ consists of the type $\I^-$ and type $\III$ boxes of $\pi$, while $B'$ consists of the type $\II$ and unlabelled type $\III$ boxes of $\pi$.  Since $(A', B')$ is an $AB$ configuration on $\pi$, we have $A'\subseteq\I^-\cup\III$ and $B'\subseteq\II\cup\III$, and by Conditions~\ref{conditions:ab on pi}.1, $A'$ must contain all type $\I^-$ boxes of $\pi$, while $B'$ must contain all type $\II$ boxes of $\pi$. Also, by Conditions~\ref{conditions:ab on pi}.2, we know that $A'$ and $B'$ contain all unlabelled type $\III$ boxes of $\pi$. So, by Conditions~\ref{conditions:ab on pi}.1 and since $A'\subseteq\I^-\cup\III$ and $B'\subseteq\II\cup\III$, $(A', B')=AB_{\base}(\pi)$ if $A'$ contains all labelled type $\III$ boxes of $\pi$ and any box $w\in B'\cap\III$ is unlabelled. For the first statement, if $w$ is a labelled type $\III$ box of $\pi$, then by Conditions~\ref{conditions:ab on pi}, $w\in\III\cap((A\cup B)\setminus(A\cap B))=\III\cap((A\setminus B)\cup(B\setminus A))\subseteq(A\setminus B)\cup(\III\cap(B\setminus A))\subseteq A'$. For the second statement, if $w\in B'\cap\III$ is a labelled box of $\pi$, then by Conditions~\ref{conditions:ab on pi}, $w\in B'\cap((A'\cup B')\setminus(A'\cap B'))=B'\cap((A'\setminus B')\cup(B'\setminus A'))\subseteq B'\setminus A'$, so $w\not\in A'\supseteq A$. This in turn implies that $w\in\III\cap(B'\setminus A)\subseteq\III\cap(B\setminus A)\subseteq A'$, a contradiction.
\end{proof}

Let $ \sAB = P^{-1}(\sAB_{\base})$. Clearly, \[\sAB_{\base} \subseteq \bigcup_{\pi\in\text{PT-box}}\sAB(\pi) \subseteq \sAB \subseteq \sAB_{\text{all}}.\]  In fact, the following lemma shows that $\bigcup\limits_{\pi\in\text{PT-box}}\sAB(\pi) = \sAB$.  Moreover, as defined, $P|_{\sAB}$ is a surjection from $\sAB$ onto $\sAB_{\base}$.

\begin{lemma}
	\label{lemma:sAB as union}
We have \[\sAB=\bigcup_{\pi\in\text{PT-box}}\sAB(\pi).\]  More precisely, $P^{-1}(AB_{\base}(\pi))=\sAB(\pi)$.  
\end{lemma}

\begin{proof}
By Lemma~\ref{lemma:well-definedness of P}, $\sAB(\pi)\subseteq P^{-1}(AB_{\base}(\pi))$.  Conversely, suppose $(A, B)\in P^{-1}(AB_{\base}(\pi))$, that is, $(A, B)$ is an $AB$ configuration such that $(A', B'):=P(A, B)=AB_{\base}(\pi)$.  To show that $(A, B)\in\sAB(\pi)$, we just need to check that Conditions~\ref{conditions:ab on pi} hold.  Since $(A', B')=AB_{\base}(\pi)\in\sAB(\pi)$, we have that $A\cup B=A'\cup B'$ is the set of boxes in $\pi$ and $A\cap B=A'\cap B'$ is the set of unlabelled type $\III$ boxes in $\pi$, as desired.  Thus, $P^{-1}(AB_{\base}(\pi))\subseteq\sAB(\pi)$.  Finally, \begin{align*}\sAB=P^{-1}(\sAB_{\base})&=P^{-1}\left(\bigcup_{\pi\in\text{PT-box}}\left\{AB_{\base}(\pi)\right\}\right)\\&=\bigcup_{\pi\in\text{PT-box}}P^{-1}(AB_{\base}(\pi))=\bigcup_{\pi\in\text{PT-box}}\sAB(\pi).\end{align*}
\end{proof}

\begin{lemma}
\label{lemma:adjacent labels}
Suppose $\pi$ is a labelled box configuration, $(A, B)\in\sAB(\pi)$, and $w\in\III\cap(A\triangle B)$ is a box that is adjacent to a cell $n\in\mathcal{L}(A, B)$. If $n\in\Cyl_l^-\cup\II_{\bar{l}}$ for some $l\in\{1, 2, 3\}$, then the label of $w$ in $\pi$ is $\vspan\{\mathbf{l}_w+\mathbb{C}\cdot(1, 1, 1)_w\}$. If $n\in\III\cap(A\triangle B)\cap BN(w)$, then $n$ is a labelled type $\III$ box of $\pi$, and if the label of $n$ in $\pi$ is $\vspan\left\{z_1\mathbf{1}_n+z_2\mathbf{2}_n+z_3\mathbf{3}_n+\mathbb{C}\cdot(1, 1, 1)_n\right\}$, then the label of $w$ in $\pi$ is $\langle z_1, z_2, z_3\rangle_w$.
\end{lemma}

\begin{proof}
By Conditions~\ref{conditions:ab on pi}, $w$ is a labelled type $\III$ box of $\pi$. Suppose $n\in\Cyl_l^-$ for some $l\in\{1, 2, 3\}$. Then $n\in\I^-\cap A$ and $n\not\in\mathbb{Z}^3_{\geq 0}$. Since $w\in\III\subseteq\mathbb{Z}^3_{\geq 0}$, $n\in BN(w)$. Since $(A, B)\in\sAB(\pi)$, $n$ is a box of $\pi$, by Conditions~\ref{conditions:ab on pi}.1. Then, note that the span $S$ of subspaces of $\frac{\mathbb{C} \cdot \mathbf{1}_w \oplus\mathbb{C} \cdot \mathbf{2}_w \oplus\mathbb{C} \cdot \mathbf{3}_w}{\mathbb{C} \cdot (1,1,1)_w}$ induced by boxes in $BN(w)$ contains the subspace $\vspan\{\mathbf{l}_w+\mathbb{C}\cdot(1, 1, 1)_w\}$, so $S$ is that subspace or $S$ is $2$-dimensional. By Conditions~\ref{conditions:labelled box stacking}.3, it follows that the label of $w$ in $\pi$ is $\vspan\{\mathbf{l}_w+\mathbb{C}\cdot(1, 1, 1)_w\}$ or $w$ is an unlabelled box of $\pi$. In the latter case, by Conditions~\ref{conditions:ab on pi}.2, $w\in A\cap B$, contradicting the fact that $w\in A\triangle B$. So, the former statement must hold.

Suppose $n\in\II_{\bar{l}}$ for some $l\in\{1, 2, 3\}$. Then $n\in\II\setminus B$. By Lemma~\ref{lemma:type III back neighbors}, $n\not\in BN(w)$, so $w\in BN(n)$. Since $A\subseteq\I^-\cup\III$, $n\not\in A$, so $n\not\in A\cup B$. Since $(A, B)\in\sAB(\pi)$, $n$ is not a box of $\pi$, by Conditions~\ref{conditions:ab on pi}.1. By Conditions~\ref{conditions:labelled box stacking}.2, the label of $w$ in $\pi$ must be $\vspan\{\mathbf{l}_w+\mathbb{C}\cdot(1, 1, 1)_w\}$. 

Finally, suppose $n\in\III\cap(A\triangle B)\cap BN(w)$. Then, by Conditions~\ref{conditions:ab on pi}, $n$ is a labelled type $\III$ box of $\pi$. Let $\ell_w$ denote the label of $w$ in $\pi$ and $\vspan\left\{z_1\mathbf{1}_n+z_2\mathbf{2}_n+z_3\mathbf{3}_n+\mathbb{C}\cdot(1, 1, 1)_n\right\}$ be the label of $n$ in $\pi$. Since $n\in BN(w)$, the span $S$ of subspaces of $\frac{\mathbb{C} \cdot \mathbf{1}_{w} \oplus\mathbb{C} \cdot \mathbf{2}_{w} \oplus\mathbb{C} \cdot \mathbf{3}_{w}}{\mathbb{C} \cdot (1,1,1)_{w}}$ induced by boxes in $BN(w)$ contains the subspace $\langle z_1, z_2, z_3\rangle_w$ induced by $n$. By Conditions~\ref{conditions:labelled box stacking}.3, $S$ is $1$-dimensional and $\ell_w=S$. Thus, $\ell_w=S=\langle z_1, z_2, z_3\rangle_w$.
\end{proof}

\begin{thm}
	\label{thm:labellable iff in sAB}
Given an $AB$ configuration $(A, B)$, Algorithm~\ref{algorithm:AB labelling algorithm} succeeds if and only if $(A, B) \in \sAB$.
\end{thm}

\begin{proof}
Let $(A, B)\in\sAB_{\text{all}}$. Suppose Algorithm~\ref{algorithm:AB labelling algorithm} succeeds. By Lemma~\ref{lemma:sAB as union}, to show that $(A, B)\in\sAB$, it suffices to find a labelled box configuration $\pi$ such that $(A, B)\in\sAB(\pi)$. To achieve this, we will show that $\pi(A, B)$ satisfies Conditions~\ref{conditions:labelled box stacking}, and then show that Conditions~\ref{conditions:ab on pi} hold. \\

\noindent \underline{Conditions~\ref{conditions:labelled box stacking}.}
Suppose $w\in\I^-$ and $n\in BN(w)\cap(A\cup B)$. Since $w\not\in\mathbb{Z}^3_{\geq 0}$, $n\not\in\mathbb{Z}^3_{\geq 0}$, so $n\not\in\II\cup\III$, implying that $n\in A$. Then, by Conditions~\ref{conditions:ab box stacking}.1, $w\in A\subseteq A\cup B$. 

Next, suppose $w\in\II_{\bar i}$ and $n\in BN(w)\cap(A\cup B)$ is not a type $\III$ box labelled $\vspan\{\mathbf{i}_n+\mathbb{C}\cdot(1, 1, 1)_n\}$. If $n\in B$, then by Conditions~\ref{conditions:ab box stacking}.2, $w\in B\subseteq A\cup B$. Otherwise, $n\not\in B$, so $n\in A\setminus B$. Then $n\in\I^-\cup\III$. If $n\in\I^-$, by Lemma~\ref{lemma:adjacent types I- and II}, $n\in\Cyl_j^-$ for some $j\in\{1, 2, 3\}\setminus\{i\}$. Since $n\in\I^-\cap A\subseteq\mathcal{L}(A, B)$ and Algorithm~\ref{algorithm:AB labelling algorithm} does not terminate at step 1, $w\in\II\setminus\mathcal{L}(A, B)\subseteq B\subseteq A\cup B$. Otherwise, $n\in\III$. In this case, suppose $w\not\in A\cup B$. Then $w\in\II\setminus B\subseteq\mathcal{L}(A, B)$ and $n\in\III\cap(A\setminus B)\subseteq\III\cap(A\triangle B)\subseteq\mathcal{L}(A, B)$, so Algorithm~\ref{algorithm:AB labelling algorithm} assigns a label of $i$ to $n$ at step 2. However, by Definition~\ref{defn:pi(A, B)}, the label of $n$ in $\pi(A, B)$ is $\vspan\{\mathbf{i}_n+\mathbb{C}\cdot(1, 1, 1)_n\}$, contradicting our assumption, so $w\in A\cup B$. 

Now, suppose $w\in\III$ and the span $S$ of subspaces of $\frac{\mathbb{C} \cdot \mathbf{1}_w \oplus\mathbb{C} \cdot \mathbf{2}_w \oplus\mathbb{C} \cdot \mathbf{3}_w}{\mathbb{C} \cdot (1,1,1)_w}$ induced by boxes in $BN(w)\cap(A\cup B)$ is nonzero. In this case, $BN(w)\cap(A\cup B)\neq\varnothing$, so by Conditions~\ref{conditions:ab box stacking}.2, $w\in A\cup B$. By Lemma~\ref{lemma:type III back neighbors}, $BN(w)\subseteq\I^-\cup\III$, so \[BN(w)\cap(A\cup B)\subseteq(\I^-\cup\III)\cap(A\cup B)=(\I^-\cap(A\cup B))\cup(\III\cap(A\cup B))\subseteq(\I^-\cap A)\cup(\III\cap(A\cup B)).\] 

Suppose the dimension of $S$ is $1$. Then no cell in $BN(w)\cap(A\cup B)$ is left unlabelled by Algorithm~\ref{algorithm:AB labelling algorithm}, for any such cell must be an unlabelled type $\III$ box in $\pi(A, B)$, and such boxes induce the whole $2$-dimensional space $\frac{\mathbb{C} \cdot \mathbf{1}_w \oplus\mathbb{C} \cdot \mathbf{2}_w \oplus\mathbb{C} \cdot \mathbf{3}_w}{\mathbb{C} \cdot (1,1,1)_w}$. As a result, $BN(w)\cap(A\cup B)\subseteq\mathcal{L}(A, B)$. We must show that the label of $w$ in $\pi(A, B)$ is $S$ or $w$ is unlabelled in $\pi(A, B)$. Suppose $w$ is not unlabelled in $\pi(A, B)$. Then Algorithm~\ref{algorithm:AB labelling algorithm} must assign a label to $w$, so $w\in\mathcal{L}(A, B)$. Thus, since $w$ is adjacent to each cell in $BN(w)\cap(A\cup B)$, $\{w\}\cup(BN(w)\cap(A\cup B))$ is contained in a single connected component $C$ of $\mathcal{L}(A, B)$, so Algorithm~\ref{algorithm:AB labelling algorithm} assigns the same label $\ell$ to each element of $\{w\}\cup(BN(w)\cap(A\cup B))$. 

Let $n\in BN(w)\cap(A\cup B)$. Since $BN(w)\cap(A\cup B)\subseteq(\I^-\cap A)\cup(\III\cap(A\cup B))$, either $n\in\I^-\cap A$, so $n\in\Cyl_i^-$ for some $i\in\{1, 2, 3\}$ and $\ell=i$, or $n\in\III\cap(A\cup B)$. In the first case, $n$ induces the subspace $\vspan\{\mathbf{i}_w+\mathbb{C}\cdot(1, 1, 1)_w\}$ of $\frac{\mathbb{C} \cdot \mathbf{1}_w \oplus\mathbb{C} \cdot \mathbf{2}_w \oplus\mathbb{C} \cdot \mathbf{3}_w}{\mathbb{C} \cdot (1,1,1)_w}$, so $\vspan\{\mathbf{i}_w+\mathbb{C}\cdot(1, 1, 1)_w\}\subseteq S$, but since $S$ is $1$-dimensional, $\vspan\{\mathbf{i}_w+\mathbb{C}\cdot(1, 1, 1)_w\}=S$. Then, since $w\in\III\cap(A\cup B)$ and Algorithm~\ref{algorithm:AB labelling algorithm} labels $w$ by $\ell=i\in\{1, 2, 3\}$, the label of $w$ in $\pi(A, B)$ is $\vspan\{\mathbf{i}_w+\mathbb{C}\cdot(1, 1, 1)_w\}=S$, according to Definition~\ref{defn:pi(A, B)}. In the second case, since $n, w\in\III\cap(A\cup B)$ and Algorithm~\ref{algorithm:AB labelling algorithm} labels $n, w\in\{w\}\cup(BN(w)\cap(A\cup B))$ by $\ell$, either $\ell\in\{1, 2, 3\}$ and the labels of $n$ and $w$ in $\pi(A, B)$ are $\vspan\{\boldsymbol{\ell}_n+\mathbb{C}\cdot(1, 1, 1)_n\}$ and $\vspan\{\boldsymbol{\ell}_w+\mathbb{C}\cdot(1, 1, 1)_w\}$, or $\ell$ is a freely chosen element $\langle z_1, z_2, z_3\rangle$ of $\mathbb{P}^1$ and the labels of $n$ and $w$ in $\pi(A, B)$ are the same freely chosen elements $\ell_n:=\vspan\left\{z_1\mathbf{1}_n+z_2\mathbf{2}_n+z_3\mathbf{3}_n+\mathbb{C}\cdot(1, 1, 1)_n\right\}$ and $\ell_w:=\langle z_1, z_2, z_3\rangle_w$. Then $n$ induces the subspace $\vspan\{\boldsymbol{\ell}_w+\mathbb{C}\cdot(1, 1, 1)_w\}$ or $\ell_w$, respectively, of $\frac{\mathbb{C} \cdot \mathbf{1}_w \oplus\mathbb{C} \cdot \mathbf{2}_w \oplus\mathbb{C} \cdot \mathbf{3}_w}{\mathbb{C} \cdot (1,1,1)_w}$, so $\vspan\{\boldsymbol{\ell}_w+\mathbb{C}\cdot(1, 1, 1)_w\}\subseteq S$ or $\ell_w\subseteq S$, respectively. Since $S$ is $1$-dimensional, $\vspan\{\boldsymbol{\ell}_w+\mathbb{C}\cdot(1, 1, 1)_w\}=S$ or $\ell_w=S$, respectively. That is, the label of $w$ in $\pi(A, B)$ is $S$. 

Suppose the dimension of $S$ is $2$. We must show that $w$ is an unlabelled box of $\pi(A, B)$. In other words, we must show that $w\not\in\mathcal{L}(A, B)$. If $BN(w)\cap A\cap B\neq\varnothing$, then by Conditions~\ref{conditions:ab box stacking}, $w\in\III\cap A\cap B$, so $w\not\in\mathcal{L}(A, B)$. Otherwise, $BN(w)\cap A\cap B=\varnothing$. In this case, since $BN(w)\cap(A\cup B)\subseteq(\I^-\cap A)\cup(\III\cap(A\cup B))$, we have $BN(w)\cap(A\cup B)\subseteq(\I^-\cap A)\cup(\III\cap(A\triangle B))\subseteq\mathcal{L}(A, B)$. Suppose $w\in\mathcal{L}(A, B)$. Then, since $w$ is adjacent to each cell in $BN(w)\cap(A\cup B)$, $\{w\}\cup(BN(w)\cap(A\cup B))$ is contained in a single connected component $C$ of $\mathcal{L}(A, B)$, so Algorithm~\ref{algorithm:AB labelling algorithm} assigns the same label $\ell$ to each element of $\{w\}\cup(BN(w)\cap(A\cup B))$. Either $\ell\in\{1, 2, 3\}$ or $\ell$ is a freely chosen element $\langle z_1, z_2, z_3\rangle$ of $\mathbb{P}^1$. By the arguments given in the previous paragraph, in the first case, each element of $BN(w)\cap(A\cup B)$ induces the subspace $\vspan\{\boldsymbol{\ell}_w+\mathbb{C}\cdot(1, 1, 1)_w\}$ of $\frac{\mathbb{C} \cdot \mathbf{1}_w \oplus\mathbb{C} \cdot \mathbf{2}_w \oplus\mathbb{C} \cdot \mathbf{3}_w}{\mathbb{C} \cdot (1,1,1)_w}$, and in the second case, each element of $BN(w)\cap(A\cup B)$ induces the same freely chosen element $\ell_w:=\langle z_1, z_2, z_3\rangle_w$ of $\mathbb{P}^1_w$. In the first case, $S=\vspan\{\boldsymbol{\ell}_w+\mathbb{C}\cdot(1, 1, 1)_w\}$, and in the second case, $S=\ell_w$. In either case, $S$ is $1$-dimensional. By contradiction, $w\not\in\mathcal{L}(A, B)$. \\

\noindent \underline{Conditions~\ref{conditions:ab on pi}.}
Conditions~\ref{conditions:ab on pi}.1 holds by construction.  For Conditions~\ref{conditions:ab on pi}.2, suppose $w\in A\cap B$.  Then, since $A\subseteq\I^-\cup\III$ and $B\subseteq\II\cup\III$, $w\in(\I^-\cup\III)\cap(\II\cup\III)\subseteq\III$, which means that $w\not\in\mathcal{L}(A, B)$. Therefore, $w$ is an unlabelled box of $\pi(A, B)$. Conversely, suppose $w$ is an unlabelled type $\III$ box of $\pi(A, B)$. Then $w\in\III\cap(A\cup B)\setminus\mathcal{L}(A, B)\subseteq A\cap B$. \\

For the converse, suppose $(A, B)\in\sAB$. Then, by Lemma~\ref{lemma:sAB as union}, $(A, B)\in\sAB(\pi)$ for some $\pi\in\text{PT-box}$. We must show that Algorithm~\ref{algorithm:AB labelling algorithm} succeeds, i.e., we must show that it passes step 1. Suppose not. Then a connected component $C$ of $\mathcal{L}(A, B)$ contains a cell $w_i\in\Cyl_i^-\cup\II_{\bar{i}}$ and a cell $w_j\in\Cyl_j^-\cup\II_{\bar{j}}$, where $i\neq j$. 

Suppose $w_i$ is adjacent to $w_j$. Without loss of generality, assume $w_i\in BN(w_j)$. Observe that $\Cyl_i^-$ is not adjacent to $\Cyl_j^-$, because $\Cyl_i^-$ and $\Cyl_j^-$ are subsets of non-adjacent octants of $\mathbb{Z}^3$, so at least one of $w_i$ and $w_j$ is a type $\II$ cell. In fact, if $w_i\in\II\subseteq\mathbb{Z}^3_{\geq 0}$, since $w_i\in BN(w_j)$, we have $w_j\in\mathbb{Z}^3_{\geq 0}$. Then $w_j\not\in\I^-$, in which case, $w_j\in\II$. In any case, we deduce that $w_j\in\II$, so $w_j\in\II_{\bar{j}}$. Suppose $w_i\in\II$. Then, by Lemma~\ref{lemma:adjacent type II}, $w_i\in\II_{\bar{j}}$. Since $w_i\in\Cyl_i^-\cup\II_{\bar{i}}$ and $i\neq j$, this is a contradiction. Consequently, $w_i\not\in\II$, so $w_i\in\Cyl_i^-\subseteq\I^-$. Furthermore, $w_i, w_j\in\mathcal{L}(A, B)$, so $w_i\in A\subseteq A\cup B$, while $w_j\not\in\I^-\cup\III\cup B$, implying that $w_j\not\in A\cup B$. By Conditions~\ref{conditions:ab on pi}.1, $w_i$ is a box of $\pi$, while $w_j$ is not. On the other hand, by Conditions~\ref{conditions:labelled box stacking}.2, $w_j$ is a box of $\pi$. By contradiction, $w_i$ is not adjacent to $w_j$. In fact, since $w_i$ and $w_j$ were arbitrary, this argument shows that $C$ cannot contain two adjacent cells $w, w'\in\I^-\cup\II$ such that $\ell(w)\neq\ell(w')$. 

Since $w_i, w_j\in C$ and $C$ is a connected subset of $\mathcal{L}(A, B)$, there is a sequence of adjacent cells $w_i:=p_0, p_1, \ldots, p_r:=w_j$, each of which is an element of $C\subseteq\mathcal{L}(A, B)$. Let $0\leq t\leq r$ be the index such that $p_t$ is the last cell in this sequence that is an element of $\Cyl_i^-\cup\II_{\bar{i}}$. Then $p_t, w_j\in C$ and $p_t, p_{t+1}, \ldots, p_r=w_j$ is a sequence of adjacent cells, each of which is an element of $C$. So, without loss of generality, assume that $w_i$ is the only cell in the sequence $w_i=p_0, p_1, \ldots, p_r=w_j$ that is an element of $\Cyl_i^-\cup\II_{\bar{i}}$. Then, let $0<t'\leq r$ be the index such that $p_{t'}$ is the first cell in the sequence $p_1, p_2, \ldots, p_r=w_j$ that is an element of $\I^-\cup\II$. Since $w_i$ is the only cell in the sequence $w_i=p_0, p_1, \ldots, p_r=w_j$ that is an element of $\Cyl_i^-\cup\II_{\bar{i}}$, $p_{t'}\in(\I^-\cup\II)\setminus(\Cyl_i^-\cup\II_{\bar{i}})$, so $p_{t'}\in\Cyl_l^-\cup\II_{\bar{l}}$ for some $l\in\{1, 2, 3\}\setminus\{i\}$. Also, $w_i, p_{t'}\in C$ and $w_i=p_0, p_1, \ldots, p_{t'}$ is a sequence of adjacent cells, each of which is an element of $C$. So, without loss of generality, assume that $p_s\not\in\I^-\cup\II$ for $0<s<r$. Then, for $0<s<r$, $p_s\in\mathcal{L}(A, B)\setminus(\I^-\cup\II)\subseteq\III\cap(A\triangle B)$. Moreover, since $w_i$ is not adjacent to $w_j$, $1<r$, so $1\leq r-1$. In particular, $p_1, \ldots, p_{r-1}\in\III\cap(A\triangle B)$. 

Since $p_1\in\III\cap(A\triangle B)$ is adjacent to $p_0=w_i\in\mathcal{L}(A, B)\cap(\Cyl_i^-\cup\II_{\bar{i}})$, Lemma~\ref{lemma:adjacent labels} shows that the label of $p_1$ in $\pi$ is $\vspan\{\mathbf{i}_{p_1}+\mathbb{C}\cdot(1, 1, 1)_{p_1}\}$. Similarly, $p_{r-1}\in\III\cap(A\triangle B)$ is adjacent to $p_r=w_j\in\mathcal{L}(A, B)\cap(\Cyl_j^-\cup\II_{\bar{j}})$, so the label of $p_{r-1}$ in $\pi$ is $\vspan\{\mathbf{j}_{p_{r-1}}+\mathbb{C}\cdot(1, 1, 1)_{p_{r-1}}\}$. Since $i\neq j$, $1<r-1$. However, by Lemma~\ref{lemma:adjacent labels}, we then find that the label of $p_2$ in $\pi$ is $\vspan\{\mathbf{i}_{p_2}+\mathbb{C}\cdot(1, 1, 1)_{p_2}\}$, since $p_1\in BN(p_2)$ or $p_2\in BN(p_1)$. Then, since $i\neq j$, $2<r-1$. By repeating this argument finitely many times, we eventually see that the label of $p_{r-1}$ in $\pi$ is $\vspan\{\mathbf{i}_{p_{r-1}}+\mathbb{C}\cdot(1, 1, 1)_{p_{r-1}}\}$, contradicting the fact that $i\neq j$. This completes the proof. 
\end{proof}

\begin{corollary}
Given $(A, B)\in\sAB$, $\pi(A, B)$ is a labelled box configuration.
\end{corollary}

\begin{proof}
According to the theorem, Algorithm 1 succeeds. So, as established by the first half of the proof, $\pi(A, B)$ is a labelled box configuration.
\end{proof}

Define $\psi_{\base}:\sAB_{\base} \rightarrow \text{PT-box}$ by letting $\psi_{\base}(A, B)=\pi(A, B)$.

\begin{lemma}
	\label{lemma:AB PT correspondence}
	\begin{align*}
		\phi_{\base} \psi_{\base} &= 1_{\sAB_{\base}}; \\
		\psi_{\base} \phi_{\base} &= 1_{\text{PT-box}}.
	\end{align*}
\end{lemma}

\begin{proof}
For the second equation, we must show for all $\pi\in\text{PT-box}$, that $\psi_{\base}(\phi_{\base}(\pi))=\pi$.  However, $\phi_{\base}(\pi)=AB_{\base}(\pi)$, so we just need to show that $\psi_{\base}(AB_{\base}(\pi))=\pi$.  And, given this equation, we have \[\phi_{\base}(\psi_{\base}(AB_{\base}(\pi)))=\phi_{\base}(\pi)=AB_{\base}(\pi)\] for all $\pi\in\text{PT-box}$, thereby also establishing the first equation.  In other words, it suffices to show for all $\pi\in\text{PT-box}$, that if $(A, B):=AB_{\base}(\pi)$, then $\pi(A, B)=\pi$.  So, let $\pi\in\text{PT-box}$ and $(A, B)=AB_{\base}(\pi)$. First, since $(A, B)=AB_{\base}(\pi)\in\sAB(\pi)$, $A\cup B$ is the set of boxes in $\pi$, and $A\cap B$ is the set of unlabelled type $\III$ boxes in $\pi$.  Furthermore, from Definition~\ref{defn:pi(A, B)}, $A\cup B$ is the set of boxes of $\pi(A, B)$. Since $A\cap B\subseteq(\I^-\cup\III)\cap(\II\cup\III)\subseteq\III$, we have $A\cap B\subseteq\III\setminus\mathcal{L}(A, B)\subseteq\III\setminus(A\triangle B)$, so by Definition~\ref{defn:pi(A, B)}, $A\cap B$ is the set of unlabelled type $\III$ boxes of $\pi(A, B)$. Therefore, the set of labelled type $\III$ boxes in $\pi$ coincides with the set of labelled type $\III$ boxes of $\pi(A, B)$, and both are equal to $\III\cap(A\cup B)\setminus(A\cap B)=\III\cap(A\triangle B)$. We need only show that $\pi$ and $\pi(A, B)$ associate the same labels to each of these boxes. More precisely, given $w\in\III\cap(A\triangle B)$, we must show that the label $\ell_w$ of $w$ in $\pi$ is equal to the label of $w$ in $\pi(A, B)$. 

Suppose $w\in\III\cap(A\triangle B)$, $C$ is the connected component of $\mathcal{L}(A, B)$ containing $w$, and Algorithm~\ref{algorithm:AB labelling algorithm} labels $C$ by $i\in\{1, 2, 3\}$. Then the label of $w$ in $\pi(A, B)$ is $\vspan\{\mathbf{i}_w+\mathbb{C}\cdot(1, 1, 1)_w\}$, and $C$ contains a cell $p_0\in\Cyl_i^-\cup\II_{\bar{i}}$. Since $C$ is connected, there is a sequence of adjacent cells $p_0, p_1, \ldots, p_r:=w$, each of which is an element of $C\subseteq\mathcal{L}(A, B)$. Without loss of generality, assume that $p_0$ is the only cell in the sequence in $\Cyl_i^-\cup\II_{\bar{i}}$. Since $(A, B)$ is a labelled $AB$ configuration, $C$ contains no cells in $\Cyl_j^-\cup\II_{\bar{j}}$, for $j\neq i$, so for $0<s\leq r$, $p_s\in\mathcal{L}(A, B)\setminus(\I^-\cup\II)\subseteq\III\cap(A\triangle B)$. Since $p_1\in\III\cap(A\triangle B)$ is adjacent to $p_0\in\mathcal{L}(A, B)\cap(\Cyl_i^-\cup\II_{\bar{i}})$, Lemma~\ref{lemma:adjacent labels} shows that the label of $p_1$ in $\pi$ is $\vspan\{\mathbf{i}_{p_1}+\mathbb{C}\cdot(1, 1, 1)_{p_1}\}$. Then, if $1<r$, by Lemma~\ref{lemma:adjacent labels}, we find that the label of $p_2$ in $\pi$ is $\vspan\{\mathbf{i}_{p_2}+\mathbb{C}\cdot(1, 1, 1)_{p_2}\}$, since $p_1\in BN(p_2)$ or $p_2\in BN(p_1)$. By repeating this argument finitely many times, we eventually see that the label of $p_r$ in $\pi$ is $\vspan\{\mathbf{i}_{p_r}+\mathbb{C}\cdot(1, 1, 1)_{p_r}\}$, i.e., $\ell_w=\vspan\{\mathbf{i}_w+\mathbb{C}\cdot(1, 1, 1)_w\}$. 

Now consider the connected components of $\mathcal{L}(A, B)$ that Algorithm~\ref{algorithm:AB labelling algorithm} labels by freely chosen elements of $\mathbb{P}^1$. Since $\mathcal{L}(A, B)\subseteq A\cup\II\cup\III\subseteq A\cup[0, M-1]^3$, $\mathcal{L}(A, B)$ is finite, so there are finitely many such components, which we will denote $C_1, C_2, \ldots, C_k$. Consider one such component $C_m$. Since $C_m$ does not contain any cells in $\Cyl_i^-\cup\II_{\bar{i}}$ for any $i\in\{1, 2, 3\}$, $C_m\subseteq\mathcal{L}(A, B)\setminus(\I^-\cup\II)\subseteq\III\cap(A\triangle B)$. Suppose $w$ and $w'$ are adjacent cells in $C_m$. Then $w, w'\in\III\cap(A\triangle B)$ are labelled type $\III$ boxes in $\pi$, and by Lemma~\ref{lemma:adjacent labels}, since $w\in BN(w')$ or $w'\in BN(w)$, the labels of $w$ and $w'$ in $\pi$ must match: if the label of $w$ in $\pi$ is $\ell_w=\langle z_1, z_2, z_3\rangle_w$, then the label of $w'$ in $\pi$ must be $\vspan\left\{z_1\mathbf{1}_{w'}+z_2\mathbf{2}_{w'}+z_3\mathbf{3}_{w'}+\mathbb{C}\cdot(1, 1, 1)_{w'}\right\}$. By the connectedness of $C_m$, this implies that $C_m$ consists of labelled type $\III$ boxes in $\pi$, all of whose labels in $\pi$ match. That is, there exists $\ell_m:=\langle z_1, z_2, z_3\rangle\in\mathbb{P}^1$ such that, for all $w\in C_m$, $w$ is a labelled type $\III$ box in $\pi$ and the label of $w$ in $\pi$ is $\ell_w=\langle z_1, z_2, z_3\rangle_w$. Since Algorithm~\ref{algorithm:AB labelling algorithm} labels $C_m$ by a freely chosen element of $\mathbb{P}^1$, the label of each $w\in C_m$ in $\pi(A, B)$ is the same freely chosen element. So, it just remains to show that $\ell_m$ can be freely chosen for $1\leq m\leq k$, i.e., regardless of the values of $\ell_1, \ell_2, \ldots, \ell_k$, $\pi$ satisfies Conditions~\ref{conditions:labelled box stacking}.

Suppose there is a choice $L_1, L_2, \ldots, L_k$ of the labels $\ell_1, \ell_2, \ldots, \ell_k$ for which $\pi$ does not satisfy Conditions~\ref{conditions:labelled box stacking}, i.e., for which the corresponding labelling $\pi'$ of $\pi$ is not a labelled box configuration. Since $\pi$ satisfies Conditions~\ref{conditions:labelled box stacking} and Conditions~\ref{conditions:labelled box stacking}.1 does not refer to labels, $\pi'$ also satisfies Conditions~\ref{conditions:labelled box stacking}.1. Suppose $\pi'$ does not satisfy Conditions~\ref{conditions:labelled box stacking}.2. Then there is a cell $w\in\II_{\bar i}\setminus(A\cup B)$ and a cell $n\in BN(w)\cap(A\cup B)$ that is not a type $\III$ box whose label in $\pi'$ is $\vspan\{\mathbf{i}_n+\mathbb{C}\cdot(1, 1, 1)_n\}$. In particular, Algorithm~\ref{algorithm:AB labelling algorithm} assigns $w\in\II\setminus B\subseteq\mathcal{L}(A, B)$ the label $i$ in step 2. Furthermore, since $\pi$ satisfies Conditions~\ref{conditions:labelled box stacking}.2, it must be the case that $n$ is a type $\III$ box whose label in $\pi$ is $\vspan\{\mathbf{i}_n+\mathbb{C}\cdot(1, 1, 1)_n\}$. However, labels in $\pi$ and $\pi'$ may only differ for boxes in $\bigcup_{j=1}^k C_j$, so from this it follows that $n\in C_m$ for some $1\leq m\leq k$. Then, since $n$ and $w$ are adjacent, $w\in C_m$, which is a contradiction. We conclude that $\pi'$ satisfies Conditions~\ref{conditions:labelled box stacking}.2, so $\pi'$ does not satisfy Conditions~\ref{conditions:labelled box stacking}.3. 

Thus, there exists a cell $w\in\III$ such that (i) the span $S'$ of subspaces of $\frac{\mathbb{C} \cdot \mathbf{1}_w \oplus\mathbb{C} \cdot \mathbf{2}_w \oplus\mathbb{C} \cdot \mathbf{3}_w}{\mathbb{C} \cdot (1,1,1)_w}$ induced by boxes of $\pi'$ in $BN(w)$ is $1$-dimensional, and $w$ is neither a box whose label in $\pi'$ is $S'$ nor an unlabelled box in $\pi'$, or (ii) $S'$ is $2$-dimensional, and $w$ is not an unlabelled box in $\pi'$. In either case, $BN(w)\cap(A\cup B)\neq\varnothing$, so by Conditions~\ref{conditions:ab box stacking}, $w\in A\cup B$. As a result, $w$ is a labelled box in $\pi'$. Let the label of $w$ in $\pi'$ be $\ell'$. In case (i), $\ell'\neq S'$. Since the set of labelled type $\III$ boxes in $\pi$ is equal to the set of labelled type $\III$ boxes in $\pi'$, $w$ is a labelled box in $\pi$. Let the label of $w$ in $\pi$ be $\ell$ and let $S$ be the span of subspaces of $\frac{\mathbb{C} \cdot \mathbf{1}_w \oplus\mathbb{C} \cdot \mathbf{2}_w \oplus\mathbb{C} \cdot \mathbf{3}_w}{\mathbb{C} \cdot (1,1,1)_w}$ induced by boxes of $\pi$ in $BN(w)$. Since $BN(w)\cap(A\cup B)\neq\varnothing$, $S$ is nonzero. Then, since $\pi$ satisfies Conditions~\ref{conditions:labelled box stacking}.3, $S$ is $1$-dimensional and $\ell=S$. So, in case (i), $\ell\neq\ell'$ or $S\neq S'$, and in case (ii), $S\neq S'$. In all cases, for some $1\leq m\leq k$, $\left(\{w\}\cup BN(w)\right)\cap C_m\neq\varnothing$, since labels in $\pi$ and $\pi'$ may only differ for boxes in $\bigcup_{j=1}^k C_j$. Suppose $w\not\in C_m$. Then there exists $n\in BN(w)\cap C_m$. Since $w$ is a labelled box in $\pi$, $w\in\III\cap(A\triangle B)\subseteq\mathcal{L}(A, B)$, so because $C_m$ is a connected component of $\mathcal{L}(A, B)$ and $w$ is adjacent to $n\in C_m$, $w\in C_m$. By contradiction, $w\in C_m$. 

Then, if $n$ is a box of $\pi'$ in $BN(w)$, i.e., $n\in BN(w)\cap(A\cup B)$, Lemma~\ref{lemma:type III back neighbors} implies that $n\in\I^-\cup\III$. Suppose $n\in\I^-$. Then $n\in\I^-\cap A\subseteq\mathcal{L}(A, B)$, since $B\subseteq\II\cup\III$, so because $C_m$ is a connected component of $\mathcal{L}(A, B)$ and $n$ is adjacent to $w\in C_m$, $n\in C_m\subseteq\III\cap(A\triangle B)$, a contradiction. It follows that $n\not\in\I^-$, so $n\in\III$. Additionally, suppose $n\in A\cap B$. Then, by Conditions~\ref{conditions:ab box stacking}, $w\in A\cap B$, contradicting the fact that $w\in C_m\subseteq\III\cap(A\triangle B)$, so $n\in\III\cap(A\triangle B)\subseteq\mathcal{L}(A, B)$. Therefore, since $C_m$ is a connected component of $\mathcal{L}(A, B)$ and $n$ is adjacent to $w\in C_m$, $n\in C_m$. We deduce that, if $L_m=\langle z_1, z_2, z_3\rangle$, then $\ell'=\langle z_1, z_2, z_3\rangle_w$, and all boxes $n$ of $\pi'$ in $BN(w)$ are labelled $\vspan\left\{z_1\mathbf{1}_n+z_2\mathbf{2}_n+z_3\mathbf{3}_n+\mathbb{C}\cdot(1, 1, 1)_n\right\}$, so $S'=\ell'$. Then $S'$ is $1$-dimensional, ruling out case (ii), and in case (i), we have $\ell'\neq S'=\ell'$. By contradiction, $\pi$ satisfies Conditions~\ref{conditions:labelled box stacking}, regardless of the values of $\ell_1, \ell_2, \ldots, \ell_k$. This completes the proof that $\pi(A, B)=\pi$.
\end{proof}

\begin{lemma}
\label{lemma:AB labelling algorithm for AB(pi)}
Let $\pi\in\text{PT-box}$. For any $(A, B), (A', B')\in\sAB(\pi)$, $\mathcal{L}(A, B)=\mathcal{L}(A', B')$. Thus, the output of Algorithm~\ref{algorithm:AB labelling algorithm} is the same for all elements of $\sAB(\pi)$. 
\end{lemma}

\begin{proof}
Suppose $(A, B), (A', B')\in\sAB(\pi)$. Then $A\cup B=A'\cup B'$ is the set of boxes of $\pi$ and $A\cap B=A'\cap B'$ is the set of unlabelled type $\III$ boxes of $\pi$. Suppose $w\in\I^-\cap A$. Then, since $B'\subseteq\II\cup\III$, $w\in A\cup B=A'\cup B'$ and $w\not\in B'$, so $w\in\I^-\cap A'$. So $\I^-\cap A\subseteq\I^-\cap A'$, and by the analogous argument, $\I^-\cap A'\subseteq\I^-\cap A$, so $\I^-\cap A=\I^-\cap A'$. Suppose $w\in\II\setminus B$. Then, since $A\subseteq\I^-\cup\III$, $w\not\in A\cup B=A'\cup B'$, so $w\in\II\setminus B'$. So $\II\setminus B\subseteq\II\setminus B'$, and by the analogous argument, $\II\setminus B'\subseteq\II\setminus B$, so $\II\setminus B=\II\setminus B'$. Finally, \[\III\cap(A\triangle B)=\III\cap((A\cup B)\setminus(A\cap B))=\III\cap((A'\cup B')\setminus(A'\cap B'))=\III\cap(A'\triangle B'),\] so $\mathcal{L}(A, B)=\mathcal{L}(A', B')$. Since Algorithm~\ref{algorithm:AB labelling algorithm} only depends on the connected components of the labelling set, we conclude that the output of Algorithm~\ref{algorithm:AB labelling algorithm} is the same for all elements of $\sAB(\pi)$. 
\end{proof}

\begin{corollary}
	\label{corollary:labelling sAB(pi)}
Given $(A, B)\in\sAB(\pi)$, $\pi(A, B)=\pi$.
\end{corollary}

\begin{proof}
Let $(A', B')=AB_{\base}(\pi)$. By Definition~\ref{defn:pi(A, B)}, Conditions~\ref{conditions:ab on pi}.1, and the lemma, $\pi(A, B)=\pi(A', B')$. Then, by Lemma~\ref{lemma:AB PT correspondence}, we have \[\pi(A, B)=\pi(A', B')=\psi_{\base}(A', B')=\psi_{\base}\left(AB_{\base}(\pi)\right)=\psi_{\base}\left(\phi_{\base}(\pi)\right)=\pi,\] as desired.
\end{proof}

\begin{corollary}
	\label{corollary:sAB(pi) disjoint}
The sets $\sAB(\pi)$, for $\pi\in\text{PT-box}$, are disjoint.
\end{corollary}

\begin{proof}
Suppose $\pi_1$ and $\pi_2$ are labelled box configurations such that $(A, B)\in\sAB\left(\pi_1\right)\cap\sAB\left(\pi_2\right)$. Then, by Corollary~\ref{corollary:labelling sAB(pi)}, we have $\pi_1=\pi(A, B)=\pi_2$.
\end{proof}

\begin{lemma}
	\label{lemma:chitop}
Let $\pi\in\text{PT-box}$. If there are $k$ connected components of freely labelled type $\III$ boxes in $\pi$, then $\chi_{\text{top}}(\pi) = 2^k$.
\end{lemma}

\begin{proof}
Let $(A, B)=AB_{\base}(\pi)$. By Corollary~\ref{corollary:labelling sAB(pi)}, $\pi=\pi(A, B)$. Suppose there are $k$ connected components of freely labelled type $\III$ boxes in $\pi$. Then there are $k$ connected components of freely labelled type $\III$ boxes in $\pi(A, B)$. By Definition~\ref{defn:pi(A, B)}, the set of freely labelled type $\III$ boxes in $\pi(A, B)$ is $\III\cap(C_1\cup C_2\cup\cdots\cup C_K)$, where $C_1, C_2, \ldots, C_K$ are the connected components of $\mathcal{L}(A, B)$ that Algorithm~\ref{algorithm:AB labelling algorithm} labels by freely chosen elements of $\mathbb{P}^1$. For $1\leq m\leq K$, since Algorithm~\ref{algorithm:AB labelling algorithm} labels $C_m$ by a freely chosen element of $\mathbb{P}^1$, $C_m$ must contain no cells in $\I^-\cup\II$, so $C_m\subseteq\mathcal{L}(A, B)\setminus(\I^-\cup\II)\subseteq\III\cap(A\triangle B)\subseteq\III$. Thus, $\III\cap(C_1\cup C_2\cup\cdots\cup C_K)=C_1\cup C_2\cup\cdots\cup C_K$, so the connected components of freely labelled type $\III$ boxes in $\pi(A, B)$ are the connected components of $C_1\cup C_2\cup\cdots\cup C_K$, which are precisely $C_1, C_2, \ldots, C_K$. In particular, by Definition~\ref{defn:pi(A, B)}, there are $K=k$ independent, freely chosen labels in $\pi(A, B)$, one for each component $C_1, C_2, \ldots, C_K=C_k$. In other words, the moduli space of labellings of $\pi(A, B)$ is $\underbrace{\mathbb{P}^1\times\mathbb{P}^1\times\cdots\times\mathbb{P}^1}_{k\text{ times}}$. The topological Euler characteristic of this space is $\chi(\mathbb{P}^1)^k=2^k$, i.e., $\chi_{\text{top}}(\pi)=\chi_{\text{top}}(\pi(A, B))=2^k$. 
\end{proof}

\begin{lemma}
	Let $\pi\in\text{PT-box}$ and $(A, B)=AB_{\base}(\pi)$.  Also, let the connected components of freely labelled type $\III$ boxes in $\pi$ be denoted $C_1, C_2, \ldots, C_k$, and let \[C(\pi)=\left\{C_{j_1}\cup\cdots\cup C_{j_m}\mid 1\leq j_1<\cdots<j_m\leq k\right\}.\]  Then \[\sAB(\pi)=\left\{(A', B')\in\sAB_{\text{all}}\mid A'=A\setminus S, B'=B\cup S\text{ for some }S\in C(\pi)\right\}.\]
\end{lemma}

\begin{proof}
Let \[\cAB(\pi)=\left\{(A', B')\in\sAB_{\text{all}}\mid A'=A\setminus S, B'=B\cup S\text{ for some }S\in C(\pi)\right\}.\]  Suppose $(A', B')\in\cAB(\pi)$.  Then $A'=A\setminus S$ and $B'=B\cup S$ for some $S\in C(\pi)$.  Note that $S$ is a set of labelled type $\III$ boxes in $\pi$, so $S\subseteq A\setminus B$.  Then, to show that $(A', B')\in\sAB(\pi)$, we just observe that \[A'\cup B'=\left(A\setminus S\right)\cup\left(B\cup S\right)=A\cup B\] is the set of boxes in $\pi$, and \[A'\cap B'=\left(A\setminus S\right)\cap\left(B\cup S\right)=A\cap B\] is the set of unlabelled type $\III$ boxes in $\pi$.

Conversely, suppose $(A', B')\in\sAB(\pi)$.  To show that $(A', B')\in\cAB(\pi)$, we must find a set $S\in C(\pi)$ such that $A'=A\setminus S$ and $B'=B\cup S$.  By Lemma~\ref{lemma:sAB as union}, $\sAB(\pi)=P^{-1}(AB_{\base}(\pi))$, so $P(A', B')=(A, B)$, i.e., $(A, B)$ is obtained from $(A', B')$ by moving all multiplicity $1$ type $\III$ boxes into $A'$.  In other words, $A=A'\cup S$ and $B=B'\setminus S$, where $S=\III\cap\left(B'\setminus A'\right)$.  Then $A'=A\setminus S$ and $B'=B\cup S$, so it just remains to show that $S\in C(\pi)$.

Given $w\in S$, since $(A', B')\in\sAB(\pi)$ and $S\subseteq\III\cap(A'\triangle B')\subseteq\mathcal{L}(A', B')$, $w\in\mathcal{L}(A', B')$ is a labelled type $\III$ box in $\pi$. We claim that $w\in C_1\cup C_2\cup\cdots\cup C_k$. For this, we must show that $w$ is freely labelled. Let $\ell$ denote the label of $w$ in $\pi$, and let $C(w)$ be the connected component of $\mathcal{L}(A', B')$ containing $w$. By Lemma~\ref{lemma:AB labelling algorithm for AB(pi)}, $\mathcal{L}(A', B')=\mathcal{L}(AB_{\base}(\pi))=\mathcal{L}(A, B)$, so $C(w)$ is the connected component of $\mathcal{L}(A, B)$ containing $w$, and the output of Algorithm~\ref{algorithm:AB labelling algorithm} is the same for $(A', B')$ and $AB_{\base}(\pi)=(A, B)$. By Lemma~\ref{lemma:AB PT correspondence}, $\pi=\pi(AB_{\base}(\pi))=\pi(A, B)$. So, either Algorithm~\ref{algorithm:AB labelling algorithm} labels $C(w)$ by $i\in\{1, 2, 3\}$ and the label of $w$ in $\pi$ is $\ell=\vspan\{\mathbf{i}_w+\mathbb{C}\cdot(1, 1, 1)_w\}$, or Algorithm~\ref{algorithm:AB labelling algorithm} labels $C(w)$ by a freely chosen element $\langle z_1, z_2, z_3\rangle$ of $\mathbb{P}^1$ and the label of $w$ in $\pi$ is the same freely chosen element $\ell=\langle z_1, z_2, z_3\rangle_w$ of $\mathbb{P}^1_w$.

Suppose Algorithm~\ref{algorithm:AB labelling algorithm} labels $C(w)$ by $i\in\{1, 2, 3\}$. Then there is a cell $n\in C(w)\cap(\Cyl_i^-\cup\II_{\bar{i}})$. Since $C(w)$ is connected, there is a sequence of adjacent cells $w:=p_0, p_1, \ldots, p_r:=n$, each of which is an element of $C(w)\subseteq\mathcal{L}(A, B)$. Without loss of generality, assume that $n$ is the only cell in the sequence in $\Cyl_i^-\cup\II_{\bar{i}}$. Then, since $(A, B)$ is a labelled $AB$ configuration, $C(w)$ contains no cells in $\Cyl_j^-\cup\II_{\bar{j}}$, for $j\neq i$, so for $0\leq s<r$, $p_s\in\mathcal{L}(A, B)\setminus(\I^-\cup\II)\subseteq\III\cap(A\triangle B)$. Since $(A, B)=AB_{\base}(\pi)\in\sAB(\pi)$ and $(A', B')\in\sAB(\pi)$, \[\III\cap(A\triangle B)=\III\cap((A\cup B)\setminus(A\cap B))=\III\cap((A'\cup B')\setminus(A'\cap B'))=\III\cap(A'\triangle B'),\]  so for $0\leq s<r$, $p_s\in\III\cap(A'\triangle B')$. 

Suppose $n\in\Cyl_i^-$. Then $n\in\I^-\cap\mathcal{L}(A, B)\subseteq\I^-\cap A$, so $n\in A\cup B=A'\cup B'$. However, $B'\subseteq\II\cup\III$, so $p_r=n\in A'\setminus B'$. Since $w\in S$, $w\in B'\setminus A'$. Therefore, there exists $0\leq s<r$ such that $p_s\in B'\setminus A'$ and $p_{s+1}\in A'\setminus B'$. Then $p_s\in\III$ is adjacent to $p_{s+1}\in\I^-\cup\III$, so $p_s\in BN\left(p_{s+1}\right)$ or $p_{s+1}\in BN\left(p_s\right)$. In the first case, since $p_s\in\III\subseteq\mathbb{Z}^3_{\geq 0}$, $p_{s+1}\in\mathbb{Z}^3_{\geq 0}\cap\left(\I^-\cup\III\right)\subseteq\III\subseteq\II\cup\III$. It is easy to see that these statements contradict Conditions~\ref{conditions:ab box stacking} in both cases.

Otherwise, $n\in\II_{\bar{i}}$. Then $n\in\II\cap\mathcal{L}(A, B)\subseteq\II\setminus B$, and since $A\subseteq\I^-\cup\III$, $n\not\in A\cup B=A'\cup B'$. By Lemma~\ref{lemma:type III back neighbors}, $n\not\in BN\left(p_{r-1}\right)$, but $p_{r-1}$ and $n$ are adjacent, so we must have $p_{r-1}\in BN\left(n\right)$. Then, by Conditions~\ref{conditions:ab box stacking}, we deduce that $p_{r-1}\not\in B'$, so $p_{r-1}\in A'\setminus B'$. Since $w\in B'\setminus A'$, there exists $0\leq s<r-1$ such that $p_s\in B'\setminus A'$ and $p_{s+1}\in A'\setminus B'$. Then $p_s\in\III$ is adjacent to $p_{s+1}\in\III$, so $p_s\in BN\left(p_{s+1}\right)$ or $p_{s+1}\in BN\left(p_s\right)$. Again, it is easy to see that these statements contradict Conditions~\ref{conditions:ab box stacking} in both cases.

In all cases, we arrived at a contradiction. We conclude that $\ell$ is freely chosen and, as a result, $w\in C_1\cup C_2\cup\cdots\cup C_k$. So, $S\subseteq C_1\cup C_2\cup\cdots\cup C_k$. Moreover, $w$ is in exactly one of the connected components $C_w$ of freely labelled type $\III$ boxes in $\pi$. We claim that $C_w\subseteq S$. Since $C_w$ is connected and $w\in C_w\cap S$, it suffices to show that if $w', w''\in C_w$ are adjacent and $w'\in S$, then $w''\in S$. Suppose $w', w''\in C_w$ are adjacent and $w'\in S$. Then $w', w''$ are freely labelled type $\III$ boxes in $\pi$. Furthermore, since $w'\in S$, $w'\in B'\setminus A'$. Since $w'$ and $w''$ are adjacent, $w'\in BN\left(w''\right)$ or $w''\in BN\left(w'\right)$. Additionally, since $w''\in A'\cup B'$ is labelled, $w''\not\in A'\cap B'$, so $w''\in A'\triangle B'$. However, by Conditions~\ref{conditions:ab box stacking}, $w'\in BN\left(w''\right)$ implies that $w''\in B'$, while $w''\in BN\left(w'\right)$ implies that $w''\not\in A'$, so in either case, we must have $w''\not\in A'\setminus B'$. It follows that $w''\in B'\setminus A'$, so $w''\in\III\cap\left(B'\setminus A'\right)=S$, as desired. Consequently, $C_w\subseteq S$, so \[S=\bigcup_{w\in S}C_w\in C(\pi).\] This completes the proof.
\end{proof}

\begin{corollary}
	\label{corollary:size of sAB(pi)}
Let $N(\pi)$ be the number of connected components of freely labelled type $\III$ boxes in $\pi$. Then $|\sAB(\pi)| = 2^{N(\pi)} = \chi_{\text{top}}(\pi)$.
\end{corollary}

\begin{proof}
Suppose $S\in C(\pi)$. Let $(A', B')=(A\setminus S, B\cup S)$. We claim that $(A', B')\in\sAB_{\text{all}}$. As we observed in the proof of the lemma, $S\subseteq\III$ and $S\subseteq A\setminus B$. Since $A$ is a finite subset of $\I^-\cup\III$, so is $A'$. Since $S\subseteq\III$ and $B\subseteq\II\cup\III$, $B'=B\cup S\subseteq\II\cup\III$. Also, $\II\cup\III\subseteq[0, M-1]^3$, so $\II\cup\III$ is finite and, thus, $B'$ is finite.

Next, suppose $w\in\I^-\cup\III$ and $n\in BN(w)\cap A'$. Since $A'\subseteq A$ and $(A, B)$ is an $AB$ configuration, $w\in A$. Suppose $w\in S$. Then $w\in C_j\subseteq S$ for some $1\leq j\leq k$, and $w\in\III\cap(A\setminus B)\subseteq\III\cap(A\triangle B)\subseteq\mathcal{L}(A, B)$. Then, by Conditions~\ref{conditions:ab box stacking}, $n\not\in B$. By Lemma~\ref{lemma:type III back neighbors}, $n\in\I^-\cup\III$, and since $n\in A'\subseteq A$, we have \[n\in(\I^-\cap A)\cup(\III\cap(A\setminus B))\subseteq(\I^-\cap A)\cup(\III\cap(A\triangle B))\subseteq\mathcal{L}(A, B).\] As shown in the proof of Lemma~\ref{lemma:chitop}, $C_1, C_2, \ldots, C_k$ are connected components of $\mathcal{L}(A, B)$. So, since $w$ and $n$ are adjacent, $n\in C_j\subseteq S$, contradicting the fact that $n\in A'=A\setminus S$. We deduce that $w\not\in S$, so $w\in A\setminus S=A'$.

Now, suppose $w\in\II\cup\III$ and $n\in BN(w)\cap B'=BN(w)\cap(B\cup S)$. If $n\in B$, then $w\in B\subseteq B'$, since $(A, B)$ is an $AB$ configuration. Otherwise, $n\in S$, so $n\in C_j\subseteq S$ for some $1\leq j\leq k$, and $n\in\III\cap(A\setminus B)\subseteq\III\cap(A\triangle B)\subseteq\mathcal{L}(A, B)$. If $w\in B$, then $w\in B'$. Otherwise, $w\not\in B$. Then $w\in\II\setminus B$ or $w\in\III$, in which case, by Conditions~\ref{conditions:ab box stacking}, since $n\in S\subseteq A\setminus B$, $w\in\III\cap(A\setminus B)$. That is, \[w\in(\II\setminus B)\cup(\III\cap(A\setminus B))\subseteq(\II\setminus B)\cup(\III\cap(A\triangle B))\subseteq\mathcal{L}(A, B).\] As discussed above, $C_1, C_2, \ldots, C_k$ are connected components of $\mathcal{L}(A, B)$. So, since $w$ and $n$ are adjacent, $w\in C_j\subseteq S\subseteq B'$. In all cases, $w\in B'$.

These arguments show that $(A', B')$ is an $AB$ configuration, or in other words, $(A', B')\in\sAB_{\text{all}}$. Then, by the lemma, $(A', B')\in\sAB(\pi)$, so there is a well-defined surjective map $f:C(\pi)\to\sAB(\pi)$ given by \[f(S)=(A\setminus S, B\cup S).\] We claim that $f$ is also injective. Suppose that $f(S_1)=f(S_2)$ for some $S_1, S_2\in C(\pi)$. Then $B\cup S_1=B\cup S_2$, and as discussed above, $S_1\subseteq A\setminus B$, $S_2\subseteq A\setminus B$, so \[S_1=\left(B\cup S_1\right)\setminus B=\left(B\cup S_2\right)\setminus B=S_2,\] as desired. Thus, $\left|C(\pi)\right|=\left|\sAB(\pi)\right|$. Since $C_1, C_2, \ldots, C_k$ are disjoint, $C(\pi)$ is in bijection with the power set of $\{1, 2, \ldots, k\}$, so \[\left|\sAB(\pi)\right|=\left|C(\pi)\right|=2^k=2^{N(\pi)}=\chi_{\text{top}}(\pi),\] the last equality holding by Lemma~\ref{lemma:chitop}.
\end{proof}

\begin{defn} Let
\label{defn:ZAB}
	\[
		Z_{\sAB} = Z_{\sAB}(q) = q^{-|\II|-2|\III|}\sum_{(A, B) \in \sAB} q^{|A|+|B|}.
	\]
\end{defn}

\begin{thm}
\label{thm:ZABW}
	\[
		Z_{\sAB} = W(\mu_1, \mu_2, \mu_3)
	\]
\end{thm}

\begin{proof}
By Lemma~\ref{lemma:sAB as union}, Corollary~\ref{corollary:sAB(pi) disjoint}, Conditions~\ref{conditions:ab on pi}, and Corollary~\ref{corollary:size of sAB(pi)}, we have 
\begin{align*}
Z_{\sAB}=q^{-|\II|-2|\III|}\sum_{(A, B) \in \sAB} q^{|A|+|B|}&=q^{-|\II|-2|\III|}\sum_{\pi\in\text{PT-box}}\sum_{(A, B) \in \sAB(\pi)} q^{|A|+|B|}\\
&=q^{-|\II|-2|\III|}\sum_{\pi\in\text{PT-box}}\sum_{(A, B) \in \sAB(\pi)} q^{|A\cup B|+|A\cap B|}\\
&=q^{-|\II|-2|\III|}\sum_{\pi\in\text{PT-box}}\sum_{(A, B) \in \sAB(\pi)} q^{|\pi|}\\
&=q^{-|\II|-2|\III|}\sum_{\pi\in\text{PT-box}}|\sAB(\pi)|q^{|\pi|}\\
&=q^{-|\II|-2|\III|}\sum_{\pi\in\text{PT-box}}\chi_{\text{top}}(\pi)q^{|\pi|}=W(\mu_1, \mu_2, \mu_3).
\end{align*}
\end{proof}

\subsection{PT theory and the labelled double-dimer model}
\label{sec:double_dimer_configs}

%\subsubsection{From $AB$ configurations to double-dimer configurations}

%Recall that our goal is to describe $W(\mu_1, \mu_2, \mu_3)$ using the tripartite double-dimer model. To this end, 

The advantage of working with $AB$ configurations is that they are unlabelled, plane partition-like objects. In addition, there is a relationship between $\sAB$ and the tripartite double-dimer model, which we will now explain. On an infinite graph, a \emph{double-dimer configuration} is the union of two dimer configurations. 

Let $(A, B)$ be an $AB$ configuration. We consider $A$ and $B$ separately. Let $R_1$ (resp.~$R_2$) denote the subset of $\mathbb{Z}^3$ consisting of the cells that have at least one negative coordinate (resp.~at least two negative coordinates). For $A$, we view the surface $\mathfrak{A}:=R_2\cup(\I^-\cup\III)\setminus A$ as a lozenge tiling of the plane. In other words, we take the surface $R_2\cup\I^-\cup\III$, remove the boxes in $A$, and view the resulting surface as a lozenge tiling. Similarly, for $B$, we view the surface $\mathfrak{B}:=R_1\cup(\II\cup\III)\setminus B$ as a lozenge tiling of the plane. The fact that these surfaces can be viewed as lozenge tilings of the plane follows from Lemma~\ref{lemma:AB surfaces} below. The resulting lozenge tilings are then equivalent to dimer configurations of the infinite honeycomb graph $H$. 

\begin{figure}[htb]
\centering
\includegraphics[width=1.5in]{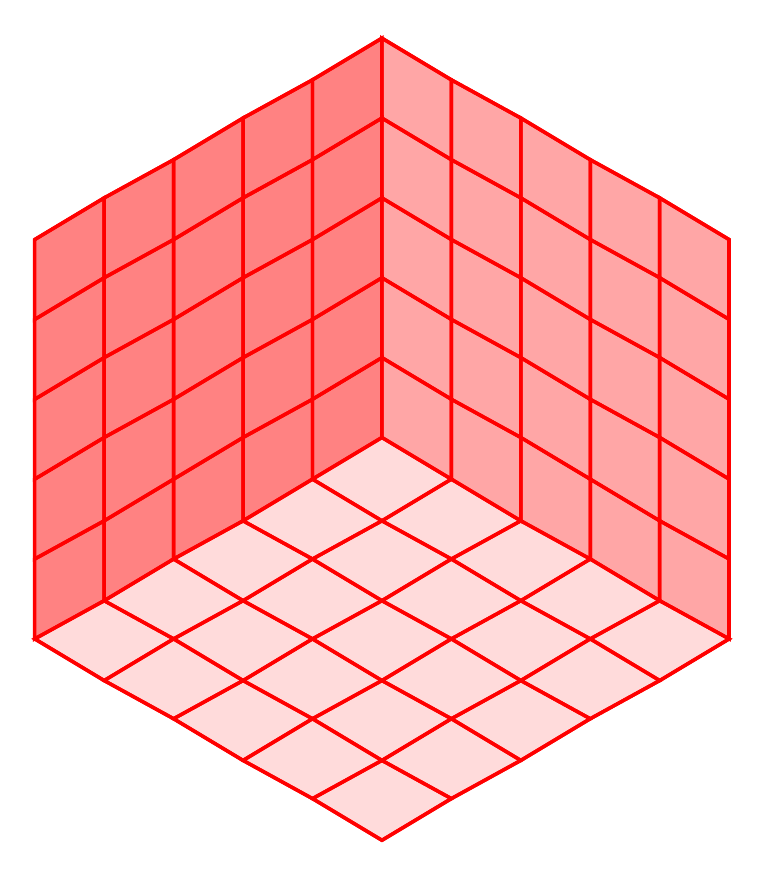}
\includegraphics[width=1.5in]{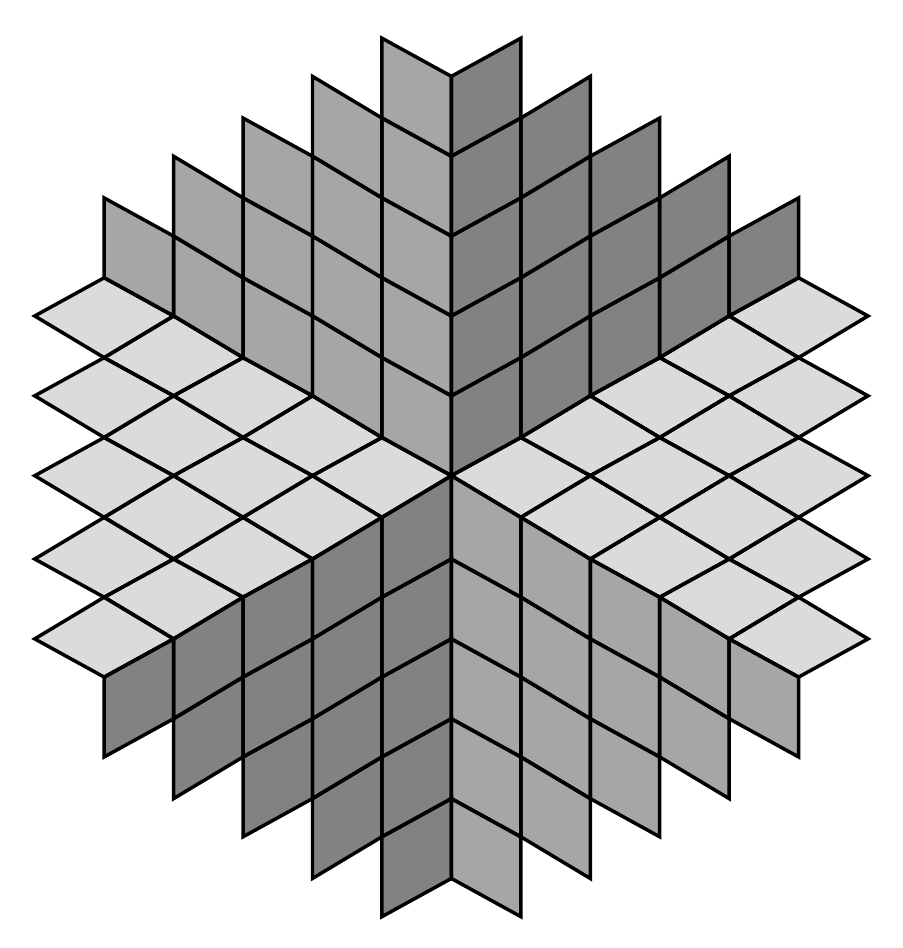}
\includegraphics[width=1.25in]{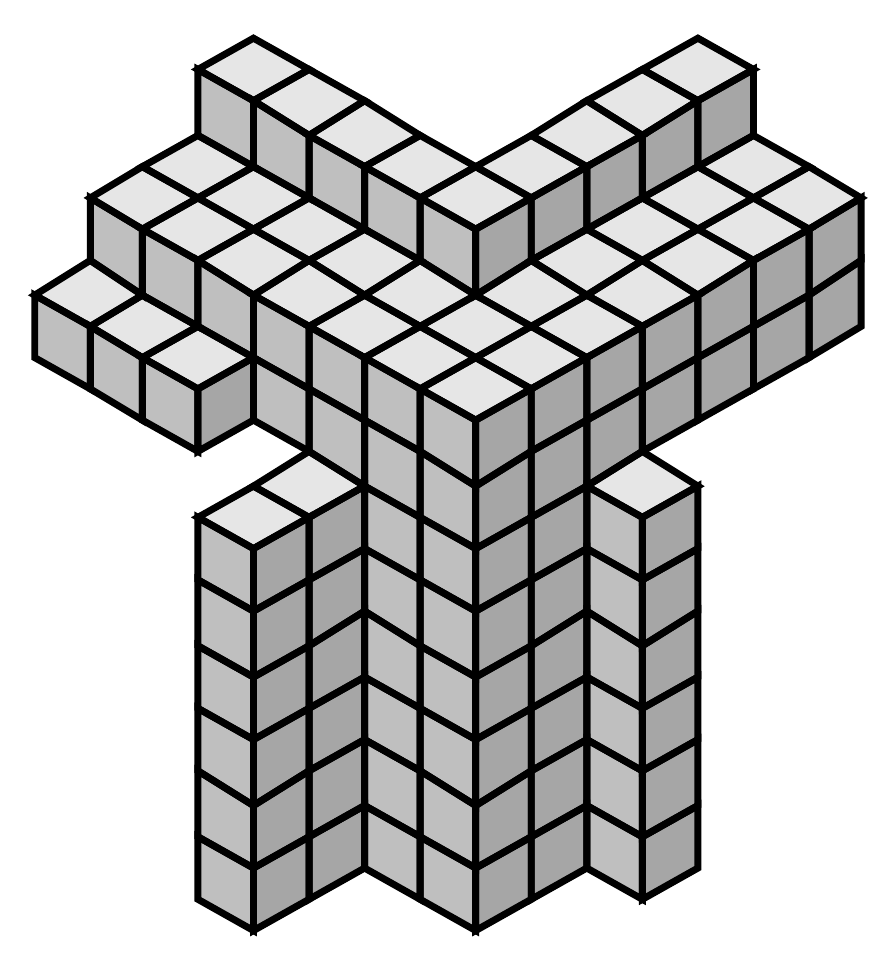}
\includegraphics[width=1.5in]{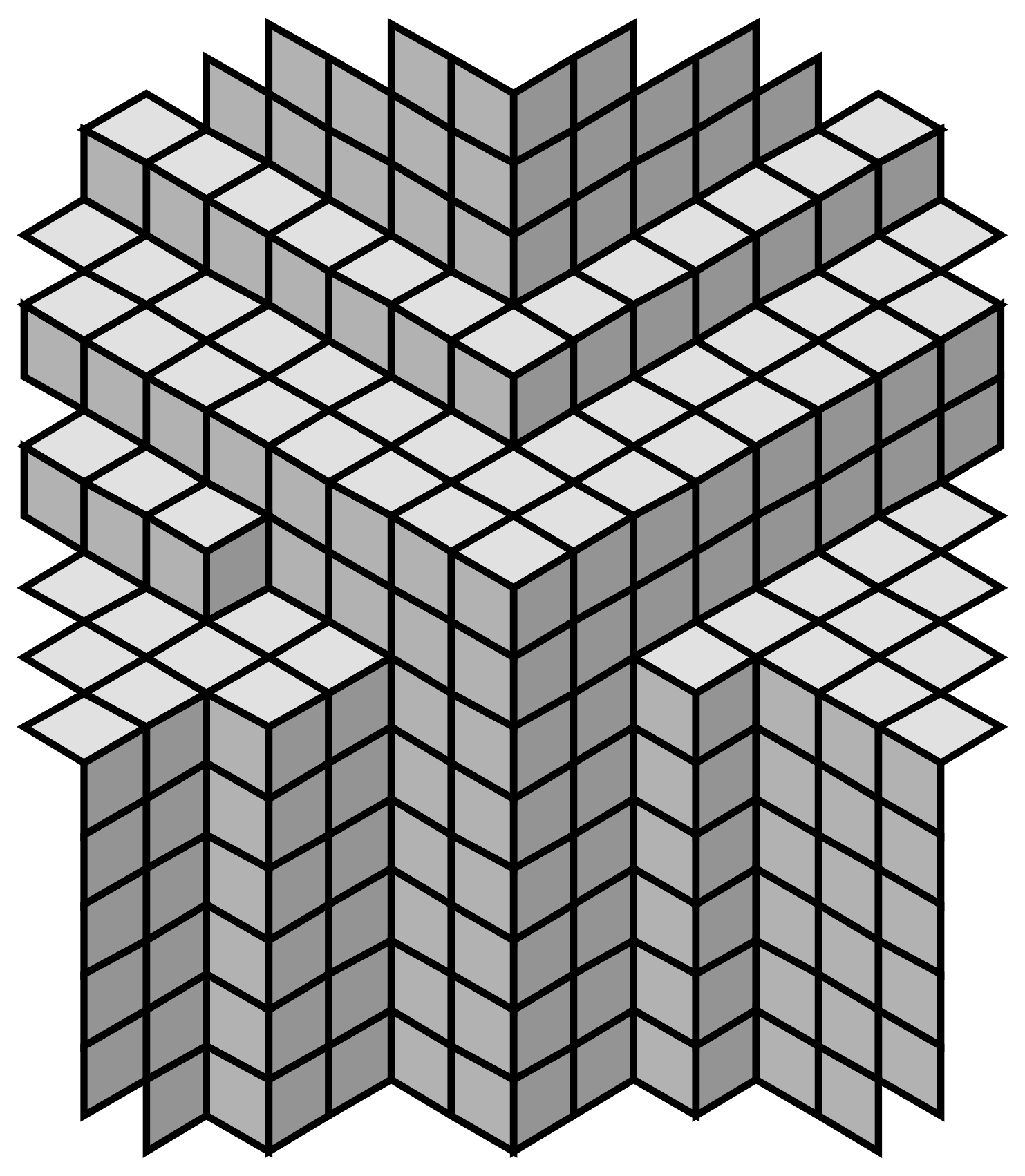}
\caption{Converting an $AB$ configuration to lozenge tilings of the plane. Left two pictures: tilings corresponding to $R_1$ and $R_2$, respectively. Right two pictures: an example of the surface $(\I^-\cup\III)\setminus A$ and the surface $R_2\cup(\I^-\cup\III)\setminus A$.}
\label{fig:empty tilings}
\end{figure}

\begin{example}
Recall the $AB$ configuration from Example~\ref{ex:spicy}. The rightmost image of Figure~\ref{fig:empty tilings} shows the lozenge tiling corresponding to $A=\{(3, -1, 0), (3, 0, 0)\}$, i.e., corresponding to the surface $R_2\cup(\I^-\cup\III)\setminus\{(3, -1, 0), (3, 0, 0)\}$. 
\end{example}

Let $M_A$ (resp.~$M_B$) denote the dimer configuration of $H$ corresponding to the tiling obtained from $A$ (resp. $B$). Superimposing $M_A$ and $M_B$ so that the origin in $\mathbb{Z}^3$ corresponds to the same face of $H$ produces a double-dimer configuration $D_{(A, B)}$ on $H$. 

\begin{figure}[htb]
\centering
\includegraphics[width=1.5in]{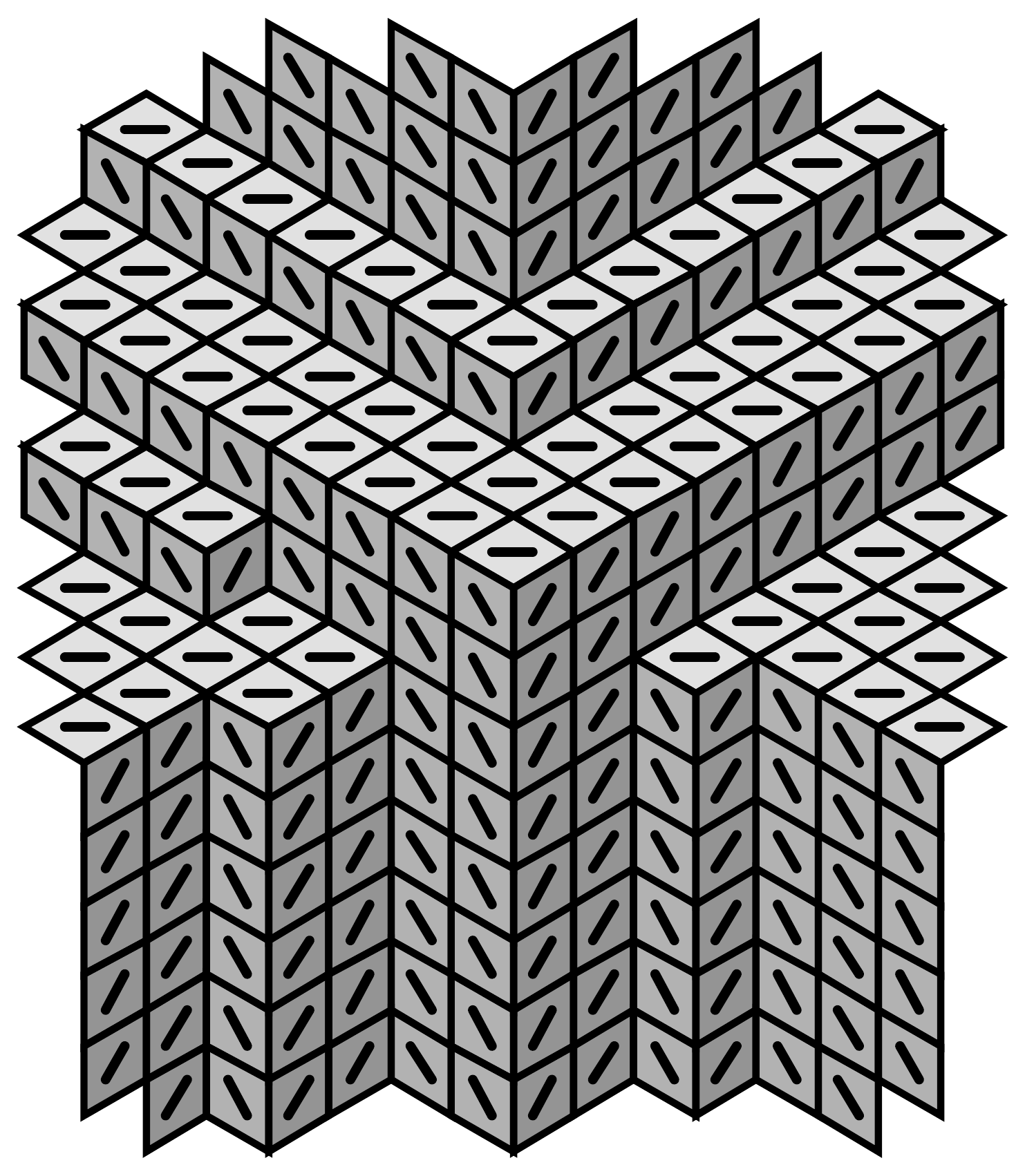}
\includegraphics[width=1.5in]{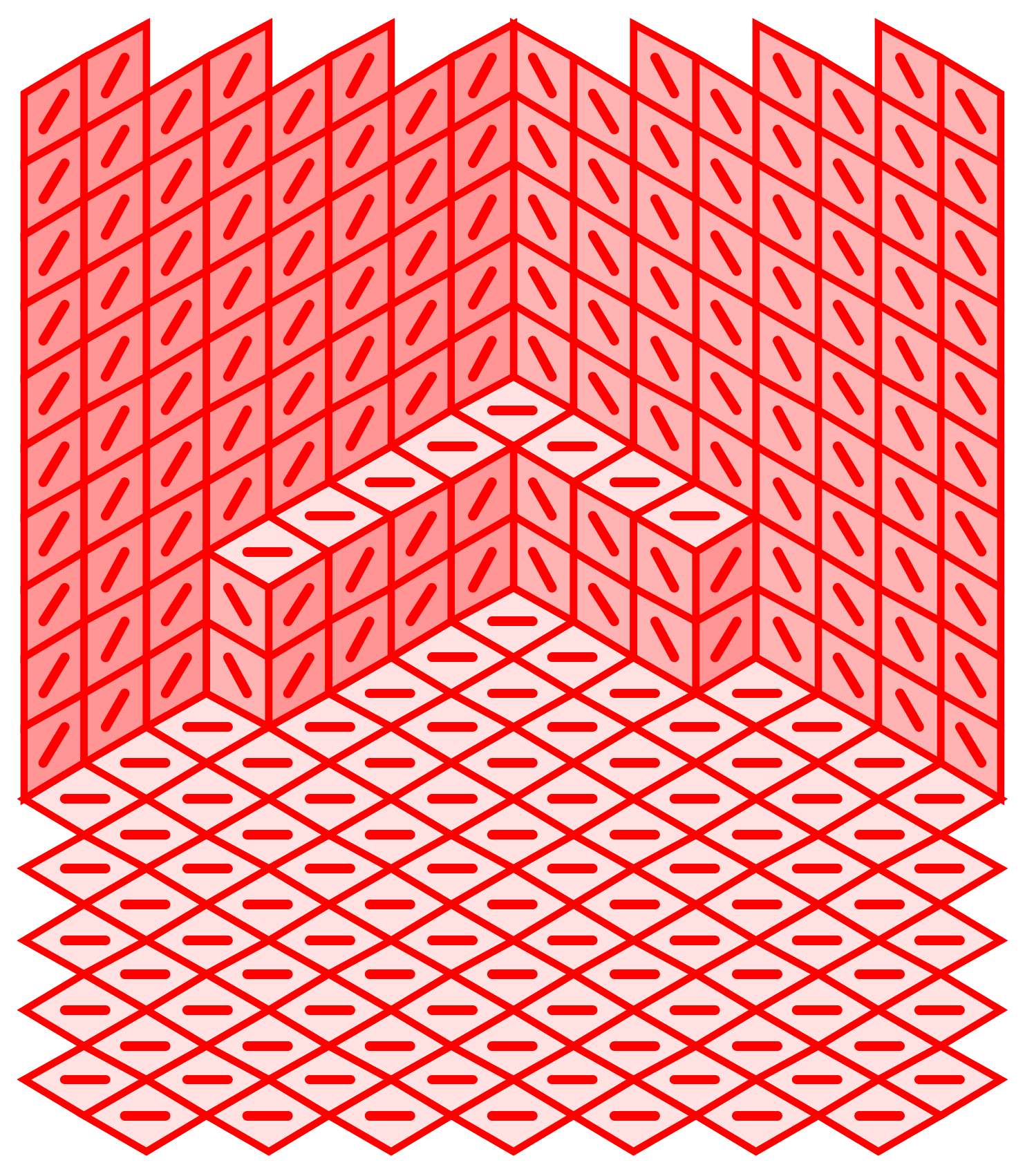}
\includegraphics[width=1.5in]{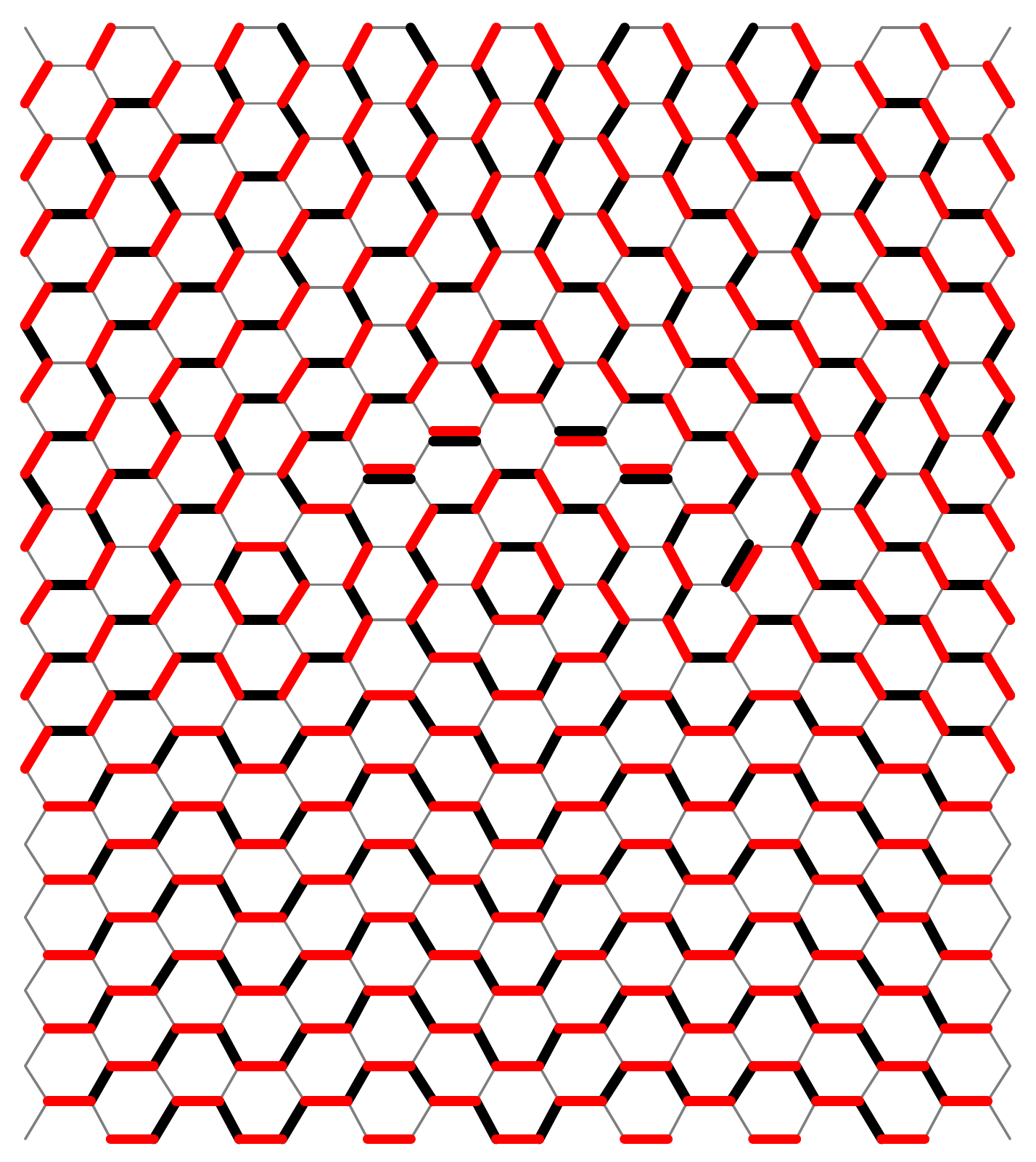}
\includegraphics[width=1.5in]{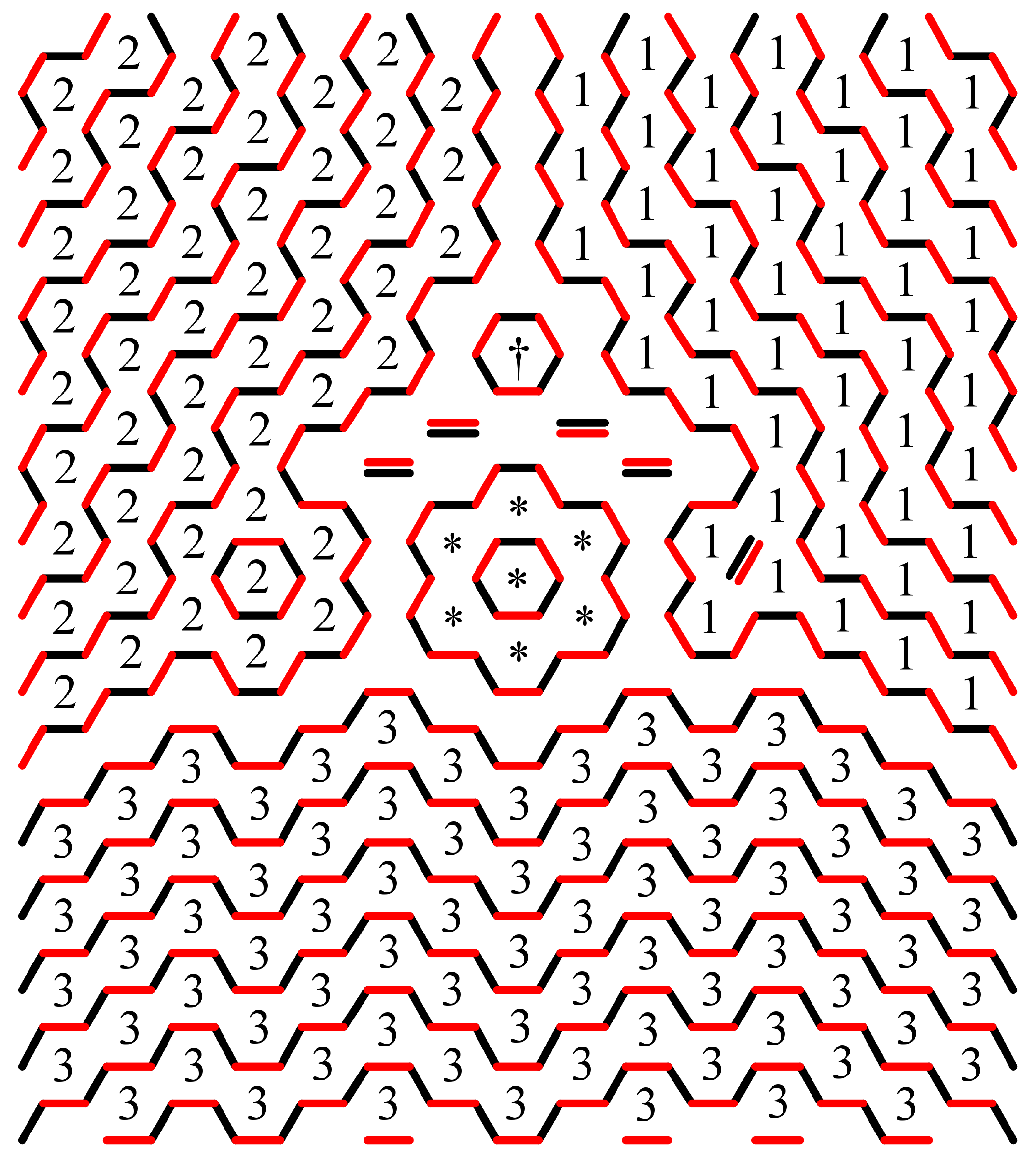}
\caption{First: The dimer configuration $M_A$. Second: The dimer configuration $M_B$. Third: The superposition of $M_A$ and $M_B$, a double-dimer configuration on $H$. Fourth: The labelled double-dimer configuration.}
\label{fig:superposition}
\end{figure}

\begin{example}
For the $AB$ configuration from Example~\ref{ex:spicy}, the dimer configurations $M_A$ and $M_B$ are shown in Figure~\ref{fig:superposition}. Their superposition, shown immediately to their right, is a double-dimer configuration $D_{(A, B)}$ on $H$. 
\end{example}

%\subsubsection{Labelled double-dimer configurations}

Just as we label certain $AB$ configurations, we label certain double-dimer configurations. Note that each double-dimer configuration on $H$ consists of doubled edges, loops, and infinite paths. Before describing a labelling algorithm for the double-dimer configurations $D_{(A, B)}$, we need the following lemmas.

Let $\mathbf{e}_i$ be the $i$th standard unit vector. 

\begin{lemma}
	\label{lemma:AB surfaces}
Let $\mathfrak{S}\in\{\mathfrak{A}, \mathfrak{B}\}$ and $q\in\mathbb{Z}^3_{\geq 0}$. If $p\not\in\mathfrak{S}$, then $p+q\not\in\mathfrak{S}$. Conversely, if $p\in\mathfrak{S}$, then $p-q\in\mathfrak{S}$. 
\end{lemma}

\begin{proof}
It suffices to establish this result for $q=\mathbf{e}_i$. Suppose $p+\mathbf{e}_i\in\mathfrak{S}$. Then $p+\mathbf{e}_i\in(\I^-\cup\III)\setminus A$ or $p+\mathbf{e}_i$ has at least two negative coordinates if $\mathfrak{S}=\mathfrak{A}$, and $p+\mathbf{e}_i\in(\II\cup\III)\setminus B$ or $p+\mathbf{e}_i$ has at least one negative coordinate if $\mathfrak{S}=\mathfrak{B}$. We will show that $p\in\mathfrak{S}$. In the first case, if $p$ has at least two negative coordinates, $p\in R_2\subseteq\mathfrak{S}$. Otherwise, since $p+\mathbf{e}_i$ having at least two negative coordinates implies that $p$ has at least two negative coordinates, we deduce from Lemmas~\ref{lemma:cylinder back neighbor} and~\ref{lemma:type III back neighbors} that $p\in\I^-\cup\III$, and by Conditions~\ref{conditions:ab box stacking}.1, $p\not\in A$. Thus, $p\in(\I^-\cup\III)\setminus A$, so $p\in\mathfrak{S}$, as desired. In the second case, if $p$ has at least one negative coordinate, $p\in R_1\subseteq\mathfrak{S}$. Otherwise, $p\in\mathbb{Z}^3_{\geq 0}$, and since $p+\mathbf{e}_i$ having at least one negative coordinate implies that $p$ has at least one negative coordinate, we deduce from Lemma~\ref{lemma:cylinder back neighbor} that $p\in\II\cup\III$. Then, by Conditions~\ref{conditions:ab box stacking}.2, $p\not\in B$, so $p\in(\II\cup\III)\setminus B$, and $p\in\mathfrak{S}$. These arguments establish the first statement of the lemma. The second statement can be established from the first by replacing $p$ with $p-q$ and taking the contrapositive of the result.
\end{proof}

\begin{remark}
In what follows, we often consider $H(N)$ as a subgraph of $H$. When doing so and some face $f$ of $H$ corresponds to the origin in $\mathbb{Z}^3$, $H(N)$ always denotes the $N\times N\times N$ honeycomb graph \emph{centered at $f$}. 
\end{remark}

\begin{lemma}
	\label{lemma:dimers covering different sectors}
Let $(A, B)\in\sAB_{\text{all}}$. If a dimer in $D_{(A, B)}$ covers vertices in two different sectors,\footnotemark~then those vertices must lie in the subgraph $H(M)\subseteq H$.
\end{lemma}

\footnotetext{When we refer to ``sectors'' in this section, we mean the sectors defined in the right-hand side of Figure~\ref{fig:sectors}.}

\begin{proof}
Suppose a dimer $e$ in $D_{(A, B)}$ covers vertices in two different sectors. Either $e\in M_A$ or $e\in M_B$. If $e\in M_A$, let $\mathfrak{S}=\mathfrak{A}$, and otherwise, let $\mathfrak{S}=\mathfrak{B}$. Then $e$ must correspond to a facet $f$ of a cell $w\in\mathfrak{S}$ having coordinates $(a, a, a)+h\mathbf{e}_i$ for some $a\in\mathbb{Z}$, $h\in\mathbb{Z}_{\geq 0}$, $i\in\{1, 2, 3\}$, such that $w+\mathbf{e}_i\not\in\mathfrak{S}$. From this, we see that if $w\in R_2$, then $w+\mathbf{e}_i\in R_2$, and if $w\in R_1$, then $w+\mathbf{e}_i\in R_1$, so considering the definitions of $\mathfrak{A}$ and $\mathfrak{B}$, we must have $w\in(\I^-\cup\III)\setminus A$ or $w\in(\II\cup\III)\setminus B$. In particular, $w\in\I^-\cup\II\cup\III$, so $a\geq 0$. Then $w\in\II\cup\III$, and $\II\cup\III$ is contained in the cube $[0, M]^3$. Projecting this cube onto the plane $x_1+x_2+x_3=0$ produces an $M\times M\times M$ hexagonal region that must contain $f$, so $e$ must be an edge of $H(M)$. The result follows.
\end{proof}

\begin{corollary}
	\label{corollary:paths in D_(A, B)}
Let $(A, B)\in\sAB_{\text{all}}$. Every path in $D_{(A, B)}$ moves between sectors finitely many times.
\end{corollary}

\begin{defn}
Given an end $\mathcal{E}$ of a path in $D_{(A, B)}$, we say that \emph{sector $i$ contains $\mathcal{E}$} if, when moving along the path toward $\mathcal{E}$, there is a point after which every dimer in the path is contained in sector $i$.
\end{defn}

\begin{remark}
\label{remark:end contained in sector}
Corollary~\ref{corollary:paths in D_(A, B)} implies that each end $\mathcal{E}$ of every path in $D_{(A, B)}$ is contained in sector $i$ for some $i$. 
\end{remark}

We also recall some facts about height functions. 

\begin{defn}
Given any dimer cover $M_0$ of $H$ and a face $f_0$ of $H$, we can associate to $M_0$ a height function $h_{M_0}$, called the \emph{absolute height function} of $M_0$, that assigns to each face of $H$ a real number as follows. Let $h_{M_0}(f_0)=0$. Then, for any other face $f$ of $H$, take a path $f_0, f_1, f_2, \ldots, f_r=f$ in the dual graph $H^{\vee}$ of $H$ from $f_0$ to $f$, and let $h_{M_0}(f)$ be the sum of the following contributions from each of the corresponding edges $e_1, e_2, \ldots, e_r$ of $H$: assuming the left vertex of $e_s$ is white (resp.~black), if $e_s\in M_0$, its contribution is $2/3$ (resp.~$-2/3$), and otherwise, its contribution is $-1/3$ (resp.~$1/3$). (Here, left and right should be interpreted from the perspective of one traversing the path from $f_0$ to $f$.) 
\end{defn}

The fact that $h_{M_0}$ is well-defined follows from the observation that such contributions sum to $0$ around any face of $H^{\vee}$. 

Given two dimer covers $M_1$ and $M_2$ of $H$, we call the difference $h_{M_1}-h_{M_2}$ the \emph{relative height function} of $M_1$ relative to $M_2$. Actually, when considering the lozenge tiling that corresponds to $M_0$ as a surface, $h_{M_0}$ gives the height above the plane $x_1+x_2+x_3=0$, divided by $\sqrt{3}$, up to a constant. Thus, $h_{M_1}-h_{M_2}$ gives the height difference, divided by $\sqrt{3}$, up to a constant, between the surfaces corresponding to $M_1$ and $M_2$. 

Given an $AB$ configuration $(A, B)$, let $h_A=h_{M_A}$ and $h_B=h_{M_B}$. In what follows, we consider the relative height function $h_{(A, B)}:=h_B-h_A$, where both absolute height functions are based on the face $f_0$ corresponding to the cell $(0, 0, M)$. Note that $\II\cup\III\subseteq[0, M-1]^3$, so $\mathfrak{A}$ and $\mathfrak{B}$ have the same height above the plane $x_1+x_2+x_3=0$ at $f_0$. Therefore, $h_{(A, B)}$ is precisely the height difference, divided by $\sqrt{3}$, between $\mathfrak{A}$ and $\mathfrak{B}$. This difference remains constant, except upon crossing an edge $e\in M_A\triangle M_B$, when it must increase or decrease by $2/3-(-1/3)=1/3-(-2/3)=1$. In other words, the loops and paths in $D_{(A, B)}$ are the contour lines for $h_{(A, B)}$. Moreover, orienting the edges in $M_B$ from white to black and those in $M_A$ from black to white produces orientations on the loops and paths so that crossing a loop or path oriented from left to right causes $h_{(A, B)}$ to increase by $1$, while crossing a loop or path oriented from right to left causes $h_{(A, B)}$ to decrease by $1$. 

\begin{lemma}
\label{lemma:labelling set implies nonzero height}
If $p\in\mathcal{L}(A, B)$, and $p$ corresponds to $f\in F$, then $p\in\mathfrak{A}\triangle\mathfrak{B}$ and $h_{(A, B)}(f)\neq 0$. 
\end{lemma}

\begin{proof}
Suppose $p\in\mathcal{L}(A, B)$, and $p$ corresponds to $f\in F$. If $p\in\I^-\cap A$, then $p\not\in(\I^-\cup\III)\setminus A$ and $p$ does not have at least two negative coordinates (it has exactly one negative coordinate), so $p\not\in\mathfrak{A}$. Since $p$ has at least one negative coordinate, $p\in\mathfrak{B}$. It follows that $h_{(A, B)}(f)>0$. If $p\in\II\setminus B$, then $p\not\in(\I^-\cup\III)\setminus A$ and $p\in\mathbb{Z}^3_{\geq 0}$ does not have at least two negative coordinates, so $p\not\in\mathfrak{A}$. Since $p\in(\II\cup\III)\setminus B$, $p\in\mathfrak{B}$. It follows that $h_{(A, B)}(f)>0$. Otherwise, $p\in\III\cap(A\triangle B)$. If $p\in\III\cap(A\setminus B)$, then $p\not\in(\I^-\cup\III)\setminus A$ and $p\in\mathbb{Z}^3_{\geq 0}$ does not have at least two negative coordinates, so $p\not\in\mathfrak{A}$. Additionally, $p\in(\II\cup\III)\setminus B$, so $p\in\mathfrak{B}$, implying that $h_{(A, B)}(f)>0$. Finally, if $p\in\III\cap(B\setminus A)$, then $p\in(\I^-\cup\III)\setminus A$, so $p\in\mathfrak{A}$. On the other hand, $p\not\in(\II\cup\III)\setminus B$ and $p\in\mathbb{Z}^3_{\geq 0}$ does not have at least one negative coordinate, so $p\not\in\mathfrak{B}$, and we find that $h_{(A, B)}(f)<0$. This completes the proof. 
\end{proof}

Let $F$ be the set of faces of $H$, and let $U_{(A, B)}=h_{(A, B)}^{-1}(0)\subseteq F$. Consider the subgraph $H^{\vee}_{(A, B)}$ of $H^{\vee}$ induced by $F\setminus U_{(A, B)}$. Then, given $f\in F$ such that $h_{(A, B)}(f)\neq 0$, denote by $C_{(A, B)}(f)$ the connected component of $H^{\vee}_{(A, B)}$ containing $f$. Also, we say that a face $f\in F$ is \emph{contained in sector $i$} if the vertices of $H$ incident to $f$ are all in sector $i$. Finally, we say that a connected component of $H^{\vee}_{(A, B)}$ is \emph{almost contained in sector $i$} if it contains only finitely many faces that are not contained in sector $i$. Note that any infinite connected component of $H^{\vee}_{(A, B)}$ is almost contained in at most one sector.

We can now describe the labelling algorithm for the double-dimer configurations $D_{(A, B)}$. Fix an $AB$ configuration $(A, B)$.

\begin{algorithm}
\begin{enumerate}
\item If there is a connected component $C$ of $H^{\vee}_{(A, B)}$ so that, given any $i$, $C$ is not almost contained in sector $i$, terminate with failure.
\item For each infinite connected component of $H^{\vee}_{(A, B)}$, there must be exactly one sector $i$ almost containing it. Label the faces it contains by $i$.
\item Label each finite connected component of $H^{\vee}_{(A, B)}$ by a single freely chosen element of $\mathbb{P}^1$.
\end{enumerate}
\label{algorithm:double-dimer labelling}
\end{algorithm}

\begin{example}
If we label the double-dimer configuration from Figure~\ref{fig:superposition}, we obtain the labelled double-dimer configuration shown in Figure~\ref{fig:superposition}. Observe that the paths in the double-dimer configuration from Figure~\ref{fig:superposition} are ``rainbow-like.'' In other words, the paths are nested and start and end in the same sector. 
\end{example}

We will first prove that this algorithm is, in some sense, equivalent to Algorithm~\ref{algorithm:AB labelling algorithm}, and then we will describe the connection between this algorithm and the double-dimer configuration $D_{(A, B)}$.

\subsection{Proofs of the equivalence of the labelling algorithms}
\label{sec:labelling_algorithm_proofs}

\begin{lemma}
\label{lemma:avoiding doubled edges}
Suppose $f$ and $f'$ are faces that belong to the same connected component of $H^{\vee}_{(A, B)}$. Then there is a sequence of adjacent faces in $F\setminus U_{(A, B)}$, beginning at $f$ and ending at $f'$, such that no pair of consecutive faces are separated by an edge in $M_A\cap M_B$. 
\end{lemma}

\begin{proof}
Since $f$ and $f'$ belong to the same connected component of $H^{\vee}_{(A, B)}$, there is a sequence of adjacent faces $f:=f_0, f_1, \ldots, f_r:=f'$ in $F\setminus U_{(A, B)}$. Suppose the edge separating $f_s$ and $f_{s+1}$ is in $M_A\cap M_B$. Then, since $M_A$ and $M_B$ are dimer configurations, the two faces adjacent to both $f_s$ and $f_{s+1}$ are separated from $f_s$ and $f_{s+1}$ by edges that are not in $M_A\cup M_B$. Therefore, for either such face $g$, we have $h_{(A, B)}(f_s)=h_{(A, B)}(g)=h_{(A, B)}(f_{s+1})$, and we may insert $g$ into the sequence $f_0, f_1, \ldots, f_r$ between $f_s$ and $f_{s+1}$ to produce a new sequence of adjacent faces in $F\setminus U_{(A, B)}$. We may continue in this way until we obtain a sequence with the desired properties. 
\end{proof}

\begin{lemma}
\label{lemma:face sequence implies cell sequence}
Suppose $f_0, f_1, \ldots, f_r$ is a sequence of adjacent faces in $F\setminus U_{(A, B)}$ such that no pair of consecutive faces are separated by an edge in $M_A\cap M_B$. Suppose $p_0\in\mathfrak{A}\triangle\mathfrak{B}$ is a cell that corresponds to $f_0$. Then there exist integers $k_s$ and cells $p_{s+1}$ for $0\leq s<r$ so that for any $i$, $j$, $k$ such that $\{i, j, k\}=\{1, 2, 3\}$, the following is a sequence of adjacent cells in $\mathfrak{A}\triangle\mathfrak{B}$, such that $p_s$ corresponds to $f_s$ for $0\leq s\leq r$: \begin{align*}p_0, p_0&+\sgn(k_0)\mathbf{e}_i, p_0+\sgn(k_0)(\mathbf{e}_i+\mathbf{e}_j), p_0+\sgn(k_0)(\mathbf{e}_i+\mathbf{e}_j+\mathbf{e}_k), \\p_0&+\sgn(k_0)(2\mathbf{e}_i+\mathbf{e}_j+\mathbf{e}_k), \ldots, p_0+(k_0\mathbf{e}_i+k_0\mathbf{e}_j+k_0\mathbf{e}_k), \\p_1, p_1&+\sgn(k_1)\mathbf{e}_i, p_1+\sgn(k_1)(\mathbf{e}_i+\mathbf{e}_j), p_1+\sgn(k_1)(\mathbf{e}_i+\mathbf{e}_j+\mathbf{e}_k), \ldots, p_r.\end{align*} Here, $\sgn(k_s)=1$ if $k_s>0$, $\sgn(k_s)=0$ if $k_s=0$, and $\sgn(k_s)=-1$ if $k_s<0$. 
\end{lemma}

\begin{proof}
Assume that we have specified the desired sequence up to $p_s$ for some $0\leq s<r$. Each pair of consecutive faces $f_s$, $f_{s+1}$ determines a direction in $\mathbb{Z}^3$. More precisely, there exist unique $\varepsilon\in\{\pm 1\}$ and $i\in\{1, 2, 3\}$ such that $p_s+\varepsilon \mathbf{e}_i$ corresponds to $f_{s+1}$. Since $h_{(A, B)}(f_{s+1})\neq 0$, there exists an integer $k_s$ such that $p_s+\varepsilon \mathbf{e}_i+(k_s, k_s, k_s)\in\mathfrak{A}\triangle\mathfrak{B}$. If $\varepsilon=-1$, assume that $k_s$ is the least such integer, and if $\varepsilon=1$, assume that $k_s$ is the greatest such integer. Define $p_{s+1}:=p_s+\varepsilon \mathbf{e}_i+(k_s, k_s, k_s)$. 

We claim that $p_s+(k_s, k_s, k_s)\in\mathfrak{A}\triangle\mathfrak{B}$. Suppose not. Of $\mathfrak{A}$ and $\mathfrak{B}$, let $\mathfrak{L}$ be the one such that $p_s+\varepsilon \mathbf{e}_i+(k_s, k_s, k_s)\not\in\mathfrak{L}$ and let $\mathfrak{U}$ be the other (i.e., the one such that $p_s+\varepsilon \mathbf{e}_i+(k_s, k_s, k_s)\in\mathfrak{U}$). Let $M_L$ and $M_U$, respectively, be the corresponding dimer configurations. By Lemma~\ref{lemma:AB surfaces}, if $\varepsilon=-1$, then $p_s+(k_s, k_s, k_s)\not\in\mathfrak{L}$, so $p_s+(k_s, k_s, k_s)\not\in\mathfrak{U}$, and if $\varepsilon=1$, then $p_s+(k_s, k_s, k_s)\in\mathfrak{U}$, so $p_s+(k_s, k_s, k_s)\in\mathfrak{L}$. In the first case, $\mathfrak{U}$ separates $p_s+(k_s, k_s, k_s)$ from $p_s+\varepsilon \mathbf{e}_i+(k_s, k_s, k_s)$, and in the second case, $\mathfrak{L}$ separates those two cells. In the first case, the edge $e$ separating $f_s$ and $f_{s+1}$ must be in $M_U$, and in the second case, $e$ must be in $M_L$. The sequence $f_0, f_1, \ldots, f_r$ is such that $e\not\in M_A\cap M_B=M_L\cap M_U$, so in either case, $e\in M_L\triangle M_U=M_A\triangle M_B$. As a result, $h_{(A, B)}$ differs by $\pm 1$ at $f_s$ and $f_{s+1}$. If $\varepsilon=-1$, $\mathfrak{U}$ must lie at $p_s+\varepsilon \mathbf{e}_i+(k_s+1, k_s+1, k_s+1)$, while $k_s$ is the least integer such that $p_s+\varepsilon \mathbf{e}_i+(k_s, k_s, k_s)\in\mathfrak{L}\triangle\mathfrak{U}$, so $\mathfrak{L}$ lies at $p_s+\varepsilon \mathbf{e}_i+(k_s, k_s, k_s)$. It follows that $h_{(A, B)}(f_{s+1})=\pm 1$. Similarly, if $\varepsilon=1$, $\mathfrak{L}$ must lie at $p_s+\varepsilon \mathbf{e}_i+(k_s, k_s, k_s)$, while $k_s$ is the greatest integer such that $p_s+\varepsilon \mathbf{e}_i+(k_s, k_s, k_s)\in\mathfrak{L}\triangle\mathfrak{U}$, so $\mathfrak{U}$ lies at $p_s+\varepsilon \mathbf{e}_i+(k_s+1, k_s+1, k_s+1)$. So, in this case, too, $h_{(A, B)}(f_{s+1})=\pm 1$. Then $h_{(A, B)}(f_s)=h_{(A, B)}(f_{s+1})\pm 1=\pm 2$, since $h_{(A, B)}(f_s)\neq 0$. Additionally, this shows that $h_{(A, B)}$ has the same sign at $f_s$ and $f_{s+1}$, so $\mathfrak{L}$ lies below $\mathfrak{U}$ at $f_s$. Consequently, if $\varepsilon=-1$, $\mathfrak{U}$ must lie at $p_s+(k_s, k_s, k_s)$ and $\mathfrak{L}$ must lie at $p_s+(k_s-2, k_s-2, k_s-2)$. On the other hand, if $\varepsilon=1$, $\mathfrak{L}$ must lie at $p_s+(k_s+1, k_s+1, k_s+1)$ and $\mathfrak{U}$ must lie at $p_s+(k_s+3, k_s+3, k_s+3)$. Then, by Lemma~\ref{lemma:AB surfaces}, in the first case, \[p_s+\varepsilon \mathbf{e}_i+(k_s-1, k_s-1, k_s-1)=p_s+(k_s-2, k_s-2, k_s-2)+\varepsilon \mathbf{e}_i+(1, 1, 1)\not\in\mathfrak{L},\] contradicting the fact that $\mathfrak{L}$ lies at $p_s+\varepsilon \mathbf{e}_i+(k_s, k_s, k_s)$. In the second case, \[p_s+\varepsilon \mathbf{e}_i+(k_s+1, k_s+1, k_s+1)=p_s+(k_s+2, k_s+2, k_s+2)+\varepsilon \mathbf{e}_i-(1, 1, 1)\in\mathfrak{U},\] contradicting the fact that $\mathfrak{U}$ lies at $p_s+\varepsilon \mathbf{e}_i+(k_s+1, k_s+1, k_s+1)$. By contradiction, $p_s+(k_s, k_s, k_s)\in\mathfrak{A}\triangle\mathfrak{B}$. Since $p_s\in\mathfrak{A}\triangle\mathfrak{B}$, by Lemma~\ref{lemma:AB surfaces}, we conclude that $p_s+\sgn(k_s)(m_1, m_2, m_3)\in\mathfrak{A}\triangle\mathfrak{B}$ for any $m_1, m_2, m_3$ such that $0\leq m_1, m_2, m_3\leq\lvert k_s\rvert$. This completes the proof. 
\end{proof}

\begin{lemma}
\label{lemma:dual component cardinality}
Suppose a cell $w$ corresponds to $f_0\in F$. If $w\in(\Cyl_{\ell}^-\cap A)\cup(\II_{\bar{\ell}}\setminus B)$ for some integer $\ell$, or Algorithm~\ref{algorithm:AB labelling algorithm} labels $w$ by an integer $\ell$, then $C_{(A, B)}(f_0)$ contains infinitely many faces contained in sector $\ell$. If Algorithm~\ref{algorithm:AB labelling algorithm} labels $w$ by $\ell$, and $\ell$ is not an integer, then $C_{(A, B)}(f_0)$ is finite.
\end{lemma}

\begin{proof}
We consider first case (i): $w\in(\Cyl_{\ell}^-\cap A)\cup(\II_{\bar{\ell}}\setminus B)$ for some integer $\ell$, or Algorithm~\ref{algorithm:AB labelling algorithm} labels $w$ by an integer $\ell$, and then case (ii): Algorithm~\ref{algorithm:AB labelling algorithm} labels $w$ by $\ell$, and $\ell$ is not an integer.\\

Case (i): Observe that $w$ must be an element of a connected component $C$ of $\mathcal{L}(A, B)$ containing a cell $n\in\Cyl_{\ell}^-\cup\II_{\bar{\ell}}$. Then there is a sequence of adjacent cells $w:=p_0, p_1, \ldots, p_r:=n$, each of which is an element of $C\subseteq\mathcal{L}(A, B)$. Furthermore, $p_r=n\in\mathcal{L}(A, B)\cap(\I^-\cup\II)\subseteq(\I^-\cap A)\cup(\II\setminus B)$. By Lemma~\ref{lemma:labelling set implies nonzero height}, assuming the cells $p_1, p_2, \ldots, p_r$ correspond to the faces $f_1, f_2, \ldots, f_r$ of $H$, we can deduce that $h_{(A, B)}(f_s)\neq 0$ for $0\leq s\leq r$. Since $p_s$ is adjacent to $p_{s+1}$, $f_s$ is adjacent to $f_{s+1}$ in $H^{\vee}$ for $0\leq s<r$. Moreover, since $h_{(A, B)}(f_s)\neq 0$ for $0\leq s\leq r$, $C_{(A, B)}(f_0)=C_{(A, B)}(f_r)$. 

Now, if $p_r\in\I^-\cap A$, let $p$ be any cell obtained by translating $p_r$ by $k>0$ units in the $x_i$-directions, for each $i\neq\ell$. Let $f(k)$ be the corresponding face of $H$. Note that $p_r\in\Cyl^-_{\ell}$, so the $\ell$th coordinate of $p$ is negative, and the other coordinates of $p$ are nonnegative. Suppose $h_{(A, B)}(f(k))=0$. Considering the definitions of $\mathfrak{A}$ and $\mathfrak{B}$, this implies that either $f(k)$ lies along one of the nonnegative coordinate axes, $f(k)$ corresponds to a cell $p'\in\I^-\setminus A$ whose single negative coordinate has the value $-1$, or $f(k)$ corresponds to a cell $p'\in\III\setminus A$. Since the $\ell$th coordinate of $p$ is negative, while the other coordinates of $p$ are nonnegative, $f(k)$ cannot lie along any of the nonnegative coordinate axes, so one of the latter cases must hold. Then, in either case, every cell above $p'$ is in $\mathbb{Z}^3_{\geq 0}$, so we conclude that $p=p'$ or $p$ is below $p'$. Thus, the only coordinate of $p'$ that may be negative is the $\ell$th coordinate, so if $p'\in\I^-$, then $p'\in\Cyl_{\ell}^-$. Additionally, if $p\not\in\Cyl_{\ell}^-$, then $p'\not\in\Cyl_{\ell}$. However, in this case, $p'\not\in\I^-\cup\III$, which is a contradiction, so we must have $p'\in\Cyl_{\ell}$ and $p\in\Cyl_{\ell}^-$. By Lemmas~\ref{lemma:cylinder back neighbor} and~\ref{lemma:type III back neighbors}, there is a sequence of back neighbors in $\I^-\cup\III$ leading from $p'$ to $p$ to $p_r$. By repeatedly applying Conditions~\ref{conditions:ab box stacking}.1, since $p_r\in A$, it follows that $p'\in A$. By contradiction, $h_{(A, B)}(f(k))\neq 0$. Finally, observe that $f(k)$ is also the face corresponding to the cell obtained by translating $p_r$ by $-k$ units in the $x_{\ell}$-direction. Therefore, since $k>0$ was arbitrary, $h_{(A, B)}$ must be nonzero at any face $f(k)$ obtained from $f_r$ by translating in the negative $x_{\ell}$-direction. This shows that $C_{(A, B)}(f_0)=C_{(A, B)}(f_r)$ contains infinitely many faces contained in sector $\ell$, since for large enough $k$, $f(k)$ is contained in sector $\ell$. 

On the other hand, if $p_r\in\II\setminus B$, let $p$ be any cell obtained by translating $p_r$ by $k<0$ units in the $x_{\ell}$-direction. Let $f(k)$ be the corresponding face of $H$. Note that $p_r\in\II_{\bar{\ell}}$, so $p_r\not\in\Cyl_{\ell}$ and $p\not\in\Cyl_{\ell}$. By Lemma~\ref{lemma:cylinder back neighbor}, though, if $p\in\mathbb{Z}^3_{\geq 0}$, then $p\in\II_{\bar{\ell}}$. In fact, in this case, there is a sequence of back neighbors in $\II_{\bar{\ell}}$ leading from $p_r$ to $p$, so by repeatedly applying Conditions~\ref{conditions:ab box stacking}.2, we find that $p\not\in B$. Then $p\in\II\setminus B\subseteq R_1\cup(\II\cup\III)\setminus B=\mathfrak{B}$ and $p\not\in R_2\cup(\I^-\cup\III)\setminus A=\mathfrak{A}$, so $h_{(A, B)}(f(k))>0$. Otherwise, the $\ell$th coordinate of $p$ is negative, while the other coordinates of $p$ are nonnegative. Since $p\not\in\Cyl_{\ell}$, $p\not\in\mathfrak{A}$. Furthermore, $p\in R_1$, so $p\in\mathfrak{B}$. Thus, in this case, too, $h_{(A, B)}(f(k))>0$. Consequently, since $k<0$ was arbitrary, $h_{(A, B)}$ must be nonzero at any face $f(k)$ obtained from $f_r$ by translating in the negative $x_{\ell}$-direction. Again, this shows that $C_{(A, B)}(f_0)=C_{(A, B)}(f_r)$ contains infinitely many faces contained in sector $\ell$. \\

Case (ii): Let $w:=p_0$. Since $\ell$ is not an integer, $w$ must be labelled in step 3 of Algorithm~\ref{algorithm:AB labelling algorithm}, so $w\in\III\cap(A\triangle B)$. If $w\in\III\cap A\setminus B$, then $w\not\in\mathfrak{A}$, while $w\in\mathfrak{B}$. Otherwise, $w\in\III\cap B\setminus A$, in which case, $w\not\in\mathfrak{B}$, while $w\in\mathfrak{A}$. In either case, $h_{(A, B)}(f_0)\neq 0$. 

So, consider $C_{(A, B)}(f_0)$. Suppose this connected component is infinite. Then, since $\mathcal{L}(A, B)$ is finite, there must be a face $f\in C_{(A, B)}(f_0)$ that doesn't correspond to any cell in $\mathcal{L}(A, B)$. By Lemma~\ref{lemma:avoiding doubled edges}, there is a sequence $f_0, f_1, \ldots, f_r:=f$ of adjacent faces in $F\setminus U_{(A, B)}$ such that no pair of consecutive faces are separated by an edge in $M_A\cap M_B$. The height function $h_{(A, B)}$ can only differ by $0$ or $\pm 1$ at adjacent faces, and $h_{(A, B)}$ is nonzero at each face in the sequence $f_0, f_1, \ldots, f_r$, so $h_{(A, B)}$ has the same sign at all of these faces. 

By Lemma~\ref{lemma:face sequence implies cell sequence}, there is a sequence of adjacent cells \begin{align*}p_0, p_0&+\sgn(k_0)(1, 0, 0), p_0+\sgn(k_0)(1, 1, 0), p_0+\sgn(k_0)(1, 1, 1), \\p_0&+\sgn(k_0)(2, 1, 1), \ldots, p_0+(k_0, k_0, k_0), p_1, p_1+\sgn(k_1)(1, 0, 0), \ldots, p_r,\end{align*} all of which are in $\mathfrak{A}\triangle\mathfrak{B}$, such that $p_s$ corresponds to $f_s$ for $0\leq s\leq r$. Recall that $f_r=f$ does not correspond to any cell in $\mathcal{L}(A, B)$, so $p_r\not\in\mathcal{L}(A, B)$. However, $p_0=w\in\III\cap(A\triangle B)\subseteq\mathcal{L}(A, B)$. So, consider the first cell $p'$ in the above sequence that is not an element of the labelling set, and let $p$ be the previous cell in the sequence. We claim that $p\in\I^-\cup\II$. Suppose not. Then $p\in\mathcal{L}(A, B)\setminus(\I^-\cup\II)=\III\cap(A\triangle B)$. Furthermore, $p$ is adjacent to $p'$, so $p\in BN(p')$ or $p'\in BN(p)$. If $p\in BN(p')$, then since $p\in\III$, we have $p'\in\Cyl_1\cup\Cyl_2\cup\Cyl_3$, and since $p\in\III\subseteq\mathbb{Z}^3_{\geq 0}$, $p'\in\mathbb{Z}^3_{\geq 0}$, implying that $p'\in\I^+\cup\II\cup\III$. But elements of $\I^+$ are not in $\I^-\cup\II\cup\III$, nor do they have any negative coordinates, so such elements are not in $\mathfrak{A}\cup\mathfrak{B}$. Since $p'\in\mathfrak{A}\triangle\mathfrak{B}$, it must be the case that $p'\in\II\cup\III$. If $p'\in BN(p)$, then by Lemma~\ref{lemma:type III back neighbors}, $p'\in\I^-\cup\III$. So, in either case, $p'\in\I^-\cup\II\cup\III$. If $p'\in\I^-$, then $p'\in\mathfrak{B}$, so $p'\not\in\mathfrak{A}$, in which case, $p'\in A$. But this means that $p'\in\I^-\cap A\subseteq\mathcal{L}(A, B)$. So, $p'\not\in\I^-$. Similarly, if $p'\in\II$, then $p'\not\in\mathfrak{A}$, so $p'\in\mathfrak{B}$, in which case, $p'\not\in B$. This means that $p'\in\II\setminus B\subseteq\mathcal{L}(A, B)$, so $p'\not\in\II$. Thus, $p'\in\III$. If $p'\not\in\mathfrak{A}$ and $p'\in\mathfrak{B}$, then $p'\in\III\cap A\setminus B\subseteq\III\cap(A\triangle B)\subseteq\mathcal{L}(A, B)$. Otherwise, if $p'\not\in\mathfrak{B}$ and $p'\in\mathfrak{A}$, then $p'\in\III\cap B\setminus A\subseteq\III\cap(A\triangle B)\subseteq\mathcal{L}(A, B)$. By contradiction, $p\in\I^-\cup\II$. Let $q$ be the first cell preceding $p'$ in the above sequence that is in $\I^-\cup\II$. Since $p'$ is the first cell in the sequence that's not in $\mathcal{L}(A, B)$, $q\in\mathcal{L}(A, B)$, so $q\in(\I^-\cap A)\cup(\II\setminus B)$. Therefore, $q$ is labelled by an integer $\ell(q)$ in step 2 of Algorithm~\ref{algorithm:AB labelling algorithm}. All of the cells $w=p_0:=q_0, q_1, \ldots, q_t$ preceding $q$ in the above sequence (written here in the same order as written in the above sequence) also precede $p'$, so they are elements of the labelling set and not in $\I^-\cup\II$, i.e., they are all elements of $\III\cap(A\triangle B)$. Since $q_0, q_1, \ldots, q_t, q$ is a sequence of adjacent cells, we see that $\{q_0, q_1, \ldots, q_t, q\}$ is contained in a single connected component of $\mathcal{L}(A, B)$, which is labelled in step 2 of Algorithm~\ref{algorithm:AB labelling algorithm} by $\ell(q)$. In particular, $w=q_0$ is labelled in step 2 of Algorithm~\ref{algorithm:AB labelling algorithm} by an integer $\ell(q)$, contradicting the fact that $\ell$ is not an integer. As a result, $C_{(A, B)}(f_0)$ is finite.
\end{proof}

\begin{lemma}
\label{lemma:nonzero height on axis implies labelling set}
If $f\in F\setminus U_{(A, B)}$ lies along one of the nonnegative coordinate axes, then $f$ corresponds to a cell $p\in\mathcal{L}(A, B)$ and all cells corresponding to $f$ that are in $\mathfrak{A}\triangle\mathfrak{B}$ must be in $\mathcal{L}(A, B)$. 
\end{lemma}

\begin{proof}
Since $f\in F\setminus U_{(A, B)}$, there exists a cell $p\in\mathfrak{A}\triangle\mathfrak{B}$ corresponding to $f$. The result will follow if we can show that any cell $q\in\mathfrak{A}\triangle\mathfrak{B}$ corresponding to $f$ is in $\mathcal{L}(A, B)$. Since $f$ lies along one of the nonnegative coordinate axes, $q=k_1\mathbf{e}_i+(k_2, k_2, k_2)$ for some $i\in\{1, 2, 3\}$, $k_1\in\mathbb{Z}_{\geq 0}$, and $k_2\in\mathbb{Z}$. If $k_2<0$, then $q$ has at least two negative coordinates, so $q\in\mathfrak{A}\cap\mathfrak{B}$, which is a contradiction. Thus, $k_2\geq 0$, so $q\in\mathbb{Z}^3_{\geq 0}$, and since $q$ is an element of exactly one of $\mathfrak{A}$ and $\mathfrak{B}$, we conclude that $q\in((\I^-\cup\III)\setminus A)\triangle((\II\cup\III)\setminus B)$. If $q\in((\I^-\cup\III)\setminus A)\setminus((\II\cup\III)\setminus B)$, then $q\not\in\I^-$, since $q\in\mathbb{Z}^3_{\geq 0}$, so we have $q\in\III\cap B\setminus A\subseteq\mathcal{L}(A, B)$. Otherwise, $q\in((\II\cup\III)\setminus B)\setminus((\I^-\cup\III)\setminus A)$, so $q\in(\II\setminus B)\cup(\III\cap A\setminus B)\subseteq\mathcal{L}(A, B)$. 
\end{proof}

\begin{lemma}
\label{lemma:exit from labelling set}
If a cell $p\in\mathcal{L}(A, B)$ is adjacent to a cell $p'\not\in\mathcal{L}(A, B)$, and $p'\in\mathfrak{A}\triangle\mathfrak{B}$, then $p'\not\in\mathbb{Z}^3_{\geq 0}\cup\I^-\cup\II\cup\III$ and $p\in(\I^-\cap A)\cup(\II\setminus B)$. 
\end{lemma}

\begin{proof}
Suppose $p'\in\mathbb{Z}^3_{\geq 0}$. Then, by the argument given in the proof of Lemma~\ref{lemma:nonzero height on axis implies labelling set}, $p'\in\mathcal{L}(A, B)$. So, by contradiction, $p'$ has at least one negative coordinate, which means that $p'\in\mathfrak{B}$. Then we must have $p'\not\in\mathfrak{A}$, so the other coordinates of $p'$ must be nonnegative. Furthermore, suppose $p'\in\I^-$. Then, since $p'\not\in\mathcal{L}(A, B)$, $p'\not\in A$, so $p'\in(\I^-\cup\III)\setminus A$, contradicting the fact that $p'\not\in\mathfrak{A}$. By contradiction, $p'\not\in\I^-$. Since $p'\not\in\mathbb{Z}^3_{\geq 0}$, $p'\not\in\II\cup\III$. 

Either $p\in BN(p')$ or $p'\in BN(p)$. If $p\in BN(p')$, then since $p'\not\in\mathbb{Z}^3_{\geq 0}$, $p$ has a negative coordinate, so $p\in\I^-\cap A$. Otherwise, $p'\in BN(p)$, so by Lemma~\ref{lemma:type III back neighbors}, $p\not\in\III$, since $p'\not\in\I^-\cup\III$. Then $p\in(\I^-\cap A)\cup(\II\setminus B)$. In either case, $p\in(\I^-\cap A)\cup(\II\setminus B)$. 
\end{proof}

\begin{lemma}
\label{lemma:finitely many unlabelled faces in sectors}
Given any $i\in\{1, 2, 3\}$, there exists $N\in\mathbb{Z}_{\geq 0}$ such that each face contained in sector $i$ that isn't a face of the subgraph $H(N)\subseteq H$ is in $F\setminus U_{(A, B)}$. 
\end{lemma}

\begin{proof}
As noted in the proof of Lemma~\ref{lemma:dual component cardinality}, if $f\in U_{(A, B)}$, then either $f$ lies along one of the nonnegative coordinate axes, $f$ corresponds to a cell $p\in\I^-\setminus A$ whose single negative coordinate has the value $-1$, or $f$ corresponds to a cell $p\in\III\setminus A$. Since the set of cells in $\I^-$ whose single negative coordinate has the value $-1$ is finite, and $\III$ is finite, the corresponding faces form a finite set. In other words, $U_{(A, B)}$ is contained in the union of faces lying along one of the nonnegative coordinate axes with finitely many other faces. In particular, since faces lying along one of the nonnegative coordinate axes are not contained in any of the sectors, finitely many faces in $U_{(A, B)}$ are contained in sector $i$. This implies the result. 
\end{proof}

\begin{lemma}
\label{lemma:infinite dual component implies label}
Suppose $C$ is a connected component of $H^{\vee}_{(A, B)}$ that contains infinitely many faces contained in sector $i$. If $p\in\mathcal{L}(A, B)$ corresponds to $f\in C$, then there exists $p'\in(\Cyl_i^-\cap A)\cup(\II_{\bar{i}}\setminus B)$ corresponding to $f'\in C$. 
\end{lemma}

\begin{proof}
Suppose $p\in\mathcal{L}(A, B)$ corresponds to $f\in C$. By Lemma~\ref{lemma:finitely many unlabelled faces in sectors}, there exists $N_1\in\mathbb{Z}_{\geq 0}$ such that each face contained in sector $i$ that isn't a face of the subgraph $H(N_1)\subseteq H$ is in $F\setminus U_{(A, B)}$. 

Consider a face $g$ contained in sector $i$ such that the face $g'$ obtained from $g$ by translating $1$ unit in the negative $x_i$-direction is separated from $g$ by an edge $e\in M_A$. Since $g$ is contained in sector $i$, if $q$ is a cell corresponding to $g$, then its $i$th coordinate $q_i$ is strictly less than each of its other coordinates. Since $g'$ is obtained from $g$ by translating $1$ unit in the negative $x_i$-direction, when crossing $e\in M_A$ from $g$ to $g'$, the left vertex of $e$ is white, so $h_A$ increases by $2/3$. That is, if $q$ is the cell corresponding to $g$ such that $\mathfrak{A}$ lies at $q$, then $\mathfrak{A}$ lies at the cell $q-\mathbf{e}_i+(1, 1, 1)$, which corresponds to $g'$. So, $q\not\in(\I^-\cup\III)\setminus A$ and $q$ has fewer than two negative coordinates, but $q-\mathbf{e}_i\in\mathfrak{A}$, so $q-\mathbf{e}_i\in(\I^-\cup\III)\setminus A$ or $q-\mathbf{e}_i$ has at least two negative coordinates. However, the $i$th coordinate of $q$ is less than each of its other coordinates, so if $q-\mathbf{e}_i$ has at least two negative coordinates, then so does $q$, which is a contradiction. Consequently, $q-\mathbf{e}_i\in(\I^-\cup\III)\setminus A$. Then, since the $i$th coordinate of $q$ is its least coordinate, the same is true of $q-\mathbf{e}_i$, so $q-\mathbf{e}_i\in\Cyl_i$. This means that $q\in\Cyl_i$. Additionally, if $q$ has one negative coordinate, it must be $q_i$, in which case $q\in\Cyl_i^-\subseteq\I^-$, implying that $q\in A$. Otherwise, each of the coordinates of $q$ is nonnegative and less than $M$, since $q\in\Cyl_i$. Therefore, since $A$ is finite, there are finitely many possibilities for $q$, so there are finitely many possibilities for $g$. So, there exists $N_2\in\mathbb{Z}_{\geq 0}$ such that each face $g$ contained in sector $i$ that isn't a face of the subgraph $H(N_2)\subseteq H$ is separated by an edge $e\not\in M_A$ from the face $g'$ obtained from $g$ by translating $1$ unit in the negative $x_i$-direction. 

Let $N=\max\{N_1, N_2\}$. Since $C$ contains infinitely many faces contained in sector $i$, it must contain a face $f_0$ contained in sector $i$ that isn't a face of $H(N)$. Consider the sequence of faces $f_0, f_1, f_2, \ldots$, where $f_{s+1}$ is obtained from $f_s$ by translating $1$ unit in the negative $x_i$-direction. Since $f_0$ is contained in sector $i$ and not a face of $H(N)$, so is $f_s$, for $0\leq s$. Then, from the above discussions, we know that $f_s\in F\setminus U_{(A, B)}$ and $f_s$ is separated by an edge $e\not\in M_A$ from $f_{s+1}$ for $0\leq s$. In addition, by Lemma~\ref{lemma:avoiding doubled edges}, there is a sequence of adjacent faces $f:=f_0', f_1', \ldots, f_r':=f_0$ in $F\setminus U_{(A, B)}$ such that no pair of consecutive faces are separated by an edge in $M_A\cap M_B$. So, we have a sequence of adjacent faces $f_0', f_1', \ldots, f_r', f_1, f_2, \ldots$ in $F\setminus U_{(A, B)}$ such that no pair of consecutive faces are separated by an edge in $M_A\cap M_B$. 

Since $f_s$ is contained in sector $i$ for $0\leq s$, either (i): every face in the sequence \\
 $f_0', f_1', \ldots, f_r', f_1, f_2, \ldots$ is contained in sector $i$ or (ii): there exists $0\leq t<r$ such that $f_t'$ is not contained in sector $i$ and $f_s'$ is contained in sector $i$ for all $t<s\leq r$. In case (i), let $t=0$, and let $p_0'=p$. In case (ii), since $f_t'$ is adjacent to $f_{t+1}'$, which is contained in sector $i$, $f_t'$ must lie along one of the nonnegative coordinate axes. Since $f_t'\in F\setminus U_{(A, B)}$, by Lemma~\ref{lemma:nonzero height on axis implies labelling set}, there is a cell $p_t'\in\mathcal{L}(A, B)$ corresponding to $f_t'$. 

In both case (i) and case (ii), $p_t'\in\mathcal{L}(A, B)$ corresponds to $f_t'$. Also, since $\mathcal{L}(A, B)\subseteq A\cup\II\cup\III$ is finite and the faces $f_1, f_2, \ldots$ are all distinct, there must be a face in the sequence $f_0', f_1', \ldots, f_r', f_1, f_2, \ldots$ that does not correspond to any cell in $\mathcal{L}(A, B)$ and that is preceded by $f_t'$. Let $g''$ be the first such face in the sequence and let $g'$ be the previous face. Either $g'$ corresponds to a cell $q'\in\mathcal{L}(A, B)$ or $g'$ is not preceded by $f_t'$, in which case, $g'=f_t'$ corresponds to $q':=p_t'\in\mathcal{L}(A, B)$. Then, by Lemmas~\ref{lemma:labelling set implies nonzero height} and~\ref{lemma:face sequence implies cell sequence}, there exist an integer $k'$ and a cell $q''$ so that for any $j$, $k$ such that $\{j, k\}=\{1, 2, 3\}\setminus\{i\}$, the following are sequences of adjacent cells in $\mathfrak{A}\triangle\mathfrak{B}$, such that $q''$ corresponds to $g''$: \begin{align*}q', q'&+\sgn(k')\mathbf{e}_i, q'+\sgn(k')(\mathbf{e}_i+\mathbf{e}_j), q'+\sgn(k')(\mathbf{e}_i+\mathbf{e}_j+\mathbf{e}_k), \\q'&+\sgn(k')(2\mathbf{e}_i+\mathbf{e}_j+\mathbf{e}_k), \ldots, q'+(k'\mathbf{e}_i+k'\mathbf{e}_j+k'\mathbf{e}_k), q'',\end{align*} \begin{align*}q', q'&+\sgn(k')\mathbf{e}_k, q'+\sgn(k')(\mathbf{e}_k+\mathbf{e}_j), q'+\sgn(k')(\mathbf{e}_k+\mathbf{e}_j+\mathbf{e}_i), \\q'&+\sgn(k')(2\mathbf{e}_k+\mathbf{e}_j+\mathbf{e}_i), \ldots, q'+(k'\mathbf{e}_k+k'\mathbf{e}_j+k'\mathbf{e}_i), q''.\end{align*} If $k'<0$, we will consider the first sequence, and if $k'\geq 0$, we will consider the second sequence. Since $g''$ does not correspond to any cell in $\mathcal{L}(A, B)$, $q''\not\in\mathcal{L}(A, B)$. On the other hand, $q'\in\mathcal{L}(A, B)$, so let $p''$ be the first cell in the sequence that is not in $\mathcal{L}(A, B)$ and let $p'$ be the previous cell. Then $p'\in\mathcal{L}(A, B)$. Let $f'$ be the face corresponding to $p'$ and let $f''$ be the face corresponding to $p''$. We must show that $p'\in(\Cyl_i^-\cap A)\cup(\II_{\bar{i}}\setminus B)$ and $f'\in C$. 

Since $p''\in(\mathfrak{A}\triangle\mathfrak{B})\setminus\mathcal{L}(A, B)$, we have $f''\in F\setminus U_{(A, B)}$ and by Lemma~\ref{lemma:nonzero height on axis implies labelling set}, $f''$ does not lie along one of the nonnegative coordinate axes. However, $g'$ is adjacent to $g''$, which is preceded by $f_t'$ in the sequence $f_0', f_1', \ldots, f_r', f_1, f_2, \ldots$ and, thus, is contained in sector $i$. As a result, $g'$ is contained in sector $i$ or $g'$ lies along one of the nonnegative coordinate axes. More precisely, $g'$ corresponds to a cell whose $i$th coordinate is zero and whose other coordinates are nonnegative, or to put it another way, the $i$th coordinate of $q'$ is less than or equal to its other coordinates. If $k'<0$, let $g_1$ be the face corresponding to $q'+\sgn(k')\mathbf{e}_i$ and let $g_2$ be the face corresponding to $q'+\sgn(k')(\mathbf{e}_i+\mathbf{e}_j)$. If $k'\geq 0$, let $g_1$ be the face corresponding to $q'+\sgn(k')\mathbf{e}_k$ and let $g_2$ be the face corresponding to $q'+\sgn(k')(\mathbf{e}_k+\mathbf{e}_j)$. Note that every cell in the sequence corresponds to one of the faces $g'$, $g_1$, $g_2$, or $g''$. We claim that $g_1$ is contained in sector $i$ or $g_1$ lies along one of the nonnegative coordinate axes, and the same holds for $g_2$. If $k'<0$, then $g_1$ corresponds to $q'-\mathbf{e}_i$ and $g_2$ corresponds to $q'-\mathbf{e}_i-\mathbf{e}_j$. Since the $i$th coordinate of $q'$ is less than or equal to its other coordinates, the same is true of $q'-\mathbf{e}_i$ and $q'-\mathbf{e}_i-\mathbf{e}_j$, so the claim holds for both $g_1$ and $g_2$. Otherwise, if $k'\geq 0$, then $g_1$ corresponds to $q'$ or $q'+\mathbf{e}_k$, while $g_2$ corresponds to $q'$ or $q'+\mathbf{e}_k+\mathbf{e}_j$. Again, since the $i$th coordinate of $q'$ is less than or equal to its other coordinates, the same is true of $q'+\mathbf{e}_k$ and $q'+\mathbf{e}_k+\mathbf{e}_j$, so the claim holds for both $g_1$ and $g_2$. In fact, we saw above that the claim also holds for both $g'$ and $g''$, and since every cell in the sequence, including $p''$, corresponds to one of the faces $g'$, $g_1$, $g_2$, or $g''$, the claim holds for $f''$. Since $f''$ does not lie along one of the nonnegative coordinate axes, we conclude that $f''$ is contained in sector $i$. It follows that the $i$th coordinate $p''_i$ of $p''$ is strictly less than its other coordinates. 

Recall that $p'\in\mathcal{L}(A, B)$ is adjacent to $p''\not\in\mathcal{L}(A, B)$, but $p''\in\mathfrak{A}\triangle\mathfrak{B}$. By Lemma~\ref{lemma:exit from labelling set}, $p''\not\in\mathbb{Z}^3_{\geq 0}\cup\I^-\cup\II\cup\III$ and $p'\in(\I^-\cap A)\cup(\II\setminus B)$. Since the $i$th coordinate of $p''$ is less than its other coordinates, $p''_i<0$. This implies that $p''\not\in\Cyl_i$. If $p'\in\I^-\cap A$, suppose the $i$th coordinate $p'_i$ of $p'$ is nonnegative. Then, since the $i$th coordinate of $p''$ is negative, while the others are nonnegative, we must have $p'=p''+\mathbf{e}_i$. Therefore, the other coordinates of $p'$ are the same as those of $p''$, so they are also nonnegative and $p'\in\mathbb{Z}^3_{\geq 0}$. By contradiction, $p'_i<0$, so $p'\in\Cyl_i^-\cap A$. On the other hand, if $p'\in\II\setminus B$, then $p'\in\mathbb{Z}^3_{\geq 0}$, so $p'=p''+\mathbf{e}_i$. Suppose $p'\in\Cyl_i$. Then $p''=p'-\mathbf{e}_i\in\Cyl_i$, so $p''\in\Cyl_i^-\subseteq\I^-$. By contradiction, $p'\not\in\Cyl_i$, so $p'\in\II_{\bar{i}}\setminus B$. Consequently, $p'\in(\Cyl_i^-\cap A)\cup(\II_{\bar{i}}\setminus B)$. 

It remains to show that $f'\in C$. Since $f'$ corresponds to $p'$, $f'$ is equal to $g'$, $g_1$, $g_2$, or $g''$, so it suffices to show that $g', g_1, g_2, g''\in C$. Since $g'$ and $g''$ are faces in the sequence $f_0', f_1', \ldots, f_r', f_1, f_2, \ldots$, which is a sequence of adjacent faces in $F\setminus U_{(A, B)}$, we have $g', g''\in C_{(A, B)}(f_0')=C_{(A, B)}(f)=C$. To see that $g_1, g_2\in C$, observe that $g_1$ and $g_2$ are adjacent to $g'$ or equal to $g'$, and according to their definitions, they correspond to cells in $\mathfrak{A}\triangle\mathfrak{B}$. So $g_1, g_2\in F\setminus U_{(A, B)}$, and we have $g_1, g_2\in C_{(A, B)}(g')=C$. This completes the proof. 
\end{proof}

\begin{lemma}
\label{lemma:dual component implies cell sequence}
If $p$ and $p'$ are cells in $\mathfrak{A}\triangle\mathfrak{B}$, corresponding to faces $f$ and $f'$, respectively, which belong to the same connected component of $H^{\vee}_{(A, B)}$, then there is a sequence of adjacent cells in $\mathfrak{A}\triangle\mathfrak{B}$, beginning at $p$ and ending at $p'$. 
\end{lemma}

\begin{proof}
By Lemmas~\ref{lemma:avoiding doubled edges} and~\ref{lemma:face sequence implies cell sequence}, there is a sequence of adjacent cells in $\mathfrak{A}\triangle\mathfrak{B}$, beginning at $p$ and ending at a cell $p''$ that corresponds to $f'$. Since $p'$ and $p''$ both correspond to $f'$, $p''=p'+(k', k', k')$ for some $k'\in\mathbb{Z}$. By Lemma~\ref{lemma:AB surfaces}, $p'+\sgn(k')(m_1, m_2, m_3)\in\mathfrak{A}\triangle\mathfrak{B}$ for any $m_1, m_2, m_3$ such that $0\leq m_1, m_2, m_3\leq\lvert k'\rvert$. That is, \begin{align*}p''=p'+(k', k', k'), p'&+\sgn(k')(\lvert k'\rvert-1, \lvert k'\rvert, \lvert k'\rvert), \\p'&+\sgn(k')(\lvert k'\rvert-1, \lvert k'\rvert-1, \lvert k'\rvert), \ldots, p'\end{align*} is a sequence of adjacent cells in $\mathfrak{A}\triangle\mathfrak{B}$. Therefore, by concatenating the aforementioned sequences, we get a sequence of adjacent cells in $\mathfrak{A}\triangle\mathfrak{B}$, beginning at $p$ and ending at $p'$. 
\end{proof}

\begin{lemma}
\label{lemma:not almost contained components}
Suppose $C$ is a connected component of $H^{\vee}_{(A, B)}$ so that, given any $i$, $C$ is not almost contained in sector $i$. Then there exist distinct $i$ and $j$ such that $C$ contains infinitely many faces contained in sector $i$ and $C$ contains infinitely many faces contained in sector $j$. 
\end{lemma}

\begin{proof}
By assumption, given any $i$, $C$ contains infinitely many faces that are not contained in sector $i$. Observe that, for $N\geq M$, the cell $N\mathbf{e}_i\in\mathbb{Z}^3_{\geq 0}$ cannot be in $\Cyl_j$, for each $j\neq i$. Thus $N\mathbf{e}_i$ has no negative coordinates, $N\mathbf{e}_i\not\in(\I^-\cup\III)\setminus A$, and $N\mathbf{e}_i\not\in(\II\cup\III)\setminus B$, so $N\mathbf{e}_i\not\in\mathfrak{A}\cup\mathfrak{B}$. Moreover, $N\mathbf{e}_i-(1, 1, 1)$ has at least two negative coordinates, so $N\mathbf{e}_i-(1, 1, 1)\in\mathfrak{A}\cap\mathfrak{B}$, which shows that $\mathfrak{A}$ and $\mathfrak{B}$ both lie at $N\mathbf{e}_i$. So, if $f_i(N)\in F$ is the face corresponding to $N\mathbf{e}_i$, then $h_{(A, B)}(f_i(N))=0$. Since this holds for all $i$ and all $N\geq M$, there are finitely many faces in $F\setminus U_{(A, B)}$ that lie along one of the nonnegative coordinate axes. Since $C\subseteq F\setminus U_{(A, B)}$ and since any face either lies along one of the nonnegative coordinate axes or is contained in one of the sectors, we deduce that, for some distinct $i$ and $j$, $C$ contains infinitely many faces contained in sector $i$ and $C$ contains infinitely many faces contained in sector $j$. 
\end{proof}

\begin{thm}
\label{thm:double-dimer labelling success}
Algorithm~\ref{algorithm:double-dimer labelling} succeeds if and only if $(A, B)\in\sAB$. 
\end{thm}

\begin{proof}
Suppose $(A, B)\not\in\sAB$. By Theorem~\ref{thm:labellable iff in sAB} and Remark~\ref{remark:N(C)}, there is a connected component $C$ of $\mathcal{L}(A, B)$ such that $\mathcal{N}(C)>1$. So, there exist $w, w'\in C\cap(\I^-\cup\II)$ such that $\ell(w)\neq\ell(w')$. Then $w\in\Cyl^-_{\ell(w)}\cup\II_{\overline{\ell(w)}}$ and $w'\in\Cyl^-_{\ell(w')}\cup\II_{\overline{\ell(w')}}$, and since $w, w'\in C\cap(\I^-\cup\II)\subseteq\mathcal{L}(A, B)\cap(\I^-\cup\II)=(\I^-\cap A)\cup(\II\setminus B)$, we have $w\in(\Cyl^-_{\ell(w)}\cap A)\cup(\II_{\overline{\ell(w)}}\setminus B)$ and $w'\in(\Cyl^-_{\ell(w')}\cap A)\cup(\II_{\overline{\ell(w')}}\setminus B)$. Let $f\in F$ and $f'\in F$ be the faces corresponding to $w$ and $w'$, respectively. By Lemma~\ref{lemma:dual component cardinality}, $C_{(A, B)}(f)$ contains infinitely many faces contained in sector $\ell(w)$, and $C_{(A, B)}(f')$ contains infinitely many faces contained in sector $\ell(w')$. Since $w, w'\in C$ and $C$ is a connected component of $\mathcal{L}(A, B)$, there is a sequence of adjacent cells $w:=p_0, p_1, \ldots, p_r:=w'$ in $\mathcal{L}(A, B)$. Then, assuming $p_s$ corresponds to the face $f_s\in F$, we obtain a sequence of adjacent faces $f=f_0, f_1, \ldots, f_r=f'$. By Lemma~\ref{lemma:labelling set implies nonzero height}, $h_{(A, B)}(f_s)\neq 0$, so $C_{(A, B)}(f)=C_{(A, B)}(f')$. Since $\ell(w)\neq\ell(w')$, this means that a connected component of $H^{\vee}_{(A, B)}$ contains infinitely many faces contained in distinct sectors. It is impossible for such a connected component to be almost contained in any sector, so Algorithm~\ref{algorithm:double-dimer labelling} fails. 

Conversely, suppose Algorithm~\ref{algorithm:double-dimer labelling} fails. Then there must be a connected component $C$ of $H^{\vee}_{(A, B)}$ so that, given any $i$, $C$ is not almost contained in sector $i$. By Lemma~\ref{lemma:not almost contained components}, for some distinct $i$ and $j$, $C$ contains infinitely many faces contained in sector $i$ and $C$ contains infinitely many faces contained in sector $j$. 

Let $f\in C$ be a face contained in sector $i$ and $f'\in C$ be a face contained in sector $j$. Since $C$ is a connected component of $H^{\vee}_{(A, B)}$, there is a sequence of adjacent faces $f:=f_0, f_1, \ldots, f_r:=f'$ in $F\setminus U_{(A, B)}$. Since $f$ is contained in sector $i$, while $f'$ is contained in sector $j$, there must exist $0<t<r$ such that $f_t$ lies along one of the nonnegative coordinate axes. Then, by Lemma~\ref{lemma:nonzero height on axis implies labelling set}, $f_t$ corresponds to a cell $p_t\in\mathcal{L}(A, B)$. Since $f_t\in C$, by Lemma~\ref{lemma:infinite dual component implies label}, there exist $p\in(\Cyl_i^-\cap A)\cup(\II_{\bar{i}}\setminus B)$ corresponding to $g\in C$ and $p'\in(\Cyl_j^-\cap A)\cup(\II_{\bar{j}}\setminus B)$ corresponding to $g'\in C$. Then, by Lemmas~\ref{lemma:labelling set implies nonzero height} and~\ref{lemma:dual component implies cell sequence}, there is a sequence of adjacent cells in $\mathfrak{A}\triangle\mathfrak{B}$, beginning at $p$ and ending at $p'$. Let $q$ be the last cell in this sequence that is in $(\Cyl_i^-\cap A)\cup(\II_{\bar{i}}\setminus B)$, and let $q'$ be the first cell in this sequence that is preceded by $q$ and in $(\Cyl_k^-\cap A)\cup(\II_{\bar{k}}\setminus B)$ for some $k\in\{1, 2, 3\}\setminus\{i\}$. Consider the part of the sequence beginning at $q$ and ending at $q'$, denoted $q:=q_0, q_1, \ldots, q_{r'}:=q'$. Each of these cells is an element of $\mathfrak{A}\triangle\mathfrak{B}$, and according to the definitions of $q$ and $q'$, $q_s\not\in\bigcup_{l\in\{1, 2, 3\}}(\Cyl_l^-\cap A)\cup(\II_{\bar{l}}\setminus B)=(\I^-\cap A)\cup(\II\setminus B)$ for $0<s<r'$. 

We claim that $q_s\in\mathcal{L}(A, B)$ for $0\leq s\leq r'$. Suppose $q_{t'}\not\in\mathcal{L}(A, B)$ for some $0\leq t'\leq r'$. Since $q, q'\in(\I^-\cap A)\cup(\II\setminus B)\subseteq\mathcal{L}(A, B)$, $0<t'<r'$. Then, by Lemma~\ref{lemma:exit from labelling set}, $t'-1=0$ or $q_{t'-1}\not\in\mathcal{L}(A, B)$, and $t'+1=r'$ or $q_{t'+1}\not\in\mathcal{L}(A, B)$. In fact, by repeating this argument, we see that $q_s\not\in\mathcal{L}(A, B)$ for $0<s<r'$. By the same lemma, $q_1, q_{r'-1}\not\in\mathbb{Z}^3_{\geq 0}\cup\I^-\cup\II\cup\III$, so $q_1, q_{r'-1}\not\in\Cyl_1\cup\Cyl_2\cup\Cyl_3$. Since $q_1, q_{r'-1}\in\mathfrak{A}\triangle\mathfrak{B}$, neither $q_1$ nor $q_{r'-1}$ has at least two negative coordinates, but $q_1, q_{r'-1}\not\in\mathbb{Z}^3_{\geq 0}$, so $q_1$ and $q_{r'-1}$ each have exactly one negative coordinate. Furthermore, $q\in\Cyl_i^-\cup\II_{\bar{i}}$ is adjacent to $q_1$, and $q'\in\Cyl_k^-\cup\II_{\bar{k}}$ is adjacent to $q_{r'-1}$. If $q\in\Cyl_i^-$, then since $q_1\not\in\Cyl_i$, $q_1\neq q\pm \mathbf{e}_i$, so the $i$th coordinate of $q_1$ is the same as that of $q$. In particular, the $i$th coordinate of $q_1$ is negative. Otherwise, $q\in\II_{\bar{i}}\subseteq\mathbb{Z}^3_{\geq 0}$. Since $q_1\not\in\mathbb{Z}^3_{\geq 0}$, this implies that $q\not\in BN(q_1)$, so $q_1\in BN(q)$. Since $q_1\not\in\Cyl_1\cup\Cyl_2\cup\Cyl_3$ and $q\in\Cyl_l$ for $l\in\{1, 2, 3\}\setminus\{i\}$, $q_1\neq q-\mathbf{e}_l\in\Cyl_l$ for $l\in\{1, 2, 3\}\setminus\{i\}$. It follows that $q_1=q-\mathbf{e}_i$, and since $q_1\not\in\mathbb{Z}^3_{\geq 0}$, while $q\in\mathbb{Z}^3_{\geq 0}$, the $i$th coordinate of $q_1$ must be negative. In both cases, the $i$th coordinate of $q_1$ is negative, and since $q_1$ has exactly one negative coordinate, the other coordinates of $q_1$ must be nonnegative. A similar argument shows that the $k$th coordinate of $q_{r'-1}$ is negative, while the other coordinates of $q_{r'-1}$ are nonnegative. Thus, denoting by $g_s$ the face corresponding to the cell $q_s$, for $0\leq s\leq r'$, we conclude that $g_1$ and $g_{r'-1}$ are contained in sector $i$ and contained in sector $k$, respectively. 

Since $q_s$ is adjacent to $q_{s+1}$, $g_s$ is adjacent to $g_{s+1}$, for $0\leq s<r'$. As a result, since $k\neq i$, there must exist $1<t''<r'-1$ such that $g_{t''}$ lies along one of the nonnegative coordinate axes. Consequently, since $q_{t''}\in\mathfrak{A}\triangle\mathfrak{B}$, we have $g_{t''}\in F\setminus U_{(A, B)}$, and by Lemma~\ref{lemma:nonzero height on axis implies labelling set}, $q_{t''}\in\mathcal{L}(A, B)$. This contradicts our previous conclusion that $q_s\not\in\mathcal{L}(A, B)$ for $0<s<r'$. By contradiction, $q_s\in\mathcal{L}(A, B)$ for $0\leq s\leq r'$. So, there is a connected component $C'$ of $\mathcal{L}(A, B)$ such that $q_s\in C'$ for $0\leq s\leq r'$, and we have $\mathcal{N}(C')\geq\lvert\{\ell(q), \ell(q')\}\rvert=\lvert\{i, k\}\rvert=2$. By Remark~\ref{remark:N(C)} and Theorem~\ref{thm:labellable iff in sAB}, $(A, B)\not\in\sAB$. 
\end{proof}

\begin{thm}
\label{thm:labelling agreement}
If $(A, B)\in\sAB$, and Algorithm~\ref{algorithm:AB labelling algorithm} labels some cell by $\ell$, then Algorithm~\ref{algorithm:double-dimer labelling} labels the corresponding face by $\ell$.
\end{thm}

\begin{proof}
Suppose $(A, B)\in\sAB$ (so, by Theorem~\ref{thm:double-dimer labelling success}, Algorithm~\ref{algorithm:double-dimer labelling} succeeds), and Algorithm~\ref{algorithm:AB labelling algorithm} labels a cell $w$ by $\ell$. Let $f\in F$ be the corresponding face. By Lemma~\ref{lemma:dual component cardinality}, if $\ell$ is an integer, then $C_{(A, B)}(f)$ contains infinitely many faces contained in sector $\ell$, and otherwise, $C_{(A, B)}(f)$ is finite. In the first case, there must be exactly one sector $i$ almost containing $C_{(A, B)}(f)$, and Algorithm~\ref{algorithm:double-dimer labelling} labels the faces in $C_{(A, B)}(f)$ by $i$. Then $C_{(A, B)}(f)$ contains only finitely many faces that are not contained in sector $i$, and since faces contained in sector $k$ are not contained in sector $i$ if $k\neq i$, we must have $\ell=i$. So Algorithm~\ref{algorithm:double-dimer labelling} labels the faces in $C_{(A, B)}(f)$, including $f$, by $\ell$. In the second case, Algorithm~\ref{algorithm:double-dimer labelling} labels the faces in $C_{(A, B)}(f)$, including $f$, by a single freely chosen element of $\mathbb{P}^1$. In this case, we must establish two statements: (i) each cell given the label $\ell$ by Algorithm~\ref{algorithm:AB labelling algorithm} corresponds to a face in $C_{(A, B)}(f)$ and (ii) each cell given a freely chosen label $\ell'\neq\ell$ by Algorithm~\ref{algorithm:AB labelling algorithm} corresponds to a face not in $C_{(A, B)}(f)$. 

Suppose $w'$ is a cell given the label $\ell$ by Algorithm~\ref{algorithm:AB labelling algorithm} and $f'\in F$ is the corresponding face. Since $\ell$ is not an integer, $w$ and $w'$ must be in a single connected component of $\mathcal{L}(A, B)$ labelled in step 3 of Algorithm~\ref{algorithm:AB labelling algorithm}. So, there must be a sequence of adjacent cells in $\mathcal{L}(A, B)$, beginning at $w$ and ending at $w'$. Then, by Lemma~\ref{lemma:labelling set implies nonzero height}, the corresponding faces form a sequence of adjacent faces, each of which is in $F\setminus U_{(A, B)}$. This sequence begins at $f$ and ends at $f'$, so $f'\in C_{(A, B)}(f)$. 

Suppose $w'$ is a cell given a freely chosen label $\ell'$ by Algorithm~\ref{algorithm:AB labelling algorithm} and $f'\in F$ is the corresponding face. We will show that if $f'\in C_{(A, B)}(f)$, then $\ell'=\ell$. Suppose $f'\in C_{(A, B)}(f)$. Then, by Lemmas~\ref{lemma:labelling set implies nonzero height} and~\ref{lemma:dual component implies cell sequence}, $w, w'\in\mathcal{L}(A, B)$ and there is a sequence of adjacent cells $w:=w_0, w_1, \ldots, w_r:=w'$ in $\mathfrak{A}\triangle\mathfrak{B}$. We claim that $w_s\in\mathcal{L}(A, B)$ for $0\leq s\leq r$. Suppose not. Let $0\leq t\leq r$ be such that $w_t$ is the first cell in the sequence that is not in $\mathcal{L}(A, B)$. Then $w_s\in\mathcal{L}(A, B)$ for $0\leq s<t$ and, by Lemma~\ref{lemma:exit from labelling set}, $w_{t-1}\in(\I^-\cap A)\cup(\II\setminus B)$. Note that $w_{t-1}$ gets labelled by an integer $j$ in step 2 of Algorithm~\ref{algorithm:AB labelling algorithm}. Since $w_0, w_1, \ldots, w_{t-1}$ is a sequence of adjacent cells, we see that $\{w_0, w_1, \ldots, w_{t-1}\}$ is contained in a single connected component of $\mathcal{L}(A, B)$, which is labelled in step 2 of Algorithm~\ref{algorithm:AB labelling algorithm} by $j$. In particular, $w=w_0$ is labelled in step 2 of Algorithm~\ref{algorithm:AB labelling algorithm} by an integer $j$, contradicting the fact that $\ell$ is not an integer. By contradiction, $w_s\in\mathcal{L}(A, B)$ for $0\leq s\leq r$. It follows that $w$ and $w'$ belong to a single connected component of $\mathcal{L}(A, B)$, so $\ell=\ell'$, as desired. 
\end{proof}

As promised, we will now describe the connection between Algorithm~\ref{algorithm:double-dimer labelling} and the double-dimer configuration $D_{(A, B)}$. 

\begin{thm}
\label{thm:in sAB iff path ends in single sector}
$(A, B)\in\sAB$ if and only if for each path in $D_{(A, B)}$, there exists $i\in\{1, 2, 3\}$ such that both ends of the path are contained in sector $i$. 
\end{thm}

\begin{proof}
Suppose there exists a path in $D_{(A, B)}$ whose ends are not contained in the same sector. Then, by Remark~\ref{remark:end contained in sector}, one end $\mathcal{E}_i$ is contained in sector $i$, the other end $\mathcal{E}_j$ is contained in sector $j$, and $i\neq j$. Let $e$ be any edge in the path, and consider a face $f\in F\setminus U_{(A, B)}$ incident to $e$, which exists because $h_{(A, B)}$ must increase or decrease upon crossing $e$ from one side of the path to the other. 

The following argument now holds for $k\in\{i, j\}$. Consider the sequence of edges $e:=e_0, e_1, e_2, \ldots$ obtained by beginning at $e$ and moving along the path toward $\mathcal{E}_k$. Each edge $e_s$ in this sequence is incident to a unique face $f_s$ on the same side of the path as $f$. In particular, $f_0=f$. In fact, since $e_s$ is adjacent to $e_{s+1}$, $f_s$ is equal to or adjacent to $f_{s+1}$ for $0\leq s$. Moreover, if $f_s$ and $f_{s+1}$ are adjacent, then since $e_s, e_{s+1}\in M_A\cup M_B$, $f_s$ and $f_{s+1}$ are separated by an edge that is in neither $M_A$ nor $M_B$. Thus, we have a sequence of equal or adjacent faces $f=f_0, f_1, f_2, \ldots$, and $h_{(A, B)}(f_s)=h_{(A, B)}(f)\neq 0$ for $0\leq s$. It follows that $f_s\in C_{(A, B)}(f)$ for $0\leq s$. Also, since $\mathcal{E}_k$ is contained in sector $k$, there exists $t\geq 0$ such that $e_s$ is contained in sector $k$ for $s>t$. Then, if $l\in\{1, 2, 3\}\setminus\{k\}$, $f_s$ is not contained in sector $l$ for $s>t$. Finally, since every face is incident to six edges and each edge appears in the sequence $e_0, e_1, e_2, \ldots$ at most once, any given face can appear in the sequence $f_0, f_1, f_2, \ldots$ at most six times. In other words, $\{f_{t+1}, f_{t+2}, f_{t+3}, \ldots\}\subseteq C_{(A, B)}(f)$ is an infinite set, so if $l\in\{1, 2, 3\}\setminus\{k\}$, $C_{(A, B)}(f)$ contains infinitely many faces that are not contained in sector $l$. As a result, if $l\in\{1, 2, 3\}\setminus\{k\}$, $C_{(A, B)}(f)$ is not almost contained in sector $l$. Since this argument holds for $k\in\{i, j\}$, and $(\{1, 2, 3\}\setminus\{i\})\cup(\{1, 2, 3\}\setminus\{j\})=\{1, 2, 3\}$, $C_{(A, B)}(f)$ is not almost contained in any sector. Consequently, Algorithm~\ref{algorithm:double-dimer labelling} fails, so by Theorem~\ref{thm:double-dimer labelling success}, $(A, B)\not\in\sAB$. 

Conversely, suppose $(A, B)\not\in\sAB$. By Theorem~\ref{thm:double-dimer labelling success}, Algorithm~\ref{algorithm:double-dimer labelling} fails, so there is a connected component $C$ of $H^{\vee}_{(A, B)}$ that is not almost contained in any sector. Then, by Lemmas~\ref{lemma:not almost contained components} and~\ref{lemma:finitely many unlabelled faces in sectors}, there exist distinct $i$ and $j$ such that $C$ contains infinitely many faces contained in sector $i$ and $C$ contains infinitely many faces contained in sector $j$, there exists $N_i\in\mathbb{Z}_{\geq 0}$ such that each face contained in sector $i$ that isn't a face of the subgraph $H(N_i)\subseteq H$ is in $F\setminus U_{(A, B)}$, and there exists $N_j\in\mathbb{Z}_{\geq 0}$ such that each face contained in sector $j$ that isn't a face of the subgraph $H(N_j)\subseteq H$ is in $F\setminus U_{(A, B)}$. So, $C$ contains a face $f_i$ contained in sector $i$ that isn't a face of $H(N_i)$, and $C$ contains a face $f_j$ contained in sector $j$ that isn't a face of $H(N_j)$. 

The following holds for $l\in\{i, j\}$. Let $k\in\{1, 2, 3\}$ satisfy $\{k\}=\{1, 2, 3\}\setminus\{i, j\}$. Observe that the set $F_l$ of faces contained in sector $l$ that are not faces of $H(N_l)$ induces a connected subgraph of $H^{\vee}$. In addition, $F_l\subseteq F\setminus U_{(A, B)}$, so $F_l$ actually induces a connected subgraph of $H^{\vee}_{(A, B)}$. Since $C$ is a connected component of $H^{\vee}_{(A, B)}$ and $f_l\in C\cap F_l$, we have $F_l\subseteq C$. For $0<s$, let $f_l(N_l+s)$ be the face corresponding to the cell $(N_l+s)\mathbf{e}_k+\mathbf{e}_m$, where $m\in\{1, 2, 3\}$ satisfies $\{m\}=\{i, j\}\setminus\{l\}$. Since the $l$th coordinate of $(N_l+s)\mathbf{e}_k+\mathbf{e}_m$ is strictly less than its other coordinates, $f_l(N_l+s)\in F_l\subseteq C$ for $0<s$. 

Since $C$ is a connected component of $H^{\vee}_{(A, B)}$, there is a sequence $f_i(N_i+1):=g_0, g_1, \ldots, g_r:=f_j(N_j+1)$ of adjacent faces in $C$. Let $N=\max\{N_i, N_j, M\}$. As discussed in the proof of Lemma~\ref{lemma:not almost contained components}, if $0\leq s$, then $h_{(A, B)}(f_k(N+s))=0$, where $f_k(N+s)$ is the face corresponding to the cell $(N+s)\mathbf{e}_k$. Also, for $0<s$, $f_k(N+s)$ is adjacent to $f_l(N+s)$ and to $f_l(N+s+1)$. Let $e_l(s)$ be the edge separating $f_k(N+s)$ and $f_l(N+s)$, and let $e_l'(s)$ be the edge separating $f_k(N+s)$ and $f_l(N+s+1)$. Since $f_l(N+s)\in C\subseteq F\setminus U_{(A, B)}$, $h_{(A, B)}(f_l(N+s))\neq 0$ for $0<s$, implying that $e_l(s), e_l'(s)\in M_A\triangle M_B$. So, the sequence of adjacent edges $e_i(1), e_i'(1), e_i(2), e_i'(2), \ldots$ constitutes one end $\mathcal{E}_i$ of a path $\gamma$ in $D_{(A, B)}$, and $\mathcal{E}_i$ is contained in sector $i$. 

Consider the other end $\mathcal{E}$ of $\gamma$. Note that $\gamma$ separates $U_{(A, B)}$ from $F\setminus U_{(A, B)}$, so $\gamma$ cannot separate two adjacent faces in the sequence \[\ldots, f_i(N_i+2), f_i(N_i+1), g_1, g_2, \ldots, g_{r-1}, f_j(N_j+1), f_j(N_j+2), \ldots,\] since each of them is in $C$ and, thus, in $F\setminus U_{(A, B)}$. On the other hand, this is a sequence of adjacent faces, so we conclude that $\gamma$ must be contained in $\left\{e_l(s), e_l'(s)\mid l\in\{i, j\}, 0<s\right\}\cup E_0$ for some finite set $E_0$. Therefore, since $e_j(s)$ and $e_j'(s)$ are contained in sector $j$, $\mathcal{E}$ must be contained in sector $j$. That is, $\gamma$ is a path in $D_{(A, B)}$ whose ends are contained in distinct sectors. This completes the proof. 
\end{proof}

%\subsubsection{From infinite to finite}

Next, in order to apply the double-dimer analogue of Kuo's graphical condensation (see Theorem~\ref{thm:DDcond}), we must truncate double-dimer configurations on $H$ to obtain {\em double-dimer configurations with nodes} on $H(N)$. 

\begin{defn}
Let $G = (V_1, V_2, E)$ be a finite, edge-weighted, bipartite planar graph embedded in the plane with $|V_1| = |V_2|$. Let ${\bf N}$ denote a set of special vertices called {\em nodes} on the outer face of $G$. A {\em double-dimer configuration} on $(G, {\bf N})$ is a multiset of the edges of $G$ with the property that each internal vertex is the endpoint of exactly two edges, and each vertex in ${\bf N}$ is the endpoint of exactly one edge.

The edge-weight of a double-dimer configuration with nodes is the product of its edge-weights. The weight of such a configuration is its edge-weight times $2^k$, where $k$ is the number of loops in the configuration.
\end{defn}

\begin{lemma}
\label{lemma:folklore for B}
For any $N\geq M$, no edge in $M_B$ is incident to a vertex in $H(N)$ and a vertex not in $H(N)$. 
\end{lemma}

\begin{proof}
Suppose $N\geq M$, an edge $e\in M_B$ is incident to vertices $u$ and $v$ of $H$, and $u$ is not in $H(N)$. We will show that $v$ is not in $H(N)$. Consider the two faces $f, f'\in F$ that are incident to $e$. Since $e\in M_B$, $h_B$ increases or decreases by $2/3$ between $f$ and $f'$. Without loss of generality, $h_B(f)=h_B(f')+2/3$, so when crossing $e$ from $f'$ to $f$, the left vertex of $e$ is white, implying that $f$ is obtained from $f'$ by translating $1$ unit in the negative $x_i$-direction for some $i\in\{1, 2, 3\}$. Let $p$ (resp.~$p'$) be the cell corresponding to $f$ (resp.~$f'$) such that $\mathfrak{B}$ lies at $p$ (resp.~$p'$). Then $p=p'-\mathbf{e}_i+(k, k, k)$ for some $k\in\mathbb{Z}$, and since $h_B(f)=h_B(f')+2/3$, $k=1$. Note that $p'-\mathbf{e}_i=p-(1, 1, 1)\in\mathfrak{B}$, so $p-(1, 1, 1)\in(\II\cup\III)\setminus B$ or $p'-\mathbf{e}_i$ has at least one negative coordinate. In the first case, $p-(1, 1, 1)\in\II\cup\III\subseteq[0, M-1]^3$, so $f$ is a face of $H(M)$ and, thus, of $H(N)$, contradicting the fact that $u$ is not in $H(N)$. In the second case, since $\mathfrak{B}$ lies at $p'$, $p'$ has no negative coordinates, so the $i$th coordinate of $p'$ must be $0$, while the other coordinates of $p'$ are nonnegative. It follows that the $i$th coordinate of $p=p'-\mathbf{e}_i+(1, 1, 1)$ is $0$, while the other coordinates of $p$ are positive, so $f$ is contained in sector $i$. Furthermore, since $f$ is obtained from $f'$ by translating in the negative $x_i$-direction, $e$ must be perpendicular to the $x_i$-axis. Any edge contained in sector $i$ and perpendicular to the $x_i$-axis is incident to vertices that are both in $H(N)$ or both not in $H(N)$, so since $u$ is not in $H(N)$, $v$ is not in $H(N)$. 
\end{proof}

The significance of this lemma is that, if $N\geq M$, $M_B$ can be truncated to a perfect matching $M_B(N)$ of $H(N)$. On the other hand, $M_A$ can be truncated to a partial matching $M_A(N)$ of $H(N)$. So, $D_{(A, B)}$ can be truncated to a double-dimer configuration with nodes, denoted by $D_{(A, B)}(N)$, on $H(N)$. In this case, the nodes are the vertices of $H(N)$ covered by dimers in $M_A$ that are not edges of $H(N)$. Such vertices must not only be on the outer face of $H(N)$, but they must be labelled by half integers, as in Figure~\ref{fig:sectors}, so they must be in sector $i^+$ or sector $i^-$ for some $i\in\{1, 2, 3\}$. 

Each double-dimer configuration with nodes is associated with a planar pairing of the nodes. On a finite graph, the notion that the paths are ``rainbow-like'' means that the pairing is {\em tripartite}.

\begin{defn}
A planar pairing $\sigma$ is {\em tripartite} if the nodes can be divided into three circularly contiguous sets $R$, $G$, and $B$ so that no node is paired with a node in the same set. We often color the nodes in the sets red, green, and blue, in which case $\sigma$ is the unique planar pairing in which like colors are not paired.
\end{defn}

\begin{figure}[htb]
\centering
\includegraphics[width=2in]{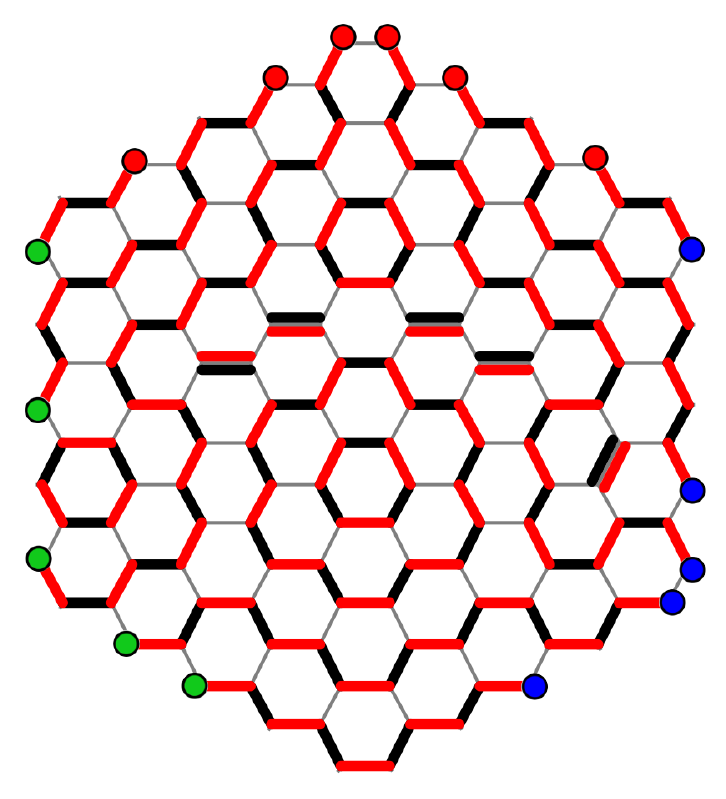}
\caption{A double-dimer configuration with nodes on $H(5)$, obtained by truncating the double-dimer configuration from Figure~\ref{fig:superposition}.}
\label{fig:tripartiteex1}
\end{figure}

\begin{example}
Truncating the double-dimer configuration from Figure~\ref{fig:superposition} to a double-dimer configuration on $H(5)$ produces the tripartite double-dimer configuration shown in Figure~\ref{fig:tripartiteex1}. 
\end{example}

We now show that if $(A, B)\in\sAB$ and $N\geq M$, then $D_{(A, B)}(N)$ is a tripartite double-dimer configuration. 

\begin{thm}
\label{thm:in sAB iff truncated path ends in single sector}
Suppose $(A, B)$ is an $AB$ configuration. Then $(A, B)\in\sAB$ if and only if, for all $N\geq M$, each path in $D_{(A, B)}(N)$ begins and ends in the same sector. 
\end{thm}

\begin{proof}
Suppose $N\geq M$. Consider a node $u$ of $D_{(A, B)}(N)$ in sector $i$, so $u$ is a vertex of $H(N)$ covered by a dimer $e\in M_A$ that is not an edge of $H(N)$. Then $e$ must be incident to another vertex $v$ that is not a vertex of $H(N)$. By Lemma~\ref{lemma:folklore for B}, $e\in M_A\setminus M_B\subseteq M_A\triangle M_B$. In other words, $e$ is a dimer in a loop or path $\gamma$ in $D_{(A, B)}$. 

Consider the sequence of vertices $u, v:=v_0, v_1, v_2, \ldots$ obtained by moving along $\gamma$, beginning at $u$, moving to $v$, and then continuing along $\gamma$. We claim that this sequence never returns to a vertex of $H(N)$ (i.e., $v_s$ is not a vertex of $H(N)$ for $s\geq 0$) and never leaves sector $i$. By Lemma~\ref{lemma:dimers covering different sectors}, $v$ is in sector $i$, and if the sequence leaves sector $i$ thereafter, it must first return to a vertex of $H(N)$, so it suffices to show that the sequence never returns to a vertex of $H(N)$. 

Suppose $v_s$ is a vertex of $H(N)$ for some $s\geq 0$. Let $r\geq 0$ be the least index such that $v_r$ is a vertex of $H(N)$. Note that $r>0$, since $v$ is not a vertex of $H(N)$. Also, $v_{r-1}$ is not a vertex of $H(N)$, so by Lemma~\ref{lemma:folklore for B}, $\{v_{r-1}, v_r\}\not\in M_B$, so $\{v_{r-1}, v_r\}\in M_A$. Since $e\in M_A$ and $\gamma$ must alternate between $M_A$ and $M_B$, we deduce that $r$ is even. Therefore, $r>1$, and if $u$ is white, then $v_r$ is black, and vice versa. As a result, the projection of the $x_i$-axis lies between $u$ and $v_r$. However, if $0\leq s<r-1$ is even, then $\{v_s, v_{s+1}\}\in M_B$, so by Lemma~\ref{lemma:folklore for B}, since neither $v_s$ nor $v_{s+1}$ is a vertex of $H(N)$, $v_s$ and $v_{s+1}$ must both be vertices of $H(N')$ and not vertices of $H(N'-1)$ for some $N'>N$. In fact, as discussed in the proof of Lemma~\ref{lemma:folklore for B}, each such dimer $\{v_s, v_{s+1}\}$ must be perpendicular to the $x_i$-axis. Since consecutive dimers in any loop or path in $D_{(A, B)}$ cannot both be perpendicular to the $x_i$-axis, this implies that the dimers $\{u, v\}$ and $\{v_s, v_{s+1}\}$, where $0<s<r$ is odd, cannot be perpendicular to the $x_i$-axis. Since the projection of the $x_i$-axis lies between $u$ and $v_r$, some dimer in $\gamma$ between $u$ and $v_r$ must cross the projection of the $x_i$-axis from the side on which $u$ lies to the side on which $v_r$ lies. Such a dimer must be perpendicular to the $x_i$-axis, so it must be of the form $\{v_t, v_{t+1}\}$, where $0\leq t<r-1$ is even. Then $v_t$ lies on the same side of the projection of the $x_i$-axis as $u$, so $u$ and $v_t$ are vertices of the same color. Since $t$ is even, this means that $u$ and $v_0=v$ are vertices of the same color, which is a contradiction. This completes the proof of the claim. We conclude that $\gamma$ is a path in $D_{(A, B)}$, and one end of $\gamma$ is contained in sector $i$. That is, if $N\geq M$, then each node of $D_{(A, B)}(N)$ in sector $i$ must be covered by a path in $D_{(A, B)}$, one of whose ends is contained in sector $i$. 

Suppose for some $N\geq M$, there is a path $\gamma'$ in $D_{(A, B)}(N)$ that begins and ends in two different sectors. Then the above discussion shows that there is a path $\gamma$ in $D_{(A, B)}$ whose ends are contained in two different sectors. By Theorem~\ref{thm:in sAB iff path ends in single sector}, $(A, B)\not\in\sAB$. Conversely, suppose $(A, B)\not\in\sAB$. By Theorem~\ref{thm:in sAB iff path ends in single sector}, there is a path $\gamma$ in $D_{(A, B)}$, one of whose ends is contained in sector $i$ and the other of whose ends is contained in sector $j$, where $i\neq j$. Then $\gamma$ consists of a sequence of dimers $\ldots, e_{-2}, e_{-1}, e_0, e_1, e_2, \ldots$, and there exist $r, t\in\mathbb{Z}$ such that $e_{-s}$ is contained in sector $i$ for $s>r$ and $e_s$ is contained in sector $j$ for $s>t$. Since consecutive dimers cannot be contained in different sectors, $-r\leq t$. Let $N'\in\mathbb{Z}_{\geq 0}$ be such that all of the dimers $e_{-r}, e_{-r+1}, \ldots, e_{t-1}, e_t$ are edges of $H(N')$, and let $N=\max\{N', M\}$. Then $N\geq M$ and all of the dimers $e_{-r}, e_{-r+1}, \ldots, e_{t-1}, e_t$ are edges of $H(N)$, so they form part of a path $\gamma'$ in $D_{(A, B)}(N)$. More precisely, $\gamma'$ must consist of the sequence of dimers $e_{-r'}, e_{-r'+1}, \ldots, e_{t'-1}, e_{t'}$ for some $r'\geq r$ and some $t'\geq t$. Let $u$ be the node covered by $e_{-r'}$ and let $v$ be the node covered by $e_{t'}$. Since $r'+1>r'\geq r$ and $t'+1>t'\geq t$, $e_{-r'-1}$ is contained in sector $i$ and $e_{t'+1}$ is contained in sector $j$. But $e_{-r'-1}$ also covers $u$ and $e_{t'+1}$ also covers $v$, so $u$ is contained in sector $i$ and $v$ is contained in sector $j$. Thus, there exists $N\geq M$ so that there is a path $\gamma'$ in $D_{(A, B)}(N)$ that begins and ends in two different sectors. 
\end{proof}

\begin{corollary}
\label{cor:in sAB iff truncated path ends in single sector}
Suppose $(A, B)$ is an $AB$ configuration. Then $(A, B)\in\sAB$ if and only if, for some $N\geq M$, each path in $D_{(A, B)}(N)$ begins and ends in the same sector. 
\end{corollary}

\begin{proof}
Suppose $N\geq M$ and there is a path $\gamma$ in $D_{(A, B)}(M)$ that begins in sector $i$ and ends in sector $j$, where $i\neq j$. Then, by the claim established in the first three paragraphs of the proof of the theorem, $\gamma$ must be a subpath of a path $\gamma'$ in $D_{(A, B)}(N)$ that begins in sector $i$ and ends in sector $j$. So, if there exists $N\geq M$ such that each path in $D_{(A, B)}(N)$ begins and ends in the same sector, then each path in $D_{(A, B)}(M)$ begins and ends in the same sector. 

Now suppose $N'\geq M$ and each path in $D_{(A, B)}(M)$ begins and ends in the same sector. Consider a path $\gamma$ in $D_{(A, B)}(N')$ that begins in sector $i$. If $\gamma$ leaves sector $i$, then by Lemma~\ref{lemma:dimers covering different sectors}, it must first enter $H(M)$. Since $\gamma$ enters $H(M)$ in sector $i$ and each path in $D_{(A, B)}(M)$ begins and ends in the same sector, $\gamma$ must exit $H(M)$ in sector $i$. In other words, if $\gamma$ leaves sector $i$, it must first enter $H(M)$ and must return to sector $i$ before exiting $H(M)$. As a result, $\gamma$ must end in sector $i$. So, by the theorem, $(A, B)\in\sAB$. This completes the proof. 
\end{proof}

We can be even more precise about the pairing of the nodes $\mathbf{N}$. Suppose there are $2r$ nodes in sector $i$. The nodes in sector $i$ are vertices on the outer face of $H(N)$, and we can number them consecutively in clockwise order. If $r>0$, we then refer to the pairing \[((1, 2r), (2, 2r-1), \ldots, (r, r+1))\] as the \emph{rainbow pairing of the nodes in sector $i$}. If $r=0$, we refer to the empty pairing as the rainbow pairing of the nodes in sector $i$. Furthermore, if the nodes in sector $i$ are paired according to the rainbow pairing in sector $i$, for each $i\in\{1, 2, 3\}$, then we call the resulting pairing of $\mathbf{N}$ the \emph{rainbow pairing of $\mathbf{N}$}. 

\begin{thm}
\label{thm:in sAB iff rainbow pairing}
Suppose $(A, B)$ is an $AB$ configuration. Then $(A, B)\in\sAB$ if and only if, for all $N\geq M$, the nodes of $D_{(A, B)}(N)$ are paired according to the rainbow pairing. 
\end{thm}

\begin{proof}
By Theorem~\ref{thm:in sAB iff truncated path ends in single sector}, it suffices to show, for $N\geq M$, that each path in $D_{(A, B)}(N)$ begins and ends in the same sector if and only if the nodes of $D_{(A, B)}(N)$ are paired according to the rainbow pairing. So, assume $N\geq M$, and let $\sigma$ denote the pairing of the nodes $\mathbf{N}$ of $D_{(A, B)}(N)$. 

Suppose each path in $D_{(A, B)}(N)$ begins and ends in the same sector. Consider the nodes in sector $i$. Each must be paired with exactly one other node in sector $i$, so there are $2r$ such nodes, for some $r\in\mathbb{Z}_{\geq 0}$. Number them consecutively in clockwise order. Then, considering the structure of $H(N)$ and the fact that each node must be incident to an edge of $H$ that is not an edge of $H(N)$, we see that the white nodes precede the black nodes. That is, given a white node numbered $m_w$ and a black node numbered $m_b$, we have $m_w<m_b$. For $1\leq j\leq 2r$, let $\gamma_j$ be the path in $D_{(A, B)}(N)$ beginning at node $j$. To show that $\sigma$ is the rainbow pairing, we must show that $\gamma_j=\gamma_{2r-j+1}$. First, since each node in sector $i$ must be paired with exactly one other such node, $\gamma_j=\gamma_k$ for some $1\leq k\leq 2r$ such that $j\neq k$. Also, since $M_B(N)$ is a perfect matching of $H(N)$, each path in $D_{(A, B)}(N)$ must begin and end with dimers in $M_B(N)$, so $\gamma_j=\gamma_k$ consists of an odd number of dimers. Consequently, if node $j$ is white, then node $k$ must be black, and vice versa. This implies that there are equally many white and black nodes in sector $i$, so nodes $1$ through $r$ are white and nodes $r+1$ through $2r$ are black. Therefore, if $j\leq r$, then $k>r$, and if $j>r$, then $k\leq r$. Moreover, since $\sigma$ is planar, there can be no crossings, i.e., no four nodes $m_1<m_2<m_3<m_4$ such that $\gamma_{m_1}=\gamma_{m_3}$ and $\gamma_{m_2}=\gamma_{m_4}$. In particular, if $\gamma_1=\gamma_k$, where $k<2r$, then $k>r$ and $\gamma_{2r}=\gamma_l$, where $1<l\leq r$, so we have a crossing. So, $\gamma_1=\gamma_{2r}$. By similar arguments, we then find that $\gamma_2=\gamma_{2r-1}$, and so on, until we find that $\gamma_r=\gamma_{r+1}$. 

Conversely, suppose $\sigma$ is the rainbow pairing, and consider a path $\gamma$ in $D_{(A, B)}(N)$. Since the rainbow pairing only pairs nodes in the same sector, and $\gamma$ is a path between two nodes $u$ and $\sigma(u)$, $\gamma$ begins and ends in the same sector. 
\end{proof}

\begin{corollary}
\label{cor:in sAB iff rainbow pairing}
Suppose $(A, B)$ is an $AB$ configuration. Then $(A, B)\in\sAB$ if and only if, for some $N\geq M$, the nodes of $D_{(A, B)}(N)$ are paired according to the rainbow pairing. 
\end{corollary}

\begin{proof}
This is a direct consequence of Corollary~\ref{cor:in sAB iff truncated path ends in single sector}. 
\end{proof}

Finally, we can explicitly describe the nodes of $D_{(A, B)}(N)$. The set of nodes {\bf N} and the coloring of these nodes is determined by the partitions $\mu_1$, $\mu_2$, and $\mu_3$. Let $S_i$ be the Maya diagram of $\mu_i$. We refer to the labelling of the graph $H(N)$ shown in the right-hand side of Figure~\ref{fig:sectors}. Given $N\in\mathbb{Z}_{\geq 0}$, let $\mathbf{N}_i^+(N)$ (resp.~$\mathbf{N}_i^-(N)$) be the set of vertices on the outer face of $H(N)$, in sector $i^+$ (resp.~sector $i^-$), that are not labelled by any of the elements of $S_i^+$ (resp.~$S_i^-$). Then let $\mathbf{N}_{\mu}(N)=\bigcup\limits_{i=1}^3 \mathbf{N}_i^+(N)\cup\mathbf{N}_i^-(N)$. 

\begin{lemma}
\label{lemma:nodes for mu}
Suppose $(A, B)$ is an $AB$ configuration and $N\geq M$ is such that each box in $A\cup B$ corresponds to a face of $H(N)$. Then the set of nodes $\mathbf{N}$ of $D_{(A, B)}(N)$ is $\mathbf{N}_{\mu}(N)$. 
\end{lemma}

\begin{proof}
Consider a node $u$ of $D_{(A, B)}(N)$ in sector $i^+$ (resp.~sector $i^-$). Then $u$ is a vertex of $H(N)$ covered by a dimer $e\in M_A$ that is not an edge of $H(N)$. We must show that $u\in\mathbf{N}_i^+(N)$ (resp.~$u\in\mathbf{N}_i^-(N)$). That is, we must show that $u$ is not labelled by any of the elements of $S_i^+$ (resp.~$S_i^-$). Let $m(u)$ denote the label associated to $u$, and let $v$ be the vertex in sector $i$ labelled by $m(u)-1$ (resp.~$m(u)+1$). There is a unique face $f\in F$ such that $e$ and $v$ are both incident to $f$. Note that $f$ is contained in sector $i$. Also, since $e$ is not an edge of $H(N)$, $f$ is not a face of $H(N)$. Let $w$ be the cell corresponding to $f$ such that $\mathfrak{A}$ lies at $w$. Then the $i$th coordinate of $w$ is strictly less than the other coordinates of $w$, and by assumption, $f$ does not correspond to any box in $A\cup B$, so $w\not\in A\cup B$. Since $\mathfrak{A}$ lies at $w$, $w$ has at most one negative coordinate and $w\not\in(\I^-\cup\III)\setminus A$. It follows that $w\not\in\I^-\cup\III$. 

Suppose $w\in\Cyl_i$. Then, since $w\not\in\I^-\cup\III$, we have $w\in\Cyl_i^+$. Since the $i$th coordinate of $w$ is the least coordinate of $w$, we deduce that $w\in[0, M-1]^3$. Then $f$ must be a face of $H(M)\subseteq H(N)$. By contradiction, $w\not\in\Cyl_i$. 

Now consider the cell $w-\mathbf{e}_j$, where $j\in\{1, 2, 3\}$ and $j\equiv i-1\pmod 3$ (resp.~$j\equiv i+1\pmod 3$). Let $f'\in F$ be the face corresponding to $w-\mathbf{e}_j$. Observe that $f'$ is the other face of $H$ to which $e$ is incident, and when crossing $e$ from $f$ to $f'$, the left vertex of $e$ is white. Since $e\in M_A$, we see that $h_A$ increases by $2/3$ between $f$ and $f'$, i.e., $h_A(f')=h_A(f)+2/3$. Thus, $\mathfrak{A}$ must lie at $w-\mathbf{e}_j+(1, 1, 1)$. In particular, $w-\mathbf{e}_j\in\mathfrak{A}$, so $w-\mathbf{e}_j$ has at least two negative coordinates, or $w-\mathbf{e}_j\in(\I^-\cup\III)\setminus A$. In the first case, since the $i$th coordinate $w_i$ of $w$ is its least coordinate and $w$ has at most one negative coordinate, the other coordinates of $w$ must be nonnegative, so $w_i<0$ and the $j$th coordinate $w_j$ of $w$ must be $0$. In this case, we conclude that $(\mu_i')_{w_k+1}=0$, where $k\in\{1, 2, 3\}$ is such that $\{k\}=\{1, 2, 3\}\setminus\{i, j\}$, so $w_k+1>(\mu_i)_1=(\mu_i)_{w_j+1}$ (resp.~$(\mu_i)_{w_k+1}=0=w_j$). In the second case, $w-\mathbf{e}_j\in\I^-\cup\III$, and the $i$th coordinate of $w$ is strictly less than the other coordinates of $w$, so the $i$th coordinate of $w-\mathbf{e}_j$ is the least coordinate of $w-\mathbf{e}_j$. In this case, we conclude that $w-\mathbf{e}_j\in\Cyl_i$. Since $w\not\in\Cyl_i$, we once again determine that $(\mu_i')_{w_k+1}=w_j$, and $w_j>0$, so $(\mu_i)_{w_j}\geq w_k+1>(\mu_i)_{w_j+1}$ (resp.~$(\mu_i)_{w_k+1}=w_j$). 

As a consequence of our choices made in defining $w$, $u$ is labelled by $m(u)=1/2+w_k-w_j$ (resp.~$m(u)=-1/2+w_j-w_k$). Therefore, we have \[(\mu_i)_{w_j+1}-(w_j+1)+1/2=(\mu_i)_{w_j+1}-1-w_j+1/2<w_k-w_j+1/2=m(u)\] and (in the case that $w_j\neq 0$) \[m(u)=w_k-w_j+1/2\leq (\mu_i)_{w_j}-1-w_j+1/2<(\mu_i)_{w_j}-w_j+1/2\] (resp.~$m(u)=-1/2+w_j-w_k=-1/2+(\mu_i)_{w_k+1}-w_k=(\mu_i)_{w_k+1}-(w_k+1)+1/2$). Since the sequence $(\mu_i)_t-t+1/2$ is a strictly decreasing sequence, $m(u)\neq (\mu_i)_t-t+1/2$ for any $t>0$, i.e., $m(u)\not\in S_i$ (resp.~$m(u)\in S_i$). So, $m(u)\not\in S_i^+$ (resp.~$m(u)\not\in S_i^-$), as desired. 

Conversely, consider $u\in\mathbf{N}_i^+(N)$ (resp.~$u\in\mathbf{N}_i^-(N)$). Then $u$ is a vertex on the outer face of $H(N)$, in sector $i^+$ (resp.~sector $i^-$), and it is not labelled by any of the elements of $S_i^+$ (resp.~$S_i^-$). We must show that $u$ is a node of $D_{(A, B)}(N)$, i.e., that $u$ is covered by a dimer in $M_A$ that is not an edge of $H(N)$. As above, let $m(u)$ denote the label associated to $u$, and let $v$ be the vertex in sector $i$ labelled by $m(u)-1$ (resp.~$m(u)+1$). There is a unique edge $e$ of $H$ that covers $u$ and is not an edge of $H(N)$, and there is a unique face $f\in F$ such that $e$ and $v$ are both incident to $f$. Note that $f$ is contained in sector $i$ and is not a face of $H(N)$. We will show that $e\in M_A$. Let $w$ be the cell corresponding to $f$ such that $\mathfrak{A}$ lies at $w$. In addition, let $j\in\{1, 2, 3\}$ such that $j\equiv i-1\pmod 3$ (resp.~$j\equiv i+1\pmod 3$), let $k\in\{1, 2, 3\}$ such that $\{k\}=\{1, 2, 3\}\setminus\{i, j\}$, and let $f'\in F$ be the face corresponding to $w-\mathbf{e}_j$. Then the $i$th coordinate of $w$ is strictly less than the other coordinates of $w$, and by assumption, $f$ does not correspond to any box in $A\cup B$, so $w\not\in A\cup B$. Since $\mathfrak{A}$ lies at $w$, $w$ has at most one negative coordinate and $w\not\in(\I^-\cup\III)\setminus A$. It follows that $w\not\in\I^-\cup\III$ and $w_j, w_k\geq 0$. Furthermore, $w-(1, 1, 1)\in\mathfrak{A}$, so (i) $w-(1, 1, 1)$ has at least two negative coordinates or (ii) $w-(1, 1, 1)\in(\I^-\cup\III)\setminus A$. In case (i), $w_j$ and $w_k$ cannot both be positive, so $w_j=0$ or $w_k=0$. In case (ii), the $i$th coordinate of $w-(1, 1, 1)$ is strictly less than the other coordinates of $w-(1, 1, 1)$, so if $w-(1, 1, 1)\in\I^-$, then $w-(1, 1, 1)\in\Cyl_i^-\subseteq\Cyl_i$. Thus, $w-(1, 1, 1)\in\Cyl_i$, so $(\mu_i)_{w_j}\geq w_k$ (resp.~$(\mu_i)_{w_k}\geq w_j$). 

As discussed above, $m(u)=1/2+w_k-w_j$ (resp.~$m(u)=-1/2+w_j-w_k$). By assumption, $0<m(u)\not\in S_i^+$ (resp.~$0>m(u)\not\in S_i^-$). Consequently, $m(u)\not\in S_i$ (resp.~$m(u)\in S_i$), so $m(u)\neq(\mu_i)_t-t+1/2$ for any $t>0$ (resp.~$m(u)=(\mu_i)_{t_0}-t_0+1/2$ for some $t_0>0$). Then $0\leq w_k-w_j\neq(\mu_i)_t-t$ for any $t>0$ (resp.~$0>w_j-w_k-1=(\mu_i)_{t_0}-t_0$). 

Suppose $w_j\neq 0$ and $w-\mathbf{e}_j\not\in\Cyl_i$. Then $(\mu_i')_{w_k+1}<w_j$, implying that $(\mu_i)_{w_j}\leq w_k$ (resp.~$(\mu_i)_{w_k+1}<w_j$). We have $(\mu_i)_{w_j}-w_j\leq w_k-w_j\neq(\mu_i)_t-t$ for any $t>0$, so $(\mu_i)_{w_j}-w_j<w_k-w_j$, which means that $(\mu_i)_{w_j}<w_k$ (resp.~$(\mu_i)_{w_k+1}-(w_k+1)<w_j-(w_k+1)=(\mu_i)_{t_0}-t_0$, so $t_0<w_k+1$, since the sequence $(\mu_i)_t-t$ is strictly decreasing). In case (i), since $w_j\neq 0$, $w_k=0$, so $(\mu_i)_{w_j}<0$, which is a contradiction (resp.~$t_0<1$, so $t_0\leq 0$, which is a contradiction). In case (ii), we have $(\mu_i)_{w_j}\geq w_k>(\mu_i)_{w_j}$, a contradiction (resp.~$(\mu_i)_{w_k}-w_k>(\mu_i)_{w_k}-w_k-1\geq w_j-w_k-1=(\mu_i)_{t_0}-t_0$, so because the sequence $(\mu_i)_t-t$ is strictly decreasing, $w_k<t_0<w_k+1$, a contradiction). We conclude that $w_j=0$ or $w-\mathbf{e}_j\in\Cyl_i$. 

If $w_j=0$, then since $w_i<w_j$, $w-\mathbf{e}_j$ has at least two negative coordinates, so $w-\mathbf{e}_j\in\mathfrak{A}$. Otherwise, $w-\mathbf{e}_j\in\Cyl_i$. Observe that $f'$ is the other face of $H$ to which $e$ is incident and is not a face of $H(N)$. Suppose $w-\mathbf{e}_j\in\Cyl_i^+$. Since the $i$th coordinate of $w$ is less than the other coordinates of $w$, the $i$th coordinate of $w-\mathbf{e}_j$ is the least coordinate of $w-\mathbf{e}_j$. We deduce that $w-\mathbf{e}_j\in[0, M-1]^3$, so $f'$ must be a face of $H(M)\subseteq H(N)$. By contradiction, $w-\mathbf{e}_j\not\in\Cyl_i^+$, so $w-\mathbf{e}_j\in\Cyl_i^-\subseteq\I^-\cup\III$. Moreover, by assumption, $f'$ does not correspond to any box in $A\cup B$, so $w-\mathbf{e}_j\not\in A\cup B$. So, $w-\mathbf{e}_j\in(\I^-\cup\III)\setminus A$, showing that $w-\mathbf{e}_j\in\mathfrak{A}$. In either case, $w-\mathbf{e}_j\in\mathfrak{A}$, so $\mathfrak{A}$ lies at or above $w+\mathbf{e}_i+\mathbf{e}_k$, which corresponds to $f'$. Since $\mathfrak{A}$ lies at $w$, which corresponds to $f$, $h_A$ must increase by at least $2/3$ between $f$ and $f'$, i.e., $h_A(f')\geq h_A(f)+2/3$. According to the definition of $h_A$, since $e$ separates $f$ and $f'$, $e\in M_A$, as desired. 
\end{proof}

\begin{corollary}
The set of nodes of $D_{(A, B)}(N)$ in sector $i$ is $\mathbf{N}_i:=\mathbf{N}_i^+(N)\cup\mathbf{N}_i^-(N)$. 
\end{corollary}

We color the nodes as follows. Recall that when given a Maya diagram, by holes, we mean elements of $\mathbb{Z}+\frac{1}{2}\setminus S$, and by beads, we mean elements of $S$. 
\begin{itemize}
\item In sector 1, the blue nodes are the holes of $S_1$ with positive coordinates and the red nodes are the beads of $S_1$ with negative coordinates.
\item In sector 2, the red nodes are the holes of $S_2$ with positive coordinates and the green nodes are the beads of $S_2$ with negative coordinates. 
\item In sector 3, the green nodes are the holes of $S_3$ with positive coordinates and the blue nodes are the beads of $S_3$ with negative coordinates. 
\end{itemize}

Since $\lvert S_i^+\rvert=\lvert S_i^-\rvert$ for $i\in\{1, 2, 3\}$, there are an equal number of nodes in sector $i$ with positive coordinates and negative coordinates. So, the rainbow pairing of $\mathbf{N}_{\mu}(N)$ pairs blue nodes in sector 1 with red nodes in sector 1, red nodes in sector 2 with green nodes in sector 2, and green nodes in sector 3 with blue nodes in sector 3. This shows that the rainbow pairing is tripartite. 

Let $D_{\sigma}(G, \mathbf{N})$ be the set of all double-dimer configurations on $G$ with nodes $\mathbf{N}$ paired according to $\sigma$, and let $Z^{DD}_{\sigma}(G, {\bf N})$ denote the weighted sum of the double-dimer configurations in $D_{\sigma}(G, \mathbf{N})$. We can now explain the relationship between $Z_{\sAB}$ and $Z^{DD}_{\sigma}(H(N), \mathbf{N}_{\mu}(N))$, where $\sigma$ is the rainbow pairing. Note that $\left\lvert\mathbf{N}_i^+(N)\right\rvert=\left\lvert\mathbf{N}_i^-(N)\right\rvert$, since $\left\lvert S_i^+\right\rvert=\left\lvert S_i^-\right\rvert$, so it makes sense to consider the rainbow pairing of $\mathbf{N}_{\mu}(N)$. 

Given a nonempty $AB$ configuration, removing a box whose back neighbors are not boxes produces another $AB$ configuration, so between any two $AB$ configurations $(A, B)$ and $(A', B')$, there always exists at least one sequence $(A, B):=(A_1, B_1), (A_2, B_2), \ldots, (A_r, B_r):=(A', B')$ of $AB$ configurations such that consecutive $AB$ configurations differ by the removal or addition of a single box. Furthermore, if $(A_{s+1}, B_{s+1})$ is obtained from $(A_s, B_s)$ by removing a box from $A_s$ or $B_s$, then $M_{A_{s+1}}$ or $M_{B_{s+1}}$ is obtained from $M_{A_s}$ or $M_{B_s}$, respectively, by performing a local move as shown in Figure~\ref{fig:removingabox}. 
\begin{figure}[htb]
\centering
\begin{minipage}{.45\textwidth}
\hfill
\includegraphics{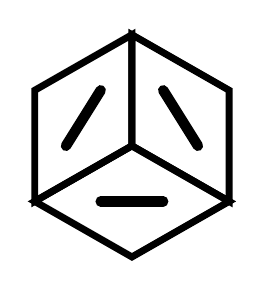}
\end{minipage}
$\longrightarrow$
\begin{minipage}{.45\textwidth}
\includegraphics{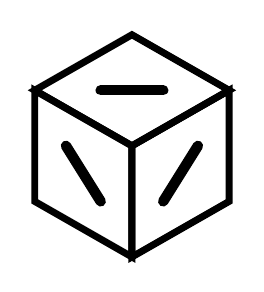}
\hfill
\end{minipage}
\caption{A local move corresponding to the removal of a box.}
\label{fig:removingabox}
\end{figure}
Similarly, if $(A_{s+1}, B_{s+1})$ is obtained from $(A_s, B_s)$ by adding a box to $A_s$ or $B_s$, then $M_{A_{s+1}}$ or $M_{B_{s+1}}$ is obtained from $M_{A_s}$ or $M_{B_s}$, respectively, by performing a local move as shown in Figure~\ref{fig:addingabox}. 
\begin{figure}[htb]
\centering
\begin{minipage}{.45\textwidth}
\hfill
\includegraphics{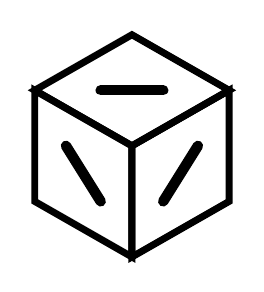}
\end{minipage}
$\longrightarrow$
\begin{minipage}{.45\textwidth}
\includegraphics{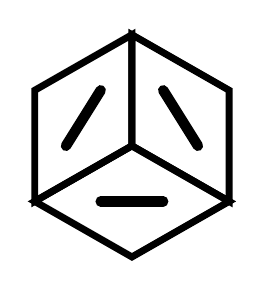}
\hfill
\end{minipage}
\caption{A local move corresponding to the addition of a box.}
\label{fig:addingabox}
\end{figure}

Recall the edge weights specified in Definition~\ref{def:kuoweighting}. Assuming that all of the boxes in $A\cup B$ and all of the boxes in $A'\cup B'$ correspond to faces of $H(N)$, the above discussion applies just as well to $M_{A_s}(N)$ and $M_{B_s}(N)$. Then, one consequence of the chosen edge weights is that removing a box increases the edge-weight by a factor of $q$, and adding a box decreases the edge-weight by a factor of $q$. Therefore, the edge-weight $q^{w_{(A, B)}(N)}$ of $D_{(A, B)}(N)$ is related to the edge-weight $q^{w_{(A', B')}(N)}$ of $D_{(A', B')}(N)$ by the following equation: \[q^{|A|+|B|+w_{(A, B)}(N)}=q^{|A'|+|B'|+w_{(A', B')}(N)}.\] In particular, if $(A', B')=(\III, \II\cup\III)$ and $N\geq M$, then we have \[|A|+|B|+w_{(A, B)}(N)=|\II|+2|\III|+w_{(\III, \II\cup\III)}(N).\] Observe that $(\III, \II\cup\III)\in\sAB(\pi)\subseteq\sAB$, where $\pi$ is the labelled box configuration consisting of the boxes $\II\cup\III$, all of which are unlabelled. So, by Theorem~\ref{thm:in sAB iff rainbow pairing} and Lemma~\ref{lemma:nodes for mu}, if $N\geq M$, $D_{(\III, \II\cup\III)}(N)\in D_{\sigma}(H(N), \mathbf{N}_{\mu}(N))$. 

\begin{defn}
\label{defn:baseDD}
The double-dimer configuration $D_{(\III, \II\cup\III)}(N)$ on $(H(N), {\bf N}_{\mu}(N))$ will be called the {\em base$_{\mu}$ double-dimer configuration} and its edge-weight will be denoted $q^{w_{base}(\mu)}$. 
\end{defn}

In other words, $w_{base}(\mu)=w_{(\III, \II\cup\III)}(N)$. Also, if $|A|+|B|\leq N-M$, and $w\in A\cup B$, then $w\in\Cyl_i^-$ for some $i\in\{1, 2, 3\}$ or $w\in\II\cup\III$. In the first case, $w\in A$, so by Conditions~\ref{conditions:ab box stacking}, $w+s\mathbf{e}_i\in A$ for $0\leq s<-w_i$. It follows that $-w_i\leq |A|\leq |A|+|B|\leq N-M$. Since the coordinates of $w$ other than the $i$th coordinate are in $[0, M-1]$, we must have $w-w_i(1, 1, 1)\in [0, M-1+N-M]^3=[0, N-1]^3$, so $w-w_i(1, 1, 1)$ corresponds to a face of $H(N)$ and, thus, so does $w$. In the second case, $w\in\II\cup\III\subseteq [0, M-1]^3\subseteq[0, N-1]^3$, so $w$ corresponds to a face of $H(N)$. If, in addition, $(A, B)\in\sAB$, then by Theorem~\ref{thm:in sAB iff rainbow pairing} and Lemma~\ref{lemma:nodes for mu}, $D_{(A, B)}(N)\in D_{\sigma}(H(N), \mathbf{N}_{\mu}(N))$. 

Consequently, assuming $N\geq M$, by Definition~\ref{defn:ZAB}, we have 
\begin{eqnarray*}
Z_{\sAB}(q^{-1}) 
&=& q^{|\II|+2|\III|}\sum_{(A, B) \in \sAB} q^{-|A|-|B|} \\
&=& q^{|\II|+2|\III|}\sum_{\substack{(A, B) \in \sAB \\ |A|+|B|\leq N-M}} q^{-|A|-|B|}+q^{|\II|+2|\III|}\sum_{\substack{(A, B) \in \sAB \\ |A|+|B|>N-M}} q^{-|A|-|B|} \\
&=& q^{-w_{base}(\mu)}\sum_{\substack{(A, B) \in \sAB \\ |A|+|B|\leq N-M}} q^{w_{(A, B)}(N)}+q^{|\II|+2|\III|}\sum_{\substack{(A, B) \in \sAB \\ |A|+|B|>N-M}} q^{-|A|-|B|}.
\end{eqnarray*}

Let $D\in D_{\sigma}(H(N), \mathbf{N}_{\mu}(N))$. Since the nodes $\mathbf{N}_{\mu}(N)$ are paired according to the rainbow pairing and $\left\lvert\mathbf{N}_i^+(N)\right\rvert=\left\lvert\mathbf{N}_i^-(N)\right\rvert$, each path in $D$ pairs a black node in $\mathbf{N}_i^+(N)$ with a white node in $\mathbf{N}_i^-(N)$, for some $i\in\{1, 2, 3\}$. Thus, each path has odd length. If there are $k(D)$ loops in $D$, this implies that there are $2^{k(D)}$ ways to decompose $D$ into a perfect matching $M_1$ of $H(N)\setminus\mathbf{N}_{\mu}(N)$ and a perfect matching $M_2$ of $H(N)$. These matchings then correspond to lozenge tilings, and since the nodes $\mathbf{N}_{\mu}(N)$ are paired according to the rainbow pairing, these tilings extend uniquely to tilings of the plane that can be interpreted as surfaces $\mathfrak{A}=R_2\cup(\I^-\cup\III)\setminus A$ and $\mathfrak{B}=R_1\cup(\II\cup\III)\setminus B$, respectively, for some $AB$ configuration $(A, B)$. Then $M_A(N)=M_1$ and $M_B(N)=M_2$, so $D_{(A, B)}(N)=D$. 

To be more precise, we must check that the penultimate statement from the previous paragraph holds for an $AB$ configuration $(A, B)$ associated with the partitions $\mu$ and not some other partitions. The fact that the nodes of $D$ are $\mathbf{N}_{\mu}(N)$ ensures that the tiling corresponding to $M_1$ can be extended so that $\mathfrak{A}=R_2\cup(\I^-(\nu)\cup\III(\nu))\setminus A$, where $A\subseteq\I^-(\nu)\cup\III(\nu)$, for any partitions $\nu$ such that $\mu_i\subseteq\nu_i$ for $i\in\{1, 2, 3\}$. It's not clear, though, that the tiling corresponding to $M_2$ can be extended so that $\mathfrak{B}=R_1\cup(\II(\mu)\cup\III(\mu))\setminus B$, where $B\subseteq\II(\mu)\cup\III(\mu)$. All we can say is that it can be extended so that $\mathfrak{B}=R_1\cup(\II(\nu)\cup\III(\nu))\setminus B$, where $B\subseteq\II(\nu)\cup\III(\nu)$, for some partitions $\nu$ such that $\mu_i\subseteq\nu_i$ for $i\in\{1, 2, 3\}$. Suppose this statement does not hold for $\nu=\mu$. Then there exists a cell $w\in\mathfrak{B}\setminus(R_1\cup\II(\mu)\cup\III(\mu))\subseteq(\II(\nu)\cup\III(\nu))\setminus (B\cup\II(\mu)\cup\III(\mu))$. 

If $w\in\II(\nu)$, then since $w\not\in\II(\mu)\cup\III(\mu)$, there exist $i, j\in\{1, 2, 3\}$ such that $i\neq j$ and $w\not\in\Cyl_j(\mu)\cup\Cyl_i(\nu)$. Since $w\in\mathfrak{B}$, by Lemma~\ref{lemma:AB surfaces}, $w-s\mathbf{e}_j\in\mathfrak{B}$ for $0\leq s\leq w_j+1$. On the other hand, since $w\not\in\Cyl_j(\mu)$, $w-s\mathbf{e}_j\not\in\Cyl_j(\mu)$ for $0\leq s\leq w_j+1$. Moreover, the $j$th coordinate of $w-(w_j+1)\mathbf{e}_j$ is $-1$, while the other coordinates are nonnegative, so $w-(w_j+1)\mathbf{e}_j\in\Cyl_j^-(\nu)\setminus\I^-(\mu)$. This implies that $w-s\mathbf{e}_j\not\in R_2\cup(\I^-(\mu)\cup\III(\mu))$ and, thus, $w-s\mathbf{e}_j\not\in\mathfrak{A}$ for $0\leq s\leq w_j+1$. Consequently, $\{w-s\mathbf{e}_j\mid 0\leq s\leq w_j+1\}\subseteq\mathcal{L}(A, B)$ is a connected set of cells containing $w\in\II_{\bar{i}}(\nu)$ and $w-(w_j+1)\mathbf{e}_j\in\Cyl_j^-(\nu)$. By Theorem~\ref{thm:labellable iff in sAB}, $(A, B)\not\in\sAB$, so by Corollary~\ref{cor:in sAB iff rainbow pairing}, the nodes of $D_{(A, B)}(N)$ are not paired according to the rainbow pairing. This contradicts the fact that $D_{(A, B)}(N)=D\in D_{\sigma}(H(N), \mathbf{N}_{\mu}(N))$. 

Otherwise, $w\in\III(\nu)$, and since $w\not\in\II(\mu)\cup\III(\mu)$, there exist $i, j\in\{1, 2, 3\}$ such that $i\neq j$ and $w\not\in\Cyl_i(\mu)\cup\Cyl_j(\mu)$. Since $w\in\mathfrak{B}$, by Lemma~\ref{lemma:AB surfaces}, $w-s\mathbf{e}_i\in\mathfrak{B}$ for $0\leq s\leq w_i+1$ and $w-t\mathbf{e}_j\in\mathfrak{B}$ for $0\leq t\leq w_j+1$. On the other hand, since $w\not\in\Cyl_i(\mu)$, $w-s\mathbf{e}_i\not\in\Cyl_i(\mu)$ for $0\leq s\leq w_i+1$. Similarly, since $w\not\in\Cyl_j(\mu)$, $w-t\mathbf{e}_j\not\in\Cyl_j(\mu)$ for $0\leq t\leq w_j+1$. Also, the $i$th coordinate of $w-(w_i+1)\mathbf{e}_i$ is $-1$, while the other coordinates are nonnegative, so $w-(w_i+1)\mathbf{e}_i\in\Cyl_i^-(\nu)\setminus\I^-(\mu)$. Similarly, the $j$th coordinate of $w-(w_j+1)\mathbf{e}_j$ is $-1$, while the other coordinates are nonnegative, so $w-(w_j+1)\mathbf{e}_j\in\Cyl_j^-(\nu)\setminus\I^-(\mu)$. Therefore, $w-s\mathbf{e}_i\not\in R_2\cup(\I^-(\mu)\cup\III(\mu))$ and, thus, $w-s\mathbf{e}_i\not\in\mathfrak{A}$ for $0\leq s\leq w_i+1$. Similarly, $w-t\mathbf{e}_j\not\in R_2\cup(\I^-(\mu)\cup\III(\mu))$ and, thus, $w-t\mathbf{e}_j\not\in\mathfrak{A}$ for $0\leq t\leq w_j+1$. Consequently, $\{w-s\mathbf{e}_i\mid 0\leq s\leq w_i+1\}\cup\{w-t\mathbf{e}_j\mid 0\leq t\leq w_j+1\}\subseteq\mathcal{L}(A, B)$ is a connected set of cells containing $w-(w_i+1)\mathbf{e}_i\in\Cyl_i^-(\nu)$ and $w-(w_j+1)\mathbf{e}_j\in\Cyl_j^-(\nu)$. By Theorem~\ref{thm:labellable iff in sAB}, $(A, B)\not\in\sAB$, so by Corollary~\ref{cor:in sAB iff rainbow pairing}, the nodes of $D_{(A, B)}(N)$ are not paired according to the rainbow pairing. Again, this contradicts the fact that $D_{(A, B)}(N)=D\in D_{\sigma}(H(N), \mathbf{N}_{\mu}(N))$. So we can, in fact, extend the tiling corresponding to $M_2$ so that $\mathfrak{B}=R_1\cup(\II(\mu)\cup\III(\mu))\setminus B$, where $B\subseteq\II(\mu)\cup\III(\mu)$. 

Now, if $D(H(N))$ denotes the set of all double-dimer configurations with nodes on $H(N)$, and $\tau: \sAB_{\text{all}}\to D(H(N))$ is the map $(A, B)\mapsto D_{(A, B)}(N)$, then $(A, B)\in\tau^{-1}(D)$. We conclude that $\lvert\tau^{-1}(D)\rvert=2^{k(D)}$. Finally, given any $(A, B)\in \tau^{-1}(D_{\sigma}(H(N), \mathbf{N}_{\mu}(N)))$, the nodes of $D_{(A, B)}(N)$ are $\mathbf{N}_{\mu}(N)$, so all boxes in $A\cup B$ must correspond to faces of $H(N)$, and we deduce that $|A|+|B|+w_{(A, B)}(N)=|\II|+2|\III|+w_{base}(\mu)$. Also, by Corollary~\ref{cor:in sAB iff rainbow pairing}, $\tau^{-1}(D_{\sigma}(H(N), \mathbf{N}_{\mu}(N)))\subseteq\sAB$. 

As a result, 
\begin{eqnarray*}
&&q^{-w_{base}(\mu)}Z^{DD}_{\sigma}(H(N), \mathbf{N}_{\mu}(N)) \\
&=&q^{-w_{base}(\mu)}\sum_{D\in D_{\sigma}(H(N), \mathbf{N}_{\mu}(N))}w(D) \\
&=&q^{-w_{base}(\mu)}\sum_{D\in D_{\sigma}(H(N), \mathbf{N}_{\mu}(N))}\sum_{(A, B)\in\tau^{-1}(D)}\frac{w(D)}{2^{k(D)}} \\
&=&q^{-w_{base}(\mu)}\sum_{D\in D_{\sigma}(H(N), \mathbf{N}_{\mu}(N))}\sum_{(A, B)\in\tau^{-1}(D)}q^{w_{(A, B)}(N)} \\
&=&q^{-w_{base}(\mu)}\sum_{(A, B)\in\tau^{-1}(D_{\sigma}(H(N), \mathbf{N}_{\mu}(N)))}q^{w_{(A, B)}(N)} \\
&=&q^{-w_{base}(\mu)}\sum_{\substack{(A, B)\in\sAB \\ |A|+|B|\leq N-M}}q^{w_{(A, B)}(N)}+q^{|\II|+2|\III|}\sum_{\substack{(A, B)\in\tau^{-1}(D_{\sigma}(H(N), \mathbf{N}_{\mu}(N))) \\ |A|+|B|>N-M}}q^{-|A|-|B|}.
\end{eqnarray*}
This discussion, along with Theorem~\ref{thm:ZABW}, leads to the following result.

\begin{thm}
\label{thm:ZDD convergence}
As $N\to\infty$, $\widetilde{Z}^{DD}_{\sigma}(H(N), \mathbf{N}_{\mu}(N)):=q^{-w_{base}(\mu)}Z^{DD}_{\sigma}(H(N), \mathbf{N}_{\mu}(N))$ converges to $Z_{\sAB}(q^{-1})=W(\mu_1, \mu_2, \mu_3; q^{-1})$. 
\end{thm}

\subsection{The condensation recurrence in PT theory}
\label{sec:pt_condensation_identity}

In \cite{jenne}, the first author showed that when $\sigma$ is tripartite, $Z^{DD}_{\sigma}(G, {\bf N})$ satisfies the condensation recurrence. 
% and certain other technical conditions hold we have the following:
%\helentodo{We can say at least some of these technical conditions without making things too confusing, I'll add something here.}
%\begin{eqnarray*}
%& & 
%Z^{DD}_{\sigma}(G, {\bf N}) Z^{DD}_{\sigma_5}(G, {\bf N} - \{a, b, c, d\})\\ & =& 
%Z^{DD}_{\sigma_1}(G, {\bf N} - \{a, b\})  
%Z^{DD}_{\sigma_2}(G, {\bf N} - \{c, d\}) +
% Z^{DD}_{\sigma_3}(G, {\bf N} - \{a, d\}) Z^{DD}_{\sigma_4}(G , {\bf N} - \{b, c\}) 
%\end{eqnarray*}

\begin{thm}\cite[Theorem 2.1.1]{jenne}
\label{thm:DDcond}
Let $G = (V_1, V_2, E)$ be a finite edge-weighted planar bipartite graph with a set of nodes {\bf N}. Divide the nodes into three circularly contiguous sets $R$, $G$, and $B$ such that $|R|$, $|G|$ and $|B|$ satisfy the triangle inequality and let $\sigma$ be the corresponding tripartite pairing.\footnote{If $|R|, |G|$, and $|B|$ do not satisfy the triangle inequality, there is no corresponding tripartite pairing $\sigma$.} Let $a, b, c, d$ be nodes appearing in a cyclic order such 
%that the set
%$\{x,y,w,v\}$ contains at least one node of each RGB color%\footnotemark. 
%If 
that $a, c \in V_1$ and $b, d \in V_2$.\footnote{Additionally, $\{a, b, c, d\}$ must contain at least one node of each RGB color. In our applications of this theorem, this assumption is always satisfied.} Then 
\begin{eqnarray}
Z^{DD}_{\sigma}(G, {\bf N}) Z^{DD}_{\sigma_{abcd}}(G, {\bf N} - \{a, b, c, d\}) &=& 
Z^{DD}_{\sigma_{ab}}(G, {\bf N} - \{a, b\}) Z^{DD}_{\sigma_{cd}}(G, {\bf N} - \{c, d\}) \label{eqn:ddcond} \\ 
&&{}+{}
Z^{DD}_{\sigma_{ad}}(G, {\bf N} - \{a, d\}) Z^{DD}_{\sigma_{bc}}(G, {\bf N} - \{b, c\}) \nonumber
\end{eqnarray}
where $\sigma_{abcd}$ is the unique planar pairing on ${\bf N} - \{a, b, c, d\}$ in which like RGB colors are not paired together, and for $i, j \in \{a, b, c, d\}$, $\sigma_{ij}$ is the unique planar pairing on ${\bf N} - \{i, j\}$ %corresponding node set 
in which like RGB colors are not paired together.
\end{thm}
%\begin{thm}
%\label{cor:cond}
%Let $G = (V_1, V_2, E)$ be a finite edge-weighted planar bipartite graph with a set of nodes {\bf N}. Divide the nodes into three circularly contiguous sets $R$, $G$, and $B$ such that $|R|$, $|G|$ and $|B|$ satisfy the triangle inequality and let $\sigma$ be the corresponding tripartite pairing\footnotemark. Let $x, y, w, v$ be nodes appearing in a cyclic order such that the set $\{x, y, w, v\}$ contains at least one node of each $RGB$ color. If $x, w \in V_1$ and $y, v \in V_2$ then
%\begin{eqnarray*}
%& & 
%Z^{DD}_{\sigma}(G, {\bf N}) Z^{DD}_{\sigma_5}(G, {\bf N} - \{x, y, w, v\})\\ & =& 
%Z^{DD}_{\sigma_1}(G, {\bf N} - \{x, y\}) 
%Z^{DD}_{\sigma_2}(G, {\bf N} - \{w, v\}) +
%Z^{DD}_{\sigma_3}(G, {\bf N} - \{x, v\}) Z^{DD}_{\sigma_4}(G, {\bf N} - \{w, y\}) 
%\end{eqnarray*}
%where $\sigma_i$ is the unique planar pairing on the corresponding node set in which like colors are not paired together. 
%\end{thm}

%\footnotetext{If $|R|, |G|,$ and $|B|$ do not satisfy the triangle inequality, there is no corresponding tripartite pairing $\sigma$}

%Given $\mu_1, \mu_2,$ and $\mu_3$, the corresponding Maya diagrams $S_1, S_2,$ and $S_3$ and the corresponding nodeset ${\bf N}$ we will apply Theorem~\ref{cor:cond} by {\em adding}
%four nodes to the graph so that ${\bf N} = {\bf \widetilde{N}} - \{r, s, g, b\}$ for four nodes $r, s, g$ and $b$. We choose the four nodes as follows:
%

We apply this recurrence with $G=H(N)$, $\mathbf{N}=\mathbf{N}_{\mu_1^{rc}, \mu_2^{rc}, \mu_3}(N)$, and the RGB coloring defined in Section~\ref{sec:labelling_algorithm_proofs}. We choose the four nodes $a$, $b$, $c$, and $d$ as follows: Let $S_i$ be the Maya diagram of $\mu_i$, and let $a$ and $b$ be the nodes in sector 1 labelled by $\max S_1^-$ and $\min S_1^+$, respectively. Similarly, we let $c$ and $d$ be the nodes in sector 2 labelled by $\max S_2^-$ and $\min S_2^+$. Note that these nodes have the same coordinates as the vertices specified in Section~\ref{sec:DTcond} but the coordinate system is different (see Figure~\ref{fig:sectors}). We remark that $a$ is a red node in sector 1, $b$ is a blue node in sector 1, $c$ is a green node in sector 2, and $d$ is a red node in sector 2. So $a$, $b$, $c$, and $d$ appear in cyclic order, alternating black and white. 

As in DT theory (see Section~\ref{sec:DTcond}), 
\begin{itemize}
\item the set of nodes ${\bf N} - \{a, b, c, d\}$ corresponds to the partitions $\mu_1$, $\mu_2$, $\mu_3$, 
\item the set of nodes ${\bf N}$ corresponds to the partitions $\mu_1^{rc}$, $\mu_2^{rc}$, $\mu_3$, 
\item the set of nodes ${\bf N} - \{a, b\}$ corresponds to the partitions $\mu_1$, $\mu_2^{rc}$, $\mu_3$, 
\item the set of nodes ${\bf N} - \{c, d\}$ corresponds to the partitions $\mu_1^{rc}$, $\mu_2$, $\mu_3$, 
\item the set of nodes ${\bf N} - \{a, d\}$ corresponds to the partitions $\mu_1^r$, $\mu_2^c$, $\mu_3$, and 
\item the set of nodes ${\bf N} - \{b, c\}$ corresponds to the partitions $\mu_1^{c}$, $\mu_2^r$, $\mu_3$. 
\end{itemize}
In Lemma~\ref{cor:ptweight}, we compute the edge-weight of the base$_{\mu}$ double-dimer configuration on $(H(N), \mathbf{N}_{\mu}(N))=(H(N), \mathbf{N}_{\mu_1^{rc}, \mu_2^{rc}, \mu_3}(N)- \{a, b, c, d\})$. We can also apply Lemma~\ref{cor:ptweight} to obtain the edge-weights of the base double-dimer configurations on $(H(N), {\bf N}_{\mu_1^{rc}, \mu_2^{rc}, \mu_3}(N))$, $(H(N), {\bf N}_{\mu_1^{rc}, \mu_2^{rc}, \mu_3}(N)- \{a, b\})=(H(N), \mathbf{N}_{\mu_1, \mu_2^{rc}, \mu_3}(N))$, and $(H(N), {\bf N}_{\mu_1^{rc}, \mu_2^{rc}, \mu_3}(N) - \{c, d\})=(H(N), \mathbf{N}_{\mu_1^{rc}, \mu_2, \mu_3}(N))$. To do so, we simply modify the partitions in the lemma statement appropriately. 

%because the unique tripartite pairing still has the property that the there are an equal number of nodes of opposite RGB coloring in sector $i$. We just have to modify the partitions in the lemma statement appropriately. 

For double-dimer configurations on $H(N)$ with nodes ${\bf N} - \{a, d\}$ or ${\bf N} - \{b, c\}$, more care is required. This is because $\mathbf{N} - \{a, d\}\neq\mathbf{N}_{\mu_1^r, \mu_2^c, \mu_3}(N)$ and $\mathbf{N} - \{b, c\}\neq\mathbf{N}_{\mu_1^c, \mu_2^r, \mu_3}(N)$. In the first case, ${\bf N} - \{a, d\}=\mathbf{N}_{\mu}(N)\cup\{b, c\}$, so we have added $b$ and $c$ (a blue node in sector 1 and a green node in sector 2) to the node set $\mathbf{N}_{\mu}(N)$. So, the unique planar pairing $\sigma_{ad}$ on ${\bf N} - \{a, d\}$ has one more blue-green path (going from a blue node in sector 1 to a green node in sector 2) than $\sigma_{abcd}$. We remark that it is no longer the case that all blue-green paths begin and end in sector 3. Similarly, the pairing $\sigma_{bc}$ has one more red-green and one more red-blue path than $\sigma_{abcd}$, and one fewer blue-green path. We illustrate this with an example. 

%\gautamtodo{edit pictures in the following example so that the nodes are obvious}

\begin{example}
Let $N = 5$ and let $\mu_1 = (3, 2)$, $\mu_2 = (2, 2)$, and $\mu_3=\emptyset$. Then the node sets $\mathbf{N}-\{a, b, c, d\}$, $\mathbf{N}-\{a, d\}$, and $\mathbf{N}-\{b, c\}$ are as shown below. 
\begin{center}
\includegraphics[width=0.3\textwidth]{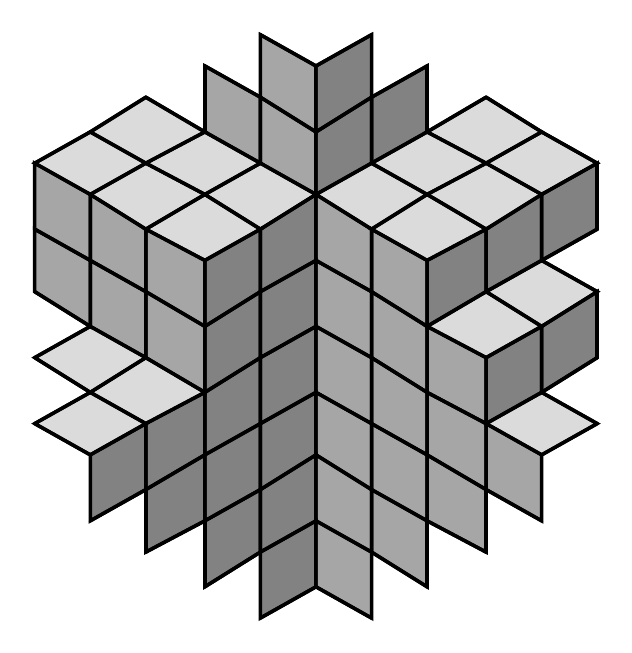}\hfill
\includegraphics[width=0.3\textwidth]{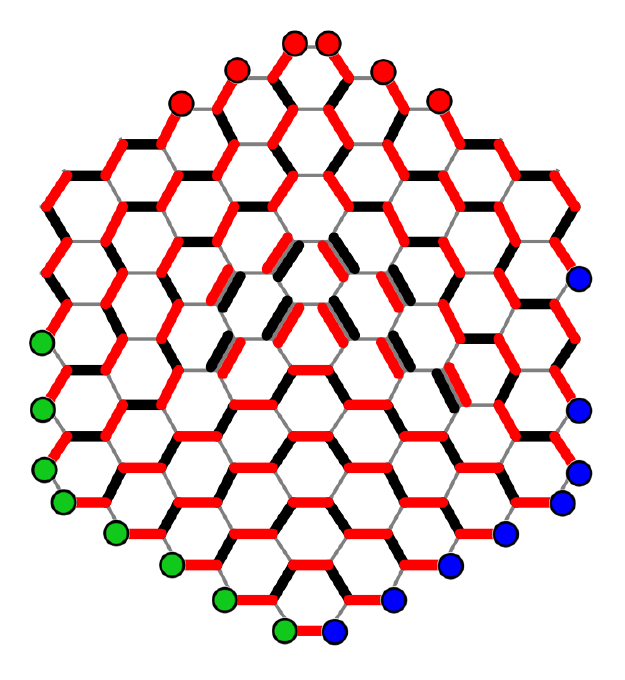}\hfill
\includegraphics[width=0.3\textwidth]{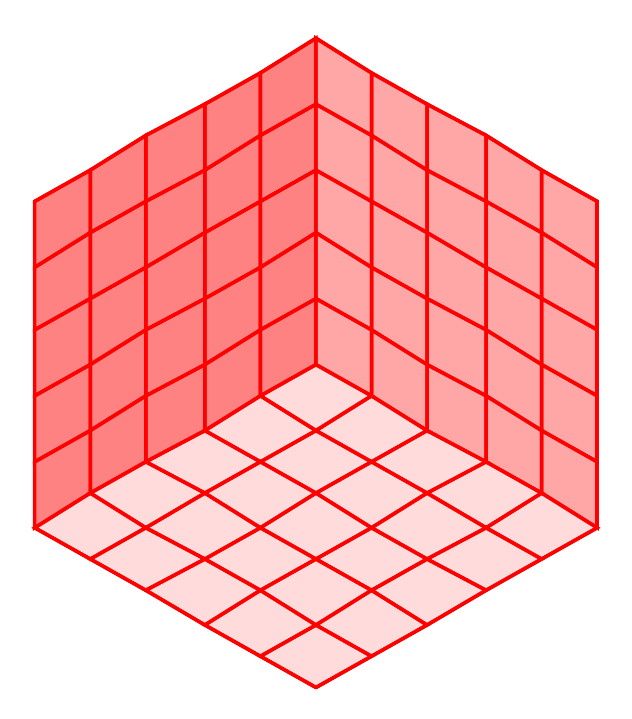}
\end{center}
When we add $b$ and $c$, we have $\mu_1^r = (4)$ and $\mu_2^c = (1, 1, 1)$, as shown below. Note that the double-dimer configuration shown has a blue-green path from sector 1 to sector 2. 
\begin{center}
\includegraphics[width=0.3\textwidth]{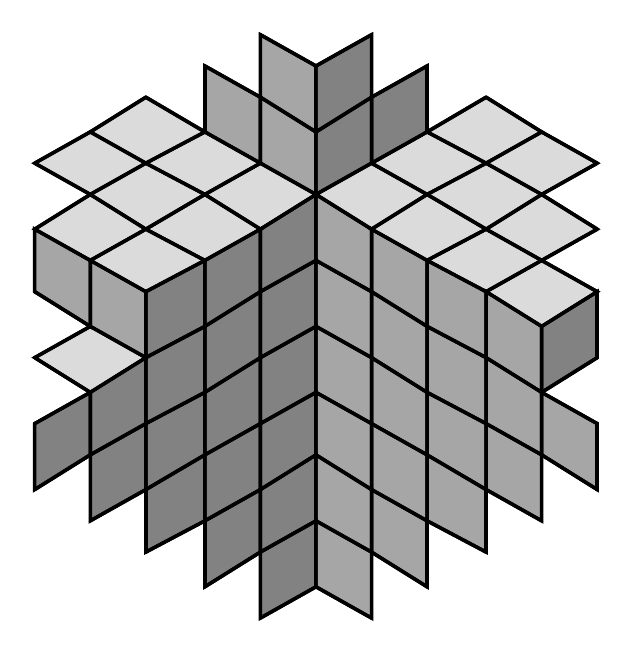}\hfill
\includegraphics[width=0.3\textwidth]{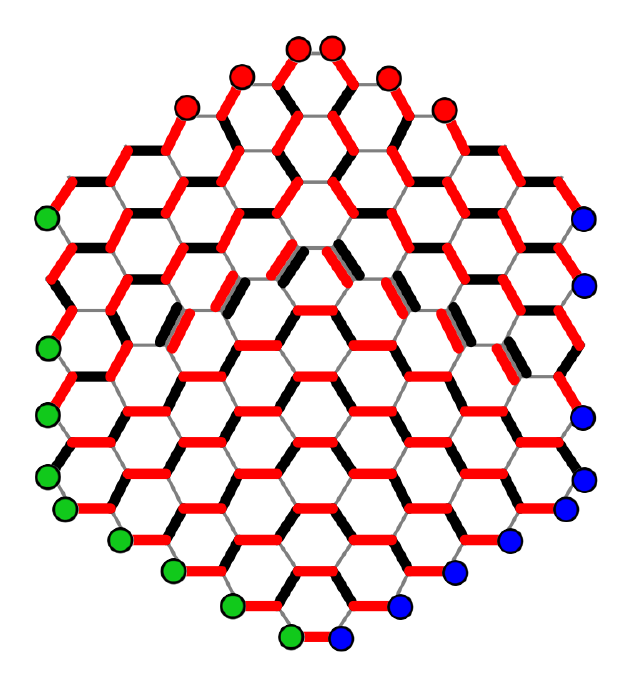}\hfill
\includegraphics[width=0.3\textwidth]{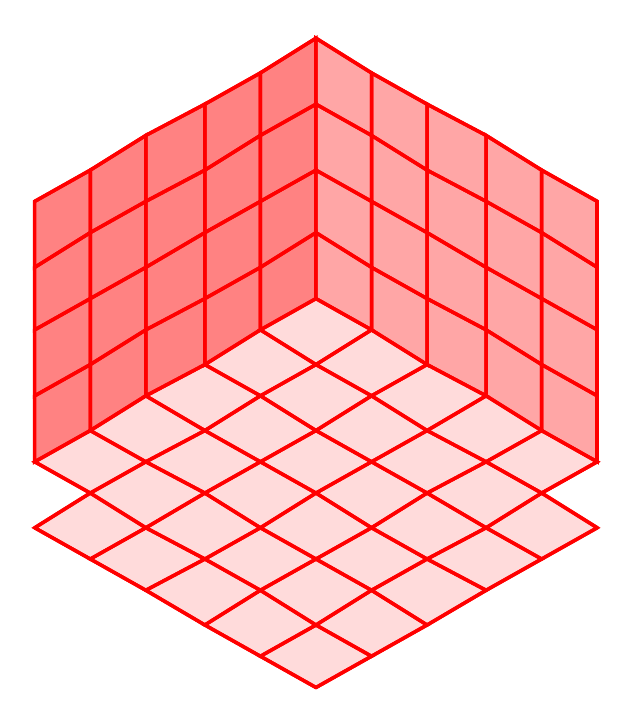}
\end{center}
When we add $a$ and $d$, we have $\mu_1^c = (2, 1, 1)$ and $\mu_2^r = (3)$, as shown below. 
\begin{center}
\includegraphics[width=0.3\textwidth]{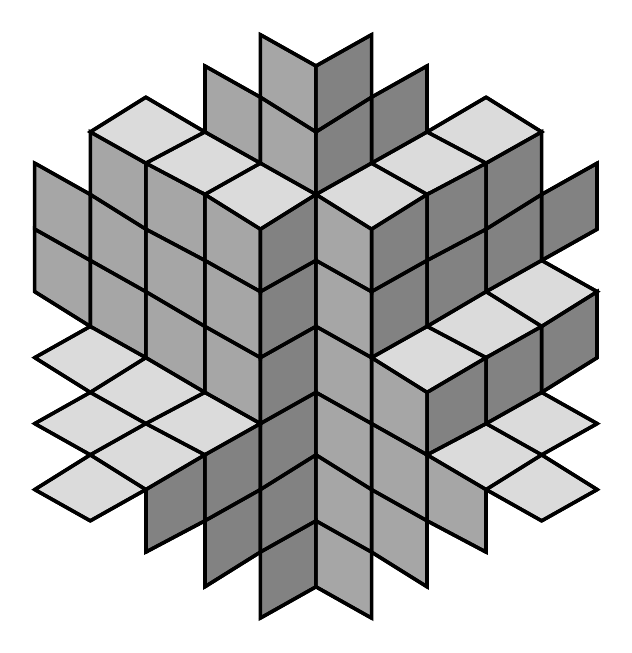}\hfill
\includegraphics[width=0.3\textwidth]{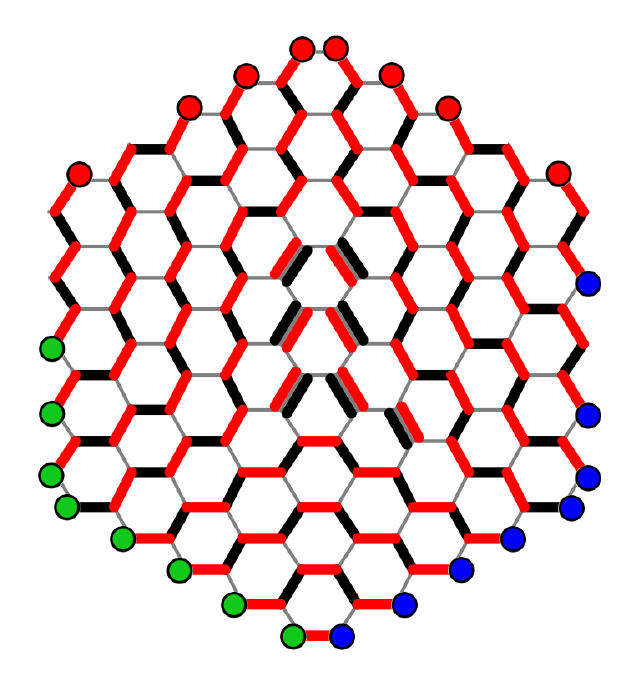}\hfill
\includegraphics[width=0.3\textwidth]{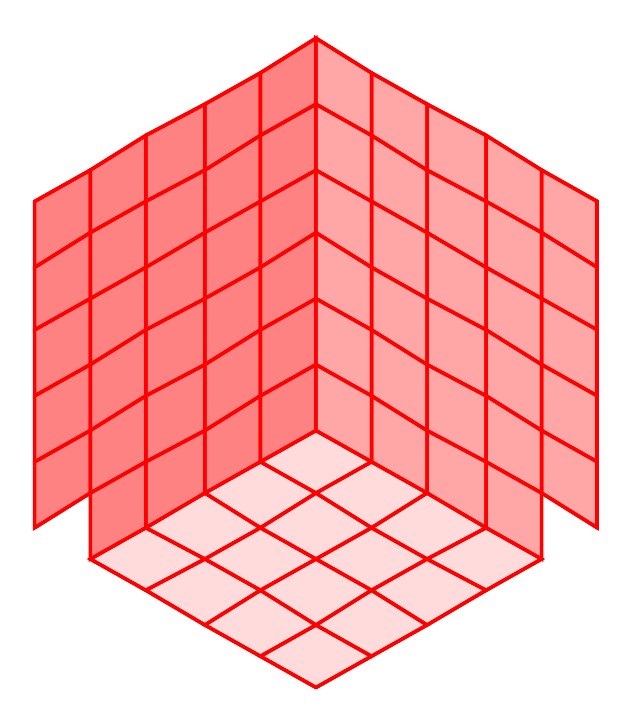}
\end{center}
\end{example}

As illustrated in the example, when the node set is ${\bf N} - \{a, d\}$, the base double-dimer configuration arises from an $AB$ configuration $(A, B)$ (associated with partitions $\mu_1^r$, $\mu_2^c$, and $\mu_3$). But, as we can see from the presence of a blue-green path from sector 1 to sector 2, this double-dimer configuration is not the result of the truncation procedure described in Section~\ref{sec:double_dimer_configs}, i.e., it is not equal to $D_{(A, B)}(N)$. Instead, the tilings and corresponding dimer configurations are shifted up by one unit prior to truncation. We will refer to this double-dimer configuration as the {\em base$_{up}$ double-dimer configuration}. We use the notation base$_{up}$ rather than  base$_{\mu_1^r, \mu_2^c, \mu_3}$, because base$_{\mu_1^r, \mu_2^c, \mu_3}$ refers to a double-dimer configuration described in Definition~\ref{defn:baseDD}, which is truncated in the usual way. 

Similarly, when the node set is ${\bf N} - \{b, c\}$, the base double-dimer configuration arises from an $AB$ configuration (associated with partitions $\mu_1^c, \mu_2^r,$ and $\mu_3$). However, the tilings and corresponding dimer configurations are shifted down by one unit prior to truncation. We will refer to this double-dimer configuration as the {\em base$_{down}$ double-dimer configuration}. 

%\helentodo{Do we need to remark on the fact that the room of the B configuration looks different in these two cases? Does this cause any issues with interpreting these double-dimer configurations as coming from AB configurations?}

Let $q^{w_{up}}$ be the edge-weight of the base$_{up}$ double-dimer configuration, and let $q^{w_{down}}$ be the edge-weight of the base$_{down}$ double-dimer configuration. We compute both of these quantities in Section~\ref{sec:baseweight} (see Lemmas~\ref{cor:ptweight_SU} and~\ref{cor:ptweight_SD}). Then let 
\begin{eqnarray*}
\widetilde{Z}^{DD}_{\sigma}(H(N), {\bf N}_{\mu_1^{rc}, \mu_2^{rc}, \mu_3}(N)) &=& q^{-w_{base}(\mu_1^{rc}, \mu_2^{rc}, \mu_3)}Z^{DD}_{\sigma}(H(N), {\bf N}_{\mu_1^{rc}, \mu_2^{rc}, \mu_3}(N)), \\
\widetilde{Z}^{DD}_{\sigma_{abcd}}(H(N), {\bf N}_{\mu_1, \mu_2, \mu_3}(N)) &=& q^{-w_{base}(\mu_1, \mu_2, \mu_3)}Z^{DD}_{\sigma_{abcd}}(H(N), {\bf N}_{\mu_1, \mu_2, \mu_3}(N)), \\
\widetilde{Z}^{DD}_{\sigma_{ab}}(H(N), {\bf N}_{\mu_1, \mu_2^{rc}, \mu_3}(N)) &=& q^{-w_{base}(\mu_1, \mu_2^{rc}, \mu_3)}Z^{DD}_{\sigma_{ab}}(H(N), {\bf N}_{\mu_1, \mu_2^{rc}, \mu_3}(N)), \\
\widetilde{Z}^{DD}_{\sigma_{cd}}(H(N), {\bf N}_{\mu_1^{rc}, \mu_2, \mu_3}(N)) &=& q^{-w_{base}(\mu_1^{rc}, \mu_2, \mu_3)}Z^{DD}_{\sigma_{cd}}(H(N), {\bf N}_{\mu_1^{rc}, \mu_2, \mu_3}(N)), \\
\widetilde{Z}^{DD}_{\sigma_{ad}}(H(N), {\bf N}_{\mu_1^{rc}, \mu_2^{rc}, \mu_3}(N)-\{a, d\}) &=& q^{-w_{up}}Z^{DD}_{\sigma_{ad}}(H(N), {\bf N}_{\mu_1^{rc}, \mu_2^{rc}, \mu_3}(N)-\{a, d\}), \text{ and} \\
\widetilde{Z}^{DD}_{\sigma_{bc}}(H(N), {\bf N}_{\mu_1^{rc}, \mu_2^{rc}, \mu_3}(N)-\{b, c\}) &=& q^{-w_{down}}Z^{DD}_{\sigma_{bc}}(H(N), {\bf N}_{\mu_1^{rc}, \mu_2^{rc}, \mu_3}(N)-\{b, c\}). 
\end{eqnarray*}
Let 
\begin{align*}
A & = w_{base}(\mu_1^{rc}, \mu_2^{rc}, \mu_3) + w_{base}(\mu_1, \mu_2, \mu_3), \\
B & = w_{base}(\mu_1, \mu_2^{rc}, \mu_3) + w_{base}(\mu_1^{rc}, \mu_2, \mu_3), \text{ and} \\
C & = w_{up} + w_{down}. 
\end{align*}

From the condensation recurrence (\ref{eqn:ddcond}) and the preceding remarks, we have 
\begin{eqnarray}
&&q^{A} \widetilde{Z}^{DD}_{\sigma}(H(N), {\bf N}_{\mu_1^{rc}, \mu_2^{rc}, \mu_3}(N)) \widetilde{Z}^{DD}_{\sigma_{abcd}}(H(N), {\bf N}_{\mu_1, \mu_2, \mu_3}(N)) \label{eqn:intermediatecond} \\
&=& 
q^{B} \widetilde{Z}^{DD}_{\sigma_{ab}}(H(N), {\bf N}_{\mu_1, \mu_2^{rc}, \mu_3}(N)) \widetilde{Z}^{DD}_{\sigma_{cd}}(H(N), {\bf N}_{\mu_1^{rc}, \mu_2, \mu_3}(N)) \nonumber \\
&&{}+{} q^{C} \widetilde{Z}^{DD}_{\sigma_{ad}}(H(N), {\bf N}_{\mu_1^{rc}, \mu_2^{rc}, \mu_3}(N)-\{a, d\}) \widetilde{Z}^{DD}_{\sigma_{bc}}(H(N), {\bf N}_{\mu_1^{rc}, \mu_2^{rc}, \mu_3}(N)-\{b, c\}). \nonumber
\end{eqnarray}
From Lemma~\ref{cor:ptweight}, we see that $A = B$, and we multiply equation (\ref{eqn:intermediatecond}) by $q^{-A}$. In Section~\ref{sec:PTalg}, we show that $C - A=K$, which does not depend on $N$. So, we can take $N \to \infty$, and each of the Laurent series $\widetilde{Z}^{DD}$ converges to an instance of $W$, with different partitions as parameters. The convergence of $\widetilde{Z}^{DD}_{\sigma}(H(N), {\bf N}_{\mu_1^{rc}, \mu_2^{rc}, \mu_3}(N))$ to $W(\mu_1^{rc}, \mu_2^{rc}, \mu_3; q^{-1})$ follows from Theorem~\ref{thm:ZDD convergence}. By the same theorem, $\widetilde{Z}^{DD}_{\sigma_{abcd}}(H(N), {\bf N}_{\mu_1, \mu_2, \mu_3}(N))$ converges to $W(\mu_1, \mu_2, \mu_3; q^{-1})$, $\widetilde{Z}^{DD}_{\sigma_{ab}}(H(N), {\bf N}_{\mu_1, \mu_2^{rc}, \mu_3}(N))$ converges to $W(\mu_1, \mu_2^{rc}, \mu_3; q^{-1})$, and $\widetilde{Z}^{DD}_{\sigma_{cd}}(H(N), {\bf N}_{\mu_1^{rc}, \mu_2, \mu_3}(N))$ converges to $W(\mu_1^{rc}, \mu_2, \mu_3; q^{-1})$. For the term $\widetilde{Z}^{DD}_{\sigma_{ad}}(H(N), {\bf N}_{\mu_1^{rc}, \mu_2^{rc}, \mu_3}(N)-\{a, d\})$, we remark that since we take $N \to \infty$, the fact that the tilings and corresponding dimer configurations are shifted before truncation does not matter, and we get convergence to $W(\mu_1^r, \mu_2^c, \mu_3; q^{-1})$. A similar argument implies convergence of $\widetilde{Z}^{DD}_{\sigma_{bc}}(H(N), {\bf N}_{\mu_1^{rc}, \mu_2^{rc}, \mu_3}(N)-\{b, c\})$ to $W(\mu_1^c, \mu_2^r, \mu_3; q^{-1})$. So, we get 
\begin{eqnarray*}
W(\mu_1^{rc}, \mu_2^{rc}, \mu_3; q^{-1}) W(\mu_1, \mu_2, \mu_3; q^{-1}) 
&=& 
W(\mu_1, \mu_2^{rc}, \mu_3; q^{-1}) W(\mu_1^{rc}, \mu_2, \mu_3; q^{-1}) \\
&&{}+{} q^{K} W(\mu_1^{r}, \mu_2^{c}, \mu_3; q^{-1}) W(\mu_1^{c}, \mu_2^{r}, \mu_3; q^{-1}). 
\end{eqnarray*}
Substituting $q$ for $q^{-1}$ and multiplying by $q^K$, we conclude that $W$ satisfies equation (\ref{eqn:vertex_condensation}), as desired. 

%Many details here have been omitted, due to space constraints.
%As in the DT section, this implies equation condensation and weights are hard TBD.

%\begin{itemize}
%\setlength\itemsep{-.25em}
%\item Let $r$ be the vertex corresponding to the negative hole with coordinate closest to 0. 
%\item Let $b$ be the vertex corresponding to the positive bead with coordinate closest to 0. 
%\item Let $s$ be the vertex corresponding to the positive bead closest to 0. 
%\item Let $g$ be the vertex corresponding to the negative hole closest to 0. 
%\end{itemize}

%We can always add four nodes in this way by the assumption that $\mu_1$ and $\mu_2$ are nonempty. 

\subsection{Example}

%For the convenience of readers familiar with~\cite{PT2}, 
We list all of the double-dimer configurations that correspond, via our maps, to the examples in Example~\ref{ex:labelledboxconfigsfrompt} (the same example as that in~\cite[Section 5.4]{PT2}, with the same numbering). The double-dimer configurations corresponding to these configurations appear in Figure~\ref{fig:pt-doubledimer-example}. 

\newlength{\ddwidth}
\setlength{\ddwidth}{0.8in}
\begin{figure}
\begin{tabular}{|r|cccccc|}
	\hline
	Length & (i) & (ii) & (iii) & (iv) & (v) & (vi) \\
	\hline
	\raisebox{0.4in}{0\rule{0in}{0.5in}} &
	\includegraphics[width=\ddwidth]{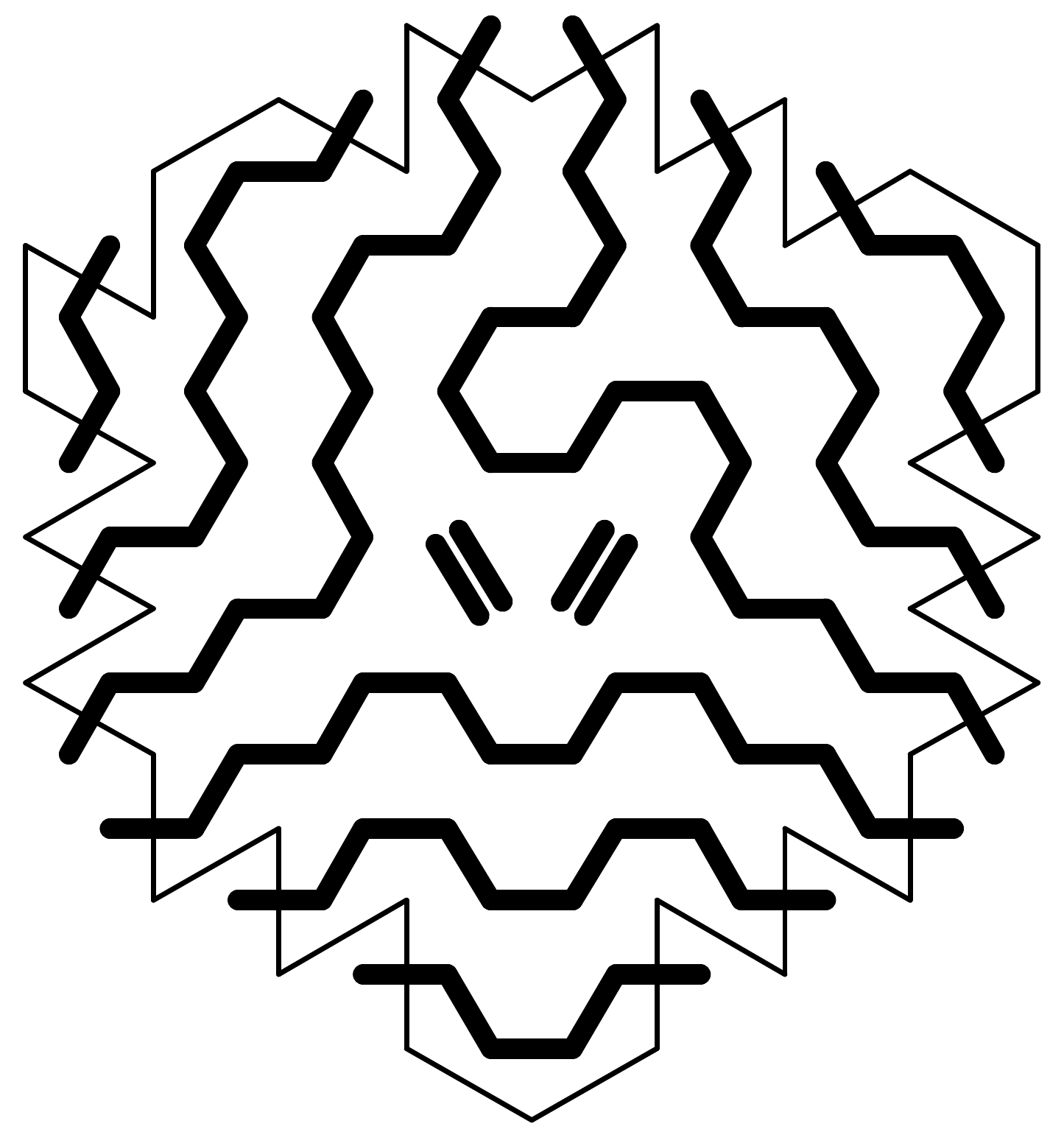} &&&&&\\
	\raisebox{0.4in}{1} &
\includegraphics[width=\ddwidth]{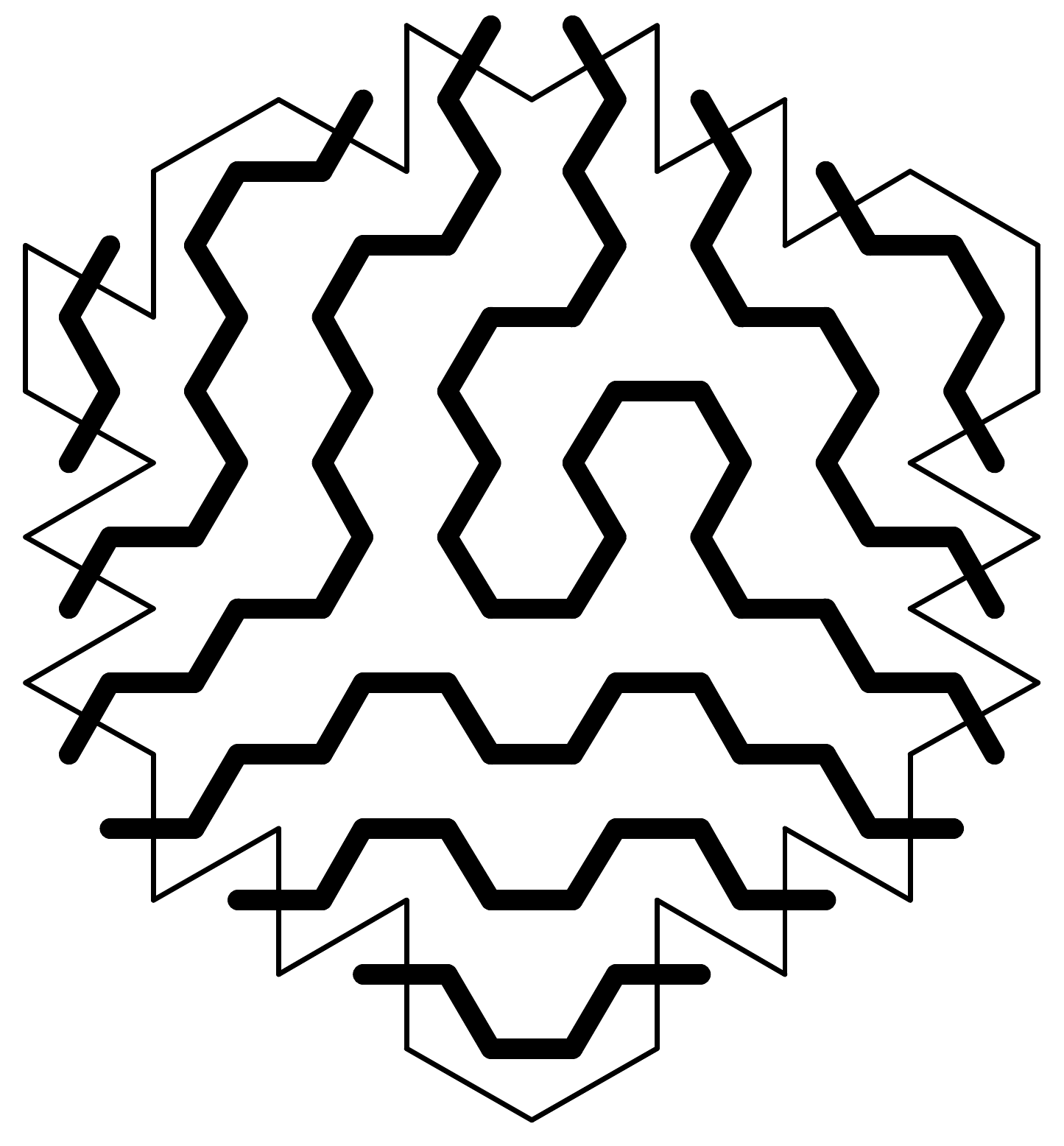} &
	\includegraphics[width=\ddwidth]{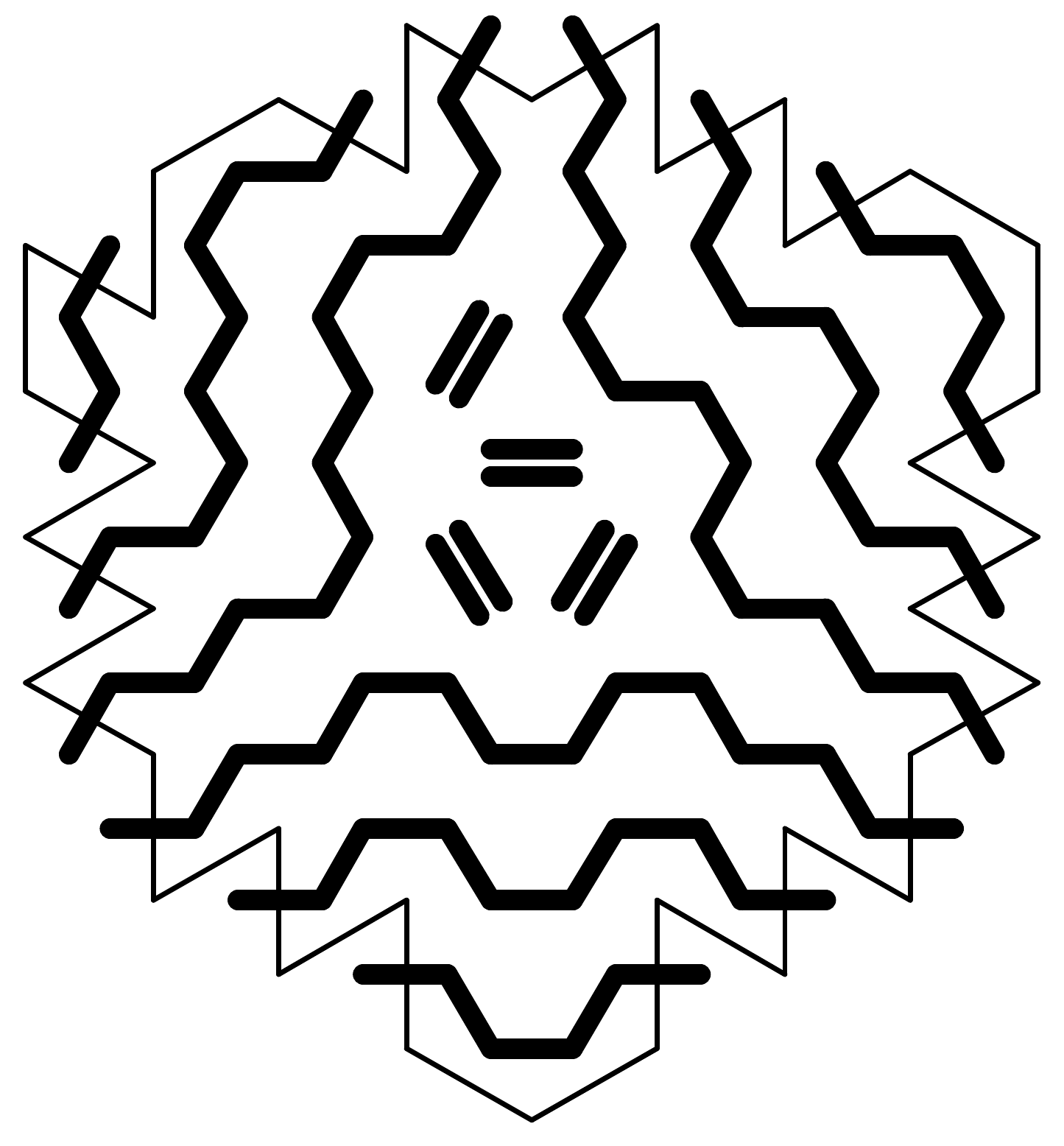} &&&&\\
	\raisebox{0.4in}{2} &
\includegraphics[width=\ddwidth]{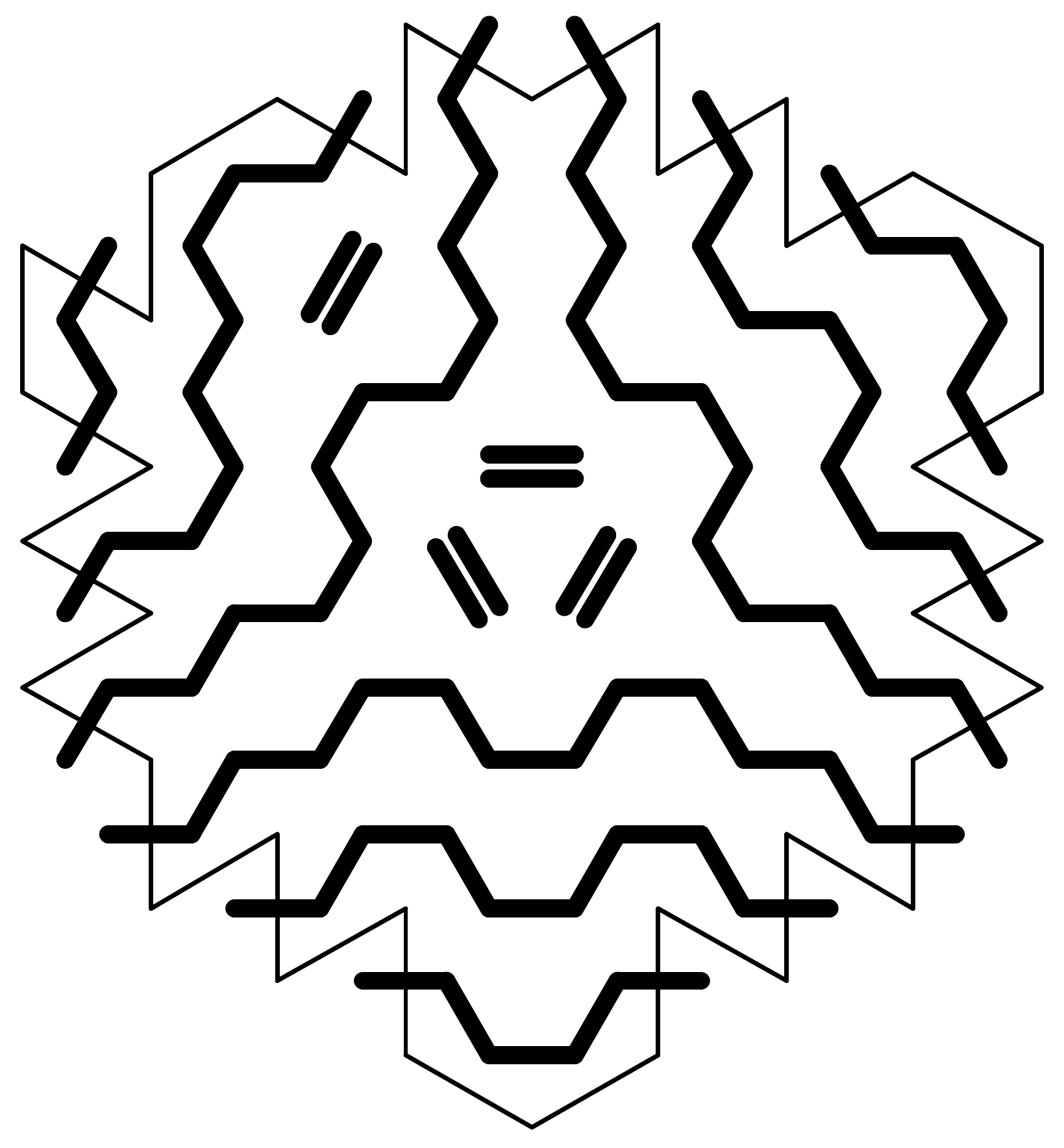} &
\includegraphics[width=\ddwidth]{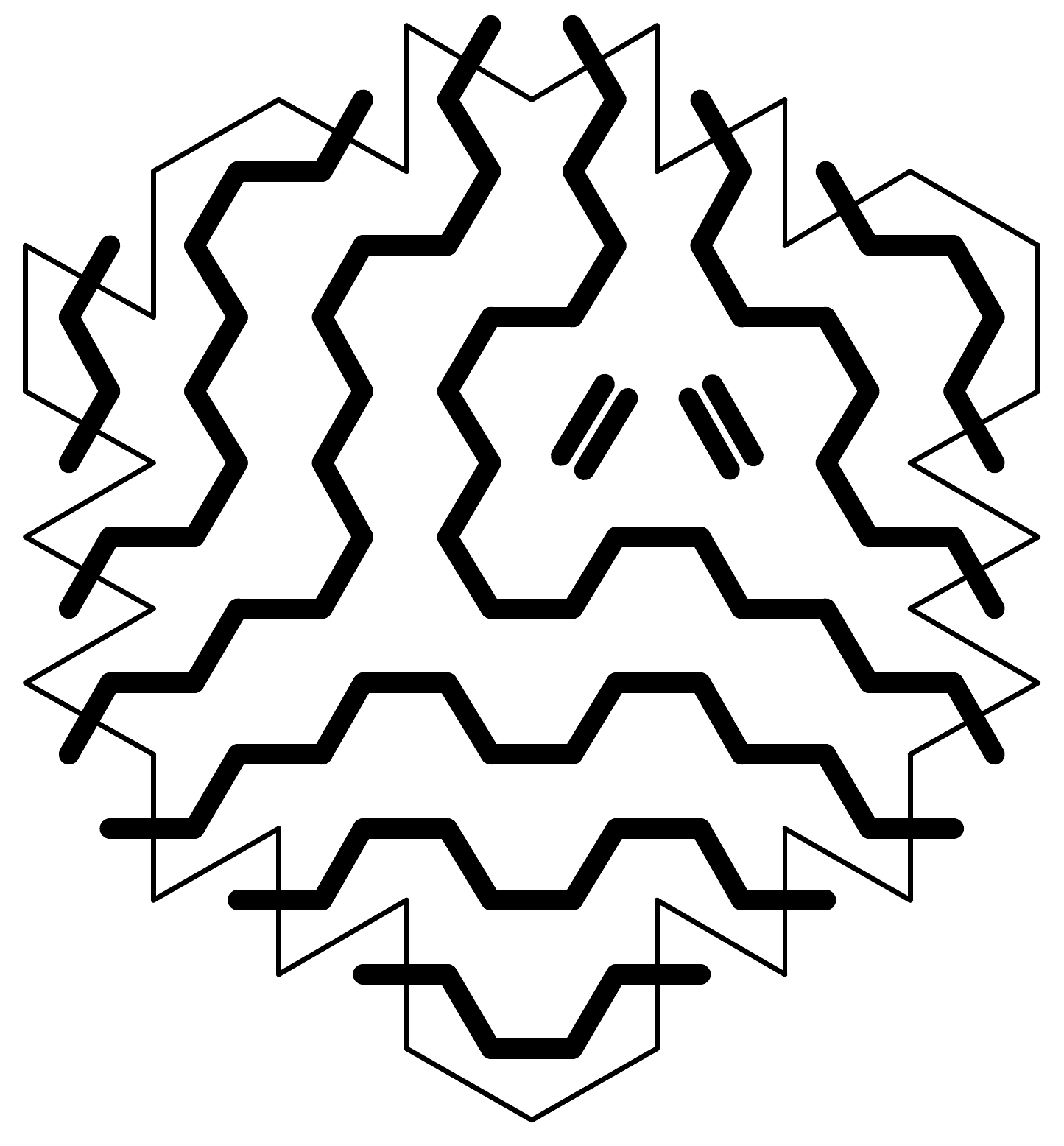} &
	\includegraphics[width=\ddwidth]{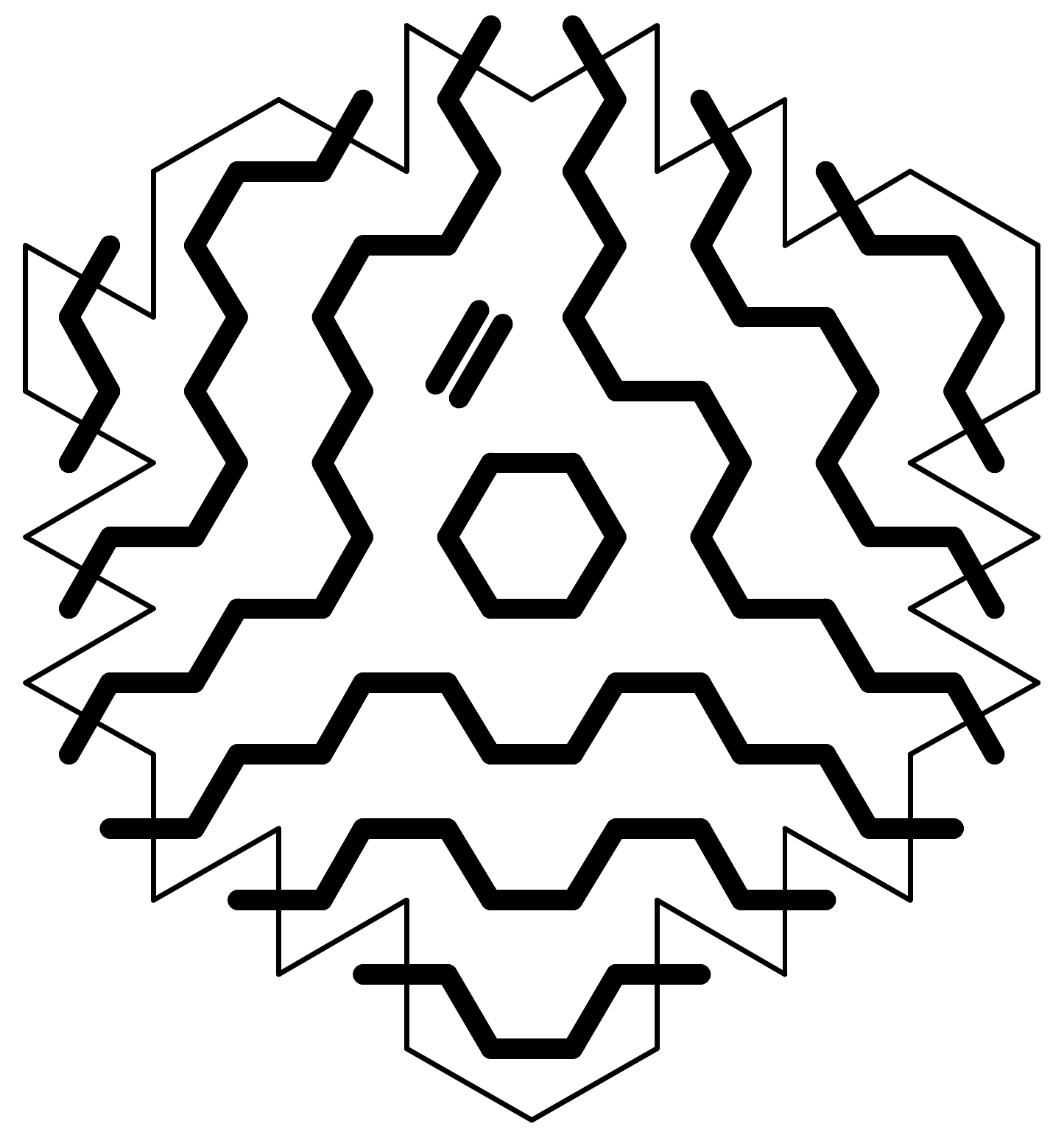} &&&\\
	\raisebox{0.4in}{3} &
\includegraphics[width=\ddwidth]{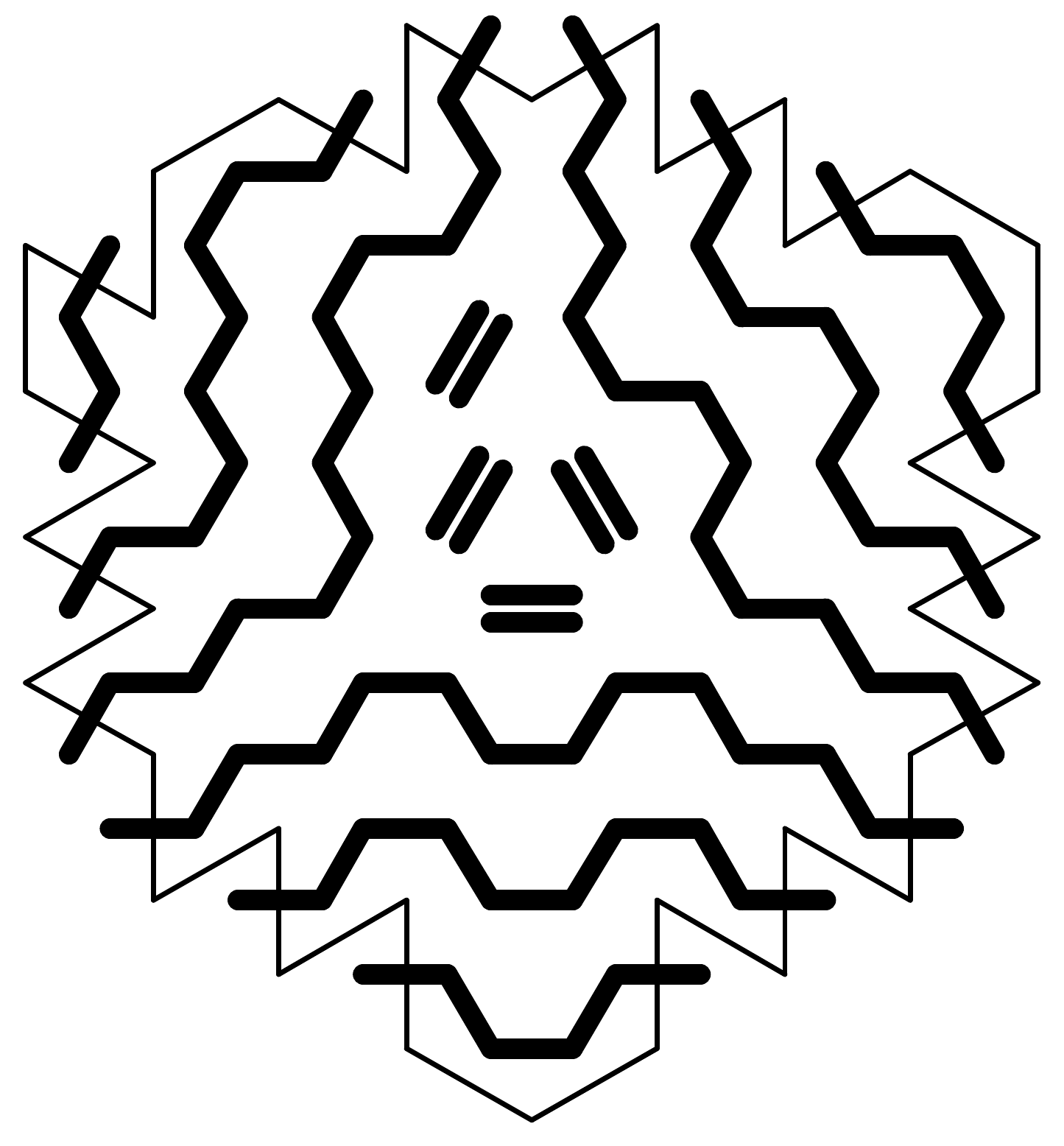} &
\includegraphics[width=\ddwidth]{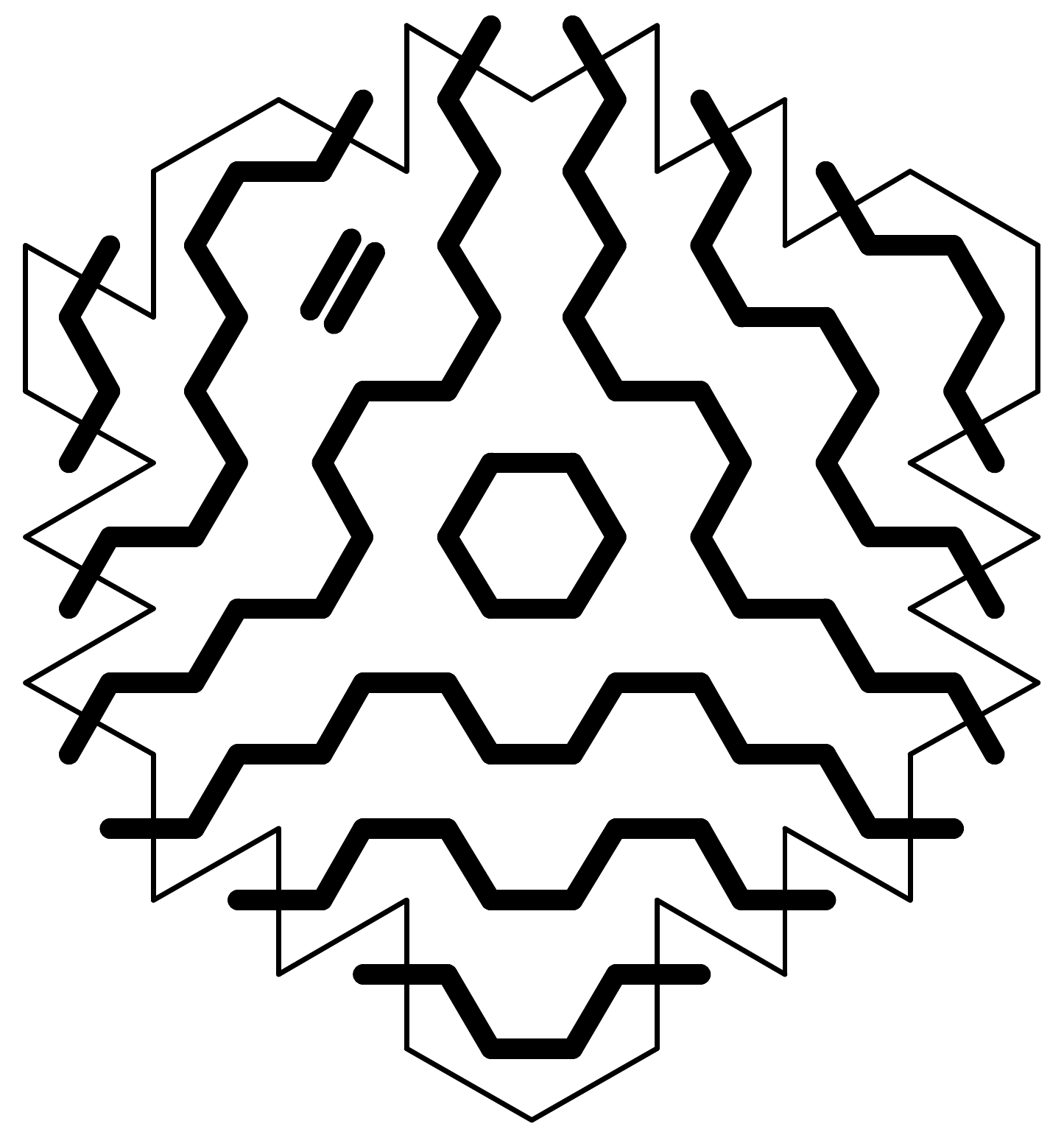} &
\includegraphics[width=\ddwidth]{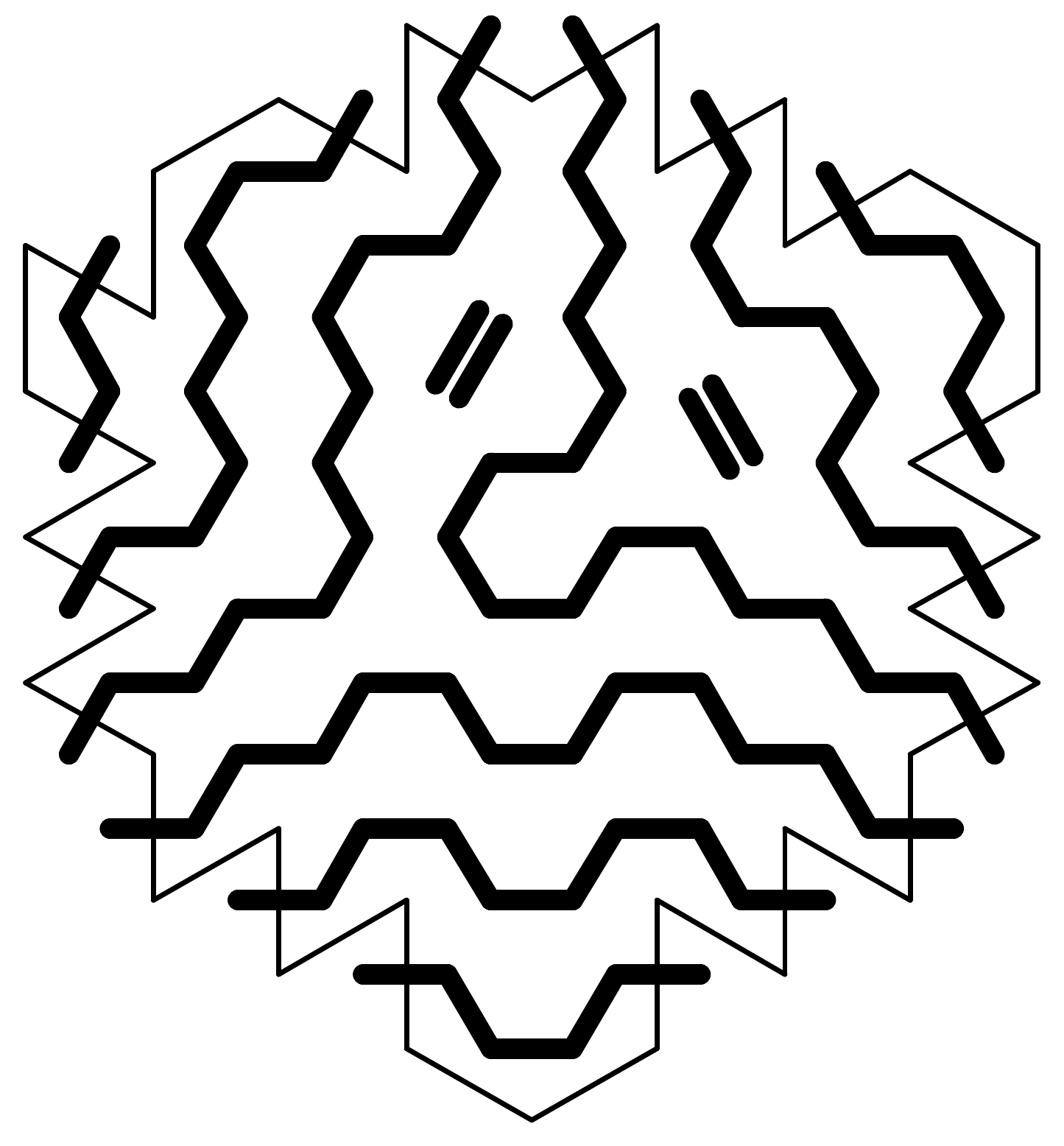} &
\includegraphics[width=\ddwidth]{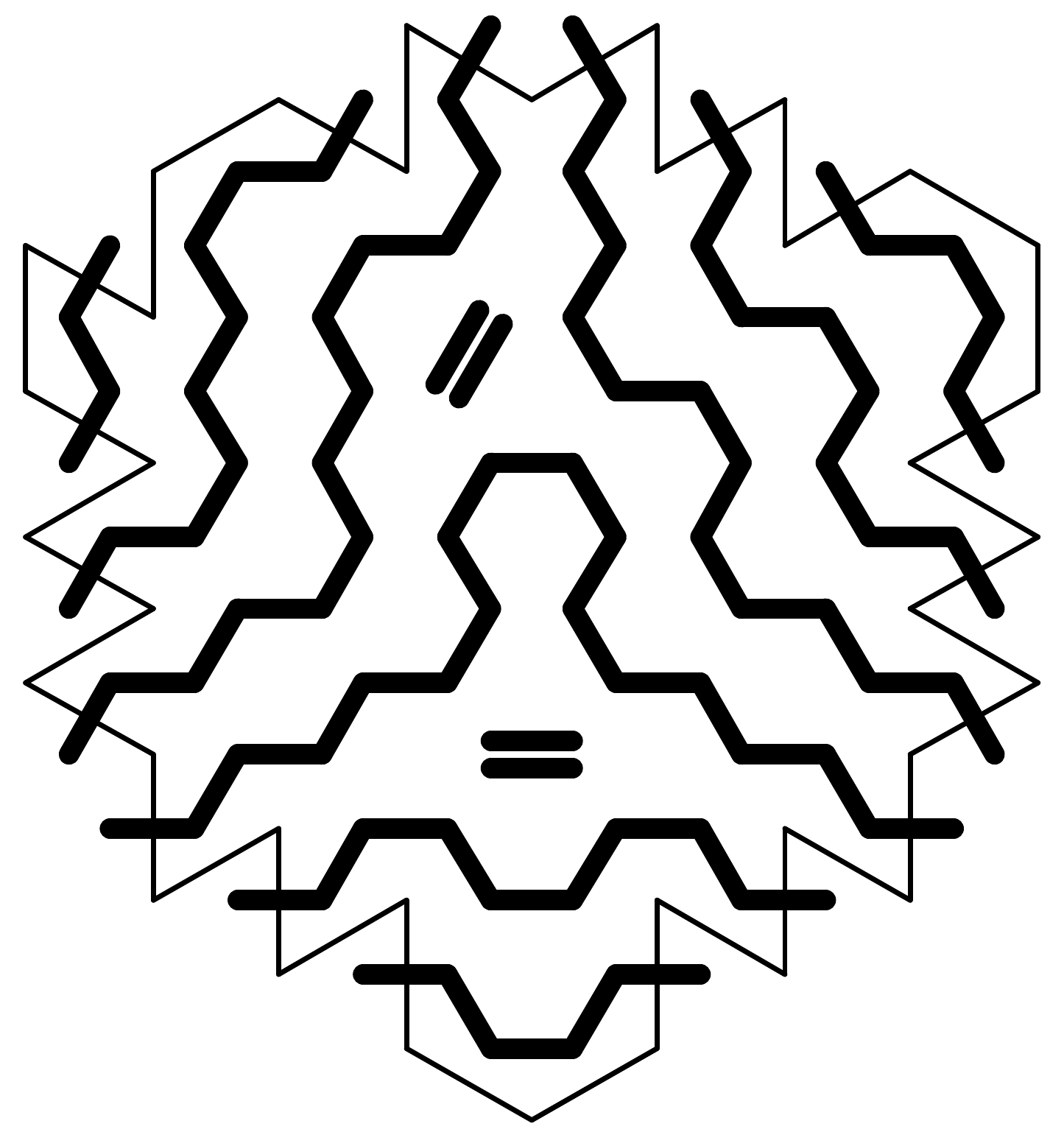} &
\includegraphics[width=\ddwidth]{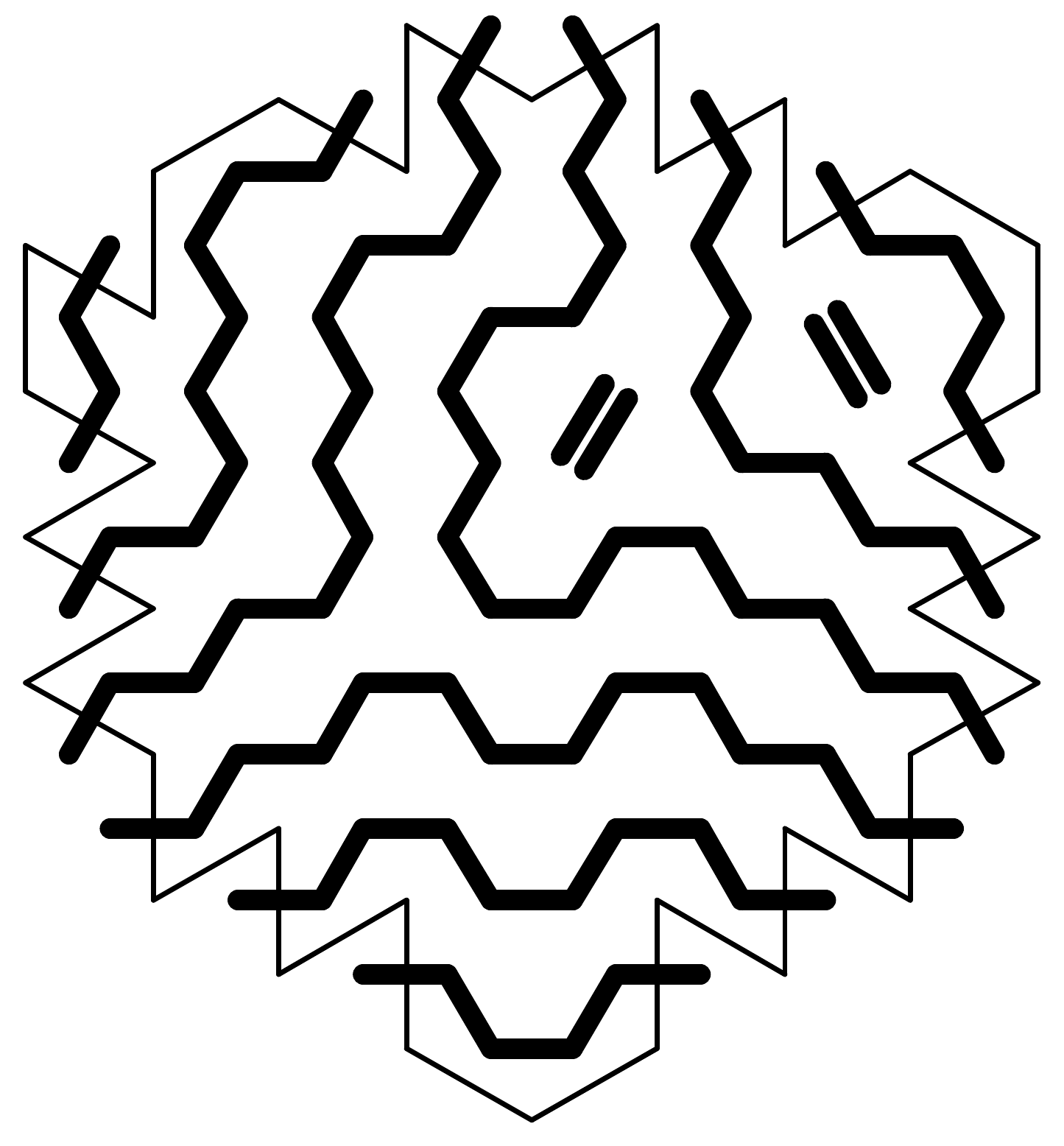} &
\includegraphics[width=\ddwidth]{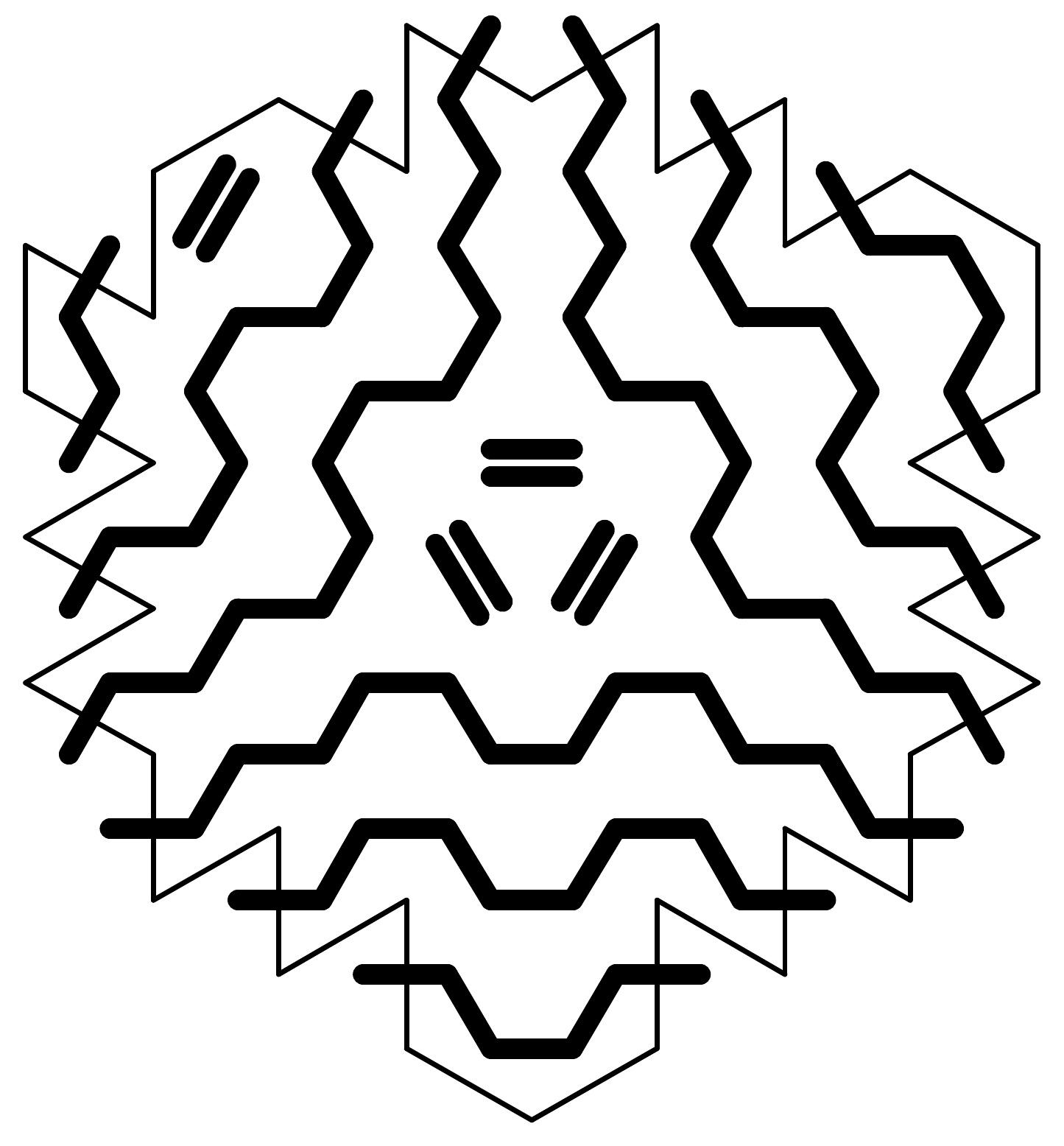} \\
\hline
\end{tabular}
\caption{Double-dimer configurations corresponding to the labelled box configurations in~\cite[Section 5.4]{PT2}.}
\label{fig:pt-doubledimer-example}
\end{figure}

\section{Weights}
\label{sec:weights}

\subsection{Modifying the partition \texorpdfstring{$\mu$}{mu}}
\label{sec:mumodify}

In this section, we collect facts about partitions that we will need to compute the DT and PT weights.

\subsubsection{The diagonal of \texorpdfstring{$\mu$}{mu}}

\begin{remark}
\label{lem:irdmu}
Let $d(\mu)$ denote the length of the diagonal of $\mu$. Then $d(\mu)$ is the largest integer $i$ such that $\mu_i \geq i$. This is immediate from the observation that $\mu_i \geq i$ if and only if $(i, i)$ is a cell in the Young diagram of $\mu$. 
\end{remark}

\begin{remark}
\label{rem:dmuplus1}
It is immediate from Remark~\ref{lem:irdmu} that $\mu_{d(\mu)+ 1} \leq d(\mu)$. For if $\mu_{d(\mu)+ 1} > d(\mu)$, then $\mu_{d(\mu)+ 1} \geq d(\mu) + 1$, contradicting that $d(\mu)$ is the length of the diagonal.
\end{remark}

In many of the computations we will make use of the fact that $d(\mu)$ is the largest integer $i$ with $\mu_i \geq i$. We will sometimes also need to know the largest integer $i$ with $\mu_i \geq i -1$. 

\begin{example}
\begin{itemize}
\item If $\mu = (4, 4, 4, 3, 1)$, then $d(\mu) = 3$ and the largest integer $i$ with $\mu_i \geq i-1$ is $4$, as $\mu_4 = 3 \geq 4-1$ and for $i>4$, $\mu_i\leq\mu_4=3<4\leq i-1$. 
\item If $\mu = (8, 8, 7, 5, 3, 2, 1, 1, 1)$, then $d(\mu) =4$ and $4$ is the largest integer $i$ with $\mu_i \geq i-1$, since $\mu_4 = 5\geq 4-1$ and for $i>4$, $\mu_i\leq\mu_5=3<4\leq i-1$.
\end{itemize}
\end{example}

The preceding example illustrates the following facts. 

\begin{lemma}
\label{rem:idie}
Let $d_{s}(\mu)$ be the largest integer $i$ such that $\mu_i \geq i-1$.
% $i \leq \mu_i + 1$. 
There are two possibilities: either $d_{s}(\mu) = d(\mu)$ or $d_{s}(\mu) = d(\mu) + 1$.  
\end{lemma}
  
\begin{proof}
Since $\mu$ is a partition, $\mu_i-i+1$ is a strictly decreasing sequence, so $d_{s}(\mu)$ is equivalently the unique integer $i$ such that $\mu_i\geq i-1$ and $\mu_{i+1}<i$. If $ \mu_i \geq i$, then  $ \mu_i \geq i-1$, so $d_{s}(\mu) \geq d(\mu)$. And since 
$$d(\mu) + 1 > d(\mu) \geq \mu_{d(\mu) + 1} \geq \mu_{d(\mu)+ 2}=\mu_{d(\mu)+ 1+1},$$
so $d_{s}(\mu) \leq d(\mu) + 1$. 
% It is possible that $d_{s}(\mu) = d(\mu) +1$, for instance if $(\mu)_{d(\mu) + 1} \geq d(\mu) = (d(\mu) + 1) - 1$. But if $\mu_{i_{d} + 2} \geq (d(\mu) + 2) -1$, then $(\mu)_{d(\mu) + 1} \geq d(\mu) + 1$, a contradiction. So $d_{s}(\mu) < d(\mu) + 2$. 
\end{proof}

\begin{lemma}
\label{lem:d'eq}
Let $\mu$ be a partition. Then 
\[d_{s}(\mu) = d(\mu) + 1 \Leftrightarrow \mu_{d(\mu)+1} = d(\mu).
\]
\end{lemma}

\begin{proof}
If $\mu_{d(\mu)+1} = d(\mu) =  (d(\mu) + 1) - 1$, then $d_{s}(\mu) = d(\mu) + 1$. If $d_{s}(\mu) = d(\mu) + 1$, then $\mu_{d(\mu)+1} \geq d(\mu)$. Since $\mu_{d(\mu)+1} < d(\mu) + 1$, we are done. 
\end{proof}

\begin{remark}
\label{rem:d'ineq}
By Lemmas~\ref{rem:idie} and~\ref{lem:d'eq},
\[d_{s}(\mu) = d(\mu) \Leftrightarrow \mu_{d(\mu)+1} < d(\mu).
\]
\end{remark}

\subsubsection{The partitions \texorpdfstring{$\mu^{r}$ and $\mu^{c}$}{mur and muc}}

To compute the weights in DT and PT, we will find it useful to have explicit descriptions of $\mu^{r}$ and $\mu^{c}$, where 
\begin{itemize}
\item $\mu^r$ is the partition associated to the charge $-1$ Maya diagram $S(\mu) \setminus \{\min S^+(\mu)\}$, and 
\item $\mu^c$ is the partition associated to the charge $1$ Maya diagram $S(\mu) \cup \{ \max S^-(\mu)\}$.
\end{itemize}
Additionally, $\mu^{rc}$ denotes the partition associated to the Maya diagram $(S(\mu)\cup \{ \max S^-(\mu)\}) \setminus \{\min S^+(\mu)\}$. Note that none of the partitions $\mu^r$, $\mu^c$, $\mu^{rc}$ are defined if $\mu=\emptyset$, so in what follows, when we refer to any of these partitions, we implicitly assume that $\mu\neq\emptyset$. 

\begin{remark}
\label{rem:MDmurmuc}
We will use the following expressions for the charge 0 Maya diagrams of $\mu^r$ and $\mu^c$. 
\begin{itemize} 
\item $\mu^r$ has charge 0 Maya diagram $S(\mu^r) = \{s + 1: s \in S(\mu) \setminus \{ \min S^{+}(\mu) \} \}$
\item $\mu^c$ has charge 0 Maya diagram $S(\mu^c) = \{s - 1: s \in S(\mu) \cup \{ \max S^{-}(\mu) \} \}$
\end{itemize}
\end{remark}

\begin{example}
Let $\mu = (4, 4, 4, 3, 1)$. Then $S(\mu) = \{ \frac{7}{2}, \frac{5}{2}, \frac{3}{2}, -\frac{1}{2}, -\frac{7}{2}, -\frac{11}{2},- \frac{13}{2}, \ldots \}$.
\begin{itemize}
\item Since $\min S^{+}(\mu) = \frac{3}{2}$, $\mu^{r}$ has charge $-1$ Maya diagram $ \{ \frac{7}{2}, \frac{5}{2}, -\frac{1}{2}, -\frac{7}{2}, -\frac{11}{2}, -\frac{13}{2}, \ldots \}$ and charge 0 Maya diagram $ \{  \frac{9}{2}, \frac{7}{2}, \frac{1}{2}, - \frac{5}{2}  , -\frac{9}{2}, -\frac{11}{2}, \ldots \}$. 
\item Since $\max S^{-}(\mu) = -\frac{3}{2} $, $\mu^{c}$ has charge $1$ Maya diagram $ \{ \frac{7}{2}, \frac{5}{2}, \frac{3}{2},   -\frac{1}{2}, -\frac{3}{2}, -\frac{7}{2}, -\frac{11}{2},- \frac{13}{2}, \ldots \}$ and charge 0 Maya diagram $ \{ \frac{5}{2}, \frac{3}{2}, \frac{1}{2},   -\frac{3}{2}, -\frac{5}{2}, -\frac{9}{2}, -\frac{13}{2}, - \frac{15}{2}, \ldots \}$. 
\end{itemize}
\end{example}

%\helentodo{left off here}

%\begin{lemma}
%\label{lem:irdmu}
% For any partition $\mu$, $i_r  = d(\mu)$. 
% \end{lemma}

%\begin{proof}
%Since $\mu_{i_{r}} \geq i_{r}$, $(i_r, i_r)$ is a cell in the diagram of $\mu$, so $i_r \leq d(\mu)$. 

%Since $(d(\mu), d(\mu))$ is a cell in the Young diagram of $\mu$, $\mu_{d(\mu)} \geq d(\mu)$, so $d(\mu) \leq i_r$. 
%\end{proof}

We begin with an explicit description of $\mu^r$ and a few facts that follow from this description.

\begin{lemma}
\label{lem:mur}
Let $\mu$ be a partition. 
%Let $i_r$ be the largest integer $i$ such that $\mu_i \geq i$. 
Then 
\[
\mu^r_i = \begin{cases} \mu_i + 1 & \text{if }i < d(\mu)  \\
\mu_{i+1} & \text{if }i \geq d(\mu).
\end{cases}
\]
That is, we obtain $\mu^r$ from $\mu$ by removing $\mu_{d(
\mu)}$ and adding 1 to the $j$th part of the partition for all $j < d(\mu)$. 
\end{lemma}

\begin{proof}
For convenience, we write $S := S(\mu)$. By definition, $\mu^r$ is the partition associated to the charge $-1$ Maya diagram $S \setminus \{\min S^+\}$, i.e., $S(\mu^r)=\{s+1: s\in S \setminus \{\min S^+\}\}$. Observe that $\min S^+$ is the least half integer $\mu_t-t+\frac{1}{2}$ such that $\mu_t-t+\frac{1}{2}>0$. Equivalently, it is the least half integer $\mu_t-t+\frac{1}{2}$ such that $\mu_t\geq t$, i.e., $\min S^+=\mu_{d(\mu)}-d(\mu)+\frac{1}{2}$. So, 
\begin{align*}
S \setminus \{\min S^+\}&=\left\{\mu_t-t+\frac{1}{2}: 1\leq t<d(\mu)\right\}\cup\left\{\mu_t-t+\frac{1}{2}: d(\mu)<t\right\} \\
&=\left\{\mu_t-t+\frac{1}{2}: 1\leq t<d(\mu)\right\}\cup\left\{\mu_{t+1}-t-1+\frac{1}{2}: d(\mu)\leq t\right\} 
\end{align*}
and 
\begin{align*}
S(\mu^r)&=\left\{\mu_t+1-t+\frac{1}{2}: 1\leq t<d(\mu)\right\}\cup\left\{\mu_{t+1}-t+\frac{1}{2}: d(\mu)\leq t\right\} \\
&=\left\{\mu^r_t-t+\frac{1}{2}: 1\leq t<d(\mu)\right\}\cup\left\{\mu^r_t-t+\frac{1}{2}: d(\mu)\leq t\right\}. 
\end{align*}
This shows that $\mu^r_t=\mu_t+1$ for $t<d(\mu)$ and $\mu^r_t=\mu_{t+1}$ for $t\geq d(\mu)$. 
\end{proof}

\begin{example}
\label{ex:mur}\leavevmode
\begin{itemize}
\item Let $\mu = (4, 4, 4, 3, 1)$. Then $d(\mu) = 3$ and $\mu_{d(\mu)} = 4$. So $\mu^r = (5, 5, 3, 1)$.
\item Let $\mu = (8, 8, 7, 5, 3, 2, 1, 1, 1)$. Then $d(\mu) = 4$ and $\mu_{d(\mu)} =5$. 
%Then the largest integer $i$ such that $\mu_i \geq i$ is $i = 4$. We remove $\mu_4 = 5$ from the partition to obtain $\tilde{\mu} = (8, 8, 7, 4, 2, 1, 1, 1)$. Then we add 1 to each part $\tilde{\mu}_j$ with $j < i$ to obtain 
So $\mu^r = (9, 9, 8, 3, 2, 1, 1, 1)$. 
\end{itemize}
\end{example}

%Let $i_r$ denote the maximum positive integer such that $\mu_i \geq i$. 

\begin{remark}
\label{rem:constructionofmur}
The following observations are immediate consequences of Lemma~\ref{lem:mur}.
\begin{itemize}
\item $| \mu^r |  = |\mu| +d(\mu) - 1 - \mu_{d(\mu)} \leq  |\mu| + d(\mu)  - 1 - d(\mu) =  |\mu| - 1$
\item $d(\mu)-1\leq d(\mu^r) \leq d(\mu)$, since by the construction of $\mu^r$, if $(i, i)$ is a cell in the Young diagram of $\mu$ and $i < d(\mu)$, it is a cell in the Young diagram of $\mu^r$. 
\item $\mu_{d(\mu) + 1} =\mu_{d(\mu)}^{r}$, and therefore $\mu_{d(\mu) + 1} = d(\mu)$ if and only if $\mu_{d(\mu)}^{r} = d(\mu)$. Also, $\mu_{d(\mu) + 1} <  d(\mu)$ if and only if $\mu_{d(\mu)}^{r} < d(\mu)$. 
\end{itemize}
\end{remark}

\begin{remark}
\label{lem:lengthmur}
$\ell(\mu^r) = \ell(\mu) -1$
\end{remark}

\begin{lemma}
\label{lem:dmur}
For any partition $\mu$, $d(\mu^r) =d(\mu)$ if and only if $\mu_{d(\mu) + 1} = d(\mu)$. 
\end{lemma}

\begin{proof}
First assume that $d(\mu^r) =d(\mu)$. We see that
\[
\mu_{d(\mu) + 1} =\mu^r_{d(\mu)}=\mu^r_{d(\mu^r)} \geq d(\mu^r) = d(\mu) \]
and since  
%$\mu_{d(\mu)} \geq d(\mu)$, 
$\mu_{d(\mu) + 1} \leq d(\mu)$, $\mu_{d(\mu) + 1} =  d(\mu)$.

Now suppose $\mu_{d(\mu) + 1} = d(\mu)$. This means that $\mu^r_{d(\mu)} = d(\mu)$, so $(d(\mu), d(\mu))$ is a cell in the Young diagram of $\mu_r$, so $d(\mu^r) \geq d(\mu)$. By Remark~\ref{rem:constructionofmur}, $d(\mu^r) = d(\mu)$. 
\end{proof}

\begin{lemma}
\label{lem:dmurneqdmu}
For any partition $\mu$, $d(\mu^r) =d(\mu) - 1$ if and only if $\mu_{d(\mu) + 1} < d(\mu)$. 
\end{lemma}

\begin{proof}
%We note that $(d(\mu), d(\mu))$ is a cell in the Young diagram of $\mu^r$ if and only if $\mu^r_{d(\mu)} \geq d(\mu)$. 

If $d(\mu^r) =d(\mu) - 1$, $(d(\mu), d(\mu))$ is not a cell in the Young diagram of $\mu^r$, so $\mu^r_{d(\mu)} < d(\mu)$. Since $\mu^r_{d(\mu)} = \mu_{d(\mu) + 1}$ by Remark~\ref{rem:constructionofmur}, this shows that $\mu_{d(\mu) + 1} < d(\mu)$. 

If $\mu_{d(\mu) + 1} < d(\mu)$, $(d(\mu), d(\mu))$ is not a cell in the Young diagram of $\mu^r$. However, $(d(\mu) -1, d(\mu)-1)$ is a cell in the Young diagram of $\mu^r$, by construction. % (see Remark~\ref{rem:constructionofmur}. 
So $d(\mu^r) =d(\mu) - 1$. 
\end{proof}

\begin{example}
Continuing Example~\ref{ex:mur}, when $\mu = (4, 4, 4, 3, 1)$, $d(\mu) = 3$, and $\mu_{d(\mu) + 1}  = d(\mu)$. As expected, $\mu^{r} = (5, 5, 3, 1)$ has $d(\mu^r) = 3$. 

When $\mu = (8, 8, 7, 5, 3, 2, 1, 1, 1)$, $d(\mu) = 4$ and $\mu_{d(\mu) + 1} < d(\mu)$. 
As expected, $\mu^r = (9, 9, 8, 3, 2, 1, 1, 1)$ has $d(\mu^r) = 3$. 
\end{example}

\begin{lemma}
\label{lem:indexset}
If there exists a positive integer $i$ such that $\mu^r_i > i+1$, then the largest such integer is $d(\mu) - 1$. In other words, the set of positive integers $i$ satisfying $\mu^r_i > i+1$ is equal to the set of positive integers $i$ satisfying $i\leq d(\mu)-1$. 
\end{lemma}
  
\begin{proof}
Since $\mu^r$ is a partition, the sequence $\mu^r_i-i-1$ is strictly decreasing, so it suffices to show that $d(\mu)-1>0$, $\mu^r_{d(\mu)-1}>d(\mu)$ and $\mu^r_{d(\mu)}\leq d(\mu)+1$. Assuming there exists a positive integer $i$ such that $\mu^r_i > i+1$, we must have $\mu^r_1>1+1=2$. Thus, by Lemma~\ref{lem:mur}, if $d(\mu)=1$, then $\mu_2>2$, so $d(\mu)\geq 2$. By contradiction, $d(\mu)>1$. Also, by Lemma~\ref{lem:mur}, $\mu^r_{d(\mu)-1} = \mu_{d(\mu)-1} + 1 \geq d(\mu) + 1 > d(\mu)$.  And $\mu^r_{d(\mu)} = \mu_{d(\mu) + 1} \leq d(\mu) < d(\mu)+1$. 
\end{proof}

Next we will give an explicit description of $\mu^c$. 

\begin{lemma}
\label{lem:muc}
Let $\mu$ be a partition. Let $i_d$ be the largest integer $i$ with $\mu_i \geq d(\mu)$. Then
\[ \mu_i^c = 
\begin{cases}
\mu_i -1 &\text{if } i \leq i_d \\
d(\mu) -1 &\text{if } i = i_d+1 \\
\mu_{i-1} &\text{if } i > i_d+1.
\end{cases} \]
%\begin{itemize}
%\item If 
%there is an integer $i$ such that 
%$(\mu)_{d(\mu)} = d(\mu)$,
%add a part of size $d(\mu)$ to $\mu$ to obtain $\tilde{\mu}$. Then $\mu^c$ is the partition obtained from $\tilde{\mu}$ by subtracting 1 from each part $(\mu)_j$ such that $(\mu)_j \geq d(\mu)$.
%\item Otherwise, let $i$ be the smallest integer such that $(\mu)_{i} < i$. 
That is, to construct $\mu^c$ we first add a part of size $d(\mu)-1$ to $\mu$ to obtain $\tilde{\mu}$. Then $\mu^c$ is the partition obtained from $ \tilde{\mu}$ by subtracting 1 from each part $\tilde{\mu}_j$ such that $\mu_j \geq d(\mu)$.
%Add a part of size $i-1$ to $\mu_2$ to obtain $\tilde{\mu_2}$. Then $\mu_2^g$ is the partition obtained from $ \tilde{\mu_2}$ by subtracting 1 from each part $(\mu_2)_j$ such that $(\mu_2)_j \geq i-1$.

%Otherwise, $S$ does not contain $-\frac{1}{2}$ but it does contain $\frac{1}{2}$. In this case, let $i$ be the smallest integer such that $(\mu_2)_i < i -1$. Add a part of size $(\mu_2)_{i-1}$ to the partition to obtain $\tilde{\mu_2}$ and then subtract $1$ from each part of the partition $(\mu_2)_j$ with $(\mu_2)_j \geq i-1$.
%then let $i$ be the integer such that $(\mu_{2})_{i} = i-1$. 
%Let $i$ be the largest integer such that $(\mu_2)_i \geq i-1$. Let $\tilde{\mu_2}$ be the partition obtained by adding a part of size $(\mu_2)_i$. Then $\mu_2^g$ is the partition obtained from $ \tilde{\mu_2}$ by subtracting 1 from each part $(\mu_2)_j$ such that $(\mu_2)_j \geq i-1$.
\end{lemma}

\begin{proof}
Let $S$ be the Maya diagram of $\mu$. By definition, $\mu^c$ is the partition associated to the charge $1$ Maya diagram $S \cup \{\max S^-\}$, i.e., $S(\mu^c)=\{s-1: s\in S \cup \{\max S^-\}\}$. Note that $\max S^-$ is the greatest half integer $h<0$ such that $h\neq\mu_t-t+\frac{1}{2}$ for all $t\geq 1$. We claim that \[\max S^-=d(\mu)-i_d-\frac{1}{2}.\] 

In the case that $i_d=d(\mu)$, suppose $-\frac{1}{2}=\mu_t-t+\frac{1}{2}$ for some $t\geq 1$. Then $\mu_t=t-1$, so $t>d(\mu)=i_d$, which means that $d(\mu)>\mu_t=t-1>i_d-1=d(\mu)-1$. This is a contradiction. Therefore, $\max S^-=-\frac{1}{2}=d(\mu)-i_d-\frac{1}{2}$, as claimed. 

Otherwise, $i_d>d(\mu)$. In this case, for all $d(\mu)<t\leq i_d$, we have $d(\mu)\leq\mu_{i_d}\leq\mu_t\leq\mu_{d(\mu)+1}<d(\mu)+1$, so $\mu_t=d(\mu)$. Then $\mu_t-t+\frac{1}{2}=d(\mu)-t+\frac{1}{2}$, so $-\frac{1}{2}, -\frac{3}{2}, \ldots, d(\mu)-i_d+\frac{1}{2}\in S$ and we deduce that $\max S^-\leq d(\mu)-i_d-\frac{1}{2}<0$. On the other hand, $\mu_{i_d+1}-(i_d+1)+\frac{1}{2}=\mu_{i_d+1}-i_d-\frac{1}{2}<d(\mu)-i_d-\frac{1}{2}$. Since the sequence $\mu_t-t+\frac{1}{2}$ is strictly decreasing, this implies that $d(\mu)-i_d-\frac{1}{2}\not\in S$, so $\max S^-\geq d(\mu)-i_d-\frac{1}{2}$. Then $\max S^-=d(\mu)-i_d-\frac{1}{2}$, proving the claim. 

So, 
\begin{eqnarray*}
S \cup \{\max S^-\}&=&\left\{\mu_t-t+\frac{1}{2}: 1\leq t\leq i_d\right\}\cup\left\{d(\mu)-(i_d+1)+\frac{1}{2}\right\} \\
&&{}\cup{}\left\{\mu_t-t+\frac{1}{2}: i_d<t\right\} \\
&=&\left\{\mu_t-t+\frac{1}{2}: 1\leq t\leq i_d\right\}\cup\left\{d(\mu)-(i_d+1)+\frac{1}{2}\right\} \\
&&{}\cup{}\left\{\mu_{t-1}-t+\frac{3}{2}: i_d+1<t\right\} 
\end{eqnarray*}
and 
\begin{eqnarray*}
S(\mu^c)&=&\left\{\mu_t-1-t+\frac{1}{2}: 1\leq t\leq i_d\right\}\cup\left\{d(\mu)-1-(i_d+1)+\frac{1}{2}\right\} \\
&&{}\cup{}\left\{\mu_{t-1}-t+\frac{1}{2}: i_d+1<t\right\} \\
&=&\left\{\mu^c_t-t+\frac{1}{2}: 1\leq t\leq i_d\right\}\cup\left\{\mu^c_{i_d+1}-(i_d+1)+\frac{1}{2}\right\} \\
&&{}\cup{}\left\{\mu^c_t-t+\frac{1}{2}: i_d+1<t\right\}. 
\end{eqnarray*}
This shows that $\mu^c_t=\mu_t-1$ for $t\leq i_d$, $\mu^c_{i_d+1}=d(\mu)-1$, and $\mu^c_t=\mu_{t-1}$ for $t>i_d+1$. 
\end{proof}

%\begin{lemma}

%begin{itemize}
%\item If $\mu_{d(\mu)} = d(\mu)$, let $i_d$ be the largest integer $i$ with $\mu_i = d(\mu)$. Then
%\[ \mu_i^c = 
%\begin{cases}
%\mu_i -1 &\mbox{ if } i \leq i_d \\
%d(\mu) -1 & \mbox{ if } i = i_d+1 \\
%\mu_{i-1} &\mbox{ if } i > i_d+1
%\end{cases} \]
%\item If $\mu_{d(\mu)} > d(\mu)$, then
%\[ \mu_i^c = 
%\begin{cases}
%\mu_i -1 &\mbox{ if } i \leq d(\mu) \\
%d(\mu) -1 & \mbox{ if } i = d(\mu) +1 \\
%\mu_{i-1} &\mbox{ if } i > d(\mu) +1
%\end{cases} \]
%\end{itemize}
%\end{lemma}

\begin{example}
\begin{itemize}
\item If $\mu = (1)$, then $d(\mu) =1$ and $i_d = 1$, so 
% When $\mu = (1), S = \{\frac{1}{2}, -\frac{3}{2}, -\frac{5}{2}, \ldots\}$. 
%Since $\frac{1}{2}$ is in $S$ and $(\mu)_1 =1$, we add a part of size $1$ to the partition to obtain $\tilde{\mu} = (1, 1)$, and then subtract 1 from each part of the partition to get 
$\mu^{c} = \emptyset$. 
%\item $\mu = (1, 1)$, so $S = \{ \frac{1}{2}, - \frac{1}{2}, - \frac{5}{2}, - \frac{7}{2}, \ldots \}$. Again, for $i = 1$, $(\mu)_i = i$. When we add $-\frac{3}{2}$ to $S$ we get the partition $\tilde{\mu} = (1, 1, 1)$. Then we subtract one from each part of $\tilde{\mu}$ to get $\mu^{g} = \emptyset$. 
\item If $\mu = (2)$, $d(\mu) =1$ and $i_d = 1$, so 
%so $S = \{ \frac{3}{2}, - \frac{3}{2}, - \frac{5}{2}, \ldots \}$. In this case $i = 2$ is the smallest integer with $(\mu)_i < i $. We add a part of size $1$ to obtain $\tilde{\mu} = (2, 1)$ and then subtract 1 from each part to obtain 
$\mu^{c} = (1)$. 
\item If $\mu = (4, 4, 3, 2)$, $d(\mu) = 3$ and $i_d = 3$, so 
%When $i = 3$, $\mu_{i} =i$. So we add a part of size $3$ to obtain $\tilde{\mu} = (4, 4, 3,3, 2)$. Then we subtract 1 from each part of $\tilde{\mu}$ that is greater than or equal to 3 to obtain 
$\mu^{c} = (3, 3, 2, 2, 2)$.
\item If $\mu = (4, 4, 4, 3, 1)$, $d(\mu) = 3$ and $i_d = 4$. 
% There is no integer $i$ such that $(\mu)_{i} = i$ but when $i = 4$, $(\mu)_{i} = i-1$. When we add a part of size $i-1 = 3$ we get the partition  $\tilde{\mu} = (4, 4, 4, 3,3, 1)$. Then we subtract one from each part of $\mu$ that is greater than or equal to $3$ to get 
We get $\mu^{c} = (3,3,3, 2,2, 1)$. 
\item If $\mu = (7, 7, 6, 1)$, $d(\mu) =3$ and $i_d = 3$, 
 %The smallest integer with $(\mu)_{i} < i$ is $i = 4$. So we add a part of size 3 to obtain 
%$\tilde{\mu} = (7, 7, 6, 3, 1)$. Then we subtract 1 from each part of $\tilde{\mu}$ that is greater than or equal to 3 to obtain 
so $\mu^{c} = (6, 6, 5, 2, 1)$.
%\item $\mu = (6, 6, 6, 5, 2, 1)$. Then $d(\mu) = 4$ and $i_d = 4$
% When $i = 5$, $(\mu)_{i} < i$. So we add a part of size $i-1 = 4$ to obtain $\tilde{\mu} = (6, 6, 6, 5, 4, 2, 1)$. Then we subtract 1 from each part of $\tilde{\mu}$ that is greater than or equal to 4 to obtain 
%We find that $\mu^{c} = (5, 5,5, 4, 3, 2, 1)$.
\end{itemize} 
\end{example}

\begin{remark}
\label{lem:lengthmuc}\leavevmode
\begin{itemize}
\item If $d(\mu) > 1$, then $\ell(\mu^{c}) = \ell(\mu) +1$.
\item If $d(\mu) =1$ and $\mu_1 > 1$, $\ell(\mu^c) =1$.
\item If $d(\mu) = 1$ and $\mu_1 = 1$,  $\ell(\mu^c) =0$.
\end{itemize}
\end{remark}

\begin{remark}
\label{rem:idlambdat}
Let $i_d$ be the largest integer with $\mu_i \geq d(\mu)$. Then $i_d = \mu'_{d(\mu)}$. 
\end{remark}

\begin{remark}
\label{rem:mucdmuplus1}
By Lemma~\ref{lem:muc}, 
%It is an immediate consequence of this lemma that
$\mu^c_{d(\mu) + 1} = d(\mu) - 1$. Because if $i_d = d(\mu)$, then $\mu^c_{d(\mu) + 1}= \mu^c_{i_d+ 1}  = d(\mu) - 1$. And if $i_d > d(\mu)$, then $\mu_{d(\mu) + 1} = d(\mu)$, so $\mu^c_{d(\mu) + 1} = d(\mu)-1$. 
\end{remark}

\begin{remark}
\label{rem:mucdiag}
We note that $\mu^{c}_{d(\mu)} = d(\mu) - 1$ if and only if $\mu_{d(\mu)} = d(\mu)$. Also, $d(\mu^c) = d(\mu) $ if and only if $\mu_{d(\mu)} > d(\mu)$, and $d(\mu^c)=d(\mu)-1$ if and only if $\mu_{d(\mu)}=d(\mu)$. 
\end{remark}

%\helentodo{Move this lemma to after remark 5.1.20}
\begin{lemma}
\label{lem:dprime_c}
Let $d_{s}(\mu^c)$ be the maximum positive integer $i$ such that $\mu^{c}_i\geq i-1$. Then 
%$d(\mu) = d((\mu)^{c})$ if and only if 
$d_{s}(\mu^c) = d(\mu)$. In other words, the set of positive integers $i$ satisfying $\mu^{c}_i\geq i-1$ is equal to the set of positive integers $i$ satisfying $i\leq d(\mu)$. 
\end{lemma}

\begin{proof}
Since $\mu^c$ is a partition, the sequence $\mu^c_i-i+1$ is strictly decreasing, so it suffices to show that $\mu^c_{d(\mu)}\geq d(\mu)-1$ and $\mu^c_{d(\mu)+1}<d(\mu)$. By Lemma~\ref{lem:muc}, $\mu^c_{d(\mu)}=\mu_{d(\mu)}-1 \geq d(\mu) -1$ and $\mu^c_{d(\mu) + 1} = d(\mu) - 1<d(\mu)$. 
\end{proof}

We will also take advantage of the following relationship between $\mu^c$ and $\mu^r$. 

\begin{lemma}
\label{lem:mucprime}
Let $\mu$ be a partition. 
\begin{enumerate}
\item[(1)] $(\mu^c)' = (\mu')^r$, and 
\item[(2)] $(\mu^r)' = (\mu')^c$.
\end{enumerate}
\end{lemma}

The expression (2) follows from (1) by substituting $\mu'$ for $\mu$. Before we proceed to the proof, we make a useful observation. 

\begin{lemma}
\label{lem:muc'}
Let $\mu$ be a partition. 
\begin{itemize}
\item[(1)] $S^{+}(\mu') = -S^{-}(\mu)$, and 
\item[(2)] $S^{-}(\mu') = -S^{+}(\mu)$.
\end{itemize}
\end{lemma}

\begin{proof}
Let $L_{\mu}$ be the {\em contour of $\mu$}, which is obtained from the Maya diagram of $\mu$ by placing a line segment of slope $-1$ where there is a hole and a line segment of slope $1$ where there is a bead (this is standard, see for instance \cite{RZ}). Then the claim follows from the observation that we obtain $\mu'$ from $\mu$ by reflecting $L_{\mu}$ across the line $x = 0$. 
\end{proof}

\begin{example}
Let $\mu = (6, 6, 5, 5, 5, 3, 1)$. Then 
\begin{itemize}
\item $S(\mu) = \{ \frac{11}{2}, \frac{9}{2}, \frac{5}{2}, \frac{3}{2}, \frac{1}{2}, -\frac{5}{2}, -\frac{11}{2}, -\frac{15}{2}, -\frac{17}{2}, \ldots \}$, 
\item $S^{+}(\mu) = \{ \frac{11}{2}, \frac{9}{2}, \frac{5}{2},  \frac{3}{2}, \frac{1}{2} \}$, and
\item $S^{-}(\mu) = \{ - \frac{1}{2}, - \frac{3}{2}, - \frac{7}{2}, - \frac{9}{2}, - \frac{13}{2} \}$. 
\end{itemize}
We see that $\max S^{-}(\mu) = -  \frac{1}{2}$. Noting that $\mu^c = (5, 5, 4, 4, 4, 4, 3, 1)$, 
%$\mu^c = (8, 7, 7, 6, 2)$, 
we see that 
\begin{eqnarray*}
S(\mu^c) &=& \left\{ \frac{9}{2}, \frac{7}{2}, \frac{3}{2}, \frac{1}{2}, -\frac{1}{2}, -\frac{3}{2}, -\frac{7}{2}, -\frac{13}{2},-\frac{17}{2},-\frac{19}{2}, \ldots \right\} \\
&=& \{s -1 : s \in S(\mu)  \cup \{\max S^{-}(\mu)\} \}.
\end{eqnarray*}
So 
\[ S^+(\mu^c) = \left\{ \frac{9}{2}, \frac{7}{2}, \frac{3}{2}, \frac{1}{2} \right\} = \left\{s-1: s \in S^+(\mu) \setminus \left\{ \frac{1}{2} \right\} \right\},\text{ and}
\]
\[
S^{-}(\mu^c) = \left\{ - \frac{5}{2}, - \frac{9}{2}, - \frac{11}{2}, - \frac{15}{2} \right\} =  \{s -1 : s \in S^-(\mu)  \setminus \{\max S^{-}(\mu)\} \}. \]
We next note that $(\mu^c)' = (8, 7, 7, 6, 2)$, 
\[S^{+}((\mu^c)') = \left\{ \frac{15}{2}, \frac{11}{2}, \frac{9}{2}, \frac{5}{2} \right\} = -S^{-}(\mu^c) = \{ - s + 1: s \in S^-(\mu)  \setminus \{\max S^{-}(\mu)\} \},\text{ and}
\]
\[S^{-}((\mu^c)') = \left\{-\frac{1}{2}, -\frac{3}{2}, -\frac{7}{2}, -\frac{9}{2} \right\} = -S^{+}(\mu^c) = \left\{ - s + 1: s \in S^+(\mu) \setminus \left\{ \frac{1}{2} \right\} \right\}.
\]
%\[
%S(\mu) \cup \max S^{-}(\mu) = \left\{ \frac{11}{2}, \frac{9}{2}, \frac{5}{2},  \frac{3}{2}, 
%\frac{1}{2}, -\frac{1}{2}, -\frac{5}{2}, -\frac{11}{2}, -\frac{15}{2}, -\frac{17}{2}, \ldots \right\}.
%\]
Since $\mu = (6, 6, 5, 5, 5, 3, 1)$, $\mu' = (7, 6, 6, 5, 5, 2)$. So
\[S^{+}(\mu') = \left\{\frac{13}{2}, \frac{9}{2}, \frac{7}{2}, \frac{3}{2}, \frac{1}{2} \right\} = - S^{-}(\mu)\text{ and}\]
\[S^{-}(\mu') = \left\{ - \frac{1}{2}, - \frac{3}{2}, - \frac{5}{2},  - \frac{9}{2},-  \frac{11}{2} \right\} = - S^{+}(\mu).\]
Then 
\begin{eqnarray*}
S^{+}((\mu')^r) = \left\{ \frac{15}{2}, \frac{11}{2}, \frac{9}{2}, \frac{5}{2}  \right\} &=& \{s +1: s \in S^{+}(\mu') \setminus \{\min S^{+}(\mu')\} \}  \\
&=& \{s +1: s \in - S^{-}(\mu) \setminus \{\min (- S^{-}(\mu))\} \} \\
& = & \{-s +1: s \in S^{-}(\mu) \setminus \{\max S^{-}(\mu)\} \}
\end{eqnarray*}
and 
\begin{eqnarray*}
S^{-}((\mu')^r) = \left\{ -\frac{1}{2}, -\frac{3}{2}, -\frac{7}{2}, -\frac{9}{2}  \right\} &=& \left\{s +1: s \in S^{-}(\mu') \setminus \left\{- \frac{1}{2} \right\} \right\} \\
&=& \left\{s +1: s \in - S^{+}(\mu) \setminus \left\{- \frac{1}{2} \right\} \right\} \\
& = & \left\{-s +1: s \in S^{+}(\mu)  \setminus \left \{ \frac{1}{2} \right\} \right\}.
\end{eqnarray*}
\end{example}

\begin{proof}[Proof of Lemma~\ref{lem:mucprime}]
We break into cases based on whether $\frac{1}{2} \in S$. 

First suppose $\frac{1}{2} \in S$. By Remark~\ref{rem:MDmurmuc}, 
\[S^+(\mu^c) = \left\{s-1: s \in  S^+(\mu) \setminus \left\{ \frac{1}{2} \right\} \right\},\text{ and}
\]
\[
S^{-}(\mu^c) =  \{s -1 : s \in S^-(\mu)  \setminus \{\max S^{-}(\mu)\} \}.
\]
Note that in the expression for $S^{-}(\mu^c)$ we used the fact that $\frac{1}{2} \in S$, so $- \frac{1}{2}$ is in $S(\mu^c)=\{s -1 : s \in S(\mu) \cup \{\max S^{-}(\mu)\} \}$ and therefore not in $S^{-}(\mu^c)$. 
Now we see that 
\begin{eqnarray*}
S^{+}((\mu^c)') = -S^{-}(\mu^c) 
&=& \{ - s + 1: s \in S^-(\mu) \setminus \{\max S^{-}(\mu)\} \} \\
& = & \{s +1: s \in - S^{-}(\mu) \setminus \{\min (- S^{-}(\mu))\} \} \\
& = & \{s +1: s \in S^{+}(\mu') \setminus \{\min S^{+}(\mu')\} \} = S^{+}((\mu')^r).
\end{eqnarray*}
Similarly, 
\begin{eqnarray*}
S^{-}((\mu^c)') = -S^{+}(\mu^c) 
&=& \left\{-s +1: s \in S^{+}(\mu) \setminus \left \{ \frac{1}{2} \right\} \right\} \\
&=&\left\{s +1: s \in - S^{+}(\mu) \setminus \left\{- \frac{1}{2} \right\}  \right\} \\
&=& \left\{s +1: s \in S^{-}(\mu') \setminus \left\{- \frac{1}{2} \right\} \right\} = S^{-}((\mu')^r). 
\end{eqnarray*}
%Where we remark that by the assumption that $\frac{1}{2} \in S(\mu)$, $-\frac{1}{2} \in S^{-}(\mu')$. 

Next we assume $\frac{1}{2} \notin S$. As in the first case, we start by noting that 
\[S^+(\mu^c) = \left\{s-1: s \in  S^+(\mu) \right\},\text{ and}
\]
\[
S^{-}(\mu^c) = \{s -1 : s \in S^-(\mu) \setminus \{\max S^{-}(\mu)\} \} \cup \left\{ - \frac{1}{2} \right\}.
\]
Note that in the expression for $S^{-}(\mu^c)$ we used the fact that $\frac{1}{2} \notin S$, so $- \frac{1}{2}$ is not in $S(\mu^c)=\{s -1 : s \in S(\mu) \cup \{\max S^{-}(\mu)\} \}$ and therefore is in $S^{-}(\mu^c)$. As in the first case, we proceed by observing that 
\begin{eqnarray*}
S^{+}((\mu^c)') = -S^{-}(\mu^c) 
&=& \{ - s + 1: s \in S^-(\mu) \setminus \{\max S^{-}(\mu)\} \} \cup \left\{ \frac{1}{2} \right\} \\
& = & \{s +1: s \in - S^{-}(\mu) \setminus \{\min (- S^{-}(\mu))\} \} \cup \left\{ \frac{1}{2} \right\} \\
& = & \{s +1: s \in S^{+}(\mu') \setminus \{\min S^{+}(\mu')\} \} \cup \left\{ \frac{1}{2} \right\} = S^{+}((\mu')^r).
\end{eqnarray*}
Similarly, 
\begin{eqnarray*}
S^{-}((\mu^c)') = -S^{+}(\mu^c) 
&=& \{-s +1: s \in S^{+}(\mu) \} \\
&=& \{s +1: s \in - S^{+}(\mu) \} = \{s +1: s \in S^{-}(\mu') \} = S^{-}((\mu')^r). 
\end{eqnarray*}
\end{proof}
 
\begin{remark}
\label{rem:sizeofmuc}
By Remark~\ref{rem:constructionofmur}, $|\mu^r| = |\mu| - \mu_{d(\mu)} + d(\mu) -1$. By Lemma~\ref{lem:mucprime}, 
\[
|\mu^c| = |( (\mu')^r)'| =| (\mu')^r| = |\mu'| - \mu'_{d(\mu')} + d(\mu') - 1 = |\mu| - \mu'_{d(\mu)} + d(\mu) - 1.\]
\end{remark}

\subsubsection{The partition \texorpdfstring{$\mu^{rc}$}{murc}}

\begin{remark}
\label{lem:murc}
Let $\mu$ be a partition. 
%Let $i_r$ be the largest integer $i$ such that the cell $(i, i)$ is in the Young diagram of $\mu$. 
Then $\mu^{rc}$ is the partition obtained by removing the hook of $(d(\mu), d(\mu))$ from $\mu$. 
\end{remark}

\begin{lemma}
\label{lem:musize}
\begin{equation*}
|\mu^r| - |\mu | + |\mu^c| - |\mu^{rc} |  = -1
\end{equation*}
\end{lemma}

\begin{proof}
By Remark~\ref{rem:sizeofmuc}, 
\[ |\mu^c| - |\mu^{rc} | = 
|\mu| - \mu'_{d(\mu)} + d(\mu) - 1 - |\mu| + h_{\mu}(d(\mu), d(\mu))
= \mu_{d(\mu)} - d(\mu), 
\]
where the last equality follows from the fact that
\[ \mu_{d(\mu)} +  \mu'_{d(\mu)}  -1 = h_{\mu}(d(\mu), d(\mu)) + 2(d(\mu) - 1). \]

Combining this with Remark~\ref{rem:constructionofmur}, we have
\begin{equation*}
|\mu^r| - |\mu | + |\mu^c| - |\mu^{rc} |  = -1.
\end{equation*}
\end{proof}

\begin{remark}
Since the hook of $(d(\mu), d(\mu))$ in $\mu$ is the same as the hook of $(d(\mu'), d(\mu'))$ in $\mu'$, $(\mu')^{rc} = (\mu^{rc})'$. 
\end{remark}

\begin{remark}
Let $i_d$ be the largest integer $i$ with $\mu_i \geq d(\mu)$. Then it follows from Remark~\ref{lem:murc} that 
\[ \mu_i^{rc} = 
\begin{cases}
\mu_i &\text{if } i < d(\mu) \\
d(\mu) -1 & \text{if } d(\mu) \leq i \leq i_d \\
\mu_{i} & \text{if } i > i_d.
\end{cases} \]
\end{remark}

\begin{remark}
\label{rem:dmurc}
It is immediate from Remark~\ref{lem:murc} that $d(\mu^{rc}) = d(\mu) - 1$. Therefore by Remark~\ref{rem:mucdiag}, $d(\mu^{rc}) = d(\mu^c)$ if and only if $\mu_{d(\mu)} = d(\mu)$ and $d(\mu^{rc}) = d(\mu^c) -1$ if and only if $\mu_{d(\mu)} > d(\mu)$. 
\end{remark}

\begin{lemma}
\label{lem:mucandmurc}
\[ \mu_i^{rc} = 
\begin{cases}
\mu^c_i + 1 & \text{if } i \leq d(\mu^{rc}) \\
\mu_{i+1}^{c}& \text{if } i > d(\mu^{rc})
\end{cases} \]
\end{lemma}

\begin{proof}
If $i \leq d(\mu^{rc})= d(\mu) - 1$, then $\mu_i^{rc} = \mu_i$ and $\mu_i^c = \mu_i -1$, since $i \leq i_d$. 

If $i > d(\mu^{rc})= d(\mu) - 1$, then we consider two cases. If $d(\mu) \leq i \leq i_d$, then $\mu_i^{rc}  = d(\mu) -1 = \mu_{i+1}^{c}$. If $i > i_d$, then $i+1 > i_d +1$, so $\mu_i^{rc} = \mu_i = \mu_{i+1}^c$. 
\end{proof}

\begin{lemma}
\label{lem:lengthmurc}
\begin{itemize}
\item $d(\mu) > 1$ if and only if $\ell(\mu^{rc}) = \ell(\mu)$, and 
\item $d(\mu) = 1$ if and only if $\ell(\mu^{rc}) = 0$. 
\end{itemize}
\end{lemma}

\begin{proof}
This is immediate by the construction of $\mu^{rc}$ from $\mu$. 
\end{proof}

\begin{corollary}
\label{cor:lengthmurclengthmuc}
If $d(\mu) > 1$ or $d(\mu) =1$ and $\mu_1 > 1$, then $\ell(\mu^{rc}) = \ell(\mu^{c}) -1$.
If $d(\mu) = 1$ and $\mu_1 = 1$, $\ell(\mu^{rc}) = \ell(\mu^c) = 0$.
\end{corollary}

\subsection{DT weights}
%\label{sec:DTcond}

In this section we compute the constants $A$, $B$, and $C$ from equation (\ref{eqn:intermediateDTcond}) in Section~\ref{sec:DTcond}. To that end, in Section~\ref{sec:minconfigs} we compute the weights of the minimal dimer configurations of the graphs 
\begin{align*}
G &= H(N; \mu_1^{rc},\mu_2^{rc},\mu_3), &
G-\{a, b, c, d\} &= H(N; \mu_1,\mu_2,\mu_3), \\
G-\{a, b\} &= H(N; \mu_1, \mu_2^{rc}, \mu_3), & G-\{c, d\} &= H(N; \mu_1^{rc}, \mu_2, \mu_3), \\
G-\{a, d\} &= H(N; \mu_1^{rc},\mu_2^{rc},\mu_3)-\{a, d\},
& G-\{b, c\} &= H(N; \mu_1^{rc},\mu_2^{rc},\mu_3)-\{b, c\}. 
\end{align*}
As in previous sections, we assume $N\geq M$. The remaining work is to compute $C - A$; this is done in Section~\ref{sec:DTalg}.

\subsubsection{Weight of minimal configuration}
\label{sec:minconfigs}

We weight the edges of $H(N)$ so that the weight of the horizontal edges on a diagonal is $q$ times the weight of the horizontal edges on the previous diagonal, moving from right to left (see Definition~\ref{def:kuoweighting}). Recall the correspondence between dimer configurations of $H(N)$ and plane partitions described in Section~\ref{sec:DTtheoryAndDimers}. With the chosen edge weights, when a box is added to a plane partition, the weight of the corresponding dimer configuration increases by a factor of $q$. So, the minimal dimer configuration of $H(N)$ corresponds to the empty plane partition and has weight $q^{N^2(N-1)/2}$. This expression is simply the product of the weights of the $N^2$ horizontal dimers that make up the ``floor'' of the empty plane partition. 

Now observe that the minimal dimer configuration of $H(N; \mu_1, \mu_2, \mu_3)$ differs from the dimer configuration corresponding to a plane partition $\pi(\mu_1, \mu_2, \mu_3)$ (with $N(|\mu_1|+|\mu_2|+|\mu_3|)-|\II|-2|\III|$ boxes) only near the boundary of $H(N)$. The minimal dimer configuration of $H(N; \mu_1, \mu_2, \mu_3)$ has extra horizontal dimers in sector 1 and sector 2, and has fewer horizontal dimers in sector 3. 

Specifically, in sector 1, if $(\mu_1')_i \geq i$, the $i$th part of $\mu_1'$ contributes $i-1$ horizontal dimers of weight $q^{N + (\mu_1')_i - i}$. If $(\mu_1')_i < i$, the $i$th part of $\mu_1'$ contributes $(\mu_1')_i$ horizontal dimers of weight $q^{N + (\mu_1')_i - i}$. Therefore, in sector 1 the weight of the minimal dimer configuration of $H(N; \mu_1, \mu_2, \mu_3)$ differs from that of the dimer configuration corresponding to $\pi(\mu_1, \mu_2, \mu_3)$ by a factor of 
\[
\prod_{i: (\mu'_1)_i \geq i\geq 1} 
q^{(i-1)(N + (\mu'_1)_i - i) } \prod_{i: (\mu'_1)_i < i\leq\ell(\mu_1')} q^{(\mu'_1)_i (N + (\mu_1)_i' - i )}.
\]

In sector 2, if $(\mu_2)_i \geq i$, the $i$th part of $\mu_2$ contributes $i-1$ horizontal dimers with weights $q^{(\mu_2)_i -i + 1}, q^{(\mu_2)_i -i + 2}, \ldots, q^{(\mu_2)_i -1}$. The total weight of these dimers is 
\[ \prod\limits_{i: (\mu_2)_i \geq i\geq 1 } \prod\limits_{j=1}^{i-1} q^{(\mu_2)_i -i + j}
= \prod\limits_{i: (\mu_2)_i \geq i\geq 1 } q^{(i-1)((\mu_2)_i -i)} q^{ (i-1)i/2}
= \prod\limits_{i: (\mu_2)_i \geq i\geq 1 } q^{(i-1) ( (\mu_2)_i- i/2 ) }. \]
If $(\mu_2)_i < i$, the $i$th part of $\mu_2$ contributes $(\mu_2)_i$ horizontal dimers with weights $q^{0}, q^{1}, \ldots, q^{(\mu_2)_i -1}$. The total weight of these dimers is 
\[
\prod\limits_{i: (\mu_2)_i < i\leq\ell(\mu_2) } \prod\limits_{j=0}^{(\mu_2)_{i}-1 }
q^j = \prod_{i: (\mu_2)_i < i\leq\ell(\mu_2) } q^{((\mu_2)_i-1)(\mu_2)_i/2}.
\]

In sector 3, the dimers in the dimer configuration corresponding to $\pi(\mu_1, \mu_2, \mu_3)$ that are not in the minimal dimer configuration of $H(N; \mu_1, \mu_2, \mu_3)$ have weight 
\[
\prod\limits_{i=1}^{\ell(\mu_3)} q^{(2N-i) (\mu_3)_i}.
\]

Since the dimer configuration corresponding to the plane partition $\pi(\mu_1, \mu_2, \mu_3)$ has weight $q^{N^2(N-1)/2+N(|\mu_1|+|\mu_2|+|\mu_3|)-|\II|-2|\III|}$, we combine these remarks to arrive at the following. 

\begin{lemma}
\label{cor:DTweightGabcd}
The weight of the minimal dimer configuration of $H(N; \mu_1,\mu_2,\mu_3)$ is $q^{w_{\min}(\mu_1, \mu_2, \mu_3)}=q^{\widetilde{w}_{\min}(\mu_1, \mu_2, \mu_3)-|\II(\mu_1, \mu_2, \mu_3)|-2|\III(\mu_1, \mu_2, \mu_3)|}$, where 
\begin{eqnarray*}
\widetilde{w}_{\min}(\mu_1, \mu_2, \mu_3) &= & \dfrac{N^2(N-1)}{2} + N(|\mu_1|+|\mu_2|+|\mu_3|) 
+ \sum\limits_{i=1}^{\ell(\mu_3) } (-2N+i)(\mu_3)_i \\
&& {}+{} \sum\limits_{i :1 \leq i \leq (\mu'_1)_i } (i-1)(N + (\mu'_1)_i - i) + 
\sum\limits_{i: (\mu'_1)_i < i \leq \ell( (\mu'_1) ) } (\mu'_1)_i (N + (\mu'_1)_i - i) \\
&& {}+{} \sum\limits_{i :1 \leq i \leq (\mu_2)_i } (i-1)\left( (\mu_2)_i- \frac{i}{2} \right) +
\sum\limits_{i: (\mu_2)_i < i \leq \ell( \mu_2 ) }
((\mu_2)_i-1) \frac{(\mu_2)_i}{2}.
\end{eqnarray*}
\end{lemma}

%\helentodo{Comment about proof}

%\helentodo{Comment about extra edges}

%\helentodo{Move over an example or two}

%\begin{remark}
%\label{rem:wminB}

Lemma~\ref{cor:DTweightGabcd} is sufficient to analyze the first four factors in the condensation recurrence (\ref{eqn:DTcond}). For the remaining two factors, more work is needed, since they are associated with Maya diagrams of nonzero charge. However, we omit the proofs of the necessary lemmas, because they are very similar to that of Lemma~\ref{cor:DTweightGabcd}. 

%For example, to obtain the expression $\widetilde{w}_{\min}$ for $G- \{a, b\}$ we change $\mu_2$ to $\mu_2^{rc}$ in Corollary~\ref{cor:DTweightGabcd}, because adding the vertices $c$ and $d$ to $G$ changes the partition $\mu_2$.

%For example, to obtain the expression $\widetilde{w}_{\min}$ for $G- \{c, d\}$ we change $\mu_1$ to $\mu_1^{rc}$ in Corollary~\ref{cor:DTweightGabcd}. 
%\end{remark}

\begin{lemma}
\label{lem:wminGad}
The weight of the minimal dimer configuration of $H(N; \mu_1^{rc},\mu_2^{rc},\mu_3)-\{a, d\}$ is $q^{w_{\min}^u}=q^{\tilde{w}_{\min}^{u}-|\II(\mu_1^r, \mu_2^c, \mu_3)|-2|\III(\mu_1^r, \mu_2^c, \mu_3)|}$, where 
\begin{eqnarray*}
\tilde{w}_{\min}^{u} &=&
\frac{N(N^2+2N-1)}{2}+(N+1)\left(\left|\mu_1^r\right|+\left|\mu_2^c\right|\right)+N+(N-1)\left|\mu_3\right|
+\sum_{i=1}^{\ell(\mu_3)}(-2N+i)(\mu_3)_i \\
&&{}+{}\sum_{i: 1\leq i\leq(\mu_1^r)'_i+1}(i-2)(N+(\mu_1^r)'_i-(i-1)) \\
&&{}+{}\sum_{i: (\mu_1^r)'_i+1<i\leq\ell((\mu_1^r)')}(\mu_1^r)'_i(N+(\mu_1^r)'_i-(i-1)) \\
&&{}+{}\sum_{i: 1\leq i\leq(\mu_2^c)_i+1}(i-2)\left((\mu_2^c)_i - \frac{i-1}{2}\right) +\sum_{i: (\mu_2^c)_i+1<i\leq\ell(\mu_2^c)}\frac{(\mu_2^c)_i((\mu_2^c)_i-1)}{2}.
\end{eqnarray*}
\end{lemma}

\begin{lemma}
\label{lem:wminGbc}
The weight of the minimal dimer configuration of $H(N; \mu_1^{rc},\mu_2^{rc},\mu_3)-\{b, c\}$ is $q^{w_{\min}^d}=q^{\tilde{w}_{\min}^{d}-|\II(\mu_1^c, \mu_2^r, \mu_3)|-2|\III(\mu_1^c, \mu_2^r, \mu_3)|}$, where 
\begin{eqnarray*}
\tilde{w}_{\min}^{d}
&=&\frac{(N-1)^2(N-2)}{2}+(N-1)\left(\left|\mu_1^c\right|+\left|\mu_2^r\right|\right)+(N+1)\left|\mu_3\right| +\sum_{i=1}^{\ell(\mu_3)}(-2N+i)(\mu_3)_i\\
&&{}+{}\sum_{i: 1\leq i\leq(\mu_1^c)'_i}i(N+(\mu_1^c)'_i-i-1)+\sum_{i: (\mu_1^c)'_i<i\leq\ell((\mu_1^c)')}(\mu_1^c)'_i(N+(\mu_1^c)'_i-i-1)\\
&&{}+{}\sum_{i: 1\leq i\leq(\mu_2^r)_i}i\left((\mu_2^r)_i - \frac{i+1}{2}\right) +\sum_{i: (\mu_2^r)_i<i\leq\ell(\mu_2^r)}\frac{(\mu_2^r)_i((\mu_2^r)_i-1)}{2}.
\end{eqnarray*}
\end{lemma}

\subsubsection{Algebraic simplification}
\label{sec:DTalg}

Since
$A = \widetilde{w}_{\min}(\mu_1, \mu_2, \mu_3) + \widetilde{w}_{\min}(\mu_1^{rc}, \mu_2^{rc}, \mu_3)$ and \\
$B = \widetilde{w}_{\min}(\mu_1^{rc}, \mu_2, \mu_3) + \widetilde{w}_{\min}(\mu_1, \mu_2^{rc}, \mu_3)$, we see that $A = B$. In addition, $C = \tilde{w}_{\min}^{u} + \tilde{w}_{\min}^{d}$. 

To compute $C- A$, we split the algebra into two pieces: we first simplify the summands involving $N$, and next simplify the summands that do not involve $N$. 

%First we consider the terms involving $N$ for the constant $A$. 
By Lemma~\ref{cor:DTweightGabcd}, the terms in $A$ that involve $N$ are 
\begin{eqnarray*}
& & N^2(N-1) + N(|\mu_1| + |\mu_1^{rc}| +|\mu_2| + |\mu_2^{rc}| +2 |\mu_3|) 
+ 2 \sum\limits_{i=1}^{\ell(\mu_3) } (-2N+i)(\mu_3)_i \\
& & {}+{} \sum\limits_{i :1 \leq i \leq (\mu_1)_i' } N(i-1) + 
\sum\limits_{i: (\mu_1)_i' < i \leq \ell( (\mu_1)' ) } N (\mu_1)_i'
+ \sum\limits_{i :1 \leq i \leq (\mu_1^{rc})_i' } N(i-1) \\
& & {}+{} \sum\limits_{i: (\mu_1^{rc})_i' < i \leq \ell( (\mu_1^{rc})' ) } N (\mu_1^{rc})_i'.
\end{eqnarray*}
Since $\lambda_i \geq i$ precisely when $i \leq d(\lambda)$, we can write 
\[ \sum\limits_{i :1 \leq  i \leq \lambda_i } N(i-1) 
= \dfrac{ Nd( \lambda) (d(\lambda) -1) }{2} .\]
%where $d( \mu_1')$ is the length of the diagonal of $\mu_1'$. 
So, the above can be written as 
\begin{eqnarray}
&& N^2(N-1) + N(|\mu_1| + |\mu_1^{rc}| +|\mu_2| + |\mu_2^{rc}| +2 |\mu_3|) 
+ 2 \sum\limits_{i=1}^{\ell(\mu_3) } (-2N+i)(\mu_3)_i \label{eqn:A} \\
&& {}+{} \dfrac{ Nd( \mu_1') (d( \mu_1') -1) }{2}
+ N\sum\limits_{i: d( \mu_1') + 1 \leq i \leq \ell( \mu'_1 ) } (\mu'_1)_i
+ \dfrac{ Nd( (\mu_1^{rc})') (d( (\mu_1^{rc})') -1) }{2} \nonumber \\
&& {}+{} N \sum\limits_{ i: d( (\mu_1^{rc})') + 1 \leq i \leq \ell( (\mu_1^{rc})' ) } (\mu_1^{rc})_i'. \nonumber 
\end{eqnarray}
%We comment that by Remark~\ref{rem:wminB}, the terms involving $N$ in the constant $B$ are exactly the same. 

Now we consider the terms in $C$ that involve $N$. By Lemmas~\ref{lem:wminGad} and~\ref{lem:wminGbc}, those terms are 
\begin{eqnarray*}
& & \frac{N(N^2+2N-1)}{2}+N \left(\left|\mu_1^r\right|+\left|\mu_2^c\right| + \left|\mu_3\right| \right)+N+ 2\sum_{i=1}^{\ell(\mu_3)}(-2N+i)(\mu_3)_i \\
&&{}+{}\sum_{i: 1\leq i\leq(\mu_1^r)'_i+1} N (i-2)
+ \sum_{i: (\mu_1^r)'_i+1<i\leq\ell((\mu_1^r)')} N (\mu_1^r)'_i + \frac{(N-1)^2(N-2)}{2} \\
&&{}+{}N\left(\left|\mu_1^c\right|+\left|\mu_2^r\right| +\left|\mu_3\right| \right)
% - |\mu_1^c| -|\mu_2^r| + \left|\mu_3\right| \\
+\sum_{i: 1\leq i\leq(\mu_1^c)'_i} Ni +
\sum_{i: (\mu_1^c)'_i<i\leq\ell((\mu_1^c)')} N(\mu_1^c)'_i.
\end{eqnarray*}
As above, we can write 
\[\sum_{i: 1\leq i\leq(\mu_1^c)'_i} Ni 
= N \sum\limits_{i=1}^{ d((\mu_1^c)')} i = \dfrac{ Nd( (\mu_1^{c})') (d( (\mu_1^{c})') +1) }{2}. 
\]
%And thus the above can be simplified to
%\begin{align}
%& N^2(N-1) + 3N - 1 +N \left(\left|\mu_1^r\right| + \left|\mu_1^c\right| + \left|\mu_2^r\right| +\left|  \mu_2^c\right| +2 \left|\mu_3\right|    \right) + 2\sum_{i=1}^{\ell(\mu_3)}(-2N+i)(\mu_3)_i \nonumber  \\ 
%& +\sum_{i: 1\leq i\leq(\mu_1^r)'_i+1} N (i-2) +
%\sum_{i: (\mu_1^r)'_i+1<i\leq\ell((\mu_1^r)')}   N (\mu_1^r)'_i \nonumber   \\
%&+ \dfrac{ Nd( (\mu_1^{c})')  (d( (\mu_1^{c})')  +1) }{2}   +
%\sum_{(\mu_1^c)'_i<i\leq\ell((\mu_1^c)')} N(\mu_1^c)' \nonumber 
%\end{align}

Recall from Lemma~\ref{rem:idie} that $d_{s}((\mu_1^{r})')$ denotes the largest integer $i$ such that $i \leq (\mu_1^{r})'_i + 1$. There are two possibilities, either $d_{s} := d_{s}((\mu_1^{r})')$ is equal to $d := d((\mu_1^{r})')$, or $d_{s} = d+1$. 
%Now, we split into cases. First we introduce some notation. Let $i_e$ (resp. $i_d$) be the maximum positive integer $i$ such that $i \leq (\mu^r_1)'_i + 1$ (resp. $i \leq (\mu^r_1)'_i$). There are two possibilities: either $i_e = i_d$ or $i_e = i_d + 1$ (see Lemma~\ref{rem:idie}). By Remark~\ref{lem:irdmu}, $i_d = d( (\mu_1)')$. 
First assume that $d_{s} = d$. 
%$i_e = i_d$. 
Then we have 
\begin{itemize}
\item $N \sum\limits_{i: 1\leq i\leq(\mu_1^r)'_i+1} (i-2)
= N \sum\limits_{i: 1\leq i\leq d((\mu_1^r)')} (i-2) = 
N \left( \frac{ (d((\mu_1^r)')-2) (d((\mu_1^r)')-1) }{2} -1 \right)$, and 
\item $N\sum\limits_{i: (\mu_1^r)'_i+1<i\leq\ell((\mu_1^r)')} (\mu_1^r)'_i 
=N\sum\limits_{i: d((\mu_1^r)') <i\leq\ell((\mu_1^r)')} (\mu_1^r)'_i$. 
\end{itemize}
If instead $d_{s} = d+1$, then 
$$ N \sum\limits_{1\leq i\leq(\mu_1^r)'_i+1} (i-2)
= N \left( \sum\limits_{1\leq i\leq d((\mu_1^r)')} (i-2) + d((\mu_1^r)') -1 \right).$$ 
Since $(\mu^r_1)'_{d(\mu^r_1) + 1} = d((\mu^r_1)')$ by Lemma~\ref{lem:d'eq}, 
\begin{eqnarray*}
& & N \left( \sum\limits_{i: 1\leq i\leq(\mu_1^r)'_i+1} (i-2) + \sum\limits_{i: (\mu_1^r)'_i+1<i\leq\ell((\mu_1^r)')} (\mu_1^r)'_i \right)\\
&=& N \left( \sum\limits_{i: 1\leq i\leq d((\mu_1^r)')} (i-2) + d((\mu_1^r)') -1 + \sum\limits_{i: d((\mu_1^r)')+1<i\leq\ell((\mu_1^r)') } (\mu_1^r)'_i \right) \\
&=&
N\left( \sum\limits_{i: 1\leq i\leq d((\mu_1^r)')} (i-2) -1 + \sum\limits_{ i: d((\mu_1^r)') <i\leq\ell((\mu_1^r)')} (\mu_1^r)'_i \right). 
\end{eqnarray*}
So the terms in $C$ that involve $N$ can be written as 
\begin{align}
& N^2(N-1) + 3N - 1 +N \left(\left|\mu_1^r\right| + \left|\mu_1^c\right| + \left|\mu_2^r\right| +\left| \mu_2^c\right| +2 \left|\mu_3\right| \right) + 2\sum_{i=1}^{\ell(\mu_3)}(-2N+i)(\mu_3)_i \label{eqn:C} \\ 
&+ N \left(-2 + \mathbbm{1}_{d_{s} = d} + \dfrac{ (d((\mu_1^r)')-2) (d((\mu_1^r)')-1) }{2} +
\sum_{i: d((\mu_1^r)') <i\leq\ell((\mu_1^r)')} (\mu_1^r)'_i \right) \nonumber \\
&+ \dfrac{ Nd( (\mu_1^{c})') (d( (\mu_1^{c})') +1) }{2} +
N \sum_{i: (\mu_1^c)'_i<i\leq\ell((\mu_1^c)')} (\mu_1^c)'_i. \nonumber 
\end{align}
%if $i_e = i_d$ and
%\begin{align}
%& N^2(N-1) + 3N - 1 +N \left(\left|\mu_1^r\right| + \left|\mu_1^c\right| + \left|\mu_2^r\right| +\left| \mu_2^c\right| +2 \left|\mu_3\right| \right) + 2\sum_{i=1}^{\ell(\mu_3)}(-2N+i)(\mu_3)_i \label{eqn:C} \\ 
%& N \left(-2 + \dfrac{ d((\mu_1^r)'-2) (d((\mu_1^r)')-1) }{2} +
%\sum_{d((\mu_1^r)') <i\leq\ell((\mu_1^r)')} (\mu_1^r)'_i \right) + \nonumber \\
%&+ \dfrac{ Nd( (\mu_1^{c})') (d( (\mu_1^{c})') +1) }{2} +
%\sum_{(\mu_1^c)'_i<i\leq\ell((\mu_1^c)')} N(\mu_1^c)' \nonumber 
%\end{align}
%otherwise. 

%\dfrac{ Nd((\mu_1^r)'-2) (d((\mu_1^r)')-1) }{2} +
%N\sum_{d((\mu_1^r)') <i\leq\ell((\mu_1^r)')} (\mu_1^r)'_i 
%Let $i_d$ be the maximum positive integer $i$ such that

Before we subtract the terms in $A$ that involve $N$ from the terms in $C$ that involve $N$, we make some remarks which will help us simplify the following sums: 
\begin{align*}
N \sum_{i: (\mu_1^c)'_i<i\leq\ell((\mu_1^c)')} (\mu_1^c)'_i, 
\qquad\qquad\qquad & N \sum_{i: d((\mu_1^r)') <i\leq\ell((\mu_1^r)')} (\mu_1^r)'_i, \\
N\sum\limits_{i: d( \mu_1') + 1 \leq i \leq \ell( \mu'_1 ) } (\mu'_1)_i,\qquad\text{ and }
\qquad & N \sum\limits_{ i: d( (\mu_1^{rc})') + 1 \leq i \leq \ell( (\mu_1^{rc})' ) } (\mu_1^{rc})_i'.
\end{align*}

\begin{remark}
\label{rem:er}
Let 
\[
e^{r}(\mu) = \sum\limits_{i: d(\mu) < i \leq \ell(\mu)} \mu_i - 
\sum\limits_{i: d(\mu^r) < i \leq \ell(\mu^r)} \mu_i^r. 
\]
There are two cases to consider. If $d(\mu) = d(\mu^r)$, then by Lemma~\ref{lem:dmur}, $\mu_{d(\mu) +1} = d(\mu)$. So, applying Lemma~\ref{lem:mur}, 
\begin{eqnarray*}
e^{r}(\mu) &=& \sum\limits_{i: d(\mu) < i \leq \ell(\mu)} \mu_i - 
\sum\limits_{i: d(\mu) < i \leq \ell(\mu^r)} \mu_{i+1} \\
&=&
\sum\limits_{i: d(\mu) < i \leq \ell(\mu)} \mu_i - 
\sum\limits_{i: d(\mu) +1 < i \leq \ell(\mu^r) + 1} \mu_{i}  = \mu_{d(\mu)+1} = d(\mu).
\end{eqnarray*}
% If $\mu_{d(\mu) +1} \neq d(\mu)$, then $\mu_{d(\mu) +1} < d(\mu)$ by Remark~\ref{rem:dmuplus1}. Then by Lemma~\ref{lem:dmurneqdmu}, 
If instead $d(\mu^r) =d(\mu) - 1$, then 
\[
e^{r}(\mu) = \sum\limits_{i: d(\mu) < i \leq \ell(\mu)} \mu_i - 
\sum\limits_{i: d(\mu) -1 < i \leq \ell(\mu^r)} \mu_{i+1} = 0.
\]

We have shown 
%\[
%e^{r}(\mu) =
%\begin{cases}
%\mu_{d(\mu)+1}  & \mbox{ if } \mu_{d(\mu) +1} = d(\mu) \\
%0 &\mbox{ otherwise}
%\end{cases}
%\]
%Equivalently, we can write
\[
e^{r}(\mu) =
\begin{cases}
d(\mu) & \mbox{if } d(\mu ) = d(\mu^r) \\
0 &\mbox{otherwise}.
\end{cases}
\]
\end{remark}

\begin{remark}
\label{rem:erc}
Let 
\[
e^{rc}(\mu) = \sum\limits_{i: d(\mu^c) < i \leq \ell(\mu^c)} \mu_i^c - 
\sum\limits_{i: d(\mu^{rc}) < i \leq \ell(\mu^{rc})} \mu_i^{rc}. 
\]
As in the previous remark, we split into cases based on whether $d(\mu^c) = d(\mu^{rc})$ or $d(\mu^c) = d(\mu^{rc}) + 1$. Applying Lemma~\ref{lem:mucandmurc}, we get 
\[
e^{rc}(\mu) = 
\begin{cases}
\mu^c_{d(\mu^c)+1} & \mbox{if } d(\mu^c) = d(\mu^{rc}) \\
0 &\mbox{otherwise}.
\end{cases} 
\]
By Remarks~\ref{rem:dmurc} and~\ref{rem:mucdiag}, if $d(\mu^c) = d(\mu^{rc})$, then $\mu^c_{d(\mu^c)+1} = d(\mu) -1$, so 
\[
e^{rc}(\mu) = 
\begin{cases}
d(\mu) -1 & \mbox{if } d(\mu^c) = d(\mu^{rc}) \\
0 &\mbox{otherwise}.
\end{cases} 
\]
%Equivalently, we can write
%
%\[
%e^{rc}(\mu) =
%\begin{cases}
%d(\mu) -1 & \mbox{if } \mu_{d(\mu)} = d(\mu) \\
%0 &\mbox{otherwise}
%\end{cases}
%\]
%by Remark~\ref{rem:dmurc}
\end{remark}
%\helentodo{come back to this -- see your notes}

%\begin{remark}
%Let 
%\[
%\tilde{e}^{r}(\mu) := \sum\limits_{i: d(\mu') < i \leq \ell(\mu)} \mu'_i - 
%\sum\limits_{i: d((\mu^r)') < i \leq \ell(\mu^r)} (\mu_i^r)'
%= \sum\limits_{i: d(\mu') < i \leq \ell(\mu)} \mu'_i -
%\sum\limits_{i: d((\mu')^c) < i \leq \ell((\mu')^c)} (\mu'_i)^c
%\]
%There are two cases to consider. If $d(\mu') = d((\mu')^c)$, then
%\begin{eqnarray*}
%e^{r}(\mu) &=&  \sum\limits_{i: d(\mu) < i \leq \ell(\mu)} \mu_i - 
%\sum\limits_{i: d(\mu) < i \leq \ell(\mu^r)} \mu_{i+1} \\
%&= &
%\sum\limits_{i: d(\mu) < i \leq \ell(\mu)} \mu_i - 
%\sum\limits_{i: d(\mu) +1 < i \leq \ell(\mu^r) + 1} \mu_{i}  = \mu_{d(\mu)%+1} 
%\end{eqnarray*}
%If $\mu_{d(\mu) +1} \neq d(\mu)$, then $\mu_{d(\mu) +1} < d(\mu)$ by Remark~\ref{rem:dmuplus1}. Then by Lemma~\ref{lem:dmurneqdmu}, $d(\mu^r) =d(\mu) - 1$. In this case, 
%\[
%e^{r}(\mu) = \sum\limits_{i: d(\mu) < i \leq \ell(\mu)} \mu_i - 
%\sum\limits_{i: d(\mu) -1 < i \leq \ell(\mu^r)} \mu_{i+1} = 0
%\]
%We have shown
%\[
%e^{r}(\mu) =
%\begin{cases}
%\mu_{d(\mu)+1} & \mbox{ if } \mu_{d(\mu) +1} = d(\mu) \\
%0 &\mbox{ otherwise}
%\end{cases}
%\]
%\end{remark}

\begin{remark}
\label{rem:erplus}
We note that 
\[
\dfrac{ d( (\mu_1^{c})') (d( (\mu_1^{c})') +1) }{2}
- \dfrac{ d( \mu_1') (d( \mu_1') -1) }{2} = 
\begin{cases}
d( \mu'_1) & \mbox{if } d((\mu_1^c)')= d(\mu'_1) \\
0 &\mbox{otherwise}.
\end{cases} 
\]
So, applying Remark~\ref{rem:er} and using the fact that $(\mu_1^{c})' = (\mu'_1)^{r}$, we have 
\[
\dfrac{ d( (\mu_1^{c})') (d( (\mu_1^{c})') +1) }{2}
- \dfrac{ d( \mu_1') (d( \mu_1') -1) }{2}
- e^{r}(\mu_1') = 0.
\]
\end{remark}

\begin{remark}
\label{rem:ercplus}
Note that 
\begin{eqnarray*}
&& \dfrac{ (d((\mu_1^r)')-2) (d((\mu_1^r)')-1) }{2} - \dfrac{ d( (\mu_1^{rc})') (d((\mu_1^{rc})') -1) }{2} \\
& = &
\begin{cases}
- (d(( \mu^r_1)') -1) & \mbox{if } d((\mu_1^r)')= d((\mu_1^{rc})') \\
0 &\mbox{otherwise}.
\end{cases}
\end{eqnarray*}
%&=&
%\begin{cases}
%-(d( \mu_1) -2) & \mbox{if } d((\mu_1^r)')= d((\mu_1^{rc})') \\
%0 &\mbox{otherwise}
%\end{cases}. 
%\end{eqnarray*}
When $d((\mu_1^r)')= d((\mu_1^{rc})')$, $-(d(( \mu^r_1)') -1) = - (d( \mu_1') -2)$. Also, the condition $d((\mu_1^r)')= d((\mu_1^{rc})')$ is equivalent to $d_{s}((\mu_1^{r})')= d((\mu_1^{r})')+1$. This is because $d_{s}((\mu_1^{r})')= d((\mu_1^{r})')+1$ if and only if $d(\mu_1')= d((\mu_1^{r})')+1$ (by Lemma~\ref{lem:dprime_c}) which holds if and only if $d(\mu_1')-1= d((\mu_1^{r})')$, which is equivalent to $d((\mu_1')^{rc}) = d((\mu_1^{r})')$.
%where $d' = d'((\mu_1^{r})')$ and $d = d((\mu_1^{r})')$.
%This is because $d' = d+1$ if and only if $(\mu_1^r)'_{d((\mu_1^r)') + 1} = d((\mu_1^r)' )$. 
%by Lemma~\ref{lem:d'eq}.
%This is equivalent to $(\mu'_1)^c_{d((\mu'_1)^c) + 1} = d((\mu'_1)^c)$. 
%Since it is always the case that $(\mu'_1)^c_{d(\mu'_1) + 1} = d(\mu'_1)-1$, we cannot have $d((\mu'_1)^c) = d(\mu'_1)$. So $d' = d+1$ implies $d((\mu_1^r)')= d((\mu_1^{rc})') $
%Conversely, if $d((\mu'_1)^c) = d(\mu'_1)-1$, then 
%$(\mu'_1)^c_{d((\mu'_1)^c) + 1} = d((\mu'_1)^c)
%= (\mu'_1)^c_{d((\mu'_1)} = (\mu'_1)_{d((\mu'_1)} -1
%$.
%\helentodo{need to come back to this, i'm getting really confused}
%this occurs 
%We conclude by observing that this condition is equivalent to
%$d((\mu'_1)^c)= d((\mu_1^{rc}))$. (Because if instead $d((\mu'_1)^c) = d(\mu'_1)$, then $(\mu'_1)^c_{d(\mu'_1) + 1} = d(\mu'_1)$, which is impossible by the construction of $(\mu'_1)^c$. )
So, by Remark~\ref{rem:erc}, if $d_s:=d_{s}((\mu_1^{r})')$ and $d:=d((\mu_1^{r})')$, then 
\[ \dfrac{ (d((\mu_1^r)')-2) (d((\mu_1^r)')-1) }{2} - \dfrac{ d( (\mu_1^{rc})') (d((\mu_1^{rc})') -1) }{2} + e^{rc}(\mu_1') = \mathbbm{1}_{d_{s} \neq d}.
\]
\end{remark}

%\[
%N e^{rc}(\mu_1') =
%\begin{cases}
%N (d(\mu'_1) - 1) & \mbox{if } d((\mu'_1)^c)= d((\mu'_1)^{rc})\\
%0 &\mbox{otherwise}
%\end{cases}. 
%\]
%(\mu'_1)^c_{d((\mu'_1)^c) +1} = d((\mu'_1)^c) \\

%\begin{eqnarray*}
%\dfrac{ Nd((\mu_1^r)')-2) (d((\mu_1^r)')-1) }{2} - \dfrac{ Nd( (\mu_1^{rc})') (d((\mu_1^{rc})') -1) }{2} 
%& = &
%\begin{cases}
%-N (d( \mu^r_1) -1) & \mbox{if } d((\mu_1^r)')= d((\mu_1^{rc})') \\
%0 &\mbox{otherwise}
%\end{cases} \\
%& =&
%\begin{cases}
%-N (d( \mu_1) -1) + N& \mbox{if } d((\mu_1^r)')= d((\mu_1^{rc})') \\
%0 &\mbox{otherwise}
%\end{cases}. 
%\end{eqnarray*}

Now we subtract the terms in $A$ that involve $N$ (see equation (\ref{eqn:A})) from the terms in $C$ that involve $N$ (see equation (\ref{eqn:C})). Each term that cancels with another term is marked with $c$. Each term that is modified between one side of an equation and the other is underlined and the relevant lemma or remark is indicated.
%We deal with both cases at the same time by putting the extra term that is present if and only if $i_e = i_d + 1$ in gray. 
%We first consider the case where $i_d = i_e$. 

\begin{eqnarray*}
&& \underbrace{N^2(N-1)}_{c} + 3N - 1 +N \bigg{(}\left|\mu_1^r\right| + \left|\mu_1^c\right| + \left|\mu_2^r\right| +\left| \mu_2^c\right| + \underbrace{2 \left|\mu_3\right| }_{c} \bigg{)} + \underbrace{ 2\sum_{i=1}^{\ell(\mu_3)}(-2N+i)(\mu_3)_i }_{c} \\
&&{}+{} N \bigg{(} -1 - \mathbbm{1}_{d_{s} \neq d} + \dfrac{ (d((\mu_1^r)')-2) (d((\mu_1^r)')-1) }{2} +
\underbrace{\sum_{d((\mu_1^r)') <i\leq\ell((\mu_1^r)')} (\mu_1^r)'_i }_{\text{Lemma }\ref{lem:mucprime}} 
%+N \sum_{1\leq i\leq(\mu_1^r)'_i+1} (i-2) +
%N\sum_{(\mu_1^r)'_i+1<i\leq\ell((\mu_1^r)')} (\mu_1^r)'_i 
\\ 
&&\phantom{{}+{} N \bigg{(}}{}+{} \dfrac{ d( (\mu_1^{c})') (d( (\mu_1^{c})') +1) }{2} +
\underbrace{ \sum_{(\mu_1^c)'_i<i\leq\ell((\mu_1^c)')} (\mu_1^c)'_i }_{ \text{Lemma}~\ref{lem:mucprime}} \bigg{)} \\
&&{}-{} \Bigg{(} \underbrace{N^2(N-1)}_{c} + N\bigg{(}|\mu_1| + |\mu_1^{rc}| +|\mu_2| + |\mu_2^{rc}| + \underbrace{2 |\mu_3| }_{c}\bigg{)} 
+ \underbrace{ 2 \sum\limits_{i=1}^{\ell(\mu_3) } (-2N+i)(\mu_3)_i }_{c} \\
&&\phantom{{}-{} \Bigg{(}}{}+{} N \bigg{(} \dfrac{ d( \mu_1') (d( \mu_1') -1) }{2}
+ \sum\limits_{i : d( \mu_1') + 1 \leq i \leq \ell( \mu_1' ) } (\mu_1)_i'
+ \dfrac{ d( (\mu_1^{rc})') (d( (\mu_1^{rc})') -1) }{2} \\
&&\phantom{\phantom{{}-{} \Bigg{(}}{}+{} N \bigg{(}}{}+{} \sum\limits_{i: d( (\mu_1^{rc})') + 1 \leq i \leq \ell( (\mu_1^{rc})' ) } (\mu_1^{rc})_i' \bigg{)} \Bigg{)} \\
&=& 3N - 1 +N \bigg{(} \underbrace{ \left|\mu_1^r\right| -|\mu_1| + \left|\mu_1^c\right| - |\mu_1^{rc}| + \left|\mu_2^r\right| - |\mu_2| +\left| \mu_2^c\right| - |\mu_2^{rc}| }_{\text{Lemma}~\ref{lem:musize} } \bigg{)} \\
&&{}+{} N \Bigg{(} -1 - \mathbbm{1}_{d_{s} \neq d} + \dfrac{ (d((\mu_1^r)')-2) (d((\mu_1^r)')-1) }{2} +
\underbrace{ \sum_{d((\mu'_1)^c ) <i\leq\ell((\mu'_1)^c )} (\mu'_1)^c_i }_{\text{Remark}~\ref{rem:erc}} \\
&&\phantom{{}+{} N \Bigg{(}}{}+{}
%+N \sum_{1\leq i\leq(\mu_1^r)'_i+1} (i-2) +
%N\sum_{(\mu_1^r)'_i+1<i\leq\ell((\mu_1^r)')} (\mu_1^r)'_i + 
\dfrac{ d( (\mu_1^{c})') (d( (\mu_1^{c})') +1) }{2} 
+ \underbrace{ \sum_{ d((\mu'_1)^r) <i\leq\ell((\mu'_1)^r)} (\mu'_1)^r_i}_{\text{Remark}~\ref{rem:er}} - \dfrac{ d( \mu_1') (d( \mu_1') -1) }{2} \\
&&\phantom{{}+{} N \Bigg{(}}{}-{} \underbrace{\sum\limits_{i : d( \mu_1') < i \leq \ell( \mu_1' ) } (\mu'_1)_i }_{\text{Remark}~\ref{rem:er}} 
- \dfrac{ d( (\mu_1^{rc})') (d( (\mu_1^{rc})') -1) }{2} - \underbrace{ \sum\limits_{i : d( (\mu'_1)^{rc}) < i \leq \ell( (\mu_1')^{rc} ) } (\mu'_1)^{rc}_i }_{\text{Remark}~\ref{rem:erc}} \Bigg{)} \\
%& = & \bigg{(} N - 1 + 
%N \sum_{1\leq i\leq(\mu_1^r)'_i+1} (i-2) +
%N\sum_{(\mu_1^r)'_i+1<i\leq\ell((\mu_1^r)')} (\mu_1^r)'_i 
%N \left(-1 + \dfrac{ d((\mu_1^r)'-2) (d((\mu_1^r)')-1) }{2} +
%\sum_{d((\mu_1^r)') <i\leq\ell((\mu_1^r)')} (\mu_1^r)'_i \right)
%\\
%&& + \dfrac{ Nd( (\mu_1^{c})') (d( (\mu_1^{c})') +1) }{2} + Ne^{rc}%(\mu'_1 )\bigg{)} - \bigg{(} \dfrac{ Nd( \mu_1') (d( \mu_1') -1) }{2} 
%\\
%%&& + N\sum\limits_{i = d( \mu_1') + 1}^{\ell( (\mu_1)' ) } (\mu_1)_i' 
%+ \dfrac{ Nd( (\mu_1^{rc})') (d( (\mu_1^{rc})') -1) }{2} \bigg{)} \\
&=& 3N - 1 - 2 N -N -N \cdot \mathbbm{1}_{d \neq d_{s}}\\
&&{}+{} N \bigg{(}
\underbrace{ \dfrac{ d( (\mu_1^{c})') (d( (\mu_1^{c})') +1) }{2} - \dfrac{ d( \mu_1') (d( \mu_1') -1) }{2} -
e^{r}(\mu_1') }_{\text{Remark}~\ref{rem:erplus}} \\
&&\phantom{{}+{} N \bigg{(}}{}+{} \underbrace{e^{rc}(\mu'_1 ) 
+ \dfrac{ (d((\mu_1^r)')-2) (d((\mu_1^r)')-1) }{2} - \dfrac{ d( (\mu_1^{rc})') (d( (\mu_1^{rc})') -1) }{2} }_{\text{Remark}~\ref{rem:ercplus}} \bigg{)} = -1.
\end{eqnarray*}
We have thus shown that the terms involving $N$ simplify to $-1$. 

Now we consider the terms that do not involve $N$. In $A$, we have 
\begin{eqnarray*}
& & \sum\limits_{1 \leq i \leq d(\mu_1) } (i-1)((\mu'_1)_i - i) + 
\sum\limits_{ d(\mu_1) < i \leq \ell( \mu'_1) } (\mu'_1)_i ((\mu'_1)_i - i) + \sum\limits_{1 \leq i \leq d(\mu_2) } (i-1)\left( (\mu_2)_i- \frac{i}{2} \right) \\
& &{}+{} \sum\limits_{ d(\mu_2) < i \leq \ell( \mu_2 ) }
((\mu_2)_i-1) \frac{(\mu_2)_i}{2} + 
\sum\limits_{1 \leq i \leq d(\mu_1^{rc}) } (i-1)( (\mu_1^{rc})_i' - i) \\
&&{}+{} \sum\limits_{d(\mu_1^{rc}) < i \leq \ell( (\mu_1^{rc})' ) } (\mu_1^{rc})_i' ( (\mu_1^{rc})_i' - i) 
+ \sum\limits_{1 \leq i \leq d(\mu_2^{rc}) } (i-1)\left( (\mu_2^{rc})_i- \frac{i}{2} \right) \\
&&{}+{} \sum\limits_{d(\mu_2^{rc}) < i \leq \ell( \mu_2^{rc} ) }
((\mu_2^{rc})_i-1) \frac{(\mu_2^{rc})_i}{2}.
\end{eqnarray*}
We remark that in Lemma~\ref{cor:DTweightGabcd}, the first sum is over $i$ such that $1 \leq i \leq (\mu'_1)_i $, but this is equivalent to writing $1 \leq  i \leq d(\mu_1)$. We have made similar replacements in the other sums.

In $C$, we have 
\begin{eqnarray}
& & \left|\mu_1^r\right|+\left|\mu_2^c\right| - \left|\mu_3\right| - \left|\mu_1^c\right|- \left|\mu_2^r\right| + \left|\mu_3\right| +\sum_{i: 1\leq i\leq(\mu_1^r)'_i+1}(i-2)((\mu_1^r)'_i-(i-1)) \label{eqn:CnoN} \\
&& {}+{}\sum_{i: (\mu_1^r)'_i+1<i\leq\ell((\mu_1^r)')}(\mu_1^r)'_i((\mu_1^r)'_i-(i-1))+\sum_{i: 1\leq i\leq(\mu_2^c)_i+1}(i-2)\left((\mu_2^c)_i - \frac{i-1}{2}\right) \nonumber \\
&& {}+{}\sum_{i: (\mu_2^c)_i+1<i\leq\ell(\mu_2^c)}\frac{(\mu_2^c)_i((\mu_2^c)_i-1)}{2} +
\sum_{1\leq i\leq d(\mu_1^c)}i((\mu_1^c)'_i-i-1) \nonumber \\
&& {}+{}\sum_{d(\mu_1^c)<i\leq\ell((\mu_1^c)')}(\mu_1^c)'_i((\mu_1^c)'_i-i-1) +\sum_{1\leq i\leq d(\mu_2^r)}i\left((\mu_2^r)_i - \frac{i+1}{2}\right) \nonumber \\
&&{}+{}\sum_{d(\mu_2^r)<i\leq\ell(\mu_2^r)}\frac{(\mu_2^r)_i((\mu_2^r)_i-1)}{2}. \nonumber
\end{eqnarray}
Like we did for $A$, we replaced $i: 1\leq i\leq(\mu_1^c)'_i$ in the fifth sum with $1\leq i\leq d(\mu_1^c)$, and similarly for the sixth, seventh, and eighth sums. 

\begin{remark}%[First four sums from $C$]
As in Remark~\ref{rem:idie}, we let $d_{s}(\mu)$ be the maximum positive integer $i$ such that $i \leq \mu_i + 1$. Then we can write the first four sums in~equation (\ref{eqn:CnoN}) as 
\begin{eqnarray*}
& & \sum_{1\leq i\leq d_{s}((\mu_1^r)')}(i-2)((\mu_1^r)'_i-(i-1)) +\sum_{d_{s}((\mu_1^r)')<i\leq\ell((\mu_1^r)')}(\mu_1^r)'_i((\mu_1^r)'_i-(i-1)) \\
& & {}+{}\sum_{1\leq i \leq d_{s}(\mu_2^c) }(i-2)\left((\mu_2^c)_i - \frac{i-1}{2}\right)
+\sum_{ d_{s}(\mu_2^c)<i\leq\ell(\mu_2^c)}\frac{(\mu_2^c)_i((\mu_2^c)_i-1)}{2}.
\end{eqnarray*}
Recall from Lemma~\ref{rem:idie} that for any partition $\mu$, either $d_{s}(\mu) = d(\mu)$ or $d_{s}(\mu) = d(\mu) + 1$. If $d_{s}((\mu_1^r)') = d((\mu_1^r)')$ (resp.~$d_{s}(\mu_2^c) = d(\mu_2^c)$), then we can replace every instance of $d_{s}((\mu_1^r)')$ (resp.~$d_{s}(\mu_2^c)$) in the sums above with $d(\mu_1^r)$ (resp.~$d(\mu_2^c)$). Otherwise, we can use the fact that by Lemma~\ref{lem:d'eq}, $d_{s}(\mu) = d(\mu) + 1$ if and only if $\mu_{d(\mu) + 1} = d(\mu)$ to see that when $d_{s}((\mu_1^r)') = d((\mu_1^r)')+1$, 
\begin{eqnarray*}
& & \sum_{1\leq i\leq d_{s}((\mu_1^r)')}(i-2)((\mu_1^r)'_i-(i-1)) \\
&=&
\sum_{1\leq i\leq d((\mu_1^r)') + 1}(i-2)((\mu_1^r)'_i-(i-1)) \\
&=&
\sum_{1\leq i\leq d((\mu_1^r)')}(i-2)((\mu_1^r)'_i-(i-1)) +
(d((\mu_1^r)') - 1) \left( d((\mu_1^r)' )- d((\mu_1^r)') \right) \\
&=&
\sum_{1\leq i\leq d((\mu_1^r)')}(i-2)((\mu_1^r)'_i-(i-1)) +
(\mu_1^r)'_{d((\mu_1^r)')+1}((\mu_1^r)'_{d((\mu_1^r)')+1}-(d((\mu_1^r)')+1-1))
\end{eqnarray*}
and when $d_{s}(\mu_2^c) = d(\mu_2^c)+1$, 
\begin{eqnarray*}
&&\sum_{1\leq i \leq d_{s}(\mu_2^c) }(i-2)\left((\mu_2^c)_i - \frac{i-1}{2}\right) \\
&=&\sum_{1\leq i \leq d(\mu_2^c)+1 }(i-2)\left((\mu_2^c)_i - \frac{i-1}{2}\right) \\
&=& 
\sum_{1\leq i \leq d(\mu_2^c) }(i-2)\left((\mu_2^c)_i - \frac{i-1}{2}\right) + (d(\mu_2^c) - 1) \left(\frac{d(\mu_2^c) }{2} \right) \\
&=& 
\sum_{1\leq i \leq d(\mu_2^c) }(i-2)\left((\mu_2^c)_i - \frac{i-1}{2}\right) + \frac{(\mu_2^c)_{d(\mu_2^c)+1}((\mu_2^c)_{d(\mu_2^c)+1}-1)}{2}.
\end{eqnarray*}
Therefore, we can write 
\begin{eqnarray*}
& & \sum_{1\leq i\leq d_{s}((\mu_1^r)')}(i-2)((\mu_1^r)'_i-(i-1)) +\sum_{d_{s}((\mu_1^r)')<i\leq\ell((\mu_1^r)')}(\mu_1^r)'_i((\mu_1^r)'_i-(i-1)) \\
& &{}+{}\sum_{1\leq i \leq d_{s}(\mu_2^c) }(i-2)\left((\mu_2^c)_i - \frac{i-1}{2}\right)
+\sum_{ d_{s}(\mu_2^c)<i\leq\ell(\mu_2^c)}\frac{(\mu_2^c)_i((\mu_2^c)_i-1)}{2} \\
& = & \sum_{1\leq i\leq d((\mu_1^r)')}(i-2)((\mu_1^r)'_i-(i-1)) +\sum_{d((\mu_1^r)')<i\leq\ell((\mu_1^r)')}(\mu_1^r)'_i((\mu_1^r)'_i-(i-1)) \\
& &{}+{}\sum_{1\leq i \leq d(\mu_2^c) }(i-2)\left((\mu_2^c)_i - \frac{i-1}{2}\right)
+\sum_{ d(\mu_2^c)<i\leq\ell(\mu_2^c)}\frac{(\mu_2^c)_i((\mu_2^c)_i-1)}{2}.
\end{eqnarray*}
\end{remark}

When we subtract the sums in $A$ from the sums in $C$, we will pair each sum in $A$ with a sum in $C$. Many terms cancel, but this is not obvious and requires the following lemmas. 

\begin{lemma} %[Red Lemma]
\label{lem:redlemma}
\begin{eqnarray*}
\sum_{d((\mu_1^r)')<i\leq\ell((\mu_1^r)')}(\mu_1^r)'_i((\mu_1^r)'_i-(i-1)) 
- \sum\limits_{d(\mu_1^{rc}) < i \leq \ell( (\mu_1^{rc})' ) } (\mu_1^{rc})_i' ((\mu_1^{rc})_i' - i) 
= 0
\end{eqnarray*}
\end{lemma}

\begin{proof}
Recall from Lemma~\ref{lem:mucprime} that $(\mu_1^r)' = (\mu'_1)^c$. Then we can rewrite the difference of sums as 
\[
\sum_{d((\mu'_1)^c)<i\leq\ell((\mu'_1)^c)}(\mu'_1)^c_i((\mu'_1)^c_i-(i-1)) 
- \sum\limits_{d((\mu'_1)^{rc})< i \leq \ell( (\mu'_1)^{rc} ) } (\mu'_1)^{rc}_i ((\mu'_1)^{rc}_i - i). 
\]
There are two cases to consider. For readability, we put $\lambda := \mu_1'$. First assume that $d(\lambda^{rc}) = d(\lambda^c)$. 
%$(\lambda^c)_{d(\lambda^c) + 1} = d(\lambda^c)$. 
Then by Lemma~\ref{lem:mucandmurc}, we have 
\begin{eqnarray*}
& & \sum_{d(\lambda^c)<i\leq\ell(\lambda^c)}\lambda^c_i(\lambda^c_i-(i-1)) 
- \sum\limits_{d(\lambda^{rc})< i \leq \ell( \lambda^{rc} ) } \lambda^{rc}_i (\lambda^{rc}_i - i) \\
& =& 
\sum_{d(\lambda^c)<i\leq\ell(\lambda^c)}\lambda^c_i(\lambda^c_i-(i-1)) 
- \sum\limits_{d(\lambda^c)< i \leq \ell( \lambda^{rc} ) } \lambda^{c}_{i+1} (\lambda^{c}_{i+1}- i) \\
& =& \sum_{d(\lambda^c)<i\leq\ell(\lambda^c)}\lambda^c_i(\lambda^c_i-(i-1)) 
- \sum\limits_{d(\lambda^c) + 1< i \leq \ell( \lambda^{rc}) + 1 } \lambda^{c}_{i} (\lambda^{c}_{i}- (i-1)) =0.
\end{eqnarray*}
In the final step we used Corollary~\ref{cor:lengthmurclengthmuc} and the fact that 
$$\lambda^c_{d(\lambda^c) + 1} =\lambda^c_{d(\lambda^{rc}) + 1}= \lambda^c_{d(\lambda)} =\lambda_{d(\lambda)}-1= d(\lambda)-1 =d(\lambda^{rc})= d(\lambda^c),$$ which follows from Lemma~\ref{lem:muc} and Remark~\ref{rem:dmurc}. Next assume that $d(\lambda^{rc}) =d(\lambda^c) -1$. Then 
\begin{eqnarray*}
& & \sum_{d(\lambda^c)<i\leq\ell(\lambda^c)}\lambda^c_i(\lambda^c_i-(i-1)) 
- \sum\limits_{d(\lambda^{rc})< i \leq \ell( \lambda^{rc} ) } \lambda^{rc}_i (\lambda^{rc}_i - i) \\
&=& \sum_{d(\lambda^c)<i\leq\ell(\lambda^c)}\lambda^c_i(\lambda^c_i-(i-1)) 
- \sum\limits_{d(\lambda^c)-1< i \leq \ell( \lambda^{rc} ) } \lambda^c_{i+1} (\lambda^c_{i+1} - i) \\
&=& \sum_{d(\lambda^c)<i\leq\ell(\lambda^c)}\lambda^c_i(\lambda^c_i-(i-1)) 
- \sum\limits_{d(\lambda^c) < i \leq \ell( \lambda^{rc} )+ 1 } \lambda^{c}_{i} (\lambda^{c}_{i}- (i-1)) =0.
\end{eqnarray*} 
\end{proof}

%\helentodo{I don't understand the $(\mu'_1)_{d(\mu_1)} = d(\mu_1)$ case of the green lemma }

\begin{lemma} %[Green Lemma]
\label{lem:greenlemma}
\begin{eqnarray*}
& & \sum_{1\leq i\leq d((\mu_1^r)')}(i-2)((\mu_1^r)'_i-(i-1)) 
- \sum\limits_{1 \leq  i \leq d(\mu_1^{rc}) } (i-1)( (\mu_1^{rc})_i' - i)\\
&=& 
\begin{cases}
%(d(\mu_1^r) - 2)((\mu_1^r)'_{d(\mu_1^r)}+ 1- d(\mu_1^r)) 
- \sum\limits_{1 \leq i \leq d(\mu_1^{rc}) } (\mu_1^{r})'_i +1 - i & \mbox{if }(\mu'_1)_{d(\mu_1)} = d(\mu_1) \\
(d(\mu_1^r) - 2)((\mu_1^r)'_{d(\mu_1^r)}+ 1- d(\mu_1^r)) - \sum\limits_{1 \leq  i \leq d(\mu_1^{rc}) } (\mu_1^{r})'_i +1 - i & \mbox{otherwise}
\end{cases}
\end{eqnarray*}
\end{lemma}

\begin{proof}
To prove the claim we begin similarly to Lemma~\ref{lem:redlemma}, using the fact that $(\mu_1^r)' = (\mu_1')^c$. Letting $\lambda = \mu'_1$, 
we split into cases based on whether $d(\lambda^{rc}) = d(\lambda^c)$, or $d(\lambda^{rc}) = d(\lambda^c) - 1$, and apply Lemma~\ref{lem:mucandmurc}. In the case where $d(\lambda^{rc}) = d(\lambda^c)$, 
\begin{eqnarray*}
& & 
\sum_{1\leq i\leq d(\lambda^c)}(i-2)(\lambda^c_i-(i-1)) 
- \sum\limits_{1 \leq i \leq d(\lambda^{rc}) } (i-1)( \lambda^{rc}_i - i) \\
&=& 
\sum_{1\leq i\leq d(\lambda^c)}(i-1)(\lambda^c_i-(i-1)) 
- \sum\limits_{1 \leq i \leq d(\lambda^{c}) } (i-1)( \lambda^{c}_i+1 - i) - \sum_{1\leq i\leq d(\lambda^c)}(\lambda^c_i-(i-1)) \\
&=& - \sum_{1\leq i\leq d(\lambda^c)}(\lambda^c_i-(i-1)).
%+( d(\lambda^c) - 1)(\lambda^c_{d((\lambda^c)}-(d(\lambda^c)-1)) - 
%( d(\lambda^c) - 1)((\lambda^{rc})_{d(\lambda^c)} - {d(\lambda^c)}) 
%& & = - \sum_{1\leq i\leq d((\lambda^c))}(\lambda^c_i-(i-1)) +( d(\lambda^c) - 1)(\lambda^c_{d((\lambda^c)}-(d(\lambda^c)-1)) - 
%( d(\lambda^c) - 1)((\lambda^{c})_{d(\lambda^c)+1} - {d(\lambda^c)})
\end{eqnarray*}
%{\bf ???}
The case where $d(\lambda^{rc}) = d(\lambda^c) - 1$ is similar. Finally, we note that $d(\lambda^{rc}) = d(\lambda^c) - 1$ if and only if $\lambda_{d(\lambda)} > d(\lambda)$ by Remark~\ref{rem:dmurc}. 
\end{proof}

\begin{lemma} %[Teal Lemma]
\label{lem:teal}
\begin{eqnarray*}
& & \sum_{1\leq i\leq d(\mu_1^c)}i((\mu_1^c)'_i-i-1) - \sum\limits_{1 \leq i \leq d(\mu_1) } (i-1)((\mu'_1)_i - i)
\\
& =& \begin{cases}
\sum\limits_{1\leq i < d(\mu_1)} ((\mu'_1)_i-i) - (d(\mu_1')-1)((\mu'_1)_{d(\mu_1)}- d(\mu_1')) -d(\mu_1) & \mbox{if } (\mu'_1)_{ d(\mu_1) + 1} = d(\mu_1) \\
\sum\limits_{1\leq i < d(\mu_1)} ((\mu'_1)_i-i) - (d(\mu_1')-1)((\mu'_1)_{d(\mu_1)}- d(\mu_1'))
& \mbox{otherwise}
\end{cases}
\end{eqnarray*}
\end{lemma}

\begin{proof}
We use the fact that $(\mu_1^c)' = (\mu_1')^r$, and then we split into cases based on whether $d((\mu_1')^r) = d(\mu'_1)$ or  $d((\mu_1')^r) = d(\mu'_1)-1$. If $d((\mu_1')^r) = d(\mu'_1)$, then by Lemma~\ref{lem:mur}, we have 
\begin{eqnarray*}
&& \sum_{1\leq i\leq d(\mu_1^c)}i((\mu_1^c)'_i-i-1) - \sum\limits_{1 \leq i \leq d(\mu_1) } (i-1)((\mu'_1)_i - i) \\
&=& \sum_{1\leq i\leq d(\mu_1)}i((\mu'_1)^r_i-i-1) - \sum\limits_{1 \leq i \leq d(\mu_1) } (i-1)((\mu'_1)_i - i) \\
& = &
\sum_{1\leq i < d(\mu_1)} i ((\mu'_1)_i-i) 
+ d(\mu_1)((\mu'_1)_{ d(\mu_1) + 1} - d(\mu_1) - 1) - \sum\limits_{1 \leq i \leq d(\mu_1) } (i-1)((\mu'_1)_i - i) \\
& = & 
\sum_{1\leq i < d(\mu_1)} ((\mu'_1)_i-i) + d(\mu_1)((\mu'_1)_{ d(\mu_1) + 1} - d(\mu_1) - 1) - (d(\mu_1')-1)((\mu'_1)_{d(\mu_1)}- d(\mu_1')).
\end{eqnarray*}
By Lemma~\ref{lem:dmur}, since $d((\mu_1')^r) = d(\mu'_1)$, $(\mu'_1)_{d(\mu_1)+ 1} = d(\mu'_1)$, so 
$$d(\mu_1)((\mu'_1)_{  d(\mu_1) + 1} - d(\mu_1) - 1) = -d(\mu_1).$$
The computation in the case where $d((\mu_1')^r) = d(\mu'_1)-1$ is very similar. 
\end{proof}

\begin{lemma} %[Brown Lemma]
\label{lem:brown}
\begin{eqnarray*}
&&\sum_{d(\mu_1^c) <i\leq\ell((\mu_1^c)')}(\mu_1^c)'_i((\mu_1^c)'_i-i-1)-
\sum\limits_{ d(\mu_1) < i \leq \ell( \mu'_1) } (\mu'_1)_i ((\mu'_1)_i - i) \\
&=& 
\begin{cases}
d(\mu_1) & \mbox{if } (\mu'_1)_{d(\mu_1) + 1} = d(\mu_1) \\
0 & \mbox{otherwise}
\end{cases}
\end{eqnarray*}
\end{lemma}

\begin{proof}
We begin by using the fact that $(\mu^c_1)' = (\mu'_1)^r$. Then we split into cases based on whether $d((\mu_1')^r) = d(\mu_1')$ or $d((\mu_1')^r) = d(\mu_1') -1$. To get the final expression we make use of the fact that $d((\mu_1')^r) = d(\mu_1')$ if and only if $(\mu'_1)_{d(\mu_1) + 1} = d(\mu_1)$.
\end{proof}

\begin{remark}
\label{rem:browntealgreen}
By combining Lemmas~\ref{lem:teal} and~\ref{lem:brown}, we get that the sums involved result in 
\[
\sum\limits_{1\leq i < d(\mu_1)} ((\mu'_1)_i-i) - (d(\mu_1')-1)((\mu'_1)_{d(\mu_1)}- d(\mu_1'))\]
in all cases. If we then include the sums from Lemma~\ref{lem:greenlemma}, we split into two cases. If $(\mu'_1)_{d(\mu_1)} = d(\mu_1')$, then by Lemma~\ref{lem:muc}, we get
\[ - (d(\mu_1')-1)((\mu'_1)_{d(\mu_1)}- d(\mu_1')) = 0.\]
If $(\mu'_1)_{d(\mu_1)} > d(\mu_1')$, then by Remark~\ref{rem:mucdiag}, $d(\mu_1^r)=d(\mu_1)$, and by Lemma~\ref{lem:muc}, we get 
\[- (d(\mu_1')-1)((\mu'_1)_{d(\mu_1)}- d(\mu_1')) + 
(d(\mu_1^r) - 2)((\mu_1^r)'_{d(\mu_1^r)}+ 1- d(\mu_1^r)) = 
d(\mu_1) - (\mu_1')_{d(\mu_1) }. \]
So in all cases, the sums from Lemmas~\ref{lem:redlemma},~\ref{lem:greenlemma},~\ref{lem:teal}, and~\ref{lem:brown} combine to produce 
\[ d(\mu_1) - (\mu_1')_{d(\mu_1) }.\]
\end{remark}

%\begin{remark}
%The sums in the green and teal lemma cancel. 
%\end{remark}

\begin{lemma} %[Cyan Lemma]
\label{lem:cyan}
\begin{eqnarray*}
& &\sum_{d(\mu_2^r)<i\leq\ell(\mu_2^r)}\frac{(\mu_2^r)_i((\mu_2^r)_i-1)}{2} 
-
\sum\limits_{ d(\mu_2) < i \leq \ell( \mu_2 ) } 
\frac{(\mu_2)_i((\mu_2)_i-1)}{2} \\
%&=&
%\begin{cases}
%- \dfrac{(\mu_2 )_{d(\mu_2 ) + 1}((\mu_2 )_{d(\mu_2 ) + 1}-1)}{2} & \mbox{if }(\mu_2)_{d(\mu_2) + 1} = d(\mu_2) \\
%0 & \mbox{otherwise} 
%\end{cases} \\
&=& \begin{cases}
- \dfrac{ d(\mu_2 )
(d(\mu_2 )-1)}{2} & \mbox{if }(\mu_2)_{d(\mu_2) + 1} = d(\mu_2) \\
0 & \mbox{otherwise} 
\end{cases}
\end{eqnarray*}
\end{lemma}

\begin{proof}
%The proof is nearly identical to the proof of Lemma~\ref{lem:graylemma}.
We split into cases based on whether $d(\mu_2^{r}) = d(\mu_2)$ or $d(\mu_2^{r}) = d(\mu_2)-1$, and then we use the fact that $d(\mu_2^{r}) = d(\mu_2)$ if and only if $(\mu_2)_{d(\mu_2)+ 1} = d(\mu_2)$.
\end{proof}

\begin{lemma} %[Blue Lemma]
\label{lem:blue}
\begin{eqnarray*}
& & 
\sum_{1\leq i\leq d(\mu_2^r)}i\left((\mu_2^r)_i - \frac{i+1}{2}\right) - \sum\limits_{1 \leq i \leq d(\mu_2) } (i-1)\left( (\mu_2)_i- \frac{i}{2} \right) \\
&=& \sum\limits_{1\leq i \leq d(\mu_2)-1}( \mu_2)_{i}
- (d(\mu_2 )-1) \left( (\mu_2 )_{ d(\mu_2 )} - \dfrac{ d(\mu_2 ) }{2} \right) \\
& &{}+{}
\begin{cases}
\dfrac{ d(\mu_2 ) (d(\mu_2) - 1)}{2} & \mbox{if } (\mu_2)_{d(\mu_2)+1} = d(\mu_2) \\
0 & \mbox{otherwise} 
\end{cases}
\end{eqnarray*}
\end{lemma}

\begin{proof}
As in the case of Lemma~\ref{lem:cyan}, we split into cases based on whether $d(\mu_2^{r}) = d(\mu_2)$ or $d(\mu_2^{r}) = d(\mu_2)-1$. We omit the details as they are similar to the details of other proofs in this section.
\end{proof}

%\helentodo{reorder so that it goes pink, then blue, then cyan. or just cyan and blue and then pink and grey.}

\begin{lemma} %[Pink Lemma]
\label{lem:pink}
\begin{eqnarray*}
&& \sum_{1\leq i \leq d(\mu_2^c) }(i-2)\left((\mu_2^c)_i - \frac{i-1}{2}\right)- \sum\limits_{1 \leq i \leq d(\mu_2^{rc}) } (i-1)\left( (\mu_2^{rc})_i- \frac{i}{2} \right) \\
&=& 
\begin{cases}
-\sum\limits_{1\leq i \leq d(\mu_2^{rc})}( \mu_2^{c})_{i}
& \mbox{if } (\mu_2)_{d(\mu_2)} = d(\mu_2) \\
-\sum\limits_{1\leq i \leq d(\mu_2^{rc})}( \mu_2^{c})_{i}
+
(d(\mu_2)-2) \left( (\mu_2)_{ d(\mu_2)} - \dfrac{ d(\mu_2) + 1}{2} \right) & \mbox{otherwise}
\end{cases}
\end{eqnarray*}
\end{lemma}

%$+
%$(d(\mu_2^c)-2) \left( (\mu_2^c)_{ d(\mu_2^c)} - \dfrac{ d(\mu_2^c) - 1}{2} \right)- \dfrac{ (d(\mu_2^c) -1) d(\mu_2^c) }{2} 
\begin{proof}
If $d(\mu_2^{rc}) = d(\mu_2^c) - 1$, the difference of sums becomes 
\[
\sum_{1\leq i \leq d(\mu_2^c) }(i-2)\left((\mu_2^c)_i - \frac{i-1}{2}\right)- \sum\limits_{1 \leq i \leq d(\mu_2^{c} )-1 } (i-2)\left( (\mu_2^{rc})_i- \frac{i}{2} \right) -
\sum\limits_{1 \leq i \leq d(\mu_2^{c}) - 1 } \left( (\mu_2^{rc})_i- \frac{i}{2} \right).
\]
Then, using Lemma~\ref{lem:mucandmurc}, we get 
\begin{eqnarray*}
&&\sum_{1\leq i \leq d(\mu_2^c) }(i-2)\left((\mu_2^c)_i - \frac{i-1}{2}\right)- \sum\limits_{1 \leq i \leq d(\mu_2^{c}) -1 } (i-2)\left( (\mu_2^{c})_i + 1- \frac{i}{2} \right) \\
&&{}-{}
\sum\limits_{1 \leq i \leq d(\mu_2^{c}) - 1 } \left( (\mu_2^{c})_i +1 - \frac{i}{2} \right).
\end{eqnarray*}
We can write this as 
\begin{eqnarray*}
& & \sum_{1\leq i \leq d(\mu_2^c) }(i-2)\left((\mu_2^c)_i - \frac{i-1}{2}\right)- 
\sum\limits_{1 \leq i \leq d(\mu_2^{c}) -1 } (i-2)\left( (\mu_2^{c})_i - \frac{i-1}{2} \right) - \sum\limits_{1 \leq i \leq d(\mu_2^{c}) -1 } \frac{ i-2}{2} \\
& & {}-{}
\sum\limits_{1 \leq i \leq d(\mu_2^{c}) - 1 } \left( (\mu_2^{c})_i - \frac{i-2}{2} \right).
\end{eqnarray*}
So, by Remark~\ref{rem:mucdiag} and Lemma~\ref{lem:muc}, the final result is 
\begin{eqnarray*}
&&-\sum_{1\leq i \leq d(\mu_2^{rc})}( \mu_2^{c})_{i}
+
(d(\mu_2^c)-2) \left( (\mu_2^c)_{ d(\mu_2^c)} - \dfrac{ d(\mu_2^c) - 1}{2} \right) \\
&=&-\sum_{1\leq i \leq d(\mu_2^{rc})}( \mu_2^{c})_{i}
+
(d(\mu_2)-2) \left( (\mu_2)_{ d(\mu_2)} -1- \dfrac{ d(\mu_2) - 1}{2} \right) \\
&=&-\sum_{1\leq i \leq d(\mu_2^{rc})}( \mu_2^{c})_{i}
+
(d(\mu_2)-2) \left( (\mu_2)_{ d(\mu_2)} - \dfrac{ d(\mu_2) + 1}{2} \right).
\end{eqnarray*}
If instead $d(\mu_2^{rc}) = d(\mu_2^c)$, the proof is similar.
% but we get an additional term.
%\[- \dfrac{ (d(\mu_2^c) -1) d(\mu_2^c) }{2} \]
%We remark that to write the term in this way, we used the fact that
%$d(\mu_2^{rc}) = d(\mu_2^c)$ if and only if $(\mu_2^c)_{d(\mu_2^{rc}) + 1} = d(\mu_2^c)$. 
%- \frac{(d(\mu_2^c) - 2)(d(\mu_2^c) - 1)}{2}\]
\end{proof}

\begin{lemma} %[Gray Lemma]
\label{lem:graylemma}
\begin{eqnarray*}
& & \sum_{ d(\mu_2^c)<i\leq\ell(\mu_2^c)}\frac{(\mu_2^c)_i((\mu_2^c)_i-1)}{2}
- \sum\limits_{d(\mu_2^{rc}) < i \leq \ell( \mu_2^{rc} ) }
\frac{(\mu_2^{rc})_i((\mu_2^{rc})_i-1)}{2} \\
%&=& 
%\begin{cases}
%\dfrac{(\mu_2^{c})_{d(\mu_2^c) + 1}((\mu_2^{c})_{d(\mu_2^c) + 1}-1)}{2} & \mbox{if }(\mu_2)_{d(\mu_2)} = d(\mu_2) \\
%0 &\mbox{otherwise}
%\end{cases} \\
&=&
\begin{cases}
\dfrac{ (d(\mu_2) -1) (d(\mu_2) -2) }{2} & \mbox{if }(\mu_2)_{d(\mu_2)} = d(\mu_2) \\
0 & \mbox{otherwise}
\end{cases}
\end{eqnarray*}
\end{lemma}

\begin{proof}
The proof is similar to the proof of Lemma~\ref{lem:greenlemma}. We split into cases based on whether $d(\mu_2^{rc}) = d(\mu_2^c)$ or $d(\mu_2^{rc}) = d(\mu_2^c)-1$. 
\end{proof}

\begin{remark}
\label{rem:CBPG}
By combining Lemmas~\ref{lem:cyan} and \ref{lem:blue}, we find that the sums involved result in 
\[
\sum\limits_{1\leq i \leq d(\mu_2)-1}( \mu_2)_{i}
- (d(\mu_2 )-1) \left( (\mu_2 )_{ d(\mu_2 )} - \dfrac{ d(\mu_2 ) }{2} \right)
\]
in all cases. By Lemma~\ref{lem:muc}, 
\[ \sum\limits_{1\leq i \leq d(\mu_2)-1}( \mu_2)_{i} - 
\sum\limits_{1\leq i \leq d(\mu^{rc}_2)}( \mu^c_2)_{i} = d(\mu_2^{rc}). \]
So, when we combine the sums from Lemmas~\ref{lem:cyan},~\ref{lem:blue}, and~\ref{lem:pink}, there are two cases. If $(\mu_2 )_{ d(\mu_2 )} = d(\mu_2 )$, then we get 
\[
d(\mu_2^{rc}) - (d(\mu_2 )-1) \left( \dfrac{ d(\mu_2 ) }{2}\right)
= 
\dfrac{ (d(\mu_2 )-1) (2 - d(\mu_2) ) }{2}.
\]
Otherwise, we get 
\begin{eqnarray*}
&&d(\mu_2^{rc})- (d(\mu_2 )-1) \left( (\mu_2 )_{ d(\mu_2 )} - \dfrac{ d(\mu_2 ) }{2} \right)+(d(\mu_2)-2) \left( (\mu_2)_{ d(\mu_2)} - \dfrac{ d(\mu_2) + 1}{2} \right) \\
&=& d(\mu_2^{rc}) -(\mu_2)_{d(\mu_2)}+ 1 =  d(\mu_2)-(\mu_2)_{d(\mu_2)}.
\end{eqnarray*}
So when we include the sums from Lemma~\ref{lem:graylemma}, we get 
\[
\begin{cases}
0 & \mbox{if } (\mu_2)_{d(\mu_2)} = d(\mu_2) \\
d(\mu_2)-(\mu_2)_{d(\mu_2)} & \mbox{otherwise}.
\end{cases}
\]
So in all cases, the sums from Lemmas~\ref{lem:cyan},~\ref{lem:blue},~\ref{lem:pink}, and~\ref{lem:graylemma} combine to produce 
$$d(\mu_2)-(\mu_2)_{d(\mu_2)}.$$
\end{remark}

%\begin{remark}
%The red lemma has been checked with code for all partitions of $n$, %where $n \leq 30$.
%\end{remark}

%\begin{remark}
%The green lemma has been checked with code for all partitions of $n$, %where $n \leq 30$.
%\end{remark}
%\begin{remark}
%The teal lemma and brown lemma have been checked with code for all %partitions of $n$, where $n \leq 30$.
%\end{remark}

%\begin{remark}
%Gray Lemma has been checked with code for all partitions of 30. 
%\end{remark}

%\begin{remark}
%The eight colored lemmas have been checked with code for all partitions of $n$, where $n \leq 30$.
%\end{remark}

% to simplify the algebra. 

Now we subtract the terms in $A$ that do not involve $N$ from the terms in $C$ that do not involve $N$, and include the difference $-1$ of the terms involving $N$: 
\begin{eqnarray*}
C-A&=& -1 + \left|\mu_1^r\right|+\left|\mu_2^c\right| - \left|\mu_3\right| - \left|\mu_1^c\right|- \left|\mu_2^r\right| + \left|\mu_3\right| \\
& & 
{}+{} \sum_{1\leq i\leq d((\mu_1^r)')}(i-2)((\mu_1^r)'_i-(i-1)) 
+\sum_{d((\mu_1^r)')<i\leq\ell((\mu_1^r)')}(\mu_1^r)'_i((\mu_1^r)'_i-(i-1)) \\
& & {}+{} \sum_{1\leq i \leq d(\mu_2^c) }(i-2)\left((\mu_2^c)_i - \frac{i-1}{2}\right)
+ \sum_{ d(\mu_2^c)<i\leq\ell(\mu_2^c)}\frac{(\mu_2^c)_i((\mu_2^c)_i-1)}{2} \\
& & {}+{}\sum_{1\leq i\leq d(\mu_1^c)}i((\mu_1^c)'_i-i-1) 
+ \sum_{d(\mu_1^c) <i\leq\ell((\mu_1^c)')}(\mu_1^c)'_i((\mu_1^c)'_i-i-1) \\
& & {}+{}\sum_{1\leq i\leq d(\mu_2^r)}i\left((\mu_2^r)_i - \frac{i+1}{2}\right) 
+ \sum_{d(\mu_2^r)<i\leq\ell(\mu_2^r)}\frac{(\mu_2^r)_i((\mu_2^r)_i-1)}{2} \\
&&{}-{} \sum\limits_{1 \leq i \leq d(\mu_1) } (i-1)((\mu'_1)_i - i)
-
\sum\limits_{ d(\mu_1) < i \leq \ell( \mu'_1) } (\mu'_1)_i ((\mu'_1)_i - i) \\
&&{}-{} \sum\limits_{1 \leq i \leq d(\mu_2) } (i-1)\left( (\mu_2)_i- \frac{i}{2} \right) 
- \sum\limits_{ d(\mu_2) < i \leq \ell( \mu_2 ) }
\frac{(\mu_2)_i((\mu_2)_i-1)}{2} \\
&&{}-{} \sum\limits_{1 \leq i \leq d(\mu_1^{rc}) } (i-1)( (\mu_1^{rc})_i' - i)
- 
\sum\limits_{d(\mu_1^{rc}) < i \leq \ell( (\mu_1^{rc})' ) } (\mu_1^{rc})_i' ( (\mu_1^{rc})_i' - i) \\
& & {}-{} \sum\limits_{1 \leq i \leq d(\mu_2^{rc}) } (i-1)\left( (\mu_2^{rc})_i- \frac{i}{2} \right) -
\sum\limits_{d(\mu_2^{rc}) < i \leq \ell( \mu_2^{rc} ) }
\frac{(\mu_2^{rc})_i((\mu_2^{rc})_i-1)}{2} \\
& = & -1 + \left|\mu_1^r\right|+\left|\mu_2^c\right| - \left|\mu_3\right| - \left|\mu_1^c\right|- \left|\mu_2^r\right| + \left|\mu_3\right|\\
&& {}+{} 
\underbrace{ \sum_{d((\mu_1^r)')<i\leq\ell((\mu_1^r)')}(\mu_1^r)'_i((\mu_1^r)'_i-(i-1)) - 
\sum\limits_{d(\mu_1^{rc}) < i \leq \ell( (\mu_1^{rc})' ) } (\mu_1^{rc})_i' ( (\mu_1^{rc})_i' - i) }_{\text{Remark}~\ref{rem:browntealgreen}} \\
& & {}+{} \underbrace{ \sum_{1\leq i\leq d((\mu_1^r)')}(i-2)((\mu_1^r)'_i-(i-1)) - \sum\limits_{1 \leq i \leq d(\mu_1^{rc}) } (i-1)( (\mu_1^{rc})_i' - i)}_{\text{Remark}~\ref{rem:browntealgreen}} \\
& & 
{}+{} \underbrace{\sum_{1\leq i\leq d(\mu_1^c)}i((\mu_1^c)'_i-i-1) 
- \sum\limits_{1 \leq i \leq d(\mu_1) } (i-1)((\mu'_1)_i - i)}_{\text{Remark}~\ref{rem:browntealgreen}}
\\
&& {}+{} \underbrace{ \sum_{d(\mu_1^c) <i\leq\ell((\mu_1^c)')}(\mu_1^c)'_i((\mu_1^c)'_i-i-1) 
-
\sum\limits_{ d(\mu_1) < i \leq \ell( \mu'_1) } (\mu'_1)_i ((\mu'_1)_i - i)}_{\text{Remark}~\ref{rem:browntealgreen}} \\
& & {}+{} \underbrace{ \sum_{d(\mu_2^r)<i\leq\ell(\mu_2^r)}\frac{(\mu_2^r)_i((\mu_2^r)_i-1)}{2} - \sum\limits_{ d(\mu_2) < i \leq \ell( \mu_2 ) } 
\frac{(\mu_2)_i((\mu_2)_i-1)}{2} }_{\text{Remark}~\ref{rem:CBPG} } \\
& & {}+{} \underbrace{\sum_{1\leq i\leq d(\mu_2^r)}i\left((\mu_2^r)_i - \frac{i+1}{2}\right) 
- \sum\limits_{1 \leq i \leq d(\mu_2) } (i-1)\left( (\mu_2)_i- \frac{i}{2} \right) }_{\text{Remark}~\ref{rem:CBPG} }\\
& &{}+{}\underbrace{\sum_{1\leq i \leq d(\mu_2^c) }(i-2)\left((\mu_2^c)_i - \frac{i-1}{2}\right) - \sum\limits_{1 \leq i \leq d(\mu_2^{rc}) }  (i-1)\left( (\mu_2^{rc})_i- \frac{i}{2} \right) }_{\text{Remark}~\ref{rem:CBPG} }\\
& & {}+{} \underbrace{ \sum_{ d(\mu_2^c)<i\leq\ell(\mu_2^c)}\frac{(\mu_2^c)_i((\mu_2^c)_i-1)}{2} - 
\sum\limits_{d(\mu_2^{rc}) < i \leq \ell( \mu_2^{rc} ) }
\frac{(\mu_2^{rc})_i((\mu_2^{rc})_i-1)}{2} }_{\text{Remark}~\ref{rem:CBPG} }\\
& = & -1 + \left|\mu_1^r\right|+\left|\mu_2^c\right| - \left|\mu_3\right| - \left|\mu_1^c\right|- \left|\mu_2^r\right| + \left|\mu_3\right| \\
&&{}+{} d(\mu_1) - (\mu_1')_{d(\mu_1) } + 
d(\mu_2)-(\mu_2)_{d(\mu_2)} \\
&= & -1 + d(\mu_1) - (\mu_1)_{d(\mu_1) } + 
d(\mu_2)-(\mu'_2)_{d(\mu_2)} = -K.
\end{eqnarray*}
In the last step we used the fact that $ \left|\mu^r\right|- |\mu^c| = \mu'_{d(\mu)} - \mu_{d(\mu)}$ (see Remark~\ref{rem:sizeofmuc}).

\subsection{PT weights}

In this section we compute the constants $A$, $B$, and $C$ from equation (\ref{eqn:intermediatecond}) in Section~\ref{sec:pt_condensation_identity}. To that end, in Section~\ref{sec:baseweight} we compute the edge-weights of the base$_{\mu}$, base$_{up}$, and base$_{down}$ double-dimer configurations. As in previous sections, we assume $N\geq M$. The remaining work is to compute $C - A$; this is done in Section~\ref{sec:PTalg}.

\subsubsection{Edge-weight of base double-dimer configuration}
\label{sec:baseweight}

In this section we compute the edge-weights of the base$_{\mu}$ double-dimer configuration and the base$_{up}$ and base$_{down}$ configurations. We prove our formula for the base$_{\mu}$ configuration, but omit the proofs for the base$_{up}$ and base$_{down}$ configurations because they are essentially the same, as the base$_{up}$ and base$_{down}$ configurations only differ from base$_{\mu}$ configurations by shifts. 

The edge-weight of the base$_{\mu}$ double-dimer configuration is given by the following lemma. 

\begin{lemma}
\label{cor:ptweight}
The edge-weight of the base$_{\mu}$ double-dimer configuration is $q^{w_{base}(\mu)}$, where 
\small
\begin{eqnarray*}
w_{base}(\mu) &=& \frac{N^2(N-1)}{2} + \sum\limits_{i=1}^{N - \ell(\mu_1') -1} \frac{(N - \ell(\mu_1') - 1)(N - \ell(\mu_1'))}{2} - \frac{(i-1)i}{2} \\
& & {}+{} \sum\limits_{i: (\mu_1')_i \geq i\geq 1 }
(N-(\mu_1')_i)(N + (\mu_1')_i - i) +
\sum\limits_{i: (\mu_1')_i < i\leq\ell(\mu_1') }
(N-i)(N + (\mu_1')_i -i ) \\ 
& & {}+{} 
\sum\limits_{i=1}^{N-1 - \ell(\mu_2) } (N+i-1)(N-i - \ell(\mu_2)) \\
& & {}+{} \sum\limits_{i: (\mu_2)_i \geq i\geq 1 }
((\mu_2)_i + N)(-(\mu_2)_i + N) + \frac{(N - (\mu_2)_i -1)(N - (\mu_2)_i)}{2} \\
&& {}+{} 
\sum\limits_{i: (\mu_2)_i < i\leq\ell(\mu_2) } 
(N - i)((\mu_2)_i + N) + \frac{(N - i-1)(N - i)}{2} + \sum\limits_{i=1}^{\ell(\mu_3) }
(N-i)(\mu_3)_i.
\end{eqnarray*}
\normalsize
\end{lemma}

We start by showing that the formula holds for $\mu_1 = \mu_2 = \mu_3 = \emptyset$. While this is not necessary to prove Lemma~\ref{cor:ptweight}, this special case will make the proof of Lemma~\ref{cor:ptweight} easier to understand.

\begin{lemma}
\label{lem:allempty}
The edge-weight of the base$_{\emptyset, \emptyset, \emptyset}$ double-dimer configuration is $q^{w_{base}(\emptyset, \emptyset, \emptyset)}$, where
%When $r= s = t = N$ and $\mu_1 = \mu_2 = \mu_3 = \emptyset$, the weight of the double-dimer configuration described above is
\[
w_{base}(\emptyset, \emptyset, \emptyset) = \frac{N^2(N-1)}{2} + \sum\limits_{i=1}^{N-1} ( N+i-1)(N-i) + \sum\limits_{i=1}^{N-1} \frac{(N-1)N}{2}- \frac{(i-1)i}{2}. \]
\end{lemma}

\begin{proof}
By Definition~\ref{defn:baseDD}, the base$_{\emptyset, \emptyset, \emptyset}$ double-dimer configuration is $D_{(\III, \II\cup\III)}(N)=D_{(\varnothing, \varnothing)}(N)$, i.e., it corresponds to the $AB$ configuration $(\varnothing, \varnothing)$. So, we have the tilings and double-dimer configuration shown in Figure~\ref{fig:DDweightexample} for $N=5$: 

\begin{figure}[htb]
\centering
\includegraphics[width=0.3\textwidth]{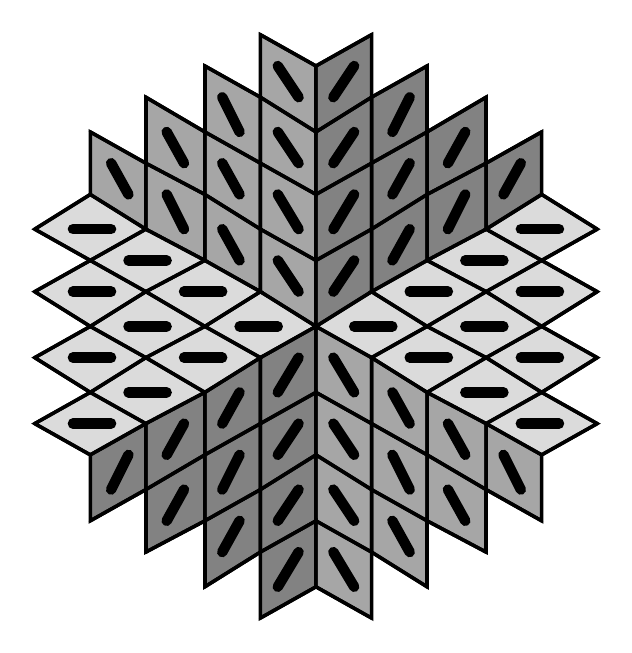}\hfill
\includegraphics[width=0.3\textwidth]{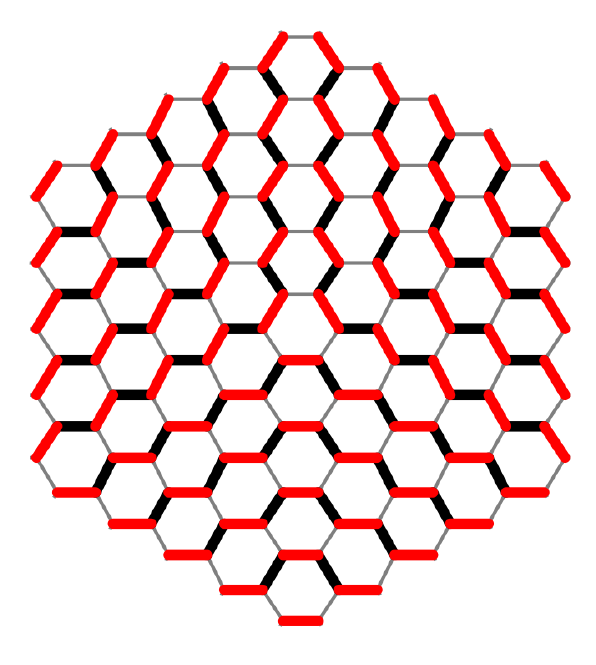}\hfill
\includegraphics[width=0.3\textwidth]{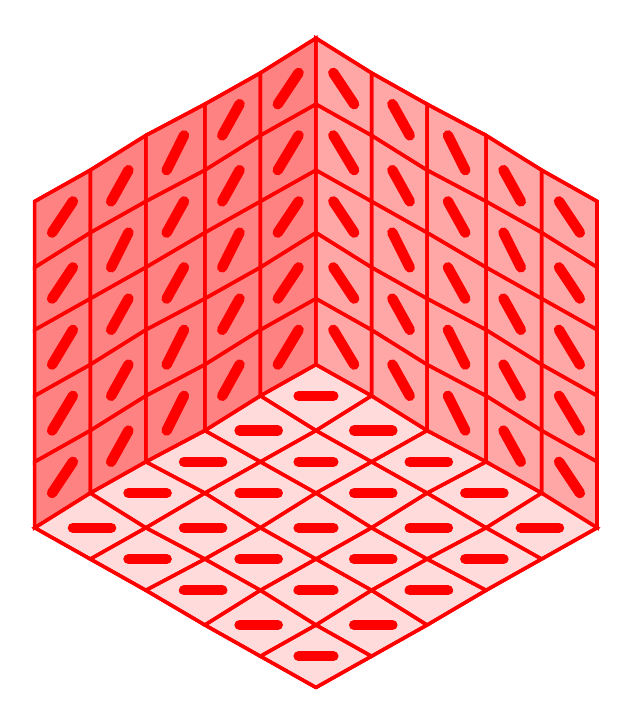}
\caption{The $AB$ configuration $(\varnothing, \varnothing)$, in the case where $\mu_1 = \mu_2 = \mu_3 = \emptyset$. Left: The tiling corresponding to $A$. Right: The tiling corresponding to $B$. Center: The corresponding double-dimer configuration.}
\label{fig:DDweightexample}
\end{figure}

Referring to Figure~\ref{fig:DDweightexample}, we see that the only horizontal dimers from the $B$ configuration are in sector 3, and these horizontal dimers contribute weight 
$$(q^{0})^{N} (q^{1})^{N} \cdots (q^{N-1})^{N} = q^{N^2 (N-1)/2}.$$ 
The horizontal dimers from the $A$ configuration in sector 2 contribute weight 
$$(q^{N})^{N-1} (q^{N+1})^{N-2} \cdots (q^{2N-2})^{1} = \prod\limits_{i=1}^{N-1} (q^{N+i-1})^{N-i}.$$ 
The horizontal dimers from the $A$ configuration in sector 1 contribute weight 
\[
(q \cdot q^2 \cdots q^{N-1} ) 
(q^2 \cdots q^{N-1} ) 
(q^3 \cdots q^{N-1} ) \cdots
q^{N-1} = \prod\limits_{i=1}^{N-1} q^{ \frac{(N-1)N}{2}- \frac{(i-1)i}{2} }.\]
There are no horizontal dimers from the $A$ configuration in sector 3. 
\end{proof}

\begin{proof}[Proof of Lemma~\ref{cor:ptweight}]
This proof has three parts. We first show that the horizontal dimers in sector 1 have weight 
\begin{eqnarray*}
& & \prod\limits_{i=1}^{N - \ell(\mu_1') -1} q^{ \frac{(N- \ell(\mu_1') - 1)(N - \ell(\mu_1'))}{2}- \frac{(i-1)i}{2} } \prod\limits_{i: (\mu_1')_i \geq i\geq 1 }
(q^{(\mu_1')_i } q^{N - i} )^{N -(\mu_1')_i } \\
& & {}\cdot{}\prod\limits_{i: (\mu_1')_i < i\leq\ell(\mu_1') }
(q^{(\mu_1')_i } q^{N-i} )^{N - i} \\
& = & 
\prod\limits_{i=1}^{N - \ell(\mu_1') -1} q^{ \frac{(N- \ell(\mu_1') - 1)(N - \ell(\mu_1'))}{2}- \frac{(i-1)i}{2} } \prod\limits_{i: (\mu_1')_i \geq i\geq 1 }
q^{ (N-(\mu_1')_i)(N + (\mu_1')_i - i) } \\
& & {}\cdot{}\prod\limits_{i: (\mu_1')_i < i\leq\ell(\mu_1') } q^{ (N-i)( N + (\mu_1')_i -i ) }.
\end{eqnarray*}
Note that when $\ell(\mu_1') = 0$, this formula agrees with the third term in Lemma~\ref{lem:allempty}. We will next show that the horizontal dimers in sector 2 have weight 
\begin{eqnarray*}
& & \prod\limits_{i=1}^{N-1 - \ell(\mu_2) } (q^{N+i-1})^{N-i - \ell(\mu_2)} \prod\limits_{i: (\mu_2)_i \geq i\geq 1 } \left(
(q^{(\mu_2)_i })^{N - (\mu_2)_i }
\prod\limits_{j=1}^{N - (\mu_2)_i } q^{N+j -1 } \right) \\
& & {}\cdot{} 
\prod\limits_{i: (\mu_2)_i < i\leq\ell(\mu_2) } \left(
(q^{(\mu_2)_i })^{N - i }
\prod\limits_{j=1}^{N - i } q^{N+j-1 } \right) \\
& = & \prod\limits_{i=1}^{N-1 - \ell(\mu_2) } (q^{N+i-1})^{N-i - \ell(\mu_2)} 
\prod\limits_{i: (\mu_2)_i \geq i\geq 1 } 
q^{ ((\mu_2)_i + N )(N - (\mu_2)_i )} 
q^{\frac{(N - (\mu_2)_i -1)(N - (\mu_2)_i)}{2} } \\
& & {}\cdot{} \prod\limits_{i: (\mu_2)_i < i\leq\ell(\mu_2) }
%\left(
%(q^{(\mu_2)_i })^{N - i }
%\prod\limits_{i=1}^{N - i } q^{N+i -1 } \right)
q^{((\mu_2)_i + N )(N - i) } q^{\frac{(N - i-1)(N - i)}{2}}.
\end{eqnarray*}
Again, when $\ell(\mu_2) = 0$, this formula agrees with the second term in Lemma~\ref{lem:allempty}. Finally, we will show that the horizontal dimers in sector 3 have weight 
\[ q^{N^2(N-1)/2}
\prod_{i=1}^{\ell(\mu_3) }
(q^{N-i})^{(\mu_3)_i}.
\]

We remark that since the base$_{\mu}$ double-dimer configuration is $D_{(\III, \II\cup\III)}(N)$, the horizontal dimers from $M_A(N)$ in sector $i$ can be completely explained by the partition $\mu_i$. Also, as in the proof of Lemma~\ref{lem:allempty}, the only horizontal dimers from $M_B(N)$ are in sector 3. \\

\noindent{\bf Sector 1.} In the case where $\mu_1=\emptyset$, we can partition the horizontal dimers from $M_A(N)$ in sector 1 (see Figure~\ref{fig:DDweightexample}) into $N-1$ groups: 
\begin{enumerate}
\item[(1)] The group of horizontal dimers that consists of the topmost horizontal dimer in each column of hexagons in sector 1. This is a group of $N-1$ dimers that each have weight $q^{N-1}$. 
\item[(2)] The group of horizontal dimers that consists of the horizontal dimers directly below the dimers in group 1. Since the leftmost dimer in group 1 does not have a horizontal dimer directly below it, this is a group of $N-2$ dimers that each have weight $q^{N-2}$. 
\item[(3)] Etc. 
\end{enumerate}
In general, group $i$ consists of the $N-i$ horizontal dimers directly below the dimers in group $i-1$ (with the exception of the leftmost dimer in group $i-1$, which does not have a horizontal dimer directly below it). The dimers in group $i$ all have weight $q^{N-i}$. 
 
Now that we have partitioned the dimers in this way, we are ready to discuss the case where $\mu_1 \neq \emptyset$. Consider $(\mu'_1)_1$. When $(\mu'_1)_1 > 0$ (compared to $(\mu'_1)_1 = 0$), the dimers in group 1 (i.e. the $N-1$ dimers with weight $q^{N-1}$) shift up $(\mu'_1)_1$ units. However, some of the dimers in group 1 shift outside $H(N)$. Specifically, the $i$th dimer from the leftmost dimer (so the leftmost dimer corresponds to $i = 0$) is still in $H(N)$ if and only if $i \leq N- (\mu'_1)_1 -1$. In total, there are $N - (\mu'_1)_1$ dimers inside $H(N)$ after this shift and these dimers each have weight $q^{(\mu'_1)_1}q^{N-1}$. 

In general, the $i$th part of $\mu_1'$ affects the weight of group $i$. For $i > \ell(\mu_1')$, the weight of group $i$ is unaffected, and the product of all such weights is $\prod\limits_{i=1}^{N - \ell(\mu_1') -1} q^{ \frac{(N- \ell(\mu_1') - 1)(N - \ell(\mu_1'))}{2}- \frac{(i-1)i}{2} }$. To determine the effect of the $i$th part of $\mu_1'$ on the weight of group $i$, we break into cases. If $(\mu_1')_i \geq i$, then as in the case where $i =1$, after the dimers in group $i$ shift up, there are $N- (\mu_1')_i$ dimers still in $H(N)$, each with weight $q^{(\mu'_1)_i}q^{N-i}$. If $(\mu_1')_i < i$, then after the dimers in group $i$ shift up, there are $N- i$ dimers still in $H(N)$, each with weight $q^{(\mu'_1)_i}q^{N-i}$. Therefore, the total weight of the dimers in sector 1 is 
\[ \prod\limits_{i=1}^{N - \ell(\mu_1') -1} q^{ \frac{(N- \ell(\mu_1') - 1)(N - \ell(\mu_1'))}{2}- \frac{(i-1)i}{2} } \prod\limits_{i: (\mu_1')_i \geq i\geq 1 } (q^{(\mu_1')_i } q^{N - i} )^{N -(\mu_1')_i }
\prod\limits_{i: (\mu_1')_i < i\leq\ell(\mu_1') }
(q^{(\mu_1')_i } q^{N-i} )^{N - i }.\]

\noindent{\bf Sector 2.} As we did in sector 1, we partition the horizontal dimers from $M_A(N)$ in sector 2 into $N-1$ groups: 
\begin{enumerate}
\item[(1)] The group of horizontal dimers that consists of the topmost horizontal dimer in each column. This is a group of $N-1$ dimers with weights $q^{N}, q^{N+1}, \ldots, q^{2N-2}$. 
\item[(2)] The group of horizontal dimers that consists of the horizontal dimers directly below the dimers in group 1. Since the rightmost dimer in group 1 does not have a horizontal dimer directly below it, this is a group of $N-2$ dimers with weights $q^N, q^{N+1}, \ldots, q^{2N-3}$. 
\item[(3)] Etc. 
\end{enumerate}
In general, group $i$ consists of the $N-i$ horizontal dimers directly below the dimers in group $i-1$ and these dimers have weights $q^N, q^{N+1}, \ldots, q^{2N-1-i}$. 

As in sector 1, the $i$th part of $\mu_2$ affects the weight of group $i$, because the dimers in group $i$ shift up $(\mu_2)_i$ units. For $i > \ell(\mu_2)$, the weight of group $i$ is unaffected, and the product of all such weights is $\prod\limits_{i=1}^{N-1 - \ell(\mu_2) } (q^{N+i-1})^{N-i - \ell(\mu_2)}$. To determine the effect of the $i$th part of $\mu_2$ on the weight of group $i$, we break into cases. If $(\mu_2)_i \geq i$, then the dimer in group $i$ with weight $q^{N+j}$ is still in $H(N)$ after being shifted if and only if $N+j + (\mu_2)_i \leq 2N-1$, that is, if and only if $j \leq N - (\mu_2)_i - 1$. So after the dimers in group $i$ are shifted, there are $N- (\mu_2)_i$ dimers still in $H(N)$, and these dimers have weights $q^{(\mu_2)_i} q^{N}, q^{(\mu_2)_i} q^{N+1}, \ldots, q^{(\mu_2)_i} q^{2N-1-(\mu_2)_i}$. If $(\mu_2)_i < i$, then after the dimers in group $i$ are shifted, there are $N- i$ dimers still in $H(N)$, and these dimers also have weights $q^{(\mu_2)_i} q^{N}, q^{(\mu_2)_i} q^{N+1}, \ldots, q^{(\mu_2)_i}  q^{2N-1-i}$. 
%after the edges from group $i$ are shifted there are $N- i$ edges, which again have weights $q^{(\mu_2)_i} q^{N+i-1}, \ldots, q^{(\mu_2)_i}  q^{2N-1}$. 
Therefore, the total weight of the dimers in sector 2 is 
\begin{eqnarray*}
&&\prod\limits_{i=1}^{N-1 - \ell(\mu_2) } (q^{N+i-1})^{N-i - \ell(\mu_2)}
\prod\limits_{i: (\mu_2)_i \geq i\geq 1 } \left(
(q^{(\mu_2)_i })^{N - (\mu_2)_i }
\prod\limits_{j=1}^{N - (\mu_2)_i } q^{N+j -1 } \right) \\
&&{}\cdot{} 
\prod\limits_{i: (\mu_2)_i < i\leq\ell(\mu_2) } \left(
(q^{(\mu_2)_i })^{N - i }
\prod\limits_{j =1}^{N - i } q^{N+j -1 } \right). 
\end{eqnarray*}
 
\noindent{\bf Sector 3.} Recall from the proof of Lemma~\ref{lem:allempty} that the horizontal dimers from $M_B(N)$ in sector 3 have weight $q^{\frac{N^2(N-1)}{2}}$. In the case where $\mu_3 = \emptyset$, there are no horizontal dimers from $M_A(N)$ in sector 3. When $\mu_3 \neq \emptyset$, there are $(\mu_3)_i$ horizontal dimers from $M_A(N)$ in sector 3, each of weight $q^{N-i}$. This gives us the desired formula. 
\end{proof}

We conclude this section with expressions for the edge-weights of the base$_{up}$ and base$_{down}$ configurations. 

\begin{lemma}
\label{cor:ptweight_SU}
The edge-weight of the base$_{up}$ double-dimer configuration is $q^{w_{up}}$, where 
%Suppose $N \geq \max\{ \ell((\mu_1^r)), ((\mu_1^r))_1,\ell(\mu_2^c), (\mu_2^c)_1 ,\ell(\mu_3), (\mu_3)_1 \}$. Then the weight of the double-dimer configuration on $(G, {\bf N} - \{a, d\})$ corresponding to the $AB$ configuration where $B$ contains all type $\II$ and type $\III$ boxes and $A$ contains all type $\III$ boxes is $q^{w_{\min}}$, where
\small
\begin{eqnarray*}
w_{up} &=& \frac{(N+1)N(N-1)}{2} + N^2 + 
\sum\limits_{i=1}^{N - \ell((\mu_1^r)') -1} \frac{(N- \ell((\mu_1^r)') + 1)(N - \ell((\mu_1^r)'))}{2}- \frac{i(i+1)}{2} \\
& & {}+{} 
\begin{cases}
0 &\text{if }(\mu_1^r)'=\emptyset \\
\sum\limits_{i: (\mu_1^r)'_i \geq i-1\geq 0 }(N- (\mu_1^{r})'_i - 1)( (\mu_1^{r})'_i + N-i+1) &\text{otherwise}
\end{cases}\\
& & {}+{}
%\cdot q^{N - i} )^{N -((\mu_1^r)')_i }
\sum\limits_{i: ((\mu_1^r)')_i < i-1\leq\ell((\mu_1^r)')-1 }
(N-i)( (\mu_1^{r})'_i + N-i+1 ) \\
& & {}+{} %\frac{ -\ell(\mu_2^c)^3 + 6\ell(\mu_2^c)^2N - 3\ell(\mu_2^c)N^2 + 4N^3 %- 3\ell(\mu_2^c)^2 + 9\ell(\mu_2^c)N - 6N^2 - 2\ell(\mu_2^c) + 2N }{6} 
\sum\limits_{i=1}^{N-1 - \ell(\mu_2^c) } (N+i)(N-i - \ell(\mu_2^c)) \\
& & {}+{} 
\begin{cases}
0 &\text{if }\mu_2^c=\emptyset \\
\sum\limits_{i: (\mu_2^c)_i \geq i -1\geq 0}((\mu_2^c)_i + N)(-(\mu_2^c)_i + N -1) +
% \frac{(N - (\mu_2^c)_i -1)(N - (\mu_2^c)_i)}{2} +
\frac{(N - (\mu_2^c)_i -1)(N - (\mu_2^c)_i)}{2} &\text{otherwise}
\end{cases}\\
& & {}+{}
\sum\limits_{i: (\mu_2^c)_i < i-1\leq\ell(\mu_2^c)-1 }
(N - i) ((\mu_2^c)_i + N ) + \frac{(N - i+1)(N - i)}{2} \\
& & {}+{} \sum\limits_{i=1}^{\ell(\mu_3) }
(N+1-i)(\mu_3)_i.
\end{eqnarray*}
\normalsize
\end{lemma}

\begin{lemma}
\label{cor:ptweight_SD}
The edge-weight of the base$_{down}$ double-dimer configuration is $q^{w_{down}}$, where 
\small
\begin{eqnarray*}
w_{down} &=& \frac{(N-1)^2(N-2)}{2} 
+ \dfrac{(N - \ell((\mu_1^c)') -2)(N - \ell((\mu_1^c)') -1) }{2} \\
& & {}+{} \sum\limits_{i=1}^{N - \ell((\mu_1^c)') -2} \frac{(N- \ell((\mu_1^c)')- 1)(N - \ell((\mu_1^c)')-2)}{2}- \frac{(i-1)i}{2} \\
& & {}+{} \sum\limits_{i: (\mu_1^c)'_i > i+1>1 }
(N- (\mu_1^c)'_i +1) ((\mu_1^c)'_i + N - i - 1) \\
%\cdot q^{N - i} )^{N -((\mu_1^c)')_i }
&& {}+{} \sum\limits_{i: (\mu_1^c)'_i \leq i+1\leq\ell((\mu_1^c)')+1 }
(N-i) ((\mu_1^c)'_i + N-i -1) \\
& & {}+{} %\frac{ -\ell(\mu_2^r)^3 + 6\ell(\mu_2^r)^2N - 3\ell(\mu_2^r)N^2 + 4N^3 %- 3\ell(\mu_2^r)^2 + 9\ell(\mu_2^r)N - 6N^2 - 2\ell(\mu_2^r) + 2N }{6} 
\sum\limits_{i=1}^{N-1 - \ell(\mu_2^r) } (N+i-2)(N-i - \ell(\mu_2^r)) \\
& & {}+{} \sum\limits_{i: (\mu_2^r)_i > i +1>1}
((\mu_2^r)_i + N -1)(-(\mu_2^r)_i +N + 1) + 
\frac{(N - (\mu_2^r)_i + 1)(N - (\mu_2^r)_i)}{2} \\
& & {}+{} 
\sum\limits_{i: (\mu_2^r)_i \leq i+1\leq\ell(\mu_2^r)+1 }
(N - i) ( (\mu_2^r)_i + N - 1) + \frac{(N - i-1)(N - i)}{2} \\
& & {}+{} \sum\limits_{i=1}^{\ell(\mu_3) }
(N-1-i)(\mu_3)_i.
\end{eqnarray*}
\normalsize
\end{lemma}

\subsubsection{Algebraic simplification}
\label{sec:PTalg}

Since
$A = w_{base}(\mu_1, \mu_2, \mu_3) + w_{base}(\mu_1^{rc}, \mu_2^{rc}, \mu_3)$ and
$B = w_{base}(\mu_1^{rc}, \mu_2, \mu_3) + w_{base}(\mu_1, \mu_2^{rc}, \mu_3)$, we see that $A = B$. In addition, $C = w_{up}+w_{down}$. 

To compute $C-A$, we split the algebra into two pieces: we first simplify the sums that have index set going from $1$ to a fixed ending point that does not depend on $\mu$, and then we simplify the remaining summands.
 
%If we let $w^{bc}_{\min}$ denote $w_{\min}$ from Corollary~\ref{cor:ptweight_SD}, and define $w^{ad}_{\min}$, $w^{abcd}_{\min}$, etc similarly, we must compute
%\[w^{bc}_{\min} + w^{ad}_{\min} - w^{abcd}_{\min} -  w_{\min}\]

\begin{remark}
Since
\[
\sum\limits_{i=1}^{\ell(\mu_3) }
(N+1 -i)(\mu_3)_i + 
\sum\limits_{i=1}^{\ell(\mu_3) }
(N-1 -i)(\mu_3)_i -
2 \sum\limits_{i=1}^{\ell(\mu_3) }
(N-i)(\mu_3)_i = 0,\]
the terms involving $\mu_3$ cancel. 
\end{remark}
 
% \subsubsection{Simplification of easy terms}
 
% \label{sec:easyterms}
 
\begin{lemma}
\label{lem:ptfirsttwoterms}
\begin{eqnarray*}
&&\sum\limits_{i=1}^{N - \ell((\mu_1^r)') -1} \dfrac{(N- \ell((\mu_1^r)') + 1)(N - \ell((\mu_1^r)'))}{2}
- \dfrac{i(i+1)}{2} \\
&&{}-{}\sum\limits_{i=1}^{N - \ell((\mu_1^{rc})') -1} \dfrac{(N- \ell((\mu_1^{rc})') - 1)(N - \ell((\mu_1^{rc})'))}{2} - \dfrac{(i-1)i}{2} \\
&=&
\begin{cases}
- \dfrac{(N-\ell((\mu_1')^c) )(N - \ell((\mu_1')^c) +1 )}{2} & \mbox{if } d(\mu_1') > 1 \text{ or } (d(\mu_1') = 1 \text{ and } (\mu_1')_1 > 1) \\
\dfrac{N(N-1)}{2} & \mbox{if } d(\mu_1') = 1 \text{ and } (\mu_1')_1 = 1
\end{cases}
\end{eqnarray*}
\end{lemma}

The terms in the lemma are from $w_{up}$ and $w_{base}(\mu_1^{rc}, \mu_2^{rc}, \mu_3)$, respectively. 
 
\begin{proof}
Recall that $(\mu_1^{rc})' = (\mu_1')^{rc}$, and $(\mu_1^r)' = (\mu_1')^c$. For convenience, we write $\lambda := \mu_1'$. There are two cases to consider. The first is when $\ell(\lambda^{rc}) = \ell(\lambda^c) -1$. This occurs precisely when $d(\lambda) > 1$ or $d(\lambda) = 1$ and $\lambda_1 > 1$. In this case, we can write the second sum as 
\[\sum\limits_{i=1}^{N - \ell(\lambda^c) }  \dfrac{(N-\ell(\lambda^c) )(N - \ell(\lambda^c) +1  )}{2} - \dfrac{(i-1)i}{2}. \]
Now we see that 
\begin{eqnarray*}
& & \sum\limits_{i=1}^{N - \ell(\lambda^c) -1} \dfrac{(N- \ell(\lambda^c) + 1)(N - \ell(\lambda^c) )}{2}
- 
\sum\limits_{i=1}^{N - \ell(\lambda^c) } \dfrac{(N-\ell(\lambda^c) )(N - \ell(\lambda^c) +1 )}{2} \\
&= & - \dfrac{(N-\ell(\lambda^c) )(N - \ell(\lambda^c) +1 )}{2}. 
\end{eqnarray*}
We have
\[ \sum\limits_{i=1}^{N - \ell(\lambda^c) -1} - \dfrac{i(i+1)}{2} - \sum\limits_{i=1}^{N - \ell(\lambda^c)} 
- \dfrac{(i-1)i}{2} =
\sum\limits_{i=1}^{N - \ell(\lambda^c) -1} - \dfrac{i(i+1)}{2} - \sum\limits_{i=0}^{N - \ell(\lambda^c)-1} 
- \dfrac{i(i+1)}{2} =0.
\]

%\[ \sum\limits_{i=1}^{N - \ell(\lambda^c) -1} - \dfrac{i(i+1)}{2}  - \sum\limits_{i=1}^{N - \ell(\lambda^c) -1} 
%- \dfrac{(i-1)i}{2}  = \sum\limits_{i=1}^{N - \ell(\lambda^c) -1} -i = 
%-\dfrac{ (N - \ell(\lambda^c) -1) (N - \ell(\lambda^c)) }{2} \]
%which cancels with the $i = N - \ell(\lambda^c) $ term of the sum 
%$\sum\limits_{i=1}^{N - \ell(\lambda^c) }  - \dfrac{(i-1)i}{2} $.
So, if $\ell(\lambda^{rc}) = \ell(\lambda^c) -1$, we have $-\frac{(N-\ell(\lambda^c) )(N - \ell(\lambda^c) +1 )}{2} $. Otherwise, $\ell(\lambda^{rc}) = \ell(\lambda^c) = 0$, and we are left with 
\[ \sum\limits_{i=1}^{N -1} \dfrac{(N+ 1)N }{2}
- \dfrac{i(i+1)}{2} - \sum\limits_{i=1}^{N -1} \dfrac{(N- 1)N}{2} - \dfrac{(i-1)i}{2} = \dfrac{N(N-1)}{2}.\]
\end{proof}

\begin{lemma}
\label{lem:ptfirstfour}
\begin{eqnarray*}
& & \dfrac{(N - \ell((\mu_1^c)') -2)(N - \ell((\mu_1^c)') -1) }{2} \\
&&{}+{} \sum\limits_{i=1}^{N - \ell((\mu_1^c)') -2} \dfrac{(N- \ell((\mu_1^c)')- 1)(N - \ell((\mu_1^c)')-2)}{2} - \dfrac{(i-1)i}{2} \\
&&{}+{} \sum\limits_{i=1}^{N - \ell((\mu_1^r)') -1} \dfrac{(N- \ell((\mu_1^r)') + 1)(N - \ell((\mu_1^r)'))}{2}
- \dfrac{i(i+1)}{2} \\
&&{}-{} \sum\limits_{i=1}^{N - \ell(\mu_1') -1} \dfrac{(N- \ell(\mu_1') - 1)(N - \ell(\mu_1'))}{2} - \dfrac{(i-1)i}{2} \\
&&{}-{} \sum\limits_{i=1}^{N - \ell((\mu_1^{rc})') -1} \dfrac{(N- \ell((\mu_1^{rc})') - 1)(N - \ell((\mu_1^{rc})'))}{2} - \dfrac{(i-1)i}{2} \\
&=& 
\begin{cases}
0 & \mbox{if } d(\mu_1') > 1 \\
\dfrac{(N - \ell(\mu_1') -1)(N - \ell(\mu_1')) }{2} - \dfrac{(N-1)N}{2} &\mbox{if } d(\mu_1') = 1 \text{ and } (\mu_1')_1 > 1 \\
\dfrac{(N - \ell(\mu_1') -1)(N - \ell(\mu_1')) }{2} + \dfrac{(N-1)N}{2} & \mbox{otherwise }
\end{cases}
\end{eqnarray*}
\end{lemma}

The terms in the lemma are from $w_{down}$, $w_{up}$, $w_{base}(\mu_1, \mu_2, \mu_3)$ and $w_{base}(\mu_1^{rc}, \mu_2^{rc}, \mu_3)$, respectively. 
 
\begin{proof}
%We begin by showing that when the first sum is subtracted from the second, we get 0. 
We use the fact that $\ell((\mu_1^c)') = \ell((\mu_1')^r)$ and then we apply Lemma~\ref{lem:lengthmur} to write the first two lines as
\[ \dfrac{(N - \ell(\mu_1') -1)(N - \ell(\mu_1')) }{2} + \sum\limits_{i=1}^{N - \ell((\mu_1') -1} \frac{(N- \ell(\mu_1'))(N - \ell(\mu_1')-1)}{2} - \frac{(i-1)i}{2}.
\]
So, when we subtract the sum from the fourth line of the lemma statement, we are left with 
\[\dfrac{(N - \ell(\mu_1') -1)(N - \ell(\mu_1')) }{2}.\]

Applying Lemma~\ref{lem:ptfirsttwoterms}, if $d(\mu_1') > 1$, then $\ell((\mu_1')^c) = \ell(\mu_1') +1$ (see Remark~\ref{lem:lengthmuc}), and so the contributions from all of the terms cancel. If $d(\mu_1') = 1$ and $(\mu_1')_1 > 1$, then $\ell((\mu_1')^c) =1$. So, we get 
\[\dfrac{(N - \ell(\mu_1') -1)(N - \ell(\mu_1')) }{2} - \dfrac{(N-1)N}{2}. \] Finally, if $d(\mu_1') = 1$ and $(\mu_1')_1 = 1$, we get
\[\dfrac{(N - \ell(\mu_1') -1)(N - \ell(\mu_1')) }{2} + \dfrac{(N-1)N}{2}. \]
\end{proof}

\begin{lemma}
\label{lem:ptsecondfour}
\begin{eqnarray*}
& & \sum\limits_{i=1}^{N-1 - \ell(\mu_2^r) } (N+i-2)(N-i - \ell(\mu_2^r)) + 
\sum\limits_{i=1}^{N-1 - \ell(\mu_2^c) } (N+i)(N-i - \ell(\mu_2^c)) \\
&&{}-{} \sum\limits_{i=1}^{N-1 - \ell(\mu_2) } (N+i-1)(N-i - \ell(\mu_2)) - 
\sum\limits_{i=1}^{N-1 - \ell(\mu_2^{rc}) } (N+i-1)(N-i - \ell(\mu_2^{rc}))
\\
&=& \begin{cases}
\ell(\mu_2)
& \mbox{if } d(\mu_2) > 1 \\
- \ell(\mu_2) N + \ell(\mu_2) & \mbox{if } d(\mu_2) = 1 \text{ and } (\mu_2)_1 > 1 \\
(N-1)(N - \ell(\mu_2)) + \dfrac{N(N-1)}{2} & \mbox{otherwise }
\end{cases}
\end{eqnarray*}
\end{lemma}

The terms in the lemma are from $w_{down}$, $w_{up}$, $w_{base}(\mu_1, \mu_2, \mu_3)$ and $w_{base}(\mu_1^{rc}, \mu_2^{rc}, \mu_3)$, respectively. 

\begin{proof}
Using the fact that $\ell(\mu_2^{r}) = \ell(\mu_2) -1 $, we write 
\begin{align*}
\sum\limits_{i=1}^{N-1 - \ell(\mu_2^r) } (N+i-2)(N-i - \ell(\mu_2^r))&=\sum\limits_{i=1}^{N - \ell(\mu_2) } (N+i-2)(N-i - \ell(\mu_2)  + 1) \\
&=\sum\limits_{i=0}^{N - \ell(\mu_2)-1 } (N+i-1)(N-i - \ell(\mu_2) ). 
\end{align*}
So when we subtract the third sum from the lemma statement, we get $$(N-1)(N - \ell(\mu_2)).$$

In the case where $\ell(\mu_2^{rc}) = \ell(\mu_2^{c}) -1$, we can write 
\begin{align*}
\sum\limits_{i=1}^{N-1 - \ell(\mu_2^{rc}) } (N+i-1)(N-i - \ell(\mu_2^{rc}))&=\sum\limits_{i=1}^{N- \ell(\mu_2^{c}) } (N+i-1)(N-i - \ell(\mu_2^{c})+1)\\
&= \sum\limits_{i=0}^{N- \ell(\mu_2^{c})-1 } (N+i)(N-i - \ell(\mu_2^{c})).
\end{align*}
So, in this case when we subtract this from the second sum in the lemma statement we have $$-N(N-\ell(\mu_2^{c})).$$ So, if $\ell(\mu_2^{rc}) = \ell(\mu_2^{c}) -1$, the contribution from all four terms is 
\[ -N - \ell(\mu_2) N + \ell(\mu_2) + \ell(\mu_2^c) N.
\]

There are two ways for $\ell(\mu_2^{rc}) = \ell(\mu_2^{c}) -1$. We could have $d(\mu_2) > 1$, in which case $\ell(\mu_2^{c}) = \ell(\mu_2) + 1$. Or we could have $d(\mu_2) = 1$ and $(\mu_2)_1 > 1$, in which case $\ell(\mu_2^{c}) = 1$. Therefore we have 
\[
\begin{cases}
\ell(\mu_2)
& \mbox{if } d(\mu_2) > 1 \\
- \ell(\mu_2) N + \ell(\mu_2) & \mbox{if } d(\mu_2) = 1 \text{ and } (\mu_2)_1 > 1,
\end{cases}
\]
% (N-1)(N - \ell(\mu_2)) + \dfrac{N(N-1)}{2} & \mbox{ otherwise }
as desired.

In the case where $\ell(\mu_2^{rc}) = \ell(\mu_2^{c}) = 0$, when we subtract the fourth sum in the lemma statement from the second sum we have
\[ \sum\limits_{i=1}^{N-1 - \ell(\mu_2^c) } (N+i)(N-i )- 
\sum\limits_{i=1}^{N-1 - \ell(\mu_2^{rc}) } (N+i-1)(N-i) = 
\dfrac{N(N-1)}{2}.
\]
So, if $\ell(\mu_2^{rc}) = \ell(\mu_2^{c}) = 0$, the contribution from all four terms is
\[ (N-1)(N - \ell(\mu_2)) + \dfrac{N(N-1)}{2}. \]
\end{proof}

%Next we consider the terms
%\begin{itemize}
%\item $\sum\limits_{i=1}^{N-1 - \ell(\mu_2) } (N+i-1)(N-i - \ell(\mu_2)) $
%\item $ \sum\limits_{i=1}^{N-1 - \ell(\mu_2^r) } (N+i-2)(N-i - \ell(\mu_2^r)) $
%\item $\sum\limits_{i=1}^{N-1 - \ell(\mu_2^{rc}) } (N+i-1)(N-i - \ell(\mu_2^{rc})) $
%\item $\sum\limits_{i=1}^{N-1 - \ell(\mu_2^c) } (N+i)(N-i - \ell(\mu_2^c)) $
%\end{itemize}

\begin{remark}
\label{rem:2n}
Note that 
\[ \frac{(N-1)^2(N-2)}{2} + \frac{(N+1)N(N-1)}{2} + N^2 -2 \cdot \frac{N^2(N-1)}{2} = 2N -1. \] 
These terms are from $w_{down}$, $w_{up}$, $w_{base}(\mu_1, \mu_2, \mu_3)$ and $w_{base}(\mu_1^{rc}, \mu_2^{rc}, \mu_3)$, respectively. 
\end{remark}

We now proceed to simplifying the terms whose index sets depend on $\mu$. As in Section~\ref{sec:DTalg}, our strategy is to pair summands that contribute to the constant $C$ with summands that contribute to the constant $A$. 

%In addition to the lemmas from Section~\ref{sec:mumodify} we will use the following lemma to simplify the sums. 

%\subsubsection{Simplification of the remaining terms}
%\begin{lemma}
%\label{lem:indexset}
%$\{ i : (\mu_1')^r_i \leq i + 1 \}
%= \{ i: d(\mu_1') \leq i \leq \ell(\mu_1') \}$
%\end{lemma}
  
%\begin{proof}
%If $i \geq d(\mu_1')$, then $(\mu_1)_i \leq i$. Since $i \geq d(\mu_1')$, %$(\mu_1')_i = (\mu_1')^{r}_{i-1}$. So $(\mu_1')^{r}_{i-1} \leq i$, and thus  %$(\mu_1')^{r}_{i} \leq i+1$. 
  
%Next assume that $(\mu_1')^{r}_{i} \leq i+1$.
%Then  $(\mu_1')^{r}_{i} -1 \leq i$.
%Suppose towards a contradiction that $i \leq d(\mu_1') - 1$. Then $(\mu_1')_{i} \leq i$, a contradiction. 
%\end{proof}

\begin{lemma}%[Lemma B]
\label{lem:lemmaB}
\begin{eqnarray*}
& & \sum\limits_{i: (\mu_1^c)'_i > i+1>1 }
(N- (\mu_1^c)'_i +1) (\mu_1^c)'_i + N - i - 1)-\sum\limits_{i: (\mu_1')_i \geq i \geq 1}
(N-(\mu_1')_i)(N + (\mu_1')_i - i) \\
&=& - (N- (\mu_1')_{d(\mu_1)})
(N + (\mu_1')_{d(\mu_1)} - d(\mu_1) )
\end{eqnarray*}
\end{lemma}

The terms in the lemma are from $w_{down}$ and $w_{base}(\mu_1, \mu_2, \mu_3)$, respectively. 

\begin{proof}
%We first note that 
 
%\begin{eqnarray*}
%& & \sum\limits_{i: ((\mu_1^c)')_i >  i+1 }
%((\mu_1^c)')_i)(N- ((\mu_1^c)')_i +1)  + 
%(N-i-1)(N- ((\mu_1^c)')_i + 1)\\
%& & = 
%\sum\limits_{i: ((\mu_1^c)')_i >  i+1 }
%(N- ((\mu_1^c)')_i +1) ((\mu_1^c)')_i + N - i - 1)
%\end{eqnarray*}
%We first note that
%\[
%\sum\limits_{i: (\mu_1')_i \geq i }
%-((\mu_1')_i)^2  + N^{2}  + i( (\mu_1')_i  - N) ) 
%= 
%\sum\limits_{i: (\mu_1')_i \geq i }
%(N-(\mu_1')_i)(N +   (\mu_1')_i  - i) 
%\]
We use the fact that $(\mu_1^c)'= (\mu_1')^{r}$ and Lemma~\ref{lem:indexset} to write 
\begin{eqnarray*}
& & \sum\limits_{i: (\mu_1^c)'_i > i+1>1 }
(N- (\mu_1^c)'_i +1) (\mu_1^c)'_i + N - i - 1) - \sum\limits_{i: (\mu_1')_i \geq i\geq 1 }
(N-(\mu_1')_i)(N + (\mu_1')_i - i) \\
&=& 
\sum\limits_{1 \leq i < d(\mu_1') }
(N- (\mu_1')^{r}_i +1) (\mu_1')^{r}_i + N - i - 1) - 
\sum\limits_{1 \leq i \leq d(\mu_1') }
(N-(\mu_1')_i)(N + (\mu_1')_i - i) \\
&=& 
\sum\limits_{1 \leq i < d(\mu_1') }
(N- (\mu_1')_i) ((\mu_1')_i + N - i ) - 
\sum\limits_{1 \leq i \leq d(\mu_1') }
(N-(\mu_1')_i)(N + (\mu_1')_i - i) \\
&=& - (N- (\mu_1')_{d(\mu_1)})
(N + (\mu_1')_{d(\mu_1)} - d(\mu_1) ).
\end{eqnarray*}
\end{proof}

\begin{lemma}%[Lemma D]
\label{lem:lemmaD}
\begin{eqnarray*}
& & \sum\limits_{i: (\mu_2^r)_i > i +1>1}
((\mu_2^r)_i + N -1)(-(\mu_2^r)_i +N + 1) + 
\frac{(N - (\mu_2^r)_i + 1)(N - (\mu_2^r)_i)}{2}\\
&&{}-{} 
\sum\limits_{i: (\mu_2)_i \geq i\geq 1}
((\mu_2)_i + N ) (-(\mu_2)_i + N ) + \frac{(N - (\mu_2)_i -1)(N - (\mu_2)_i)}{2} \\
&=& ((\mu_2)_{d(\mu_2) })^2 - N^2 - \frac{(N - (\mu_2)_{d(\mu_2) } -1)(N - (\mu_2)_{d(\mu_2) })}{2}
\end{eqnarray*}
\end{lemma}

The terms in the lemma are from $w_{down}$ and $w_{base}(\mu_1, \mu_2, \mu_3)$, respectively. 

\begin{proof}
The details of the proof are omitted as it is similar to the proof of Lemma~\ref{lem:lemmaB}. We use Lemma~\ref{lem:indexset} and the fact that when $i < d(\mu_2)$, $(\mu_2^r)_i =( \mu_2)_i+1$. Then all terms cancel except the $i = d(\mu_2)$ term of the second sum. 
\end{proof}

\begin{lemma}%[Lemma A]
\label{lem:lemmaA}
\begin{eqnarray*}
& & \sum\limits_{i: (\mu_1^c)'_i \leq i+1 \leq\ell((\mu_1^c)')+1}
(N-i) ((\mu_1^c)'_i + N-i -1)
- \sum\limits_{i: (\mu_1')_i < i \leq\ell(\mu_1')}
(N-i)( (\mu_1')_i + N-i ) \\
&=&
\begin{cases}
0 &\mbox{if } \ell(\mu_1') = d(\mu_1') \\
\sum\limits_{d(\mu_1') + 1 \leq i\leq\ell(\mu_1')} (\mu_1')_i + N - i & \mbox{otherwise}
\end{cases}
\end{eqnarray*}
\end{lemma}

The terms in the lemma are from $w_{down}$ and $w_{base}(\mu_1, \mu_2, \mu_3)$, respectively. 

\begin{proof}
%First we note that 
%\[
%\sum\limits_{i: ((\mu_1^c)')_i \leq i+1 }
%((\mu_1^c)')_i (N-i) + (N-i-1)(N-i) 
%= 
%\sum\limits_{i: ((\mu_1^c)')_i \leq i+1 }
%(N-i)   ((\mu_1^c)')_i + N-i -1) \]
%First we note that
%\[
%\sum\limits_{i: (\mu_1')_i < i }
%(\mu_1')_i (N-i) + (N-i)^2 
%= 
%\sum\limits_{i: (\mu_1')_i < i }
%(N-i)( (\mu_1')_i + N-i ).
%\]
We use the fact that $(\mu_1^c)'= (\mu_1')^{r}$ and Lemma~\ref{lem:indexset}. We get 
\begin{eqnarray*}
& & \sum\limits_{i: (\mu_1^c)'_i \leq i+1\leq\ell((\mu_1^c)')+1 }
(N-i) (\mu_1^c)'_i + N-i -1)
- \sum\limits_{i: (\mu_1')_i < i\leq\ell(\mu_1') }
(N-i)( (\mu_1')_i + N-i )\\
&=& 
\sum\limits_{d(\mu_1') \leq i\leq\ell((\mu_1')^r) }
(N-i) (\mu_1')^{r}_i + N-i -1)
- \sum\limits_{d(\mu_1' )< i \leq\ell(\mu_1')}
(N-i)( (\mu_1')_i + N-i )\\
&=& \sum\limits_{d(\mu_1') \leq i \leq\ell(\mu_1')-1}
(N-i) ((\mu_1')_{i+1} + N-i -1)
- \sum\limits_{d(\mu_1' )< i\leq\ell(\mu_1') }
(N-i)( (\mu_1')_i + N-i )\\
&=& \sum\limits_{d(\mu_1') < i\leq\ell(\mu_1') }
(N-i + 1) ((\mu_1')_{i} + N-i)
- \sum\limits_{d(\mu_1' )< i\leq\ell(\mu_1') }
(N-i)( (\mu_1')_i + N-i )\\
&=& \sum\limits_{d(\mu_1') + 1 \leq i\leq\ell(\mu_1')} (\mu_1')_i + N - i. 
\end{eqnarray*}
Note that we have used the fact that since $i \geq d(\mu_1')$, $(\mu_1')^{r}_i = (\mu_1')_{i+1}$. In the case where $\ell(\mu_1') = d(\mu_1')$, both sums are empty. 
\end{proof}

\begin{lemma}%[Lemma C]
\label{lem:lemmaC}
\begin{eqnarray*}
& & \sum\limits_{i: (\mu_2^r)_i \leq i+1\leq\ell(\mu_2^r)+1 }
(N - i) ( (\mu_2^r)_i + N - 1) + \frac{(N - i-1)(N - i)}{2} \\
&&{}-{}
\sum\limits_{i: (\mu_2)_i < i\leq\ell(\mu_2) }
(N - i) ((\mu_2)_i + N ) + \frac{(N - i-1)(N - i)}{2}\\
&=&
\begin{cases}
0 & \mbox{if } \ell(\mu_2) = d(\mu_2) \\
\sum\limits_{d(\mu_2) + 1 \leq i\leq\ell(\mu_2) } ((\mu_2)_i + N -1) & \mbox{otherwise}
\end{cases}
\end{eqnarray*}
\end{lemma}

The terms in the lemma are from $w_{down}$ and $w_{base}(\mu_1, \mu_2, \mu_3)$, respectively. 

\begin{proof} Similar to the proof of Lemma~\ref{lem:lemmaA}, we see that
\begin{eqnarray*}
& & \sum\limits_{i: (\mu_2^r)_i \leq i+1\leq\ell(\mu_2^r)+1 }
(N - i) ( (\mu_2^r)_i + N - 1) + \frac{(N - i-1)(N - i)}{2} \\
&&{}-{}
\sum\limits_{i: (\mu_2)_i < i\leq\ell(\mu_2) } 
(N - i) ( (\mu_2)_i +N ) + \frac{(N - i-1)(N - i)}{2} \\
&=& 
\sum\limits_{d(\mu_2) \leq i \leq \ell(\mu_2^r) }
(N - i) ( (\mu_2)_{i+1} + N - 1) + \frac{(N - i-1)(N - i)}{2} \\
&&{}-{}
\sum\limits_{d(\mu_2) < i \leq \ell(\mu_2) } 
(N - i) ( (\mu_2)_i +N ) + \frac{(N - i-1)(N - i)}{2} \\
&=& 
\sum\limits_{d(\mu_2) < i \leq \ell(\mu_2) }
(N - i + 1) ( (\mu_2)_{i} + N - 1) + \frac{(N - i+1)(N - i)}{2} \\
&&{}-{}
\sum\limits_{d(\mu_2) < i \leq \ell(\mu_2) } 
(N - i) ( (\mu_2)_i +N ) + \frac{(N - i-1)(N - i)}{2} \\
&=& 
\sum\limits_{d(\mu_2) < i \leq \ell(\mu_2) }
(N - i + 1) \left( (\mu_2)_{i} + N - 1 + \frac{N-i}{2} \right)\\
&&{}-{}
\sum\limits_{d(\mu_2) < i \leq \ell(\mu_2) } 
(N - i) \left( (\mu_2)_i +N + \frac{N - i-1}{2}\right) = \sum\limits_{d(\mu_2) + 1 \leq i\leq\ell(\mu_2) } ((\mu_2)_i + N -1).
\end{eqnarray*}
\end{proof}

%\begin{proof}

%We will prove this by proving $d(\mu_1') = d((\mu_1')^{c}) \Rightarrow i_e = i_d$ and $d(\mu_1') \neq d((\mu_1')^{c}) \Rightarrow i_e \neq i_d$. 

%Suppose $d(\mu_1') = d((\mu_1')^{c})$. Put $\lambda := \mu_1'$. 
%By Remark~\ref{rem:mucdmuplus1}, $\lambda^{c}_{d(\lambda) + 1} = d(\lambda) -1$. So, since $d(\lambda) = d(\lambda^c)$,  $\lambda^{c}_{d(\lambda^c) + 1} = d(\lambda^c) -1$. It is immediate that $i_e = i_d$. 
 
%Now suppose that $d((\mu_1')^{c}) = d(\mu_1') -1$. Then $\lambda^{c}_{d(\lambda)}= d(\lambda) -1$. It follows that $\lambda^{c}_{d(\lambda^c) + 1}= d(\lambda^c)$. This means that $i_e = i_d + 1$. 
 
%\end{proof}

\begin{lemma}%[Lemma E]
\label{lem:lemE}
\begin{eqnarray*}
& & 
%\sum\limits_{i: ((\mu_1^r)')_i \geq i-1 }
%((\mu_1^r)')_i)(N- ((\mu_1^r)')_i - 1) + 
%(N-i+1)(N- ((\mu_1^r)')_i - 1) 
\begin{cases}
0 &\text{if }(\mu_1^r)'=\emptyset \\
\sum\limits_{i: (\mu_1^r)'_i \geq i-1\geq 0 }(N- (\mu_1^{r})'_i - 1)( (\mu_1^{r})'_i + N-i+1) &\text{otherwise}
\end{cases}
\\
&&{}-{}
\sum\limits_{i: (\mu_1^{rc})'_i \geq i\geq 1 }
(N-(\mu_1^{rc})'_i)(N + (\mu_1^{rc})'_i - i) \\
&=& 
\begin{cases}
0 &\text{if }(\mu_1^r)'=\emptyset \\
(N- (\mu_1^r)'_{d(\mu_1')} - 1)( (\mu_1^r)'_{d(\mu_1')}+ N- d(\mu_1')+1) &\text{otherwise}
\end{cases}
\end{eqnarray*}
\end{lemma}

The terms in the lemma are from $w_{up}$ and $w_{base}(\mu_1^{rc}, \mu_2^{rc}, \mu_3)$, respectively. 

\begin{proof}
If $(\mu_1^r)'=\emptyset$, then $\ell(\mu_1^r)=0$, so by Remark~\ref{lem:lengthmur}, $\ell(\mu_1)=1$, in which case $d(\mu_1)=1$ and we get \[-\sum_{1\leq i<d(\mu_1)}(N-(\mu_1^{rc})'_i)(N+(\mu_1^{rc})'_i-i)=0.\]

Otherwise, using the fact that $(\mu_1^r)' = (\mu_1')^{c}$, we write the first sum as
\[ \sum\limits_{i: (\mu_1')^{c}_{i} \geq i-1\geq 0 } (N- (\mu_1')^{c}_i - 1)( (\mu_1')^{c}_i + N-i+1). \]
Applying Lemma~\ref{lem:dprime_c}, we can write this sum as 

%As in the DT weights section, we let $i_e$ (resp. $i_d$) be the maximum position integer such that $i \leq ((\mu_1')^{c})_i + 1$ (resp. $i \leq ((\mu_1')^{c})_i$). 

%We claim that $d(\mu_1') = d((\mu_1')^{c})$ if and only if $i_e = i_d$. 
%We will prove this by proving $d(\mu_1') = d((\mu_1')^{c}) \Rightarrow %i_e = i_d$ and 
%$d(\mu_1') \neq d((\mu_1')^{c}) \Rightarrow i_e \neq i_d$. 

%Suppose $d(\mu_1') = d((\mu_1')^{c})$. Put $\lambda := \mu_1'$. 
%By Remark~\ref{rem:mucdmuplus1}, $\lambda^{c}_{d(\lambda) + 1} = d(\lambda) -1$. So, since $d(\lambda) = d(\lambda^c)$,  $\lambda^{c}_{d(\lambda^c) + 1} = d(\lambda^c) -1$. It is immediate that $i_e = i_d$. 
 
%Now suppose that $d((\mu_1')^{c}) = d(\mu_1') -1$. Then $\lambda^{c}_{d(\lambda)}= d(\lambda) -1$. It follows that $\lambda^{c}_{d(\lambda^c) + 1}= d(\lambda^c)$. This means that $i_e = i_d + 1$. 
 
%It follows that in both cases, the first sum is 
\[ \sum\limits_{i: 1\leq i \leq d(\mu_1') } (N- (\mu_1')^{c}_i - 1)( (\mu_1')^{c}_i + N-i+1). \]
Noting that $(\mu_1^{rc})' = (\mu_1')^{rc}$ and applying Lemma~\ref{lem:mucandmurc} to the second sum, we get 
\[\sum\limits_{i: (\mu_1^{rc})'_i \geq i\geq 1 }
(N-(\mu_1')^{c}_i -1 )(N + (\mu_1')^{c}_i +1 - i).\]
Since the second sum runs over $i$ such that $1\leq i \leq d(\mu_1^{rc})$, and $d(\mu_1^{rc}) = d(\mu_1) - 1$, this completes the proof. 
%By Lemma {\bf add lemma reference} either $i_e = i_d = d(\mu_1^r)$ or $i_e = i_d + 1$. By Lemma~\ref{lem:d'eq}, $i_e = i_d + 1$ if and only if $(\mu_1^r)')_{d((\mu_1^r)') + 1 } = d((\mu_1^r)')$. 
%By Remark, $d(\mu^rc) = d(\mu^r) -1$ if and only if $\mu^r_{d(\mu^r)} = d(\mu^r)$ $d(\mu^rc) = d(\mu^r) $ if and only if $\mu^r_{d(\mu^r)} > d(\mu^r)$
\end{proof}

\begin{lemma}%[Lemma F]
\label{lem:lemmaF}
\begin{eqnarray*}
& & \sum\limits_{i: (\mu_1^r)'_i < i-1\leq\ell((\mu_1^r)')-1 }
(N-i)( (\mu_1^{r})'_i + N-i+1 )
- \sum\limits_{i: (\mu_1^{rc})'_i < i\leq\ell((\mu_1')^{rc}) }
(N-i) ( (\mu_1')^{rc}_i + N-i ) \\
&=& 
\begin{cases}
0 &\mbox{ if } \ell ((\mu_1^r)') \leq d(\mu_1) \\
\sum\limits_{i: d(\mu_1)<i\leq\ell((\mu_1')^c)} - \left( (\mu_1')^{c}_i + N - i + 1 \right) & \mbox{ otherwise}
\end{cases}
\end{eqnarray*}
\end{lemma}

The terms in the lemma are from $w_{up}$ and $w_{base}(\mu_1^{rc}, \mu_2^{rc}, \mu_3)$, respectively. 

\begin{proof}
We apply many of the same arguments as in Lemma~\ref{lem:lemE}. Namely, we use the facts that $(\mu_1^r)' = (\mu_1')^{c}$ and Lemma~\ref{lem:dprime_c} to write the first sum as 
\[ \sum\limits_{i: d(\mu_1)<i\leq\ell((\mu_1')^c) }
(N-i)( (\mu_1')^{c}_i + N-i+1 ).\]
%The second sum can be written
%\[
%\sum\limits_{i:  i > d(( \mu_1)^{rc} ) }
%(N-i) ( ((\mu_1')^{rc})_i + N-i ).
%\]
Now, applying Lemma~\ref{lem:mucandmurc}, we have 
\begin{eqnarray*}
& & \sum\limits_{i: d( \mu_1^{rc})<i\leq\ell((\mu_1')^{rc}) }
(N-i) ( (\mu_1')^{rc}_i + N-i ) = \sum\limits_{i: d(\mu_1^{rc})<i\leq\ell((\mu_1')^c)-1 }
(N-i) ( (\mu_1')^{c}_{i+1}+ N-i ) \\
&=& \sum\limits_{i: d(\mu_1^{rc} ) + 1<i\leq\ell((\mu_1')^c) }
(N-i + 1) ( (\mu_1')^{c}_{i}+ N-i + 1 ).
\end{eqnarray*}
So, subtracting this from the first sum, we get 
\[
\sum\limits_{i: d(\mu_1)<i\leq\ell((\mu_1')^c)} - \left( (\mu_1')^{c}_i + N - i + 1 \right).
\]
\end{proof}

\begin{lemma}%[Lemma G]
\label{lem:lemmaG}
\begin{eqnarray*}
&&\begin{cases}
0 &\text{if }\mu_2^c=\emptyset \\
\sum\limits_{i: (\mu_2^c)_i \geq i -1\geq 0}((\mu_2^c)_i + N)(-(\mu_2^c)_i + N -1) +
%\frac{(N - (\mu_2^c)_i -1)(N - (\mu_2^c)_i)}{2} +
\frac{(N - (\mu_2^c)_i -1)(N - (\mu_2^c)_i)}{2} &\text{otherwise}
\end{cases}\\
&&{}-{} \sum\limits_{i: (\mu_2^{rc})_i \geq i \geq 1}
((\mu_2^{rc})_i + N ) (-(\mu_2^{rc})_i + N ) + \frac{(N - (\mu_2^{rc})_i -1)(N - (\mu_2^{rc})_i)}{2} \\
&=& 
\begin{cases}
0 & \text{if } \mu_2^c=\emptyset \\
(N -( \mu_{2}^{c})_{d(\mu_2)} - 1) 
\left( \dfrac{ ( \mu_{2}^{c})_{d(\mu_2) } }{2} + \dfrac{3N}{2} \right) & \text{otherwise}
\end{cases}
\end{eqnarray*}
\end{lemma}

The terms in the lemma are from $w_{up}$ and $w_{base}(\mu_1^{rc}, \mu_2^{rc}, \mu_3)$, respectively. 

%\begin{remark}
%If $\mu_2^{c} = \emptyset$, we consider both of the sums above to be empty. 
%\end{remark}

\begin{proof}
If $\mu_2^c=\emptyset$, then $\ell(\mu_2^c)=0$, so by Remark~\ref{lem:lengthmuc}, $d(\mu_2)=1$ and $(\mu_2)_1=1$, in which case $\mu_2^{rc}=\emptyset$ and we get \[-\sum\limits_{i: (\mu_2^{rc})_i \geq i \geq 1}
((\mu_2^{rc})_i + N)(-(\mu_2^{rc})_i + N) + \frac{(N - (\mu_2^{rc})_i - 1)(N - (\mu_2^{rc})_i)}{2}=0.\]

%By the same reasoning as in Lemma~\ref{lem:lemE}, we find that $\{ i: (\mu_2^{c})_{i} \geq i-1 \} = \{ i: i \leq d(\mu_2) \}$. Then the first sum can be written as
%Applying Lemma~\ref{lem:dprime_c}, we write the first sum as
%\[ \sum\limits_{i: i \leq d(\mu_2) } 
%((\mu_2^c)_i + N)(-(\mu_2^c)_i + N -1) +
%\frac{(N - (\mu_2^c)_i -1)(N - (\mu_2^c)_i)}{2}. \]
%The second sum can be written as
%\[
%\sum\limits_{i: i \leq d(\mu_2^{rc}) }
%(-(\mu_2^{rc})_i + N )  ((\mu_2^{rc})_i + N )  +  \frac{(N - (\mu_2^{rc})_i -1)(N - (\mu_2^{rc})_i)}{2}.
%\]
%Now
Otherwise, using Lemma~\ref{lem:mucandmurc}, we see that 
\begin{eqnarray*}
& & \sum\limits_{i: 1\leq i \leq d(\mu_2^{rc}) }
(-(\mu_2^{rc})_i + N ) ((\mu_2^{rc})_i + N ) + \frac{(N - (\mu_2^{rc})_i -1)(N - (\mu_2^{rc})_i)}{2} 
\\
& = & \sum\limits_{i: 1\leq i \leq d(\mu_2^{rc}) }
(-(\mu_2^{c})_i -1 + N ) ((\mu_2^{c})_i + 1 + N ) + \frac{(N - (\mu_2^{c})_i -2)(N - (\mu_2^{c})_i -1)}{2}. 
\end{eqnarray*}
We see that 
\begin{eqnarray*}
& & \sum\limits_{i: 1\leq i \leq d(\mu_2^{rc}) } 
((\mu_2^c)_i + N)(-(\mu_2^c)_i + N -1) + \frac{(N - (\mu_2^c)_i -1)(N - (\mu_2^c)_i)}{2} \\
&&{}-{} \sum\limits_{i: 1\leq i \leq d(\mu_2^{rc}) }
(-(\mu_2^{c})_i -1 + N ) ((\mu_2^{c})_i + 1 + N ) + \frac{(N - (\mu_2^{c})_i -2)(N - (\mu_2^{c})_i -1)}{2} \\
&=& \sum\limits_{i: 1\leq i \leq d(\mu_2^{rc}) } 
(-(\mu_2^c)_i + N -1) (( \mu_2^c)_i + N - (\mu_2^{c})_i - 1 - N ) \\
& &\phantom{\sum\limits_{i: 1\leq i \leq d(\mu_2^{rc}) }} + \frac{(N - (\mu_2^{c})_i -1)}{2} ( N - (\mu_2^{c})_i - (N - (\mu_2^{c})_i -2) ) 
= 0.
\end{eqnarray*}
So, all that remains is the $i = d(\mu_2)$ term of the first sum:
\begin{eqnarray*}
& & ((\mu_2^c)_{d(\mu_2)} + N)(-(\mu_2^c)_{d(\mu_2)} + N -1) +
\frac{(N - (\mu_2^c)_{d(\mu_2)} -1)(N - (\mu_2^c)_{d(\mu_2)})}{2} \\
& = & (N -( \mu_{2}^{c})_{d(\mu_2)} - 1) 
\left( \dfrac{ ( \mu_{2}^{c})_{d(\mu_2) } }{2} + \dfrac{3N}{2} \right).
\end{eqnarray*}
\end{proof}

\begin{lemma}%[Lemma H]
\label{lem:lemmaH}
\begin{eqnarray*}
& & \sum\limits_{i: (\mu_2^c)_i < i-1\leq\ell(\mu_2^c)-1 }
(N - i) ((\mu_2^c)_i + N ) + \frac{(N - i+1)(N - i)}{2} \\
&&{}-{} \sum\limits_{i: (\mu_2^{rc})_i < i\leq\ell(\mu_2^{rc}) } 
(N - i) ((\mu_2^{rc})_i + N ) + \frac{(N - i-1)(N - i)}{2} \\
& = & 
\begin{cases}
0 &\mbox{if } \ell(\mu_2^{c} ) < d(\mu_2) \\
\sum\limits_{i: d(\mu_2)<i\leq\ell(\mu_2^c) } - ((\mu_2^c)_i + N) & \mbox{otherwise }
\end{cases}
\end{eqnarray*}
\end{lemma}

The terms in the lemma are from $w_{up}$ and $w_{base}(\mu_1^{rc}, \mu_2^{rc}, \mu_3)$, respectively. 

\begin{proof}
%Applying similar reasoning as in the proofs of the previous lemmas, we write the first sum as
%\[ \sum\limits_{i: i > d(\mu_2) }
%(N - i) ((\mu_2^c)_i  + N ) + \frac{(N - i+1)(N - i)}{2} 
%\]
%We write the second sum as
%\[ \sum\limits_{i: i > d(\mu_2^{rc}) }
%(N - i) ((\mu_2^{rc})_i  + N ) +  \frac{(N - i-1)(N - i)}{2}
%\]
Applying Lemma~\ref{lem:mucandmurc}, we have 
\begin{eqnarray*}
& & \sum\limits_{i: d(\mu_2^{rc})<i\leq\ell(\mu_2^{rc}) }
(N - i) ((\mu_2^{rc})_i + N ) + \frac{(N - i-1)(N - i)}{2} \\
& = & \sum\limits_{i: d(\mu_2^{rc})<i\leq\ell(\mu_2^c)-1 }
(N - i) ((\mu_2^{c})_{i+1} + N ) + \frac{(N - i-1)(N - i)}{2} \\
& = & \sum\limits_{i: d(\mu_2^{rc}) + 1<i\leq\ell(\mu_2^c) }
(N - i + 1) ((\mu_2^{c})_{i} + N ) + \frac{(N - i)(N - i+1)}{2}. \\
\end{eqnarray*}
So, subtracting this from the first sum, we're left with 
\[ \sum\limits_{i: d(\mu_2)<i\leq\ell(\mu_2^c) } - ((\mu_2^c)_i + N).\]
If $\ell(\mu_2^{c} ) < d(\mu_2)$, both sums are empty. 
\end{proof}

Now that we have paired all of the sums, we simplify the results from Lemmas~\ref{lem:lemmaB} through \ref{lem:lemmaH}.
  
\begin{lemma} %[Combining B \& E]
\label{lem:BandE}
The terms from Lemmas~\ref{lem:lemmaB} and \ref{lem:lemE} cancel, unless $(\mu_1^r)'=\emptyset$, in which case we are left with $-N(N-1)$. 
\end{lemma}
  
\begin{proof}
Lemma~\ref{lem:muc} states that $(\mu_1')^{c}_{d(\mu_1')} = (\mu_1')_{d(\mu_1')} -1$. Applying this to Lemma~\ref{lem:lemE} completes the proof in the case that $(\mu_1^r)'\neq\emptyset$. If $(\mu_1^r)'=\emptyset$, then $\mu_1^r=\emptyset$, so $\ell(\mu_1^r)=0$ and by Remark~\ref{lem:lengthmur}, $\ell(\mu_1)=1$, implying that $d(\mu_1)=1$ and $(\mu_1')_1=1$. Then the term from Lemma~\ref{lem:lemmaB} is \[-(N-(\mu_1')_{d(\mu_1)})(N+(\mu_1')_{d(\mu_1)}-d(\mu_1))=-(N-1)(N+1-1)=-N(N-1).\]
\end{proof}

\begin{lemma} %[Combining G \& D]
\label{lem:GandD}
The terms from Lemmas~\ref{lem:lemmaG} and~\ref{lem:lemmaD} sum to 
\[
\begin{cases}
0 & \mbox{if } \mu_2^c \neq \emptyset \\
1 - N^2 - \dfrac{ (N-2)(N-1) }{2} & \mbox{if } \mu_2^c = \emptyset.
\end{cases}
\]
\end{lemma}

\begin{proof}
To get the expression when $\mu_2^c = \emptyset$, we use the fact that if $\mu_2^c = \emptyset$, then it must be the case that $(\mu_2)_{1} = 1$. 
\end{proof}
%If $\mu_2^c \neq \emptyset$, the terms from lemmas G and D cancel. If $\mu_2^c = \emptyset$, then it must be the case that $(\mu_2)_{1} = 1$, so we are left with
%\[ 1 - N^2 - \dfrac{ (N-2)(N-1) }{2} \]
%\end{lemma} 

\begin{lemma} %[Combining  A \& F]
\label{lem:AandF}
%Let $J$ be the largest integer with $(\mu_1')_i \geq d(\mu_1')$. 
When we add the terms from Lemmas~\ref{lem:lemmaA} and~\ref{lem:lemmaF}, we get 
\[
\begin{cases}
\ell(\mu_1') + N \ell(\mu_1') - N - \dfrac{\ell(\mu_1') (\ell(\mu_1') +1)}{2} &\mbox {if } d(\mu_1') = 1 \\
- (d(\mu_1') + N -( (\mu_1)_{d(\mu_{1} )} +1) ) &\mbox {otherwise}.
\end{cases}
\]
\end{lemma}

\begin{proof}
First we deal with the case that $d(\mu_1') = 1$. In this case, $(\mu_1')_i \leq 1$ for all $i \geq 2$ and $(\mu_1')^{c}_i = 0$ for all $i \geq 2$. So, the contribution from Lemma~\ref{lem:lemmaF} is $0$ and the sum from Lemma~\ref{lem:lemmaA} becomes 
\begin{eqnarray*}
\sum\limits_{d(\mu_1') + 1 \leq i\leq\ell(\mu_1')} (\mu_1')_i + N - i
= \sum\limits_{i = 2}^{\ell(\mu_1') } 1+ N - i
& = & ( \ell(\mu_1') - 1)(1+N) - \dfrac{\ell(\mu_1') (\ell(\mu_1') +1) - 2}{2} \\
& = &
\ell(\mu_1') + N \ell(\mu_1') - N - \dfrac{\ell(\mu_1') (\ell(\mu_1') +1)}{2}.
\end{eqnarray*}

In the case where $d(\mu_1') > 1$, let $i_d$ be the largest integer $i$ with $(\mu_1')_i \geq d(\mu_1')$. Then applying Lemma \ref{lem:muc}, we see that the sum from Lemma~\ref{lem:lemmaF} becomes 
\begin{eqnarray*}
& & - \sum\limits_{i: d(\mu_1)<i\leq\ell((\mu_1')^c)} \left( (\mu_1')^{c}_i + N - i + 1 \right) \\
& = & - \sum\limits_{i: d(\mu_1) + 1 \leq i \leq i_d } \left( (\mu_1')_i + N - i \right) \\
&&{}-{} (d(\mu_1') + N -(i_d+1) ) - \sum\limits_{i: i_d + 1 < i \leq \ell((\mu_1')^c) }  \left( (\mu_1')_{i-1} + N - i + 1 \right) \\
& = & - \sum\limits_{i: d(\mu_1) + 1 \leq i \leq i_d } \left( (\mu_1')_i + N - i \right) \\
&&{}-{} \sum\limits_{i: i_d < i \leq \ell((\mu_1')^c) -1 } \left( (\mu_1')_{i} + N - i \right) - (d(\mu_1') + N -(i_d+1) ). 
\end{eqnarray*}
Since the first two sums cancel with the sum from Lemma~\ref{lem:lemmaA}, we are left with \[- (d(\mu_1') + N -(i_d+1) ) = - (d(\mu_1') + N -((\mu_1)_{d(\mu_{1} )}+1) )\] by Remark~\ref{rem:idlambdat}. 
\end{proof}

\begin{lemma}%[Combining  C \& H]
\label{lem:CandH}
% Let $J$ be the largest integer with $(\mu_2)_i \geq d(\mu_2)$. 
When we combine the terms from Lemmas~\ref{lem:lemmaC} and~\ref{lem:lemmaH}, we get 
\[
\begin{cases}
( \ell(\mu_2)-1) N &\mbox {if } d(\mu_2) = 1 \\
- d(\mu_2) - N + 1 - \ell(\mu_2) + (\mu'_2)_{d(\mu_2)} &\mbox {otherwise}. 
\end{cases}
\]
\end{lemma}

\begin{proof}
As in the previous lemma, we begin with the case where $d(\mu_2) = 1$. In this case, the sum from Lemma \ref{lem:lemmaH} is empty and the sum from Lemma~\ref{lem:lemmaC} becomes 
\[
\sum\limits_{i: d(\mu_2) + 1 \leq i\leq\ell(\mu_2) } ((\mu_2)_i + N -1) = \sum\limits_{i: 2 \leq i \leq \ell(\mu_2) } N = (\ell(\mu_2)-1)N.
\]

In the case where $d(\mu_2) > 1$, let $i_d$ be the largest integer with $(\mu_2)_i \geq d(\mu_2)$. Then we can write the sum from Lemma~\ref{lem:lemmaH} as 
\begin{eqnarray*}
& & - \sum\limits_{i: d(\mu_2)<i\leq\ell(\mu_2^c) } ((\mu_2^{c})_i + N) \\
& = & - \sum\limits_{i: d(\mu_2) + 1 \leq i \leq i_d } ((\mu_2)_i -1 + N) -
(d(\mu_2) -1 + N )
- \sum\limits_{i: i_d+1 < i\leq\ell(\mu_2)+1 } ((\mu_2)_{i-1} + N) \\
& = & - \sum\limits_{i: d(\mu_2) + 1 \leq i \leq i_d } ((\mu_2)_i + N -1) 
- \sum\limits_{i: i_d < i\leq\ell(\mu_2) } ((\mu_2)_{i} + N) -
(d(\mu_2) -1 + N ).
\end{eqnarray*}
Writing the sum from Lemma~\ref{lem:lemmaC} as 
\[
\sum\limits_{i: d(\mu_2) + 1 \leq i\leq\ell(\mu_2) } ((\mu_2)_i + N -1) = 
\sum\limits_{i: d(\mu_2) + 1 \leq i \leq i_d } ((\mu_2)_i + N -1) + 
\sum\limits_{i: i_d <  i\leq\ell(\mu_2) } ((\mu_2)_i + N -1),
\]
we see that the first sum cancels with the first sum from Lemma \ref{lem:lemmaH}. Combining the second sum with the second sum from Lemma \ref{lem:lemmaH} we get $-(\ell(\mu_2) - i_d) = -(\ell(\mu_2) - (\mu'_2)_{d(\mu_2)})$. 
\end{proof}

%\begin{remark}
%Lemmas A, B, C, D, E, F, G, and H have been checked with code for all partitions of $n$, where $n \leq 20$
%\end{remark}

We conclude the computation of $C-A$ by adding the results from Lemmas \ref{lem:ptfirstfour} and \ref{lem:ptsecondfour}, Remark~\ref{rem:2n}, and Lemmas~\ref{lem:BandE} through \ref{lem:CandH}. \\

%\subsubsection{Conclusion}

\noindent{\bf{Terms involving $\mu_1$} } 
From Lemma~\ref{lem:ptfirstfour} we have
\[
\begin{cases}
0 &\mbox{if } d(\mu_1') > 1 \\
\dfrac{(N - \ell(\mu_1') -1)(N - \ell(\mu_1')) }{2} - \dfrac{(N-1)N}{2} 
&\mbox{if } d(\mu_1') = 1 \text{ and } (\mu_1')_1 > 1 \\
\dfrac{(N - \ell(\mu_1') -1)(N - \ell(\mu_1')) }{2} + \dfrac{(N-1)N}{2} &\mbox{otherwise}.
\end{cases}
\]
From Lemmas~\ref{lem:BandE} and~\ref{lem:AandF} we have 
\begin{eqnarray*}
&&\begin{cases}
- (d(\mu_1') + N -( (\mu_1)_{d(\mu_1)}+1) ) &\mbox {if } d(\mu_1') > 1 \\
\ell(\mu_1') + N \ell(\mu_1') - N - \dfrac{\ell(\mu_1') (\ell(\mu_1') +1)}{2} &\mbox {if } d(\mu_1') = 1 
\end{cases}\\
&&{}+{}\begin{cases}
0 &\text{if }(\mu_1^r)'\neq\emptyset \\
-N(N-1) &\text{otherwise}.
\end{cases}
\end{eqnarray*}
Note that, by Remark~\ref{lem:lengthmur}, $(\mu_1^r)'=\emptyset$ if and only if $\mu_1^r=\emptyset$ if and only if $\ell(\mu_1^r)=0$ if and only if $\ell(\mu_1)=1$ if and only if $d(\mu_1')=1$ and $(\mu_1')_1=1$. 

So there are three cases to consider. If $d(\mu_1') > 1$, we have $$(\mu_1)_{d(\mu_1)} -d(\mu_1) -N + 1.$$ When $d(\mu_1') = 1$ and $(\mu_1')_1 > 1$, we get $$\ell(\mu_1') - N = (\mu_1)_{d(\mu_1)} -N -d(\mu_1) + 1.$$ When $d(\mu_1') = 1$ and $(\mu_1')_1 = 1$, we get $$\ell(\mu_1')-N+(N-1)N-N(N-1)=\ell(\mu_1')-N=(\mu_1)_{d(\mu_1)}-N-d(\mu_1)+1.$$

\noindent{\bf{Terms involving $\mu_2$} }
Recall from Lemma~\ref{lem:ptsecondfour} that the terms involving $\mu_2$ are 
\[
\begin{cases}
\ell(\mu_2)
& \mbox{if } d(\mu_2) > 1 \\
- \ell(\mu_2) N + \ell(\mu_2) & \mbox{if } d(\mu_2) = 1 \text{ and } (\mu_2)_1 > 1 \\
(N-1)(N - \ell(\mu_2)) + \dfrac{N(N-1)}{2} & \mbox{otherwise}.
\end{cases}
\]
The terms involving $\mu_2$ from Lemmas~\ref{lem:GandD} and~\ref{lem:CandH} are 
%\helentodo{Add lemma from combining g and d after double checking it}
\begin{eqnarray*}
& & \begin{cases}
- d(\mu_2) - N + 1 - \ell(\mu_2) + (\mu_2')_{d(\mu_2)} &\mbox {if }
d(\mu_2) > 1 \\
( \ell(\mu_2)-1) N &\mbox {if } d(\mu_2) = 1
\end{cases}\\
&&{}+{}
\begin{cases}
0 &\mbox { if } d(\mu_2) > 1 \text{ or } (d(\mu_2) = 1 \text{ and } (\mu_2)_{1} > 1 ) \\
1 - N^2 - \dfrac{ (N-2)(N-1) }{2} &\mbox { if } d(\mu_2) = 1 \text{ and } (\mu_2)_{1} = 1. 
\end{cases}
\end{eqnarray*}

So there are three cases. If $d(\mu_2) > 1$, then we are left with 
\[ - d(\mu_2) - N + 1 + (\mu_2')_{d(\mu_2)}. \]
If $d(\mu_2) = 1$ and $(\mu_2)_1 > 1$ then we have 
$$ \ell(\mu_2) - N = - d(\mu_2) + 1 +
(\mu_2')_{d(\mu_2)} -N.$$
Finally, in the case where $d(\mu_2) = 1$ and $(\mu_2)_1 = 1$, we have 
$$ \ell(\mu_2) - N = - d(\mu_2) + 1 +
(\mu_2')_{d(\mu_2)} -N.$$

\noindent{\bf{Combining all terms} }
In all cases, we have 
\[
(\mu_1)_{d(\mu_1)} -2N -d(\mu_1) + 2 - d(\mu_2) + (\mu_2')_{d(\mu_2)}. \]
By Remark~\ref{rem:2n}, we must add $2N-1$ to this sum, so we conclude that $$C-A = (\mu_1)_{d(\mu_1)} -d(\mu_1) + (\mu_2')_{d(\mu_2)} -d(\mu_2) + 1=K.$$

\section{Acknowledgements}

We would like to thank Jim Bryan, Rick Kenyon, Rahul Pandharipande, Richard Thomas, Jim Propp, Karel Faber, Kurt Johansson, and frankly countless other geometers, probabilists and combinatorialists for helpful conversations. 

Helen Jenne has received funding from the European Research Council (ERC) through the European Union's Horizon 2020 research and innovation programme under the Grant Agreement No 759702. 
Gautam Webb was partially supported by the National Science Foundation, DMS-2039316. 
Benjamin Young was partially supported by the Knut and Alice Wallenberg Foundation Grant KAW:2010.0063. This work was supported by a grant from the Simons Foundation (637746, BY). 

%% if you use biblatex then this generates the bibliography
%% if you use some other method then remove this and do it your own way
\bibliographystyle{plain}
\bibliography{ptdt}

\end{document}